%% file: main.tex
\documentclass[11pt]{article}
\pdfoutput=1 
\usepackage[top=1.3in,bottom=1.2in,left=1in,right=1in]{geometry}
\usepackage{amsmath}
\usepackage{amssymb}
\usepackage{amsthm}
\usepackage{mathrsfs}
\usepackage[english]{babel}
\usepackage{times}
\usepackage{graphicx}
\usepackage[font=small,figurewithin=section]{caption}
\usepackage{subcaption}
\usepackage{float}
\usepackage[authoryear]{natbib}
\usepackage{hyperref}
\usepackage[all]{xy}
\usepackage{multirow}
\usepackage{array}
\usepackage{enumitem}
\usepackage{tikz}
\usepackage{color}

\numberwithin{equation}{section}
\makeatletter
\g@addto@macro\bfseries{\boldmath}
\makeatother

\bibpunct{(}{)}{;}{a}{,}{,}

\hypersetup{colorlinks=true,citecolor=black,filecolor=black,linkcolor=black,urlcolor=black}
\hypersetup{pdfstartview=FitB,pdfpagemode=UseNone}

\setlist{nolistsep}

\newcolumntype{L}[1]{>{\raggedright\let\newline\\\arraybackslash\hspace{0pt}}m{#1}}
\newcolumntype{C}[1]{>{\centering\let\newline\\\arraybackslash\hspace{0pt}}m{#1}}
\newcolumntype{R}[1]{>{\raggedleft\let\newline\\\arraybackslash\hspace{0pt}}m{#1}}
\newcolumntype{N}{@{}m{0pt}@{}}

\input{preamble}

\begin{document}

\renewcommand*{\thepage}{Title}
\input{title.tex}

\newpage
\pagenumbering{roman}
\input{abstract}

\newpage
\tableofcontents
\newpage
\pagenumbering{arabic}
\input{introduction}
\newpage
\input{notation}
\newpage
\input{segment}
\newpage
\input{quad}
\newpage
\input{triangle}
\newpage
\input{hexa}
\newpage
\input{tet}

\newpage
\input{prism}
\newpage
\input{pyramid}
\newpage
\input{conclusion}

\newpage
\phantomsection
\addcontentsline{toc}{section}{References}
\bibliographystyle{kbib}
\bibliography{main}

\appendix
\newpage
\pagenumbering{arabic}
\renewcommand*{\thepage}{\thesection-\arabic{page}}
\input{genfamilies}

\newpage
\pagenumbering{arabic}
\renewcommand*{\thepage}{\thesection-\arabic{page}}
\input{pyrappendix}

\newpage
\pagenumbering{arabic}
\renewcommand*{\thepage}{\thesection-\arabic{page}}
\input{integration}
\newpage
\pagenumbering{arabic}
\renewcommand*{\thepage}{\thesection-\arabic{page}}
\input{verification}
\newpage
\pagenumbering{arabic}
\renewcommand*{\thepage}{\thesection-\arabic{page}}
\input{tables}

\end{document}

%% file: preamble.tex
\newcommand{\R}{\mathbb{R}}
\newcommand{\e}{\mathrm{E}}
\newcommand{\E}{\mathrm{E}}
\newcommand{\Tri}{\Delta}

\newcommand{\oo}{o}
\newcommand{\T}{\mathsf{T}}
\newcommand{\bcdot}{\,\boldsymbol{\cdot}\,}
\DeclareMathOperator{\spann}{span}

\newtheorem*{definition*}{Definition}
\newtheorem{lemma}{Lemma}
\newtheorem*{remark}{Remark}
\newcommand*\boxednum[1]{\!\scalebox{0.9}{
	\tikz[baseline=(char.base)]{
		\node[shape=rectangle,draw,inner sep=2pt](char){#1};}}\,}
\makeatletter
\newcommand{\dashto}[1][2pt]{
	\settowidth{\@tempdima}{${}\rightarrow{}$}
	\makebox[\@tempdima]{${}\rightarrow{}$}
  \makebox[-\@tempdima]{\hspace{-0.1\@tempdima}\color{white}\rule[0.5ex]{#1}{1pt}}
  \makebox[\@tempdima]{}
	}
\makeatother
\newcommand{\tto}{\!\to\!}
\newcommand{\tdashto}{\!\dashto\!}
\newcommand{\sxi}{\textstyle{\frac{\xi_1}{1-\zeta}}}
\newcommand{\sxii}{\textstyle{\frac{\xi_2}{1-\zeta}}}
\newcommand{\sxia}{\textstyle{\frac{\xi_a}{1-\zeta}}}
\newcommand{\sxib}{\textstyle{\frac{\xi_b}{1-\zeta}}}
\newcommand{\xieta}{\xi,\zeta}
\newcommand{\sxz}{\textstyle{\frac{x}{1+z}}}
\newcommand{\sxxz}{\textstyle{\frac{1-x}{1+z}}}
\newcommand{\syz}{\textstyle{\frac{y}{1+z}}}
\newcommand{\syyz}{\textstyle{\frac{1-y}{1+z}}}
\newcommand{\szz}{\textstyle{\frac{1}{1+z}}}
\newcommand{\szzz}{\textstyle{\frac{z}{1+z}}}
\let\tilde\widetilde


%% file: title.tex
\begin{titlepage}
\thispagestyle{empty}
\begin{center}

{}~\\[0.4cm]
\includegraphics[width=0.45\textwidth]{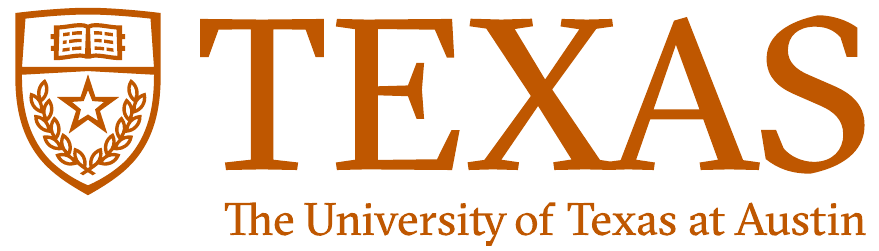}~\\[0.5cm]
\includegraphics[width=0.41\textwidth]{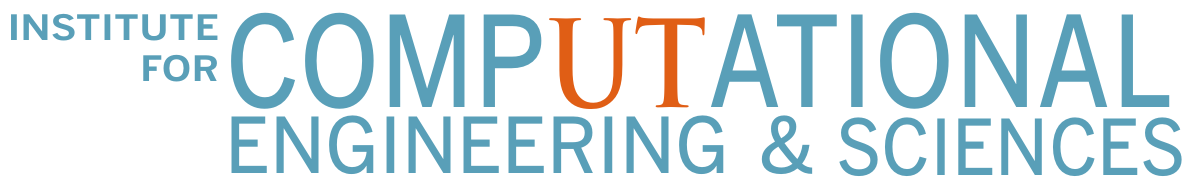}~\\[3cm] 


\linespread{1.8}\selectfont
{ \huge Orientation Embedded High Order Shape Functions for the Exact Sequence Elements of All Shapes } \\[3cm]

\linespread{1.2}\selectfont
{\Large Federico Fuentes}\\[0.1cm]
{\Large Brendan Keith}\\[0.1cm]
{\Large Leszek Demkowicz}\\[0.1cm]
{\Large Sriram Nagaraj}\\[4cm]

{\Large April 2015}\\[1.0cm]
\vfill

\end{center}

\end{titlepage}

%% file: abstract.tex
\renewcommand{\abstractname}{\large Abstract}
\begin{abstract}
\normalsize 
A unified construction of high order shape functions is given for all four classical energy spaces ($H^1$, $H(\mathrm{curl})$, $H(\mathrm{div})$ and $L^2$) and for elements of ``all'' shapes (segment, quadrilateral, triangle, hexahedron, tetrahedron, triangular prism and pyramid). 
The discrete spaces spanned by the shape functions satisfy the commuting exact sequence property for each element.
The shape functions are conforming, hierarchical and compatible with other neighboring elements across shared boundaries so they may be used in hybrid meshes.
Expressions for the shape functions are given in coordinate free format in terms of the relevant affine coordinates of each element shape. 
The polynomial order is allowed to differ for each separate topological entity (vertex, edge, face or interior) in the mesh, so the shape functions can be used to implement local $p$ adaptive finite element methods.
Each topological entity may have its own orientation, and the shape functions can have that orientation embedded by a simple permutation of arguments.
\end{abstract}

%% file: introduction.tex
\section{Preliminaries}
\label{sec:Introduction}

\subsection{Introduction}

In the context of finite elements, construction of higher order shape functions for elements forming the exact sequence has been a long standing activity in both engineering and numerical analysis communities.
A comprehensive review of the subject can be found for example in \citet{hpbook}, \citet{hpbook2} and references therein.

This document presents a self-contained systematic theory for the construction of a particular set of hierarchical, orientation embedded, $H^1$, $H(\mathrm{curl})$, $H(\mathrm{div})$, and $L^2$ conforming shape functions for elements of ``all shapes'', forming the 1D, 2D, and 3D commuting exact sequences discussed within.
By elements of ``all shapes'', we specifically mean the segment (unit interval) in 1D, the quadrilateral and triangle in 2D, and the hexahedron, tetrahedron, prism (wedge) and pyramid in 3D.

\begin{figure}[!ht]
\begin{center}
\includegraphics[scale=0.75]{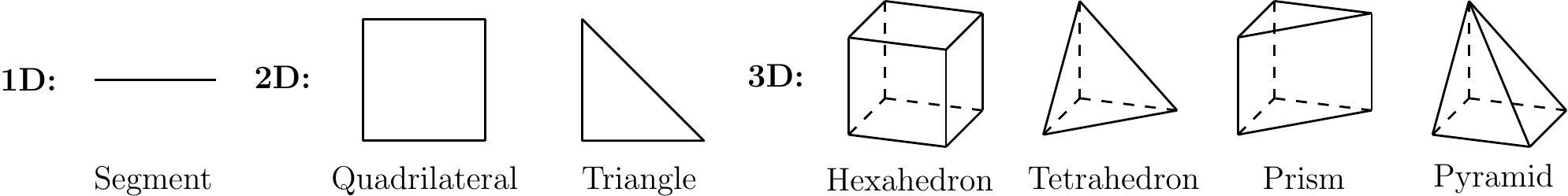}
\caption{Elements of ``all shapes''.}\label{fig:elementsallshapes}
\end{center}
\end{figure}

There are many ways to construct sets of shape functions satisfying the aforementioned properties.
However, we believe that in this work we have constructed a set which strikes an uncommon balance between simplicity and applicability. 
For all elements, and each associated energy space, we rely upon a simple methodology and a very small collection of ancillary functions to generate all of our shape functions.
Furthermore, we have supplemented this text with a package written in Fortran 90 defining each function presented in this work.\footnote{See the ESEAS library available at https://github.com/libESEAS/ESEAS.}
For these reasons, when reproducing our work in their own software, the readers should find the burden of implementation minimal.
We hope that our exposition will be clear and useful, particularly to those less familiar with the subject.



For those at the forefront of shape function construction, we hope that our work will be intriguing if only for the elegance of our construction.
Particularly, we evidence \S\ref{sec:Pyramid} on pyramid shape functions.
The higher order discrete commuting exact sequence for this element appeared only recently in the work of \citet{Nigam_Phillips_11}.
Our construction for the pyramid presents shape functions spanning each of their discrete energy spaces while maintaining compatibility with the other 3D elements.
We also remark that, for any given mesh, our shape functions are fully compatible across adjacent interelement boundaries due to considering so-called orientation embeddings.
Hence, no alterations of the shape functions are necessary at the finite element assembly procedure.
Moreover, these orientation embeddings are handled almost effortlessly by simply permuting the entries of a few relevant functions.

With regard to the choice of geometry (shape and size) of master elements, we followed \citet{hpbook}.
However, one of the key points is that our construction naturally applies to any other choice of master element geometries.
Other specific choices we made when enumerating element vertices, edges, and faces, can also be modified with little effort to the preferences of the reader.

For completeness, we have chosen to thoroughly verify the mathematical properties and to give a sound geometrical interpretation of our constructions rather than only give the necessary shape functions to the reader.
We concede that due to the depth of our presentation, and the expanse of our coverage, our offense lies only in the length of this document.
However, in a sense, an abridged version of this work is already present in a set of tables summarizing all the shape functions.
These can be conveniently consulted in Appendix~\ref{app:ShapeFunctionTable}. 



\subsection{Energy Spaces and Exact Sequences}
\label{sec:Exactsequences}
Let $\Omega\subseteq\mathbb{R}^N$, with $N=1,2,3$ be a domain.
One arrives naturally at the \textit{energy spaces} $H^1(\Omega)$, $H(\text{curl},\Omega)$, $H(\text{div},\Omega)$ and $L^2(\Omega)$ in context of various variational formulations, see e.g. Chapter 1 in \citet{hpbook2}, \citet{DemkowiczGopalakrishnan13_2} and \citet{DemkowiczClosedRange}.
Along with operations of gradient, curl and divergence (understood in the sense of distributions), these spaces form the so-called \textit{complexes}, i.e. the composition of any two operators in the sequence reduces to the trivial operator.
The 1D complex, where $\Omega\subseteq\R$, provides the simplest example:
\begin{equation*}
	\mathbb{R} \stackrel{\text{id}}{\longrightarrow} H^1(\Omega) \stackrel{\nabla}{\longrightarrow} L^2(\Omega)
	\stackrel{0}{\longrightarrow} \{ 0 \}\,.
	\label{eq:1D_exact_sequence}
\end{equation*}
Here, the symbol $\mathbb{R}$ denotes constant functions, and $\{0\}$ is the trivial vector space consisting of the zero function only. By using the name of \textit{complex}, we communicate two simple facts: a) the derivative of a constant function is zero, and b) the composition of derivative (in fact, any linear operator) with the trivial (zero) operator is trivial as well. Equivalently, we can express the same facts by using null spaces and ranges of the involved operators:
\begin{equation*}
	\mathsf{R}(\text{id}) \subseteq \mathsf{N}(\nabla) \quad \text{and} \quad
	\mathsf{R}(\nabla) \subseteq \mathsf{N}(0) \, .
\end{equation*}
If instead of inclusions above, we have equalities, then we say that the complex (sequence) is \textit{exact}. This is indeed the case for the simply connected domain $\Omega = (0,1)$. 
By using the name \textit{exact sequence}, we communicate more information: a) the derivative of a function is zero \textit{if and only if} the function is a constant, and b) the function $\nabla : H^1(\Omega) \to L^2(\Omega)$ is a surjection (onto). From now on, we remove mention of the first and final terms of the exact sequence. The two spaces $\mathbb{R}$ and $\{0\}$, and the operators $\text{id}$ and $0$, are always assumed to buttress each of the sequences we later present. Moreover, whenever possible, we absorb the $\Omega$ assignment within the notation of each energy space. The domain $\Omega$ will always be assumed to be a simply connected domain in the relevant $\R^N$.

\paragraph{1D Exact Sequence.} We now present the first exact sequence of simply connected domains in $\R$:
\begin{equation}
H^1 \stackrel{\nabla}{\longrightarrow} L^2\, .
\end{equation}

\paragraph{2D Exact Sequence.} The exact sequence for simply connected domains in $\R^2$ is of the form
\begin{equation}
H^1 \xrightarrow{\,\,\nabla\,\,} H(\mathrm{curl}) \xrightarrow{\nabla\times} L^2 \,,
\label{eq:2DExactSeq}
\end{equation}
where $\nabla\times$ and $\times$ are understood in two dimensions:
\begin{equation}
    \nabla\times E=\nabla\times\begin{pmatrix}E_1\\E_2\end{pmatrix}
        =\frac{\partial E_2}{\partial \xi_1}-\frac{\partial E_1}{\partial \xi_2}\,,\qquad\quad
    E\times F=\begin{pmatrix}E_1\\E_2\end{pmatrix}\times\begin{pmatrix}F_1\\F_2\end{pmatrix}
        =E_1 F_2 - E_2 F_1\,.\label{eq:2Dcurlandcross}
\end{equation}
By ``rotating'' $H(\mathrm{curl})$, the space $H(\mathrm{div})$ arises naturally:
\begin{equation}
    H(\mathrm{div})=\Big\{V_E=\Big(\begin{smallmatrix}0&1\\[2pt]-1&0\end{smallmatrix}\Big)E=
        \Big(\begin{smallmatrix}E_2\\-E_1\end{smallmatrix}\Big):
            E=\Big(\begin{smallmatrix}E_1\\E_2\end{smallmatrix}\Big)\in H(\mathrm{curl})\Big\}\,.
\label{eq:Hdiv2Ddef}
\end{equation}
Defined in this way, the ``rotated'' exact sequence is immediately satisfied:
\begin{equation}
H^1 \xrightarrow{\mathrm{curl}} H(\mathrm{div}) \xrightarrow{\,\nabla\cdot\,} L^2 \,,
\label{eq:2DExactSeqRotated}
\end{equation}
where, for all $\phi\in H^1$ and all $E\in H(\mathrm{curl})$, the operations satisfy the following relations: 
\begin{equation}
    \mathrm{curl}\,(\phi)=\begin{pmatrix}\frac{\partial\phi}{\partial \xi_2}\\[4pt]-\frac{\partial\phi}{\partial \xi_1}\end{pmatrix}
        =\begin{pmatrix}0&1\\[4pt]-1&0\end{pmatrix}\nabla\phi\,,\qquad\quad
            \nabla\cdot V_E=\nabla\cdot\begin{pmatrix}0&1\\[4pt]-1&0\end{pmatrix}E=\nabla\times E\,.
\end{equation}

\paragraph{3D Exact Sequence.} Finally, for a simply connected domain in $\mathbb{R}^3$, we have the 3D exact sequence
\begin{equation}
H^1 \xrightarrow{\,\,\nabla\,\,} H(\mathrm{curl}) \xrightarrow{\nabla\times} H(\mathrm{div}) \xrightarrow{\,\nabla\cdot\,} L^2 \, .
\label{eq:3D_exact_sequence}
\end{equation}

For all elements, these exact sequences will be reproduced on the discrete level by replacing the energy spaces with appropriate polynomial subspaces.\footnote{Or rational polynomial subspaces in the case of the pyramid (see \S\ref{sec:Pyramid}).} We shall use the standard notation:
\begin{equation}
\begin{aligned}
	\text{1D}:&\quad\qquad W^p \xrightarrow{\,\,\nabla\,\,\,}Y^p \, ,\\
	\text{2D}:&\qquad\!\!\Bigg\{\begin{array}{c}
		W^p \xrightarrow{\,\,\nabla\,\,\,} Q^p \xrightarrow{\nabla\times} Y^p  \, ,\\[4pt]
		W^p \xrightarrow{\mathrm{curl}} V^p \xrightarrow{\,\nabla\cdot\,} Y^p  \, ,\end{array}\\
	\text{3D}:&\quad\qquad W^p \xrightarrow{\,\,\nabla\,\,\,} Q^p \xrightarrow{\nabla\times} V^p \xrightarrow{\,\nabla\cdot\,} Y^p\,.
\end{aligned}
\label{eq:polynomial_exact_sequences}
\end{equation}
The symbol $p$ loosely denotes the polynomial order and should not be interpreted literally.\footnote{By this we mean that $p$ should, in fact, be interpreted as a multi-index for the Cartesian product elements.}

In this work, we shall consider only spaces of the \textit{first type} which, with the exception of the pyramid, were introduced by \citet{Nedelec80} in the \textit{first} of his two famous papers.
The pyramid spaces were taken as the first set of spaces proposed by \citet{Nigam_Phillips_11}.
All these spaces satisfy a number of fundamental properties.
First, the spaces of the different elements are said be tracewise \textit{compatible} at the level of spaces.
This allows them to be used in \textit{hybrid meshes}, which may contain elements of all shapes.
Second, for each element, $W^p$ contains polynomials of total order $p$, while $Q^p$, $V^p$ and $Y^p$ contain polynomials of total order $p-1$, meaning that the overall drop in polynomial degree from the first discrete energy space in the exact sequence to the last discrete energy space is one.
Thirdly, for a given element and energy space, the discrete spaces form a nested sequence of spaces as the order increases (for instance, $W^p\subseteq W^{p+1}$ and so on).
This is a necessary condition for the construction of hierarchical sets of shape functions.
Lastly, the spaces form \textit{commuting} exact sequences for each element.
This, coupled with the previous properties, ensures (global) interpolation estimates for \textit{all} of our energy spaces subject to affine transformations of the master element geometries \citep{monk_demkowicz}.

\subsection{Shape Functions}
We shall always identify the discrete spaces first (like $W^p$, $Q^p$, $V^p$ and $Y^p$ in \eqref{eq:polynomial_exact_sequences}), and only afterward introduce the corresponding shape functions that provide bases for those spaces.
This is a good place to remind the reader that there are, in fact, two competing schools of thought when it comes to the theory of shape functions.

The classical definition of \citet{Ciarlet} starts with \textit{degrees of freedom} that are functionals defined on some large subset $\mathcal{X}$ of an energy space $U$ (like $C^\infty\cap H^1\subseteq H^1$). 
The shape functions, which are elements spanning some discrete (finite dimensional) space $X\subseteq\mathcal{X}$ (e.g. $W^p\subseteq C^\infty\cap H^1$), are then defined as the dual basis to the linearly independent (when restricted to $X$) degrees of freedom.
An \textit{interpolation operator} from $\mathcal{X}$ to $X$ is then naturally defined.
In this construction, we must (usually) precompute the shape functions, e.g. in terms of combinations of monomials whose corresponding coefficients are stored.


The competing approach of Szab\'o \citep{SzaboBabuska91} starts with a direct construction of shape functions by following a topological classification (of vertices, edges, faces and element interiors) induced by conformity requirements.
The shape functions are defined in terms of families of polynomials (e.g. Legendre) and their integrals, and are computed using simple recursive formulas. This is the approach taken in this work.
The so-called \textit{projection-based interpolation} \citep{hpbook,hpbook2} defined through local projections over element edges, faces and interior, is introduced \textit{independently} of the construction of shape functions.
Therefore, with no need to precompute coefficients defining the shape functions, following Szab\'o's approach is perhaps more convenient and straightforward.

\subsection{Hierarchy in \texorpdfstring{$p$}{p}}

Given an energy space (like $H^1$) and a conforming discrete space of order $p$ (like $W^p\subseteq H^1$), denote the shape functions forming a basis for the discrete space by $\mathcal{B}^p$ (e.g. $\mathrm{span}(\mathcal{B}^p)=W^p\subseteq H^1$). 
A construction is said to be hierarchical in $p$ if $\mathcal{B}^p\subseteq\mathcal{B}^{p+1}$ for all $p$, so that the set of shape functions spanning a space of a certain order is found in all subsequent enriched spaces of higher order.
This implies that as $p$ increases, all one has to do is to add a few functions to a smaller previously constructed set of shape functions.


In our construction, hierarchy in $p$ will be enforced. 
In a given mesh, it will allow comparison of shape functions between adjacent elements that have different order, so that at least some of the shape functions of the neighboring elements will match.
This is crucial with regard to the notion of local $p$ adaptivity, which some methods employ.
In our work, this flexibility in the variability of the order will be partly reflected by the natural anisotropies present in Cartesian product elements, such as quadrilaterals, hexahedra and prisms, where each independent direction can have a different order.
More information can be consulted in the literature (see \citet{hpbook2} and references therein).

\subsection{Traces and Compatibility}
\label{sec:compatibility}

For a function to be contained in a given energy space it must satisfy some global conformity conditions which depend upon the space.
For instance, functions in $H^1$ are almost globally continuous, but functions in $L^2$ can be much more discontinuous.

Due to these conformity requirements, each energy space has a different definition of \textit{trace} at the boundaries. 
The different traces only make sense on certain parts of the boundary.
For instance, consider a polyhedral element.
\begin{itemize}
	\item The $H^1$ trace is the value of the function itself at the boundary. In 3D, it may take values at element vertices, edges and faces which lie along the boundary.
	\item The $H(\text{curl})$ (tangential) trace is the tangential component of the vector valued function across the boundary. It may take values at edges and faces in the boundary, but not at vertices, since these do not have a concept of tangent. In fact, the $H(\text{curl})$ (tangential) trace is scalar valued across edges (which have 1D tangent spaces) and has two components across faces (which have 2D tangent spaces).
	\item The $H(\text{div})$ (normal) trace is the normal component of the vector valued function across the boundary. In 3D, it can take values at the faces of the boundary, but not at vertices or edges, since they do not have a unique notion of normal. In 2D, edges along the boundary \textit{do} have a notion of normal, so the (normal) edge trace does exist.
	\item There is no notion of trace for $L^2$.
\end{itemize}


At the discrete level, these considerations lead naturally to a classification of shape functions according to topological entities, which in turn depend on the number of spatial dimensions:
\begin{equation*}
	\begin{aligned}
		\text{1D:}&\quad\qquad\text{vertex and edge shape functions,}\\
		\text{2D:}&\quad\qquad\text{vertex, edge, and face shape functions,}\\
		\text{3D:}&\quad\qquad\text{vertex, edge, face and interior shape functions.}
	\end{aligned}
\end{equation*}

For any given mesh and energy space, shape functions of adjacent elements \textit{must} be continuous at the trace level across the shared interelement boundaries.
This is referred to as \textit{compatibility}, and results in the global conformity of the (disjoint union of) shape functions. 
For instance, in 2D, the order $p$ edge shape function of a quadrilateral would need to be compatible with the order $p$ edge shape function of an adjacent triangle (see Figure \ref{fig:compatibilitydefinition}).

\begin{figure}[!ht]
\begin{center}
\includegraphics[scale=0.7]{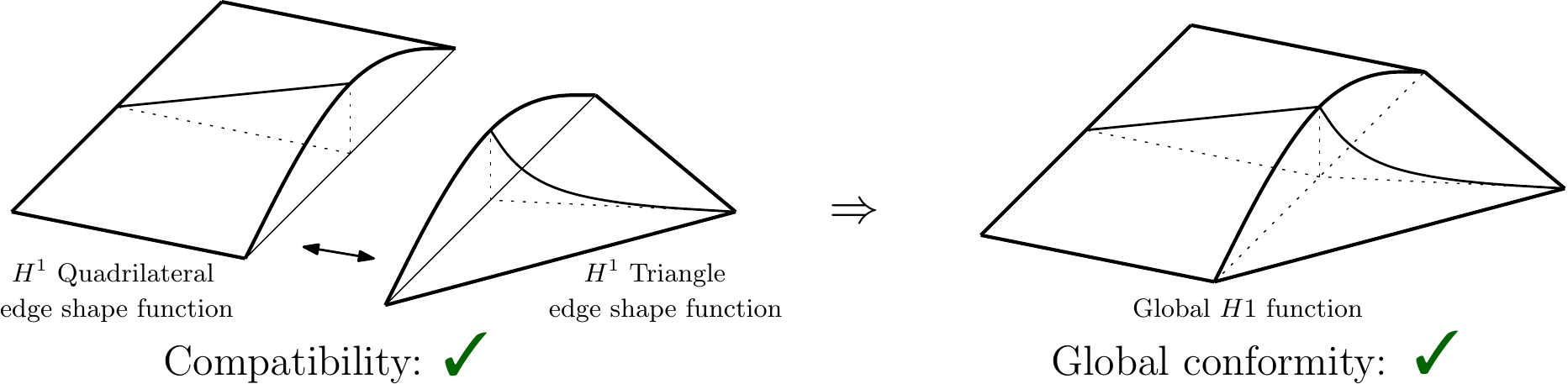}
\caption{Example of compatible $H^1$ edge functions resulting in a globally conforming $H^1$ function.}
\label{fig:compatibilitydefinition}
\end{center}
\end{figure}


\subsection{Embedded Sequences and Dimensional Hierarchy}
\label{sec:dimensionalhierarchy}

To enforce compatibility, it is useful to actually begin with a known trace over the (shared) boundary and then extend (or lift) it to the rest of the element.
This is the approach inherently present in our constructions.
In fact, this idea is reinforced when looking at the exact sequences.
Note the crucial fact that the lower dimensional sequences are ``embedded'' in the higher dimensional sequences if one considers the appropriate restrictions.
This is better represented by the following diagram,
\begin{equation}
	\begin{gathered}
  	\xymatrix{
        {\text{3D:}} & H^1 \ar[r]^{\nabla\,\,\,\,} \ar@{.>}[d]^{\mathrm{tr}} & H(\mathrm{curl})
          \ar[r]^{\,\,\,\nabla\times} \ar@{.>}[d]^{\mathrm{tr}} & H(\mathrm{div})
          	\ar[r]^{\,\,\,\nabla\cdot} \ar@{.>}[d]^{\mathrm{tr}} & L^2\\
        {\text{2D:}} & H^1 \ar[r]^{\nabla\,\,\,} \ar@{.>}[d]^{\mathrm{tr}} & H(\mathrm{curl})
          \ar[r]^{\,\,\,\nabla\times} \ar@{.>}[d]^{\mathrm{tr}} & L^2\\
        {\text{1D:}} & H^1 \ar[r]^{\nabla\,\,\,\,} & \,\,\,\,L^2\,,\, }
	\end{gathered}\label{eq:traceexactsequences}
\end{equation}
where the ``mapping arrows'', $\xymatrix{{}\ar@{.>}[r]^{\mathrm{tr}}&{}}$, indicate that the range (of the trace) is actually a larger space.
These arrows are meant to be motivational only.
At the discrete level, we always reproduce the above diagram precisely with the dotted lines being replaced by well defined maps.

Indeed, we will enforce a \textit{dimensional hierarchy} through traces which is consistent with the previous discussion.
The template for this new form of hierarchy is precisely \eqref{eq:traceexactsequences}, but it is satisfied at the level of the shape functions themselves (not only the spaces).
One will begin the construction by first defining the 1D shape functions (over the segment), then defining all the 2D functions (over the quadrilateral and triangle), and finish with the 3D shape functions.
Throughout the construction, the higher dimensional shape functions will have as (nonzero) trace a lower dimensional function,
so that they are actually extensions of these lower dimensional functions.
Indeed, the shape functions are nested through the trace operation at each topological entity lying in the boundary.

For example, given an edge of a 2D element, the nonzero edge traces of the 2D $H^1$ shape functions should reproduce the 1D $H^1$ shape functions.
Hence, some 2D $H^1$ shape functions are said to be extensions of the 1D $H^1$ shape functions.
Similarly, the nonzero edge trace of 2D $H(\mathrm{curl})$ shape functions should be the 1D $L^2$ shape functions (see \eqref{eq:traceexactsequences}).
These relations hold per topological entity as well.
For example, the nonzero edge trace of 2D $H^1$ \textit{vertex} functions should coincide with the 1D $H^1$ \textit{vertex} functions, and the nonzero trace of the 2D $H^1$ \textit{edge} functions should coincide with the 1D $H^1$ \textit{edge} functions (see Figures \ref{fig:2Dvertexcompatibility} and \ref{fig:2Dedgecompatibility} later on).
When enforced, these nice relationships not only aid in the compatibility, but, from the computational standpoint, have the benefit of allowing us to recycle a large amount of code when moving from one element construction to another.

\subsection{Basic Properties of Shape Functions}

In this section we will describe the basic properties that our shape functions should satisfy.
These properties are specific to the topological entity (vertices, edges, faces, interior) and the energy space to which the shape functions are associated.
For example, vertex functions satisfy different properties than edge functions in $H^1$, and 2D edge functions satisfy different properties in $H^1$ and $H(\mathrm{curl})$.

For a given element and energy space, each topological entity owns a set of shape functions, which is said to be \textit{associated} to the entity.
These functions are nonzero at the associated topological entity.
Indeed, each vertex is associated to \textit{one} $H^1$ vertex shape function.
Meanwhile, each edge, face, and interior of the element is associated to a \textit{set} of edge, face, and interior shape functions of size of the order of $p$, $p^2$ and $p^3$ respectively.
For example, in $H^1$, each edge is associated to a set of $p-1$ edge shape functions.

Now, for a given dimension, space, and topological entity, we will cover three main aspects of the shape functions.
First, are the vanishing properties that they should satisfy.
These properties establish a form of trivial compatibility along some parts of the boundary.
Functions whose trace vanishes everywhere along the boundary are called \textit{bubbles}, and they are trivially compatible with each other.
Second, come the nonzero trace properties.
These are, in general, nontrivial, and ensure either full compatibility or compatibility modulo ``orientations'' (see \S\ref{sec:introorientations} for a discussion on ``orientations'').
Fortunately, dimensional hierarchy will determine the form of these nonzero traces.
Third, are some properties that the shape functions should satisfy along the element itself in order to have hierarchy in $p$.


\subsubsection{1D}

In one dimension, classification of shape functions is:
\begin{align*}
  H^1:&\quad\text{vertex and edge functions,}\\
  L^2:&\quad\text{\phantom{vertex and} edge functions.}
\end{align*}
There is only one simply connected 1D element, which is the segment (or edge), and its boundary is just its two vertices.

\paragraph{$H^1$.}
\textit{Vertex functions:} 
First, there are the vanishing properties.
The vertex functions should vanish at the other (unassociated) vertex.
Second, there are the nonzero trace properties.
The vertex functions should take the value $1$ at the associated vertex.
This will ensure \textit{full} compatibility in 1D.
Third, there is the form of the function itself.
There are multiple ways in which the vertex function can decay towards the other vertex (see Figure \ref{fig:vertexcompatibility}).
If the hierarchy in $p$ is not an issue, the decay could be nonlinear and dependent on $p$, and this may have some computational advantages.
For example, when having a mesh with uniform order $p$ across all elements, this nonlinear decay might lead to a better conditioning of finite element matrices.
However, as mentioned before, we want our shape functions to be useful in $p$ adaptive environments in higher dimensions.
Hence, we enforce hierarchy in $p$, and this restricts our choice to $p=1$, so the decay must be linear.
Indeed, this is the typical and simplest choice.

\begin{figure}[!ht]
\begin{center}
\includegraphics[scale=0.7]{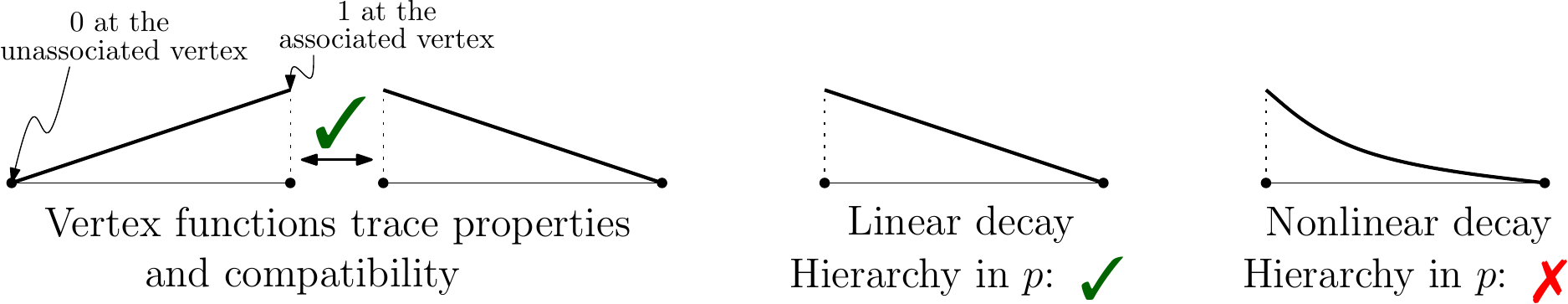}
\caption{Natural vertex function compatibility and the different possible decays.}
\label{fig:vertexcompatibility}
\end{center}
\end{figure}

\textit{Edge functions:} In 1D the $H^1$ edge shape functions are called \textit{edge bubbles} and should vanish at the two endpoints of the edge (or segment).
This makes them automatically compatible in 1D.
When $p=1$ there are no edge bubbles, when $p=2$ there is one $p=2$ edge bubble, when $p=3$ there is the $p=2$ bubble and an extra $p=3$ bubble giving a total of two bubbles, and so on.
This results in a hierarchical construction of the shape functions in $p$.


\begin{figure}[!ht]
\begin{center}
\includegraphics[scale=0.7]{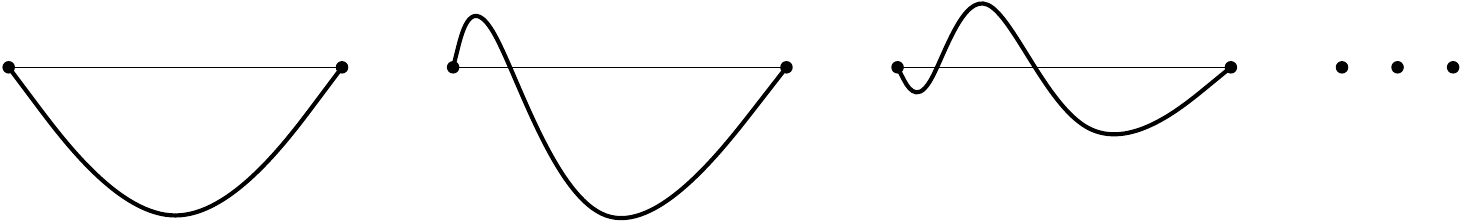}
\caption{Potential set of of 1D $H^1$ edge bubbles vanishing at both endpoints.}
\label{fig:1DH1bubbles}
\end{center}
\end{figure}

\paragraph{$L^2$.}
\textit{Edge functions:}
The 1D $L^2$ edge functions do not need to satisfy any trace properties (neither vanishing nor nonzero) because there is no notion of trace.
They span the space of the (1D) gradients of the $H^1$ conforming shape functions, and should be hierarchical in their construction.

\subsubsection{2D}

In two dimensions, classification of shape functions is:
\begin{align*}
  	H^1:&\quad\text{vertex, edge, and face functions,}\\
  	H(\mathrm{curl}):&\quad\text{\phantom{vertex,} edge, and face functions,}\\
  	L^2:&\quad\text{\phantom{vertex, edge, and} face functions.}
\end{align*}
There are two 2D elements: the quadrilateral and the triangle.
Their boundaries are composed of edges and vertices.

\paragraph{$H^1$.}
\textit{Vertex functions:}
The vertex functions should vanish at the other (unassociated) vertices and disjoint edges.
They should take the value $1$ at the associated vertex.
In 2D, at this point, this does \textit{not} guarantee compatibility.
However, dimensional hierarchy requires the (nonzero) trace of vertex functions over the adjacent edges to be precisely a 1D $H^1$ vertex function (associated to the vertex in question).
This results in full compatibility, and implies the 2D vertex functions are extensions of their 1D analogues.
Regarding the form of the shape functions across the (quadrilateral or triangle) face itself, they can have different forms of decay.
Again, hierarchy in $p$ will restrict our choice, so that the vertex functions we define lie in the lowest order space possible.
Indeed, quadrilateral vertex functions present a bilinear decay, while the decay is linear for triangles.

\begin{figure}[!ht]
\begin{center}
\includegraphics[scale=0.7]{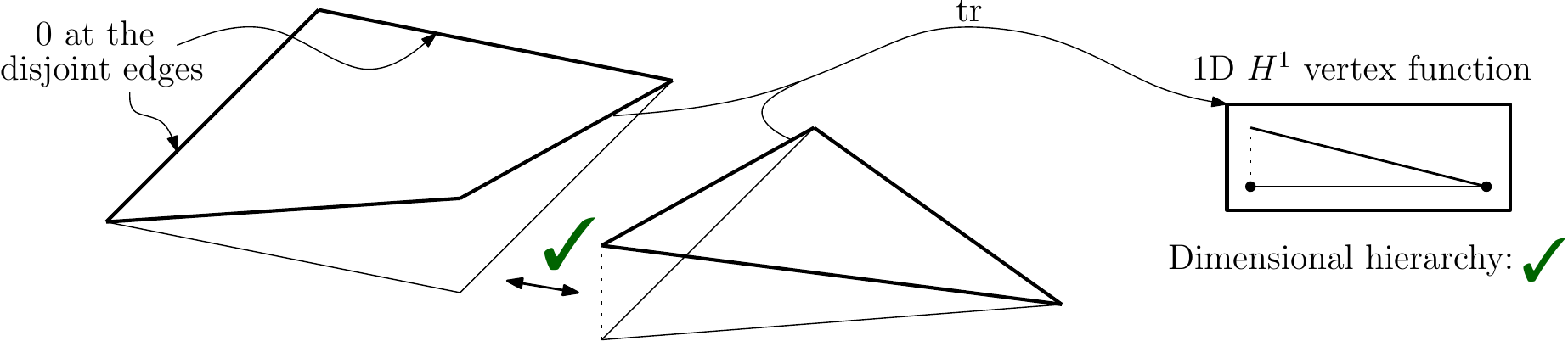}
\caption{Trace properties of 2D $H^1$ vertex functions. Dimensional hierarchy implies full compatibility.}
\label{fig:2Dvertexcompatibility}
\end{center}
\end{figure}

\textit{Edge functions:}
Edge functions should vanish at all other (unassociated) edges of the element.
By dimensional hierarchy, for a given order $p$, the nonzero trace over the (associated) edge itself should take the form of a 1D $H^1$ edge bubble of order $p$.
This gives compatibility modulo edge ``orientations'' in 2D, and it implies the edge functions are extensions of the original 1D $H^1$ edge bubbles.
Now, the functions themselves should present a certain decay from the (associated) edge towards the rest of the element.
Again, this choice is generally restricted by the hierarchy in $p$.
For the quadrilateral element the decay we invoke is linear, but things are more complicated for the triangle element. 

\begin{figure}[!ht]
\begin{center}
\includegraphics[scale=0.7]{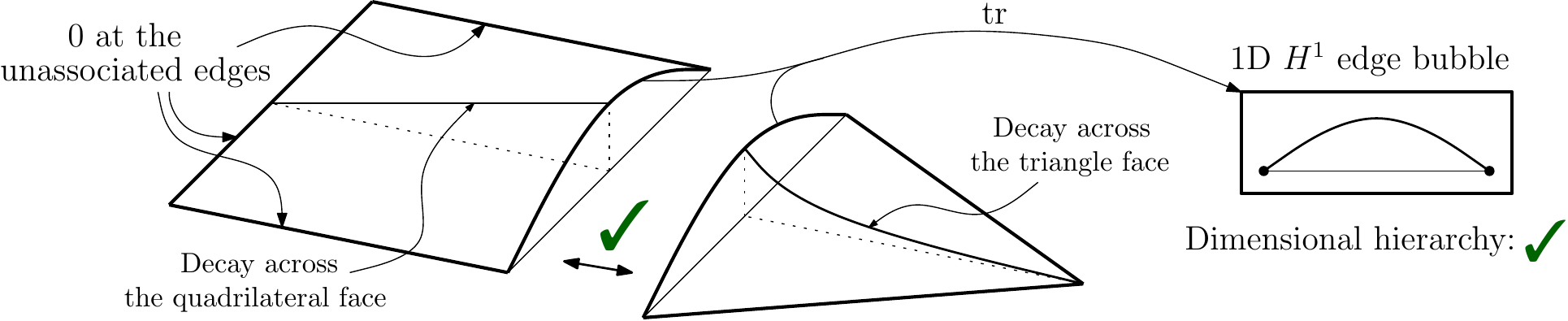}
\caption{Compatibility of edge functions. Depicted as well is the decay across the faces.}
\label{fig:2Dedgecompatibility}
\end{center}
\end{figure}

\textit{Face functions:}
The 2D \textit{face bubbles} vanish at all the edges of the element, so they are automatically compatible in 2D.
They are constructed usually by using the edge functions previously defined (since these already satisfy some vanishing properties) and making some modifications to establish the remaining vanishing conditions.
Again, their construction is hierarchical in $p$.

\paragraph{$H(\mathrm{curl})$.}
\textit{Edge functions:}
The edge functions must have vanishing (tangential) trace at all other edges.
For a given order $p$, the nonzero (tangential) trace over the edge itself should take the form of a 1D $L^2$ edge function of order $p$.
This ensures compatibility modulo edge ``orientations''.
Now, to respect hierarchy in $p$, the edge functions themselves should present a decay (in each of their two components) which is consistent with the lowest order possible decay.
Ultimately, the decay will be similar to that of $H^1$ edge functions.

\textit{Face functions:}
The face bubbles have zero (tangential) trace over all edges, so they are automatically compatible.
They are hierarchically constructed using the same ideas as for their $H^1$ counterparts.

\paragraph{$L^2$.}
\textit{Face functions:}
The $L^2$ face functions do not need to satisfy any trace properties, because there is no notion of trace in $L^2$.
They span the space of (2D) curls of the $H(\mathrm{curl})$ shape functions, and their construction should be hierarchical in $p$.

\subsubsection{3D}

In three dimensions, classification of shape functions is:
\begin{align*}
  H^1:&\quad\text{vertex, edge, face, and interior shape functions,}\\
  H(\mathrm{curl}):&\quad \text{\phantom{vertex,} edge, face, and interior shape functions,}\\
  H(\mathrm{div}): &\quad\text{\phantom{vertex, edge,} face, and interior shape functions,}\\
  L^2:&\quad\text{\phantom{vertex, edge, face, and} interior shape functions.}
\end{align*}
There are four 3D elements: the hexahedron, the tetrahedron, the prism and the pyramid.
Their boundaries are composed of faces, edges, and vertices.

\paragraph{$H^1$.}
\textit{Vertex functions:}
Vertex functions have to vanish at all other (unassociated) vertices and at all disjoint edges and faces.
Dimensional hierarchy requires the (nonzero) trace of vertex functions over the adjacent faces to be precisely a 2D $H^1$ vertex function (associated to the vertex in question).
This ensures full compatibility of the vertex functions in 3D.

\textit{Edge functions:}
Edge shape functions should vanish at all other (unassociated) edges and disjoint faces.
By dimensional hierarchy, for a given order $p$, the nonzero trace over the adjacent faces should be a 2D $H^1$ edge function of order $p$ (associated to the edge in question).
Again, this ensures compatibility modulo edge ``orientations''.

\textit{Face functions:}
Face functions should vanish at all other (unassociated) faces.
Dimensional hierarchy establishes that, for a given order, the nonzero trace over the (associated) face itself should be a 2D $H^1$ face bubble of the same order.
This establishes compatibility modulo face ``orientations''.

\textit{Interior functions:}
The 3D \textit{interior bubbles} vanish at all faces of the element, and are fully compatible in 3D.
Our construction of these functions usually involves making some changes to the face functions previously defined in order to establish the remaining vanishing conditions.

\paragraph{$H(\mathrm{curl})$.}
\textit{Edge functions:}
The edge functions must have vanishing (tangential) trace at all other edges and disjoint faces.
For a given order $p$, the nonzero trace over the adjacent faces should be a 2D $H(\mathrm{curl})$ edge function of order $p$.
This gives compatibility modulo edge ``orientations''.

\textit{Face functions:}
The face functions have vanishing (tangential) trace over all other faces.
Requiring the trace over the face itself to be a 2D $H(\mathrm{curl})$ face bubble of a given order ensures compatibility modulo face ``orientations''.

\textit{Interior functions:}
The interior bubbles have zero (tangential) trace over all the faces and are fully compatible in 3D.
They are constructed using the same ideas as for interior bubbles in $H^1$.

\paragraph{$H(\mathrm{div})$.}
\textit{Face functions:}
The face functions should have vanishing (normal) trace at all other faces.
The trace over the face itself should be a 2D $L^2$ face function of a given order.
Compatibility modulo face ``orientations'' is then established.

\textit{Interior functions:}
The interior bubbles have zero normal component at all the faces, and they are automatically compatible.
They are constructed using the same ideas as in $H^1$.

\paragraph{$L^2$.}
\textit{Interior functions:}
These do not have to satisfy any trace properties and they span the space of the (3D) divergences of the $H(\mathrm{div})$ shape functions.

\subsection{Orientation Embedded Shape Functions}
\label{sec:introorientations}

\begin{figure}[!ht]
\begin{center}
\includegraphics[scale=0.7]{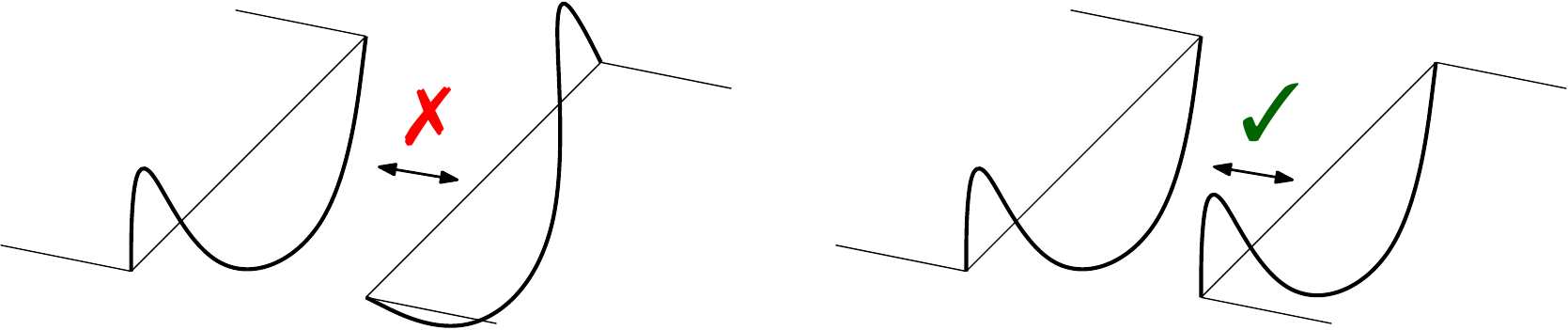}
\caption{Potential edge function ``orientation'' mismatch in 2D due to disregarding the global mesh.}
\label{fig:edgemismatchintro}
\end{center}
\end{figure}

To ensure \textit{full} compatibility of the shape functions along the boundaries, the concept of \textit{orientations} needs to be introduced.
The simplest example occurs in 2D, where edge functions might not match in the \textit{global} mesh due to a simple change in coordinates over the edge itself.
This ``orientation'' mismatch occurs when the edge functions in the (local) master element are transformed to the global mesh.
Indeed, in the standard Szab\'o's approach, master element shape functions are constructed with no regard for global edge or face coordinates, and element shape functions contributing to an edge or face basis function may not coincide with each other along the shared boundary.
This leads to the necessity of an additional action during the finite element assembly process to account for local-to-global orientation changes (change of coordinates).
Usually, these involve \textit{sign factors} in the case of edges and quadrilateral faces, and more complicated adjustments in the case of triangle faces (see the discussion in \citet[p.50]{hpbook2}).
In the context of $h$ adaptive codes involving hanging nodes, the implementation of these modifications in the assembly procedure can become quite involved, and alternative solutions to this problem are therefore desired.

One such solution involves the concept of orientation embedding \citep{GattoDemkowicz10}.
Here, a given topological entity (an edge in 2D, or a face or edge in 3D), regardless of what elements it is adjacent to, is given a \textit{global orientation} at the mesh level, so that it effectively owns a system of coordinates.
In the mesh, this is equivalent to ordering the vertices in a certain order for that given topological entity.
For instance, for an edge with vertices $a$ and $b$, the vertices can be ordered as $a\to b$ or $b\to a$. 
This choice defines a certain global edge orientation.
This information is then passed to the master element, and the shape functions are defined depending on this new information.
The resulting shape functions are then \textit{automatically} compatible with each other.
Naturally, at the local level, this constitutes an extension to the typical approach by Szab\'o, but at the assembly level, it simplifies the implementation of constrained approximation (hanging nodes) by an order of magnitude.

In view of these observations, in this work we constructed orientation embedded shape functions which take into account the information regarding the ``orientation'' of each relevant topological entity.
For each element, we explain these orientation embeddings only after first presenting a complete construction of the classical (``unoriented'') shape functions.
Hence, the information is conveniently decoupled for ease of consultation.

\subsection{Affine Coordinates}
\label{sec:affinecoordinates}

In this work, we chose to exploit simplex (barycentric) \textit{affine coordinates} to formulate all shape function constructions. 
It is well known that affine coordinates are useful when constructing shape functions for the triangle and tetrahedron, which are simplices.
However, we note that with the exception of the pyramid, all elements are either a simplex or a Cartesian product of simplices. 
Indeed, we use these coordinates for \textit{all} the elements, including the pyramid, where we define affine-related coordinates to complement the construction.

Using affine coordinates has many desirable advantages.
Firstly, they give a solid geometrical intuition to the shape functions.
Secondly, they allow the expressions for the shape functions to be used in many other master element geometries.
Lastly, they play a vital role in the context of orientation embedded shape functions.
Indeed, orientation changes are handled almost effortlessly by simple permutations in the arguments of a few crucial \textit{ancillary functions} (or \textit{operators}).
The arguments of these functions are precisely affine coordinates (or affine-related), and they are permuted in accordance to a simple auxiliary permutation function.
This property might be somewhat intuitive in the case of $H^1$ functions, but what is remarkable is that it also holds for the relevant $H(\mathrm{curl})$ and $H(\mathrm{div})$ functions, where technically speaking, nontrivial pullback maps (sometimes called Piola transforms) are required to make these coordinate changes.
Hence, these pullback maps become superfluous with the aid of ancillary operators having affine coordinate functions as their arguments.

\begin{figure}[!ht]
\begin{center}
\includegraphics[scale=0.58]{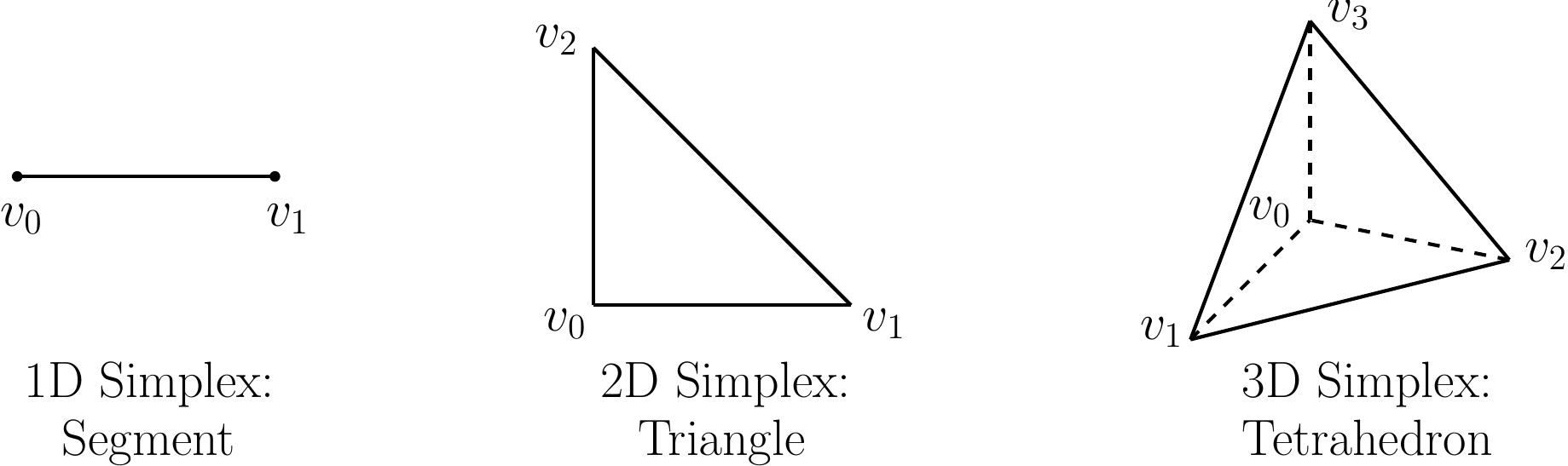}
\caption{Simplices in 1D, 2D and 3D.}
\label{fig:affinesimplices}
\end{center}
\end{figure}

We now define the affine coordinates.
Let $ v_0,\ldots, v_N$, denote the vertices of some simplex, $\Delta$.
Any point $ x\in\Delta$ can be expressed as a convex combination of the vertices:
\begin{equation}
 x = \sum_{a=0}^N s_a v_a\, .\label{eq:affinerepresentation}
\end{equation}
The weights in the sum above, $s_0,\ldots,s_N$, are the affine coordinates for $\Delta$.
We can think of them both as coordinates in and of themselves, or functions of the Cartesian variable $x$. 
Due to being a convex combination, for all $x\in\Delta$ it holds that
\begin{equation}
\sum_{a=0}^N s_a(x)=1\,,\quad\text{ and }\quad s_a(x)\geq0\,.\label{eq:affinesumtoone}
\end{equation}

Throughout this document, to ease the understanding, we shall use the following convention for affine coordinates.
\begin{itemize}
	\item 1D: $\mu_0,\mu_1$ will be affine coordinates for edges ($N=1$, $\mu_a=s_a$ and $a=0,1$).
	\item 2D: $\nu_0,\nu_1,\nu_2$ will be affine coordinates for triangles ($N=2$, $\nu_a=s_a$ and $a=0,1,2$).
	\item 3D: $\lambda_0,\lambda_1,\lambda_2,\lambda_3$ will be affine coordinates for tetrahedra ($N=3$, $\lambda_a=s_a$ and $a=0,1,2,3$).
\end{itemize}
We will often use the following ``vector'' notation for compactness:
\begin{equation}
	\vec{s}_{ab}=(s_a,s_b)\,,\quad\qquad\vec{s}_{abc}=(s_a,s_b,s_c)\,,
\end{equation}
where $s=\mu,\nu,\lambda$. Hence, for example, $\vec{\nu}_{12}=(\nu_1,\nu_2)$ and $\vec{\lambda}_{031}=(\lambda_0,\lambda_3,\lambda_1)$.

Explicit formulas  for the affine coordinates (in terms of Cartesian coordinates) used for each element will be given at the beginning of each corresponding section.


\subsection{Outline}

The document will be organized naturally starting with the simplest element in 1D and then, as dimension and complexity increase, leading into the most complicated elements in 3D.
The order of sections is: segment (\S\ref{sec:Segment}), quadrilateral (\S\ref{sec:Quad}), triangle (\S\ref{sec:Tri}), hexahedron (\S\ref{sec:Hexa}), tetrahedron (\S\ref{sec:Tet}), prism (\S\ref{sec:Prism}) and pyramid (\S\ref{sec:Pyramid}).
Those interested only in simplicial elements may simply read segment, triangle and tetrahedron, while those interested only in quadrilateral and hexadedral elements can also skip the nonrelevant sections.
The prism and pyramid sections are better appreciated after reading through all of the previous sections.
As mentioned before, with regard to orientations, for each element there will always be a final subsection describing the necessary notions and modifications to implement orientation embedded shape functions.
Therefore, as a first iteration in trying to implement the shape functions, or for those readers for which this aspect is not of interest, we suggest skipping those subsections.

As a prelude to all the constructions, there is a section introducing the concept of polynomial scaling and the versions of Legendre and Jacobi polynomials used in our constructions.
More importantly, the concept of \textit{homogenization} is defined.
This is fundamental for the elements involving triangle faces (triangle, tetrahedron, prism, and pyramid).

Finally, as mentioned before, a set of tables available in Appendix \ref{app:ShapeFunctionTable} give a thorough definiton of all ancillary functions and shape functions presented in the text.
These tables should be used as a reference by the reader when looking at the provided code or when implementing their own version.

\subsection{Previous Work}

Our constructions are often based either in part or in full in previous work by various collaborators in the field.
Construction of shape functions for the quadrilateral follows \citet{AinsworthCoyle01} (see also \citet{hpbook2}).

For the triangle and tetrahedron, our construction is based on the concept of \textit{scaled polynomials} as described by \citet{karniadakisbook}, \citet{Schoeberl_Zaglmayr_05}, \citet{SabineThesis} and the subsequent work of \citet{Beuchler_Pillwein_Schoeberl_Zaglmayr_12}.
See also \citet{Beuchler_Schoeberl_06, Beuchler_Pillwein_07, Beuchler_Pillwein_Zaglmayr_12} and \citet{Beuchler_Pillwein_Zaglmayr_13} for more details on obtaining good sparsity properties via appropriate selection of Jacobi polynomials.
Contrary to their work, here we study the classical $H(\text{curl})$ and $H(\text{div})$ conforming N\'{e}d\'{e}lec and Raviart-Thomas spaces having the property that they are affine invariant, and being compatible (at the space level) with the spaces proposed for the pyramid.
Other interesting shape functions for the tetrahedron include those of \citet{AinsworthBezier} based on Bernstein polynomials. Lastly, it is worth noting that \citet{SabineThesis} also presents a unified construction of the hexahedron and prism to complement the tetrahedron, but does not include the pyramid.

The prism element is a Cartesian product of the 2D triangle and 1D segment.
The prism shape functions therefore utilize constructs from the triangle and segment.

Construction of pyramid shape functions builds on the fundamental work of \citet{Nigam_Phillips_11} and their first family of pyramid spaces.
The spaces are natural for (parallelogram-based) affine pyramids, but as evidenced by \citet{Bergot_Durufle_14}, they also have other attractive properties in a non-affine setting.
We also note that \citet{Bergot_Gary_Durufle_10} and \citet{Bergot_Durufle_14} have contributed to the work on higher order pyramid shape functions, but their spaces and shape functions are different.

The idea of orientation embedded shape functions follows the work of \citet{GattoDemkowicz10} and stems from discussions with Joachim Sch\"oberl dating back to the Vienna WCCM congress in 2002.

%% file: notation.tex
\section{Polynomials Prelude}
\label{sec:Notation}

\subsection{Notation}



The polynomials of order $p$ with arguments $x\in\mathbb{R}$ will be denoted by
\begin{equation}
   	\mathcal{P}^p(x)=\mathrm{span}\{x^j:j=0,\ldots,p\}\,.
\end{equation}
Similarly, in two dimensions  the polynomials of total order $p$ with arguments $(x,y)\in\mathbb{R}^2$ are denoted by
\begin{equation}
    \mathcal{P}^p(x,y)=\mathrm{span}\{x^iy^j:i\geq0,j\geq0,n=i+j\leq p\}\,,
\end{equation}
while the \textit{homogeneous} polynomials of total order $p$ are denoted by
\begin{equation}
    \tilde{\mathcal{P}}^p(x,y)=\mathrm{span}\{x^iy^j:i\geq0,j\geq0,i+j=p\}\,.
\end{equation}
Similar definitions apply to polynomials of three variables. 
Moreover, when the domain is clear from the context, we will simply refer to $\mathcal{P}^p(x,y)$ and $\tilde{\mathcal{P}}^p(x,y)$ as $\mathcal{P}^p$ and $\tilde{\mathcal{P}}^p$ respectively.

Define
\begin{equation}
 \mathcal{Q}^{p,q}(x,y)=\mathcal{P}^p(x)\otimes\mathcal{P}^q(y)
        =\spann\{x^i y^j \,:\,0\leq i\leq p,\:0\leq j\leq q\}\,,
\end{equation}
and similarly for $\mathcal{Q}^{p,q,r}(x,y,z)$. When the variables are clear from the context, these spaces are simply written as $\mathcal{Q}^{p,q}$ and $\mathcal{Q}^{p,q,r}$ respectively.

The notation for vector valued polynomial spaces will be
\begin{equation}
	(\mathcal{P}^p)^2 = \mathcal{P}^p\times\mathcal{P}^p\,,
\end{equation}
and similarly for $\left(\mathcal{P}^p\right)^3$, and the vector valued homogeneous polynomials, $(\tilde{\mathcal{P}}^p)^N$, $N=2,3$.

\subsection{Scaled Polynomials}

Given an order $i$ univariate polynomial, $\psi_i(x)\in\mathcal{P}^i(x)$, we define the corresponding {\em scaled polynomial}
\begin{equation}
	\psi_i(x; t) = \psi_i\Big(\frac{x}{t}\Big) t^i \, .
	\label{eq:scaledpolyomials}
\end{equation}
Obviously, $\psi_i(x;1) = \psi_i(x)$, so the scaled polynomials define two variable polynomial extensions into the $(x,t)$ space.
Furthermore, the reader may observe that $\psi_i(x;t)$ is homogenous of order $i$ as a polynomial in this space, i.e. $\psi_i(x;t)\in\tilde{\mathcal{P}}^i(x,t)$.

\subsection{Legendre Polynomials}
\label{sec:LegendrePol}

In this work, we will use Legendre and Jacobi polynomials for the construction of all shape functions. Should the reader wish to work with different families of polynomials, our construction easily generalizes as discussed in Appendix~\ref{app:GeneratingFamilies}.

The classical Legendre polynomials comprise a specific orthogonal basis for $L^2(-1,1)$. Truncated to the first $p$ elements, $\{\tilde{P}_i:i=0,\ldots,p-1\}$, the Legendre polynomials\footnote{Although the common notation for the classical Legendre polynomials is $P_i$, we choose to denote the elements in this set with $\sim$ as we will only need this definition temporarily.} are a basis for the space of (single variable) polynomials of order $p-1$.

Of many properties of the Legendre polynomials, we list the following recursion formula
\begin{equation}
\begin{aligned}
	\tilde{P}_0(y)&=1\,,\\
	\tilde{P}_1(y)&=y\,,\\
	i\tilde{P}_i(y)&=(2i-1)y\tilde{P}_{i-1}(y) - (i-1)\tilde{P}_{i-2}(y)\,, \quad \text{for }\, i\geq2\,,
\end{aligned}
\label{eq:recursion1}
\end{equation}
and the derivative formula for $i\geq1$,
\begin{equation}
(2i+1) \tilde{P}_i(y) = \frac{\partial}{\partial y} \Big(\tilde{P}_{i+1}(y)-\tilde{P}_{i-1}(y)\Big)\,,
\label{eq:Dervivative1}
\end{equation}
which are both well known in the literature. We also make note of the $L^2$ orthogonality relationship
\begin{equation}
\int_{-1}^{1} \tilde{P}_i(y)\tilde{P}_j(y)\,\mathrm{d}y=0\,, \quad \text{if } i\neq j\,.
\label{eq:Orthogonality1}
\end{equation}
This relationship, and the definition $\tilde{P}_0(y) = 1$, leads to the zero average property
\begin{equation}
\int_{-1}^{1} \tilde{P}_i(y)\,\mathrm{d}y=0\,, \quad \text{for } i\geq1\,.
\label{eq:ZeroAverage1}
\end{equation}

\paragraph{Shifting.}
The range of affine coordinates is always $[0,1]$ (see \eqref{eq:affinesumtoone}).
Although not clear at the moment, this implies that we want to have the zero average property over the interval $[0,1]$ instead of $[-1,1]$.
We can obtain this property by composing each Legendre polynomial above with the shifting operation
\begin{equation}
y \mapsto 2x-1\,.
\label{eq:shift}
\end{equation}
The (shifted) Legendre polynomials over $[0,1]$ are defined for $i\geq0$,
\begin{equation}
P_i(x) = \tilde{P}_i\left(2x-1\right)\,.
\end{equation}




\paragraph{Scaling.}
The (shifted) scaled Legendre polynomials are defined by \eqref{eq:scaledpolyomials} as
\begin{equation}
	P_i(x;t) = P_i\left(\frac{x}{t}\right)t^i = \tilde{P}_i\left(2\left(\frac{x}{t}\right)-1\right)t^i = \tilde{P}_i(2x-t;t) \,,
\end{equation}
where the domain in $x$ is $[0,t]$.

This set of scaled Legendre polynomials obeys a recursion formula similar to (\ref{eq:recursion1}):
\begin{equation}
	\begin{aligned}
		P_0(x;t)&=1 \,,\\
		P_1(x;t)&=2x-t \,,\\
		iP_i(x;t)&=(2i-1)( 2x - t) P_{i-1}(x; t) - (i-1)t^2 P_{i-2}(x; t) \,, \quad \text{for }\, i\geq2\, .
	\end{aligned}
\label{eq:recursion2}
\end{equation}
Moreover, we carry over the orthogonality and zero average properties from \eqref{eq:Orthogonality1} and \eqref{eq:ZeroAverage1} to the scaled domain $[0,t]$:
\begin{equation}
	\begin{gathered}
		\int_{0}^{t} P_i(x;t)P_j(x;t)\,\mathrm{d}x=0\,, \quad \text{if } i\neq j\,,\\
		\int_{0}^{t} P_i(x;t)\,\mathrm{d}x=0\,, \quad \text{for } i\geq1\,.
	\end{gathered}
	\label{eq:ZeroAverage2}
\end{equation}

\paragraph{Integrated Legendre Polynomials.}

For all $i\geq1$, we define the (scaled) integrated Legendre polynomials,
\begin{equation}
	L_{i}(x; t) = \int_0^{x} P_{i-1}(\tilde{x}; t)\, \mathrm{d}\tilde{x}\,,
\end{equation}
where of course $L_i(x)=L_i(x;1)$. Notice that $\mathcal{P}^{p}(x)=\mathrm{span}(\{1\}\cup\{L_i:i=1,\ldots,p\})$. By construction, the $L_i$ are seen as elements of $H^1$ and as a result, their pointwise evaluation is understood to be well defined. Therefore, recalling the zero average property of the Legendre polynomials, we observe that,
\begin{equation}
	L_i(0)=L_i(0;t)=0=L_i(t;t)=L_i(1)\,,\quad\text{for }\,i\geq2\,.\label{eq:Lvanishatendpoints}
\end{equation}

Next, \eqref{eq:Dervivative1} motivates the formulas for computing the integrated Legendre polynomials:
\begin{equation}
	\begin{aligned}
		L_1(x;t)&=x\,,\\
		2(2i-1) L_i(x; t)&=P_i(x; t) - t^2 P_{i-2}(x; t) \,, \quad\text{for }\, i\geq2\,.
	\end{aligned}
	\label{eq:shifted_scaled_lobatto}
\end{equation}

Clearly,
\begin{equation}
	\frac{\partial}{\partial x} L_{i}(x; t) = P_{i-1}(x; t)\,.
\end{equation}
Derivatives of $L_i(x;t)$ with respect to $t$ will also be necessary in our computations.
For this, we define
\begin{equation}
	R_i(x) = (i+1)L_{i+1}(x) - xP_i(x)\,,\quad\text{for }\, i\geq0\,,
	\label{eq:RDef}
\end{equation}
which the reader may observe is an order $i$ polynomial. In Appendix~\ref{app:GeneratingFamilies} we show that
\begin{equation}
	\frac{\partial}{\partial t} L_i(x;t) = R_{i-1}(x;t)\,.
\end{equation}
Obviously $R_0(x;t)=0$, and by use of \eqref{eq:recursion2} and \eqref{eq:shifted_scaled_lobatto}, one can reduce \eqref{eq:RDef} to
\begin{equation}
	R_i(x;t) = -\frac{1}{2} \Big(P_i(x;t)+tP_{i-1}(x;t)\Big)\,,\quad\text{for }\,i\geq1\,.
\end{equation}

\subsection{Jacobi Polynomials}

Motivated by \citet{Beuchler_Schoeberl_06} and \citet{Beuchler_Pillwein_07}, we use Jacobi polynomials in our constructions of elements involving triangle faces.
The (shifted to $[0,1]$) Jacobi polynomials, $P^{(\alpha,\beta)}_i(x)$, $\alpha,\beta>-1$, form a two parameter family of polynomials including the Legendre polynomials previously defined ($P_i^{(0,0)} = P_i$). 
Jacobi polynomials have similar recursion formulas as the Legendre polynomials.
One may find a selection of such formulas in \citet{Beuchler_Pillwein_07}.
For our purposes, we will only consider the case $\beta=0$, so that from now on $P_i^\alpha=P_i^{(\alpha,0)}$.

Jacobi polynomials are also orthogonal in a weighted $L^2$ space. Assuming the scaling operation discussed previously, we have the orthogonality relation
\begin{equation}
	\int_{0}^{t} x^\alpha P_i^\alpha(x;t)P_j^\alpha(x;t)\,\mathrm{d}x=0\,, \quad \text{if } i\neq j\,,
\end{equation}
which for $\alpha\neq0$ no longer implies the zero average property.

The following is the recursion formula we use to compute the $[0,t]$ Jacobi polynomials:
\begin{equation}
	\begin{aligned}
		P_0^\alpha(x;t) &= 1\,, \\
		P_1^\alpha(x;t) &= 2x-t+\alpha x\,,\\
		a_i {P}_i^\alpha(x;t) &= b_i \left( c_i(2x-t) + \alpha^2 t \right) P_{i-1}^\alpha(x;t) - d_i t^2 P_{i-2}^\alpha(x;t)\,,
		\quad \text{for }\, i\geq2\, ,
	\end{aligned}
\label{eq:RecursionJacobi}
\end{equation}
where
\begin{equation*}
	\begin{aligned}
		a_i &=  2i(i+\alpha)\, (2i + \alpha - 2)\,,\\
		b_i &=  2i + \alpha - 1\,,\\
		c_i &=  (2i+\alpha)\, (2i + \alpha - 2)\,,\\
		d_i &=  2(i+\alpha-1)\, (i-1)\, (2i+\alpha)\,.
	\end{aligned}
\end{equation*}

We remark that other recursive relations in weight and order to compute Jacobi polynomials, such as $(\alpha+i)P_i^\alpha(x;t)=(\alpha+2i)P_i^{\alpha-1}(x;t)+itP_{i-1}^\alpha(x;t)$, were experimentally found to be numerically unstable as compared to fixing a value of $\alpha$ and using \eqref{eq:RecursionJacobi}, so that the latter approach is recommended.


\paragraph{Integrated Jacobi Polynomials.}
Finally, we define the (scaled) integrated Jacobi polynomials for $i\geq1$:
\begin{equation}
	L^{\alpha}_i(x;t) = \int_0^x P^{\alpha}_{i-1} (\tilde{x};t) \, \mathrm{d}\tilde{x} \, ,
\end{equation}
with $L^{\alpha}_i(x)=L^{\alpha}_i(x;1)$. Note that because of the absence of the zero average property, we cannot deduce that $L^\alpha_i(1)=0$, and in general, this does not hold. However, it is obvious that for all $\alpha>-1$,
\begin{equation}
	L^\alpha_i(0)=L^\alpha_i(0;t)=0\,,\quad\quad\text{for }\,i\geq1\,.\label{eq:Lalphavanishzero}
\end{equation}

We evaluate the integrated Jacobi polynomials using the following relations:\footnote{cf. (2.9) in \citet{Beuchler_Pillwein_07}.}
\begin{equation}
	\begin{aligned}
		L^{\alpha}_1(x;t) & = x\,, \\
		L^{\alpha}_i(x;t) & = a_i P^{\alpha}_i(x,t) + b_i t P^{\alpha}_{i-1}(x;t)- c_i t^2 P^{\alpha}_{i-2}(x;t)\,,
			\quad\text{for }\, i\geq2\,,
	\end{aligned}
	\label{eq:IntegratedJacobiFormula}
\end{equation}
where
\begin{equation*}
	\begin{aligned}
		a_i = & \frac{i + \alpha}{(2i + \alpha -1)(2i + \alpha)}\,, \\
		b_i = & \frac{\alpha}{(2i + \alpha -2)(2i + \alpha)}\,, \\
		c_i = & \frac{i-1}{(2i + \alpha -2)(2i + \alpha-1)}\,.
	\end{aligned}
\end{equation*}

As in the case of the integrated Legendre polynomials, we find that
\begin{equation}
\frac{\partial }{\partial x}L^\alpha_{i}(x;t) = P^\alpha_{i-1}(x;t)\,,\qquad\quad\frac{\partial}{\partial t} L^\alpha_{i}(x;t) = R^\alpha_{i-1}(x;t)\,,
\end{equation}
where again,
\begin{equation}
	R^\alpha_i(x) = (i+1)L^\alpha_{i+1}(x) - xP^\alpha_i(x)\,.
\label{eq:RalphaDef}
\end{equation}
Obviously $R^\alpha_0(x;t)=0$, and by use of \eqref{eq:RecursionJacobi} and \eqref{eq:IntegratedJacobiFormula}, one can reduce \eqref{eq:RalphaDef} to\footnote{cf. (2.16) in \citet{Beuchler_Pillwein_07}.}
\begin{equation}
	R^\alpha_i(x;t) = -\frac{i}{2i+\alpha} \Big(P^\alpha_i(x;t)+tP^\alpha_{i-1}(x;t)\Big)\,,\quad\text{for }\,i\geq1\,.
\end{equation}

\subsection{Homogenization}

\begin{definition*}
For an order $i$ polynomial
\begin{equation*}
	\psi_i \in \mathcal{P}^i(s_1,\ldots,s_d)\,,
\end{equation*}
we define the operation of homogenization $($of order $i$$)$ as a linear transformation
\begin{equation}
 [\,\cdot\,]:\mathcal{P}^i(s_1,\ldots,s_d) \longrightarrow \tilde{\mathcal{P}}^i(s_0,s_1,\ldots,s_d)\,,
\end{equation}
where
\begin{equation*}
	[\psi_i](s_0,s_1,\ldots,s_d)=\psi_i\left(\frac{s_1}{s_0+\cdots+s_d},\ldots,\frac{s_d}{s_0+\cdots+s_d}\right)\,(s_0+\cdots+s_d)^i\,.
\end{equation*}
\end{definition*}

Notice that homogenization is a form of scaling, and as such, it is forming an extension of the particular case in which $s_0+s_1\cdots+s_d=1$.
It is not a coincidence that this is precisely the property that affine coordinates satisfy (see \eqref{eq:affinesumtoone} in \S\ref{sec:affinecoordinates}).
Moreover, note that $[\psi_i]$ is always a homogeneous polynomial of degree $i$, so we have the following scaling property for all scalars $\gamma$,
\begin{equation}
	[\psi_i](\gamma s_0,\gamma s_1,\ldots,\gamma s_d) = \gamma^i [\psi_i](s_0,s_1,\ldots,s_d)\,.
	\label{eq:ScalingProperty}
\end{equation}

One will observe that for the particular case $d=1$,
\begin{equation}
	[\psi_i](s_0,s_1) = \psi_i(s_1;s_0+s_1)\, .
	\label{eq:univariate}
\end{equation}
Therefore, we see that
\begin{equation}
	[\psi_i](s_0,s_1) = \psi_i(s_1;1) = \psi_i(s_1) \,,\quad \quad\text{if }\,s_0+s_1=1\,,
	\label{eq:homogfor1Daffine}
\end{equation}
where we remind the reader that the 1D affine coordinates satisfy precisely this property.

Moreover, take the case of the integrated Legendre polynomials and recall property \eqref{eq:Lvanishatendpoints}. It follows that for all $s_0,s_1$,
\begin{equation}
	[L_i](s_0,0)=L_i(0;s_0)=0=L_i(s_1;s_1)=[L_i](0,s_1)\,,\quad\text{for }\,i\geq2\,.\label{eq:Lhomogvanishatendpoints}
\end{equation}

In the $d=2$ case, one can make another useful observation.
Let $\chi_j$ be a one variable polynomial of order $j$, and consider the homogenization of the $i+j$ order (two variable) polynomial
\begin{equation}
	\psi_{ij}(s_0,s_1) = [\psi_i](s_0,s_1)\chi_j(1-s_0-s_1)\,.
\end{equation}
In this case, using \eqref{eq:ScalingProperty}, we find
\begin{equation}
	\begin{aligned}
		{}[\psi_{ij}](s_2,s_0,s_1)&=[\psi_i]\Big(\frac{s_0}{s_0+s_1+s_2},\frac{s_1}{s_0+s_1+s_2}\Big)
				\chi_j\Big(1-\frac{s_0+s_1}{s_0+s_1+s_2}\Big)(s_0+s_1+s_2)^{i+j}\\
			&=\frac{1}{(s_0+s_1+s_2)^{i}}[\psi_i](s_0,s_1)\chi_j\Big(\frac{s_2}{s_0+s_1+s_2}\Big)(s_0+s_1+s_2)^{i+j}\\
			&=[\psi_i](s_0,s_1)[\chi_j](s_0+s_1,s_2)\,.
			\label{eq:homogproduct}
	\end{aligned}
\end{equation}
This inspires the following definition for single variable polynomials, $\psi_i$ and $\chi_j$, of order $i$ and $j$ respectively:
\begin{equation}
	[\psi_i,\chi_j](s_0,s_1,s_2)= [\psi_i](s_0,s_1)[\chi_j](s_0+s_1,s_2)\,,
\end{equation}
where it is clear $[\psi_i,\chi_j]\in\tilde{\mathcal{P}}^{i+j}(s_0,s_1,s_2)$ is a homogeneous polynomial of order $i+j$.

Again, observe that
\begin{equation}
	[\psi_i,\chi_j](s_0,s_1,s_2)=[\psi_i](s_0,s_1)\chi_j(s_2)\,,\quad\quad\text{if }\,s_0+s_1+s_2=1\,,
\end{equation}
where we remind the reader that the 2D affine coordinates satisfy precisely this property.

Finally, using properties \eqref{eq:Lvanishatendpoints} and \eqref{eq:Lalphavanishzero} of the integrated Legendre and Jacobi polynomials, it follows that for all $s_0,s_1,s_2$,
\begin{equation}
	[L_i,L_j^\alpha](s_0,s_1,0)=[L_i,L_j^\alpha](s_0,0,s_2)=[L_i,L_j^\alpha](0,s_1,s_2)=0\,,\quad\text{for }\,i\geq2\,,\,j\geq1\,.
	\label{eq:LiLjvanishing}
\end{equation}

%% file: segment.tex
\section{Segment}
\label{sec:Segment}

\begin{figure}[!ht]
\begin{center}
\includegraphics[scale=0.5]{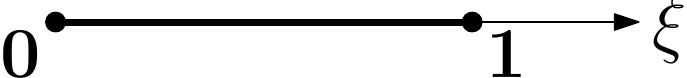}
\caption{Master segment with numbered vertices.}
\label{fig:masterseg}
\end{center}
\end{figure}

The 1D simplex is the segment or edge. 
The master element is defined as the unit interval $(0,1)$, and it is illustrated in Figure \ref{fig:masterseg} with a parameterization given by $\xi\in[0,1]$.

Denote vertex $a$ by $v_a$.
The definition of affine coordinates (see \eqref{eq:affinerepresentation}) states that $\xi=\mu_0(\xi)v_0+\mu_1(\xi)v_1$, with $\mu_0(\xi)+\mu_1(\xi)=1$, $\mu_0(\xi)\geq0$ and $\mu_1(\xi)\geq0$ for all $\xi\in[0,1]$.
Hence, $\mu_0$ is the weight related to $v_0$ and $\mu_1$ is the weight related to $v_1$.
For our master element, $v_0=0$ and $v_1=1$.
It then follows that the 1D affine coordinates for the segment, $\mu_0,\mu_1$, are the most basic linear functions.
They are written explicitly below for our master element:
\begin{equation}
		\mu_0(\xi) = 1 - \xi\,,\qquad\qquad \mu_1(\xi) = \xi \,.\label{eq:H1_1DAffine}
\end{equation}
The gradients of the affine coordinates (in 1D) are
\begin{equation}
		\nabla\mu_0(\xi)=-1\,,\qquad\qquad\nabla\mu_1(\xi)=1\,.\label{eq:gradH1_1DAffine}
\end{equation}

\subsubsection*{Exact Sequence}
The 1D exact sequence can be consulted in \S\ref{sec:Exactsequences}. 
Polynomial spaces are subsets of both $H^1$ and $L^2$ (in $(0,1)$), so we may consider a truncated polynomial approximation (of order $p$) to $H^1$ and induce a new (discrete) polynomial exact sequence
\begin{equation}
	\mathcal{P}^p \xrightarrow{\nabla} \mathcal{P}^{p-1}\,,\label{eq:ESsegment}
\end{equation}
where $\mathcal{P}^p=\mathcal{P}^p(\xi)$.
In the notation of \eqref{eq:polynomial_exact_sequences}, $W^p=\mathcal{P}^p$ and $Y^p=\mathcal{P}^{p-1}$.

\subsection{\texorpdfstring{$H^1$}{H1} Shape Functions}

The set of all shape functions defined in this section will form a basis for the space $\mathcal{P}^p$ which has dimension $p+1$.
In fact, there will be $2$ vertex shape functions and $p-1$ edge shape functions.
They will all be linearly independent and be contained in $\mathcal{P}^p$, so they will clearly form the desired basis.

\subsubsection{\texorpdfstring{$H^1$}{H1} Vertices}

As previously mentioned, each vertex is linked to an affine coordinate.
For instance, $v_0$ is linked to $\mu_0$.
It is then quite natural to consider the affine coordinate itself as the \textit{associated} vertex shape function to $v_0$:
\begin{equation*}
	\phi^\mathrm{v}(\xi)=\mu_0(\xi)\,.
\end{equation*}
Indeed, it satisfies all the desired trace properties, since it takes the value $1$ at $v_0=0$, and $0$ at $v_1=1$.
Moreover, it decays linearly to the other vertex, so that it lies in $\mathcal{P}^1$, and respects the hierarchy.
Having the vertex function of the form $\mu_0(\xi)^p$ instead, would give a (faster) nonlinear decay, but then the function would be dependent on $p$ and the hierarchy would be broken.

In general, the vertex functions, along with their gradients are,
\begin{equation}
    \phi^\mathrm{v}(\xi) = \mu_a(\xi)\,,\qquad \quad \nabla\phi^\mathrm{v}(\xi)=\nabla\mu_a(\xi)\,,
\end{equation}
for $a=0,1$. Clearly, there are a total of $2$ vertex functions (one associated to each vertex).


\subsubsection{\texorpdfstring{$H^1$}{H1} Edge Bubbles}
Recall from \S\ref{sec:LegendrePol} that the Legendre polynomials are seen as elements of $L^2$ which have the zero average property (over $[0,1]$), so that the integrated Legendre polynomials (of order $2$ and higher) are elements of $H^1$ which vanish at $0$ and $1$ (see \eqref{eq:Lvanishatendpoints}).
These are precisely the desired characteristics for $H^1$ edge bubbles.
Hence, the edge functions are defined as:
\begin{equation*}
    \phi^\mathrm{e}_i(\xi) = L_{i}(\xi),\quad i=2,\ldots,p\,.
\end{equation*}
This formula is perfectly valid and quite simple.
However, along this document, the use of affine coordinates will be enforced as much as possible.
The reasons for this will become clear as we move into higher dimensions.
Indeed, notice that due to $\mu_1(\xi)=\xi$, one can write $L_i(\xi)=L_i(\mu_1(\xi))$.
Moreover, since $\mu_0+\mu_1=1$, by \eqref{eq:homogfor1Daffine} it follows
\begin{equation*}
	L_i(\mu_1)=L_i(\mu_1;1)=[L_i](\mu_0,\mu_1)\,.
\end{equation*}
With this in mind, consider the following more general setting.

\begin{definition*}
Let $s_0$ and $s_1$ be arbitrary functions of some spatial variable in $\R^N$, with $N=1,2,3$. Denote by $p_s$ the order in the coordinate pair $(s_0,s_1)$. Then
\begin{equation}
    \phi^\E_i(s_0,s_1) = [L_i] (s_0,s_1) = L_i(s_1;s_0+s_1)\,,
\end{equation}
for $i=2,\ldots,p_s$. The gradients, understood in $\R^N$, are
\begin{equation}
    \begin{aligned}
    \nabla\phi^\E_i(s_0,s_1)&=[P_{i-1}](s_0,s_1)\nabla s_1 + [R_{i-1}](s_0,s_1)\nabla (s_1+s_0)\\
        &=P_{i-1}(s_1;s_0+s_1)\nabla s_1 + R_{i-1}(s_1;s_0+s_1)\nabla (s_1+s_0)\,.
    \end{aligned}
\end{equation}
\end{definition*}

Clearly, the definition of $\phi_i^\E$, involving homogenization, can be thought of as an extension of our more simple case.
This is the first of the so-called \textit{ancillary functions} which are defined in this work.
It is highlighted as an important definition, because it will be used multiple times throughout the text in more general settings.
Here, the superscript $\E$ stands for \textit{edge}, and one should think of this topological entity when looking at this function.
Also, note its arguments, $s_0,s_1$, are meant to be affine coordinates (or at least affine-related).

Rewriting \eqref{eq:Lhomogvanishatendpoints}, it follows that for any $s_0,s_1$, and all $i\geq2$,
\begin{equation}
	\phi^\E_i(0,s_1)=\phi^\E_i(s_0,0)=0\,.\label{eq:phiEvanishing}
\end{equation}

As observed, when the coordinates are 1D affine coordinates, like in this case, the formulas for $\phi^\E_i$ and its gradient are simplified.
For this, record the next remark.

\begin{remark}
Let $\mu_0=1-\mu_1$, where $\mu_1$ is an arbitrary function of some spatial variable in $\R^N$, $N=1,2,3$, and where $p$ is the order in the coordinates $(\mu_0,\mu_1)$. Then for all $i=2,\ldots,p$,
\begin{equation}
    \phi_i^\E(\mu_0,\mu_1)=L_i(\mu_1)\,,\quad\qquad \nabla\phi_i^\E(\mu_0,\mu_1)=P_{i-1}(\mu_1)\nabla\mu_1\,.
    \label{eq:H11Dspecialcase}
\end{equation}
\end{remark}

From now on, shape functions will be written in terms of ancillary functions and the affine coordinates of the element being analyzed.
Indeed, all that is required to compute $\phi_i^\E$ and $\nabla\phi^\E_i$ are $s_0,s_1$ and $\nabla s_0,\nabla s_1$ (since we already know from \S\ref{sec:LegendrePol} how to compute the scaled versions of $L_i$, $P_i$ and $R_i$).
For the segment, $s_0=\mu_0,s_1=\mu_1$, so this information is in \eqref{eq:H1_1DAffine} and \eqref{eq:gradH1_1DAffine}.

At first, this approach might seem to be overcomplicated given the simplicity of the initial formula (which does not involve scaling).
However, computationally speaking, this motivates coding $\phi^\E_i$ and $\nabla\phi^\E_i$, which will be observed to be fundamental as the document progresses.
If desired, within the $\phi^\E_i$ subroutine, one could decide to separately handle the special situation where $s_0,s_1$ are 1D affine coordinates, in which case the simplification shown in \eqref{eq:H11Dspecialcase} would then hold.
More importantly, when orientations become relevant, they will be handled through permutations of the arguments of $\phi_i^\E$.
Hence, having everything written in terms of $\phi_i^\E$, and in general, in terms of ancillary functions, is highly desirable.

Recalling that $\vec{\mu}_{01}=(\mu_0,\mu_1)$, the shape functions and their gradients are then defined as
\begin{equation}
    \phi^\mathrm{e}_i(\xi) = \phi^\E_i(\vec{\mu}_{01}(\xi))\,,\qquad\quad
    	\nabla\phi^\mathrm{e}_i(\xi) = \nabla\phi^\E_i(\vec{\mu}_{01}(\xi))\,,\label{eq:phiEgeneral}
\end{equation}
for $i=2,\ldots,p$. There are $p-1$ edge bubbles for the segment.

As previously mentioned, it is clear that $\phi^\mathrm{e}_i(0)=\phi^\mathrm{e}_i(1)=0$, by the properties of the integrated Legendre polynomials when $i\geq2$ (see \eqref{eq:Lhomogvanishatendpoints} or \eqref{eq:phiEvanishing}).
Hence the trace properties are satisfied.


\subsection{\texorpdfstring{$L^2$}{L2} Shape Functions}

The collection of $L^2$ conforming shape functions is simple and motivated from the exact sequence. 
In 1D, all $L^2$ functions are realized as gradients of $H^1$ functions. 
To resemble this property at the discrete level, we simply consider the linearly independent derivatives of $\phi^\mathrm{v}$ and $\phi^\mathrm{e}_i$.
Clearly they will be a basis for $\mathcal{P}^{p-1}$.

\subsubsection{\texorpdfstring{$L^2$}{L2} Edges}
The $L^2$ edge shape functions are the Legendre polynomials, which written in terms of affine coordinates are
\begin{equation}
    \psi^\mathrm{e}_i(\xi) = P_i(\mu_1(\xi))= [P_i](\vec{\mu}_{01}(\xi))\nabla\mu_1(\xi)\,,\label{eq:the1DL2edge}
\end{equation}
for $i=0,\ldots,p-1$. There are $p$ such edge functions and they span $\mathcal{P}^{p-1}$. 
The apparently trivial factor $\nabla\mu_1(\xi)$ makes the expression coordinate free, so it takes the same form independent of any (possibly nonlinear) transformations.

\subsection{Orientations}
\label{sec:fulledgeorientations}
In 1D, the trace is simply the two endpoints of each element, and it is clear that shape functions of adjacent elements will have full compatibility at the vertices. 
However, in 2D and 3D, the boundaries involve edges and faces.
Achieving this compatibility is nontrivial.
By dimensional hierarchy, edge functions of elements in higher dimensions will involve (through the trace operation) the 1D edge functions defined in this section (see \S\ref{sec:compatibility}).
In view of this, it is natural to explain \textit{edge} orientations at this time.

\subsubsection{Edge Orientations Explained}
\label{sec:edgeorientations}

\begin{figure}[!ht]
\begin{center}
\includegraphics[scale=0.75]{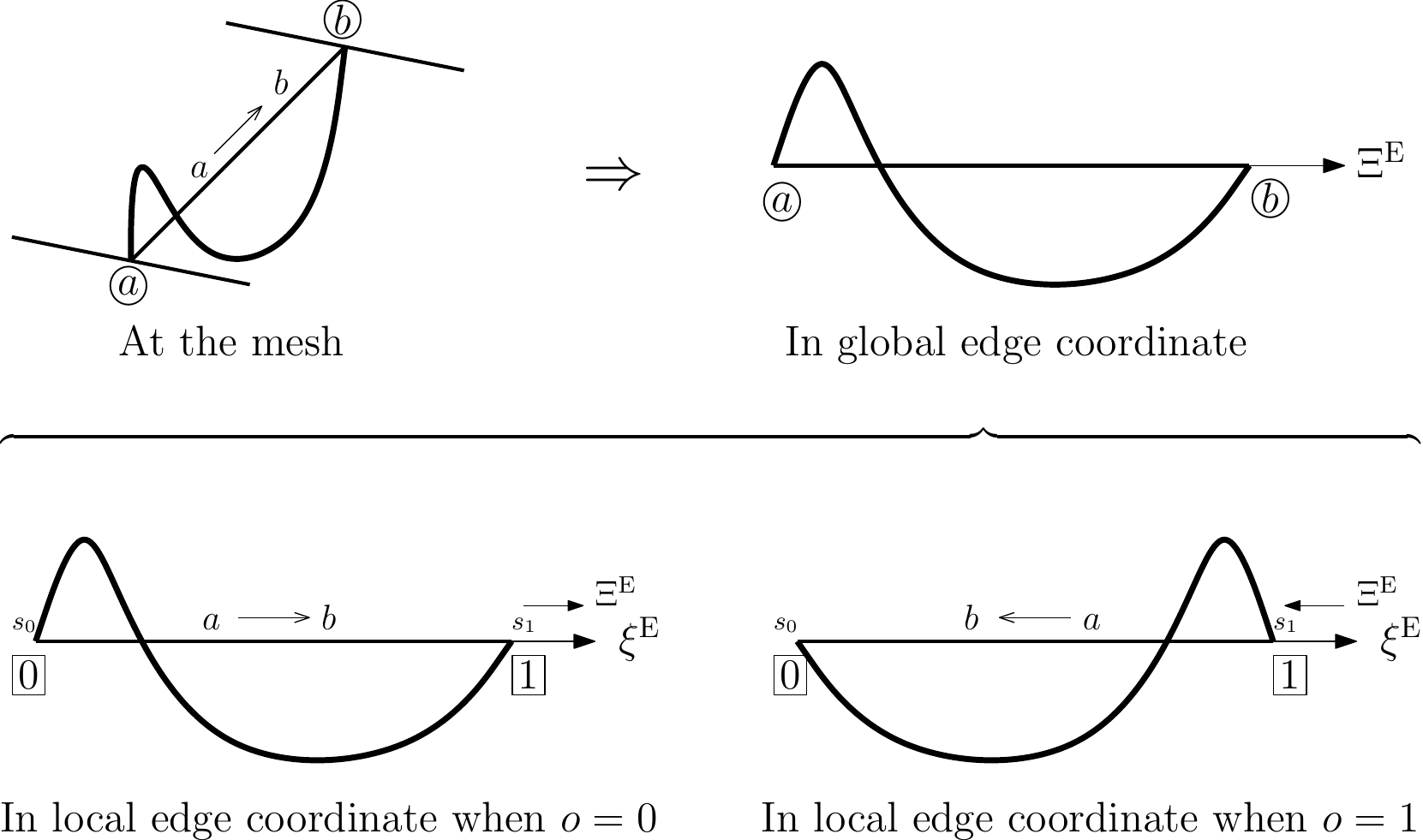}
\caption{Edge orientations.}
\label{fig:orientationsedge}
\end{center}
\end{figure}

In 2D, if one naively disregards how the elements are placed in the global mesh, and proceeds to define all shape functions at the (local) master element level, one might end up with shape functions that, when transformed back into the original mesh, are incompatible across shared edges (see Figure \ref{fig:edgemismatchintro}).
With orientation embedded shape functions this problem is avoided by taking into account more information of the global mesh. 
This is done by giving each mesh edge its own \textit{global orientation}, and is represented by a global coordinate $\Xi^\E$, or equivalently by a \textit{global edge vertex-ordering}.
For example, given an edge at the mesh with vertices $a$ and $b$, a global edge vertex-ordering of the form $a\tto b$ means that $\Xi^\E$ has its origin at $a$ and points from $a$ to $b$.
This information is then passed to the particular master element, where the edge has its own fixed \textit{local orientation}, represented by the local coordinate $\xi^\E$, or equivalently by the the fixed \textit{local} edge vertex-ordering of the form $\boxednum{0}\!\tdashto\boxednum{1}$ (note the dashed arrow for \textit{local} orderings).
Viewed at the local level, the global coordinate $\Xi^\E$ can either coincide with the local coordinate $\xi^\E$ or point in the opposite direction.
To reflect these \textit{two} possibilities, the orientation parameter $\oo$ is introduced.
If $\oo=0$, this means the local and global coordinates coincide, and otherwise $\oo=1$.
All this is depicted in Figure \ref{fig:orientationsedge}. 

Due to the use of affine coordinates and the form of the ancillary functions proposed in this work, these orientation problems can be readily tackled.
To ensure full compatibility, we want the shape functions over a given edge to be immovable when observed in the global coordinates (like in Figure \ref{fig:orientationsedge}).
This is achieved by evaluating the ancillary functions with global coordinates.
Unfortunately, the available coordinates produced by the master element are the (fixed) local coordinates.
Hence, the idea is to apply a \textit{local-to-global transformation} over the edge, which will obviously depend on the orientation parameter $\oo$. 
Such a transformation is completely natural in the context of affine coordinates, since this only involves permutations of these coordinates.
Indeed, a simple permutation function dependent on $\oo$, denoted by $\sigma_\oo^\E$, will represent this transformation.

\begin{definition*}
Let $s_0$ and $s_1$ be arbitrary variables, and let $\oo=0,1$ be the edge orientation parameter. 
The edge orientation permutation function, $\sigma_\oo^\E$, is defined as
\begin{equation}
	\sigma_\oo^\E(s_0,s_1)=\begin{cases}\sigma_0^\E(s_0,s_1)=(s_0,s_1)&\quad\text{if  }\,\oo=0\,,\\
		\sigma_1^\E(s_0,s_1)=(s_1,s_0)&\quad\text{if  }\,\oo=1\,.\end{cases}\label{eq:orientEdge}
\end{equation}
\end{definition*}

To explain the definition of $\sigma_\oo^\E$, note that in \eqref{eq:orientEdge}, if one links $s_0$ to the local vertex $\boxednum{0}$ and $s_1$ to the local vertex $\boxednum{1}$, then the \textit{locally ordered} pair $(s_0,s_1)$ represents the local coordinates.
It is ordered in the sense that $s_0$ comes first and $s_1$ comes second, and this is meant to correspond with the fixed \textit{local} ordering $\boxednum{0}\!\tdashto\boxednum{1}$, where $\boxednum{0}$ comes first and $\boxednum{1}$ comes second.
Similarly, there are \textit{globally ordered} pairs which depend on the parameter $\oo$.
Indeed, in Figure \ref{fig:orientationsedge}, looking at the \textit{global} edge vertex-ordering $a\tto b$, there is an induced \textit{global} vertex-ordering of the two vertices $\boxednum{0}$ and $\boxednum{1}$.
It is $\boxednum{0}\!\tto\boxednum{1}$ if $\oo=0$, and $\boxednum{1}\!\tto\boxednum{0}$ if $\oo=1$.
Hence, the global coordinates are represented by the globally ordered pairs $(s_0,s_1)$ if $\oo=0$ and $(s_1,s_0)$ if $\oo=1$.
Therefore, in this sense, $\sigma_\oo^\E$ is a local-to-global transformation.

Now, all that is required is to compose the \textit{edge} ancillary functions and their differential form (those with superscript $\e$) with this permutation function $\sigma_\oo^\E$.
Thus, in 2D and 3D, all instances of $\phi_i^\E$ and $\nabla\phi_i^\E$ in the shape functions should be replaced with $\phi_i^\E\circ\sigma_\oo^\E$ and $\nabla\phi_i^\E\circ\sigma_\oo^\E$ respectively.
The resulting functions are then said to be \textit{orientation embedded} shape functions.
More concrete examples will be given in the 2D and 3D elements as the document progresses.

%% file: quad.tex
\section{Quadrilateral}
\label{sec:Quad}

\begin{figure}[!ht]
\begin{center}
\includegraphics[scale=0.5]{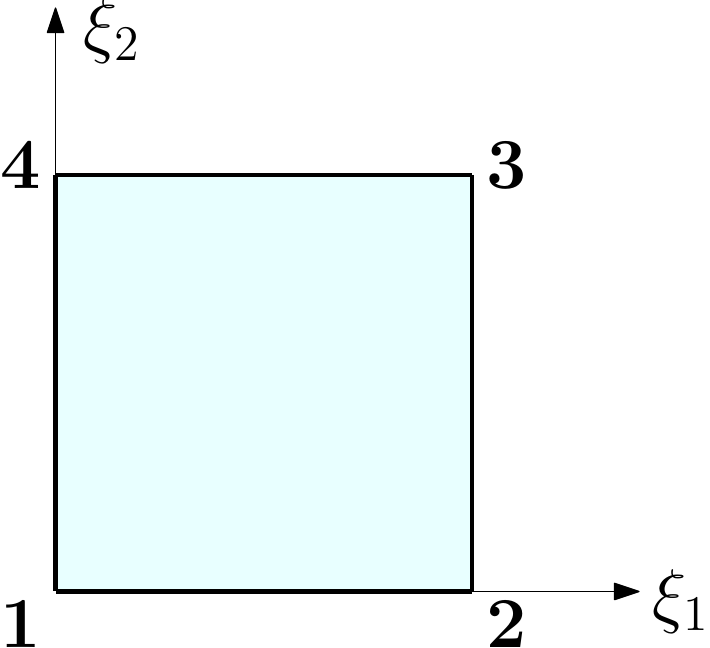}
\caption{Master quadrilateral with numbered vertices.}
\label{fig:MasterQuad}
\end{center}
\end{figure}

The master element for quadrilaterals, which is $(0,1)^2$, is shown in Figure \ref{fig:MasterQuad} in $\xi=(\xi_1,\xi_2)$ space.
The master quadrilateral is clearly a Cartesian product of two segments.

Due to the product structure, there are \textit{two} pairs of 1D affine coordinates:
\begin{equation}
    \begin{alignedat}{4}
        \mu_0(\xi_1)&=1-\xi_1\,,\quad \mu_1(\xi_1)=\xi_1\,\qquad&\Rightarrow\qquad
            \nabla\mu_0(\xi_1)&=\Big(\begin{smallmatrix}-1\\[2pt]0\end{smallmatrix}\Big)\,,\quad
                \nabla\mu_1(\xi_1)=\Big(\begin{smallmatrix}1\\[2pt]0\end{smallmatrix}\Big)\,,\\
        \mu_0(\xi_2)&=1-\xi_2\,,\quad \mu_1(\xi_2)=\xi_2\,\qquad&\Rightarrow\qquad
            \nabla\mu_0(\xi_2)&=\Big(\begin{smallmatrix}0\\[2pt]-1\end{smallmatrix}\Big)\,,\quad
                \nabla\mu_1(\xi_2)=\Big(\begin{smallmatrix}0\\[2pt]1\end{smallmatrix}\Big)\,.
    \end{alignedat}
\end{equation}
These will be used explicitly or implicitly in the formulas that follow.

Again, vertex $a$ is denoted by $v_a$, so that $v_1=(0,0)$, $v_2=(1,0)$, $v_3=(1,1)$ and $v_4=(0,1)$.
These vertices are related to the affine coordinates just as in 1D.
For example, over edge 12 (or edge 43), $\mu_0(\xi_1)$ is the weight related to $v_1$, while $\mu_1(\xi_1)$ is the weight related to $v_2$ (and similarly with $v_4$ and $v_3$).
Indeed, given a point $(\xi_1,0)$ on edge 12, it holds that $(\xi_1,0)=\mu_0(\xi_1)v_1+\mu_1(\xi_1)v_2$.
The same way, $\mu_0(\xi_2)$ is related to $v_1$ in edge 14, so that $v_1$ is linked to both $\mu_0(\xi_1)$ and $\mu_0(\xi_2)$.
A similar assertion holds for each vertex, where each is linked to \textit{two} affine coordinates.
Now, looking at the edges, note that $\mu_0(\xi_2)$ takes the value $1$ over edge 12 and $0$ at opposite edge 43.
This way, each edge is linked to \textit{one} affine coordinate.

\subsubsection*{Exact Sequence}

Recall the 2D exact sequence for simply connected domains \eqref{eq:2DExactSeq} and its rotated analogue \eqref{eq:2DExactSeqRotated}.
The corresponding polynomial exact sequences are
\begin{equation}
	\begin{alignedat}{4}
    &\mathcal{Q}^{p,q} \xrightarrow{\,\,\nabla\,\,} & \mathcal{Q}^{p-1,q}\times\mathcal{Q}^{p,q-1} 
    	\xrightarrow{\nabla\times} &&\mathcal{Q}^{p-1,q-1} \,,\\
    &\mathcal{Q}^{p,q} \xrightarrow{\mathrm{curl}\,} &\mathcal{Q}^{p,q-1}\times\mathcal{Q}^{p-1,q} 
    	\xrightarrow{\,\nabla\cdot\,} &&\mathcal{Q}^{p-1,q-1} \,,
	\end{alignedat}
	\label{eq:QuadES}
\end{equation}
where $\mathcal{Q}^{p,q}=\mathcal{Q}^{p,q}(\xi_1,\xi_2)=\mathcal{P}^p(\xi_1)\otimes\mathcal{P}^q(\xi_2)$. 
These are the standard N\'{e}d\'{e}lec's spaces \citeyearpar{Nedelec80} of the first type for the quadrilateral.
Note here the natural anisotropy of the element, which has order $p$ in the $\xi_1$ direction and a potentially different $q$ in the $\xi_2$ direction. 
The hierarchy should be maintained in both $p$ and $q$ separately. 
This is associated to the notion of local $p$ adaptivity. 
It will sometimes be convenient to refer to $p_a$ as the order in the $\xi_a$ direction, so that $p_1=p$ and $p_2=q$.

\subsection{\texorpdfstring{$H^1$}{H1} Shape Functions}

It will be clear that all shape functions lie in $\mathcal{Q}^{p,q}$ and that they span the space. 
For this, one will only require the linear independence of the shape functions, which will be evident, and a judicious count of them, which will give $(p+1)(q+1)$ (the dimension of $\mathcal{Q}^{p,q}$).

Also, note that due to the Cartesian product structure, there is a natural separation of variables, and one expects the shape functions to be tensor products of the relevant 1D functions for the edges and vertices. 
That is, tensor products of $\mu_a(\xi_b)$ and $\phi_i^\E(\vec{\mu}_{01}(\xi_b))$, for $a=0,1$ and $b=1,2$.
Fortunately, this is the case.

\subsubsection{\texorpdfstring{$H^1$}{H1} Vertices}
\label{sec:H1QuadVertices}

As mentioned before, each vertex is linked with two affine coordinates, and the \textit{associated} vertex function is precisely the tensor product of these two coordinates.
For instance, $v_1$ is linked to $\mu_0(\xi_1)$ and $\mu_0(\xi_2)$, so its associated vertex function is  
\begin{equation*}
	\phi^\mathrm{v}(\xi)=\mu_0(\xi_1)\mu_0(\xi_2)\,.
\end{equation*}
It satisfies all the desired properties, since it vanishes at the disjoint edges 23 and 34, and more importantly, its trace over the adjacent edges is a 1D $H^1$ vertex function associated to the vertex.
For instance, over edge 12, where $\mu_0(\xi_2)=1$, its trace is $\mu_0(\xi_1)$, which is the 1D $H^1$ vertex function associated to $v_1$ over the edge 12.
Finally, the function decays bilinearly and is in the lowest order possible space, $\mathcal{Q}^{1,1}$, so that it respects the hierarchy in both $p$ and $q$.

More generally, the vertex functions and their gradients are,
\begin{equation}
    \phi^\mathrm{v}(\xi)=\mu_a(\xi_1)\mu_b(\xi_2)\,,\quad\qquad
    \nabla\phi^\mathrm{v}(\xi)=\mu_a(\xi_1)\nabla\mu_b(\xi_2)+\mu_b(\xi_2)\nabla\mu_a(\xi_1)\,.\label{eq:H1vertexquad}
\end{equation}
for $a=0,1$ and $b=0,1$.
%
There is a total of $4$ vertex functions (one for each vertex). 

\subsubsection{\texorpdfstring{$H^1$}{H1} Edges}
\label{sec:H1edgesQuad}
To ease the comprehension, take for example edge 12, where $\xi_2=0$.
The idea for the edge functions is to use the segment bubbles $\phi_i^\E$ in $\xi_1$ and \textit{blend} them with a linear function in $\xi_2$.
That is, blend them with the linked 1D affine coordinate, $\mu_0(\xi_2)$.
Hence, the associated edge functions will be the tensor products of $\phi_i^\E(\vec{\mu}_{01}(\xi_1))$ and $\mu_0(\xi_2)$,
\begin{equation*}
    \phi_i^\mathrm{e}(\xi)=\mu_0(\xi_2)\phi_i^\E(\vec{\mu}_{01}(\xi_1))=(1-\xi_2)L_i(\xi_1)\,,
\end{equation*}
with $i=2,\ldots,p$.
The trace properties are satisfied mainly due to the vanishing conditions of the $L_i$ at the endpoints (see \eqref{eq:Lvanishatendpoints}), which are restated in terms of $\phi_i^\E$ in \eqref{eq:phiEvanishing}. 
The vanishing properties are easily observed in the simplified form, $(1-\xi_2)L_i(\xi_1)$, but we will write these traces in terms of the ancillary functions and the affine coordinates. 
For this, it is useful to write the boundary restrictions in terms of affine coordinates.
For example, over edge 12, which has equation $\xi_2=0$, $\vec{\mu}_{01}(\xi_2)=(\mu_0(\xi_2),\mu_1(\xi_2))=(1,0)$. 
Using these natural relationships, the edge traces are
\begin{align*}
    \phi_i^\mathrm{e}(\xi)|_{\xi_2=0}&=\mu_0(\xi_2)\phi_i^\E(\vec{\mu}_{01}(\xi_1))|_{\vec{\mu}_{01}(\xi_2)=(1,0)}
    	=1\cdot\phi_i^\E(\vec{\mu}_{01}(\xi_1))=\phi_i^\E(\vec{\mu}_{01}(\xi_1))\,,\\
    \phi_i^\mathrm{e}(\xi)|_{\xi_1=1}&=\mu_0(\xi_2)\phi_i^\E(\vec{\mu}_{01}(\xi_1))|_{\vec{\mu}_{01}(\xi_1)=(0,1)}
    	=\mu_0(\xi_2)\phi_i^\E(0,1)=0\,,\\
  	\phi_i^\mathrm{e}(\xi)|_{\xi_2=1}&=\mu_0(\xi_2)\phi_i^\E(\vec{\mu}_{01}(\xi_1))|_{\vec{\mu}_{01}(\xi_2)=(0,1)}
    	=0\cdot\phi_i^\E(\vec{\mu}_{01}(\xi_1))=0\,,\\
    \phi_i^\mathrm{e}(\xi)|_{\xi_1=0}&=\mu_0(\xi_2)\phi_i^\E(\vec{\mu}_{01}(\xi_1))|_{\vec{\mu}_{01}(\xi_1)=(1,0)}
    	=\mu_0(\xi_2)\phi_i^\E(1,0)=0\,.
\end{align*}
Hence, the desired vanishing properties are satisfied, and more importantly, over the edge 12 itself, the trace is $\phi_i^\E(\vec{\mu}_{01}(\xi_1))$ which as expected is a 1D $H^1$ edge bubble.
Note here the decay towards the rest of the element, represented by the blending function $\mu_0(\xi_2)$, is linear.
Indeed, the edge functions for edge 12 lie in $\mathcal{Q}^{p,1}$ and they respect the hierarchy in $q$.

Next, we will give a geometrical representation of the edge 12 shape functions presented above.
Recall that over edge 12, $v_1$ is linked to $\mu_0(\xi_1)$, while $v_2$ is linked to $\mu_1(\xi_1)$. 
Hence, $\vec{\mu}_{01}(\xi_1)=(\mu_0(\xi_1),\mu_1(\xi_1))$ actually represents a point in the edge,
\begin{equation*}
	\mu_0(\xi_1)\Big(\begin{smallmatrix}0\\[2pt]0\end{smallmatrix}\Big)
		+\mu_1(\xi_1)\Big(\begin{smallmatrix}1\\[2pt]0\end{smallmatrix}\Big)
			=\Big(\begin{smallmatrix}\xi_1\\[2pt]0\end{smallmatrix}\Big)\,.
\end{equation*} 
This can be interpreted as a projection to edge 12 from an arbitrary point $(\xi_1,\xi_2)$,
\begin{equation*}
    (\xi_1,\xi_2)\;\longmapsto\;(\xi_1,0)\,.
\end{equation*}
The trivial projection consists simply of finding the intersection $P'=(\xi_1,0)$ of the edge with the normal projecting line passing through the original point $P=(\xi_1,\xi_2)$. 
It is better illustrated in Figure \ref{fig:QuadProjection}.

\begin{figure}[!ht]
\begin{center}
\includegraphics[scale=0.55]{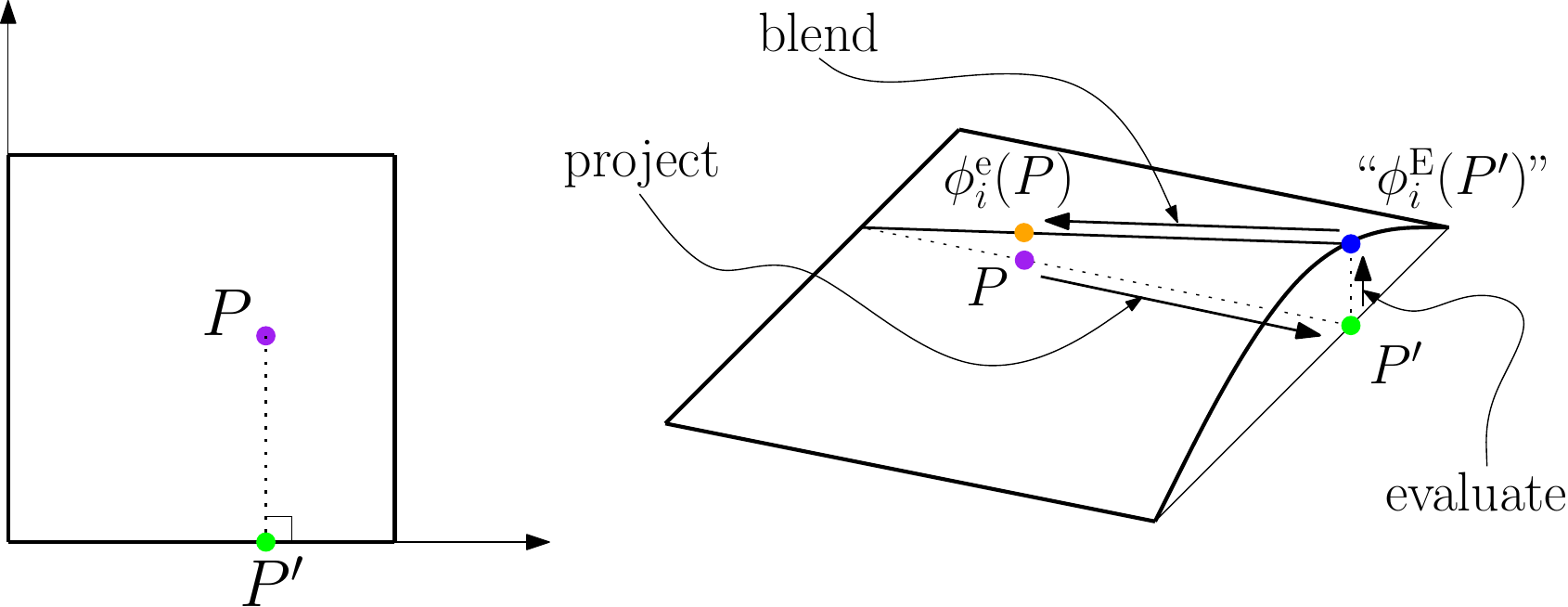}
\caption{Edge projection from $P$ to $P'$, and the logic project$\,\to\,$evaluate$\,\to\,$blend.}
\label{fig:QuadProjection}
\end{center}
\end{figure}

After the original point is projected to the desired edge, it is evaluated at that edge, and finally it is blended linearly:
\begin{equation*}
    \phi_i^\mathrm{e}(\xi)=\underbrace{\mu_0(\xi_2)}_{\text{blend}}
        \underbrace{\phi_i^\E(\underbrace{\vec{\mu}_{01}(\xi_1)}_{\text{project}})}_{\text{evaluate}}
    	=\underbrace{(1-\xi_2)}_{\text{blend}}\underbrace{L_i(\underbrace{\xi_1}_{\text{project}})}_{\text{evaluate}}\,.
\end{equation*}
This blending represents and extension or lifting to the rest of the element.
The whole process of
\begin{equation*}
	\text{projecting}\longrightarrow\text{evaluating}\longrightarrow\text{blending}
\end{equation*}
is extremely important, since it is used in the construction of shape functions for all remaining elements. 
Moreover, as depicted in Figure \ref{fig:QuadProjection}, it gives a geometrical interpretation to the formulas. 
It should be mentioned that if orientations are to be handled, they are taken care of only at the level of evaluating, where a local-to-global transformation will need to be prepended to the original evaluating procedure.
Projecting and blending are unaffected by orientations.

The general formula for edge functions is
\begin{equation}
    \phi_i^\mathrm{e}(\xi)=\mu_c(\xi_b)\phi_i^\E(\vec{\mu}_{01}(\xi_a))\,,\quad\qquad
    	\nabla\phi_i^\mathrm{e}(\xi)=\mu_c(\xi_b)\nabla\phi_i^\E(\vec{\mu}_{01}(\xi_a))
        +\phi_i^\E(\vec{\mu}_{01}(\xi_a))\nabla\mu_c(\xi_b)\,,\label{eq:QuadH1Edge}
\end{equation}
where $i=2,\ldots,p_a$, $(a,b)=(1,2),(2,1)$ and $c=0,1$, with $p_a$ being the order in the $\xi_a$ coordinate. 
For example for edge 12 (linked to $\mu_c(\xi_b)=\mu_0(\xi_2)$), this would correspond to $(a,b)=(1,2)$, $c=0$ and $p_a=p$, for edge 23 (linked to $\mu_1(\xi_1)$) it is $(a,b)=(2,1)$, $c=1$ and $p_a=q$, and so on. 
For each edge there are $p_a-1$ shape functions, leading to a total of $2(p-1)+2(q-1)$ edge shape functions.

\subsubsection{\texorpdfstring{$H^1$}{H1} Face Bubbles}

The quadrilateral face bubbles can be naturally defined as the tensor product of 1D edge bubbles,
\begin{equation*}
    \phi_{ij}^\mathrm{f}(\xi)=\phi_i^\E(\vec{\mu}_{01}(\xi_1))\phi_j^\E(\vec{\mu}_{01}(\xi_2))\,,
\end{equation*}
for $i=2,\ldots,p$ and $j=2,\ldots,q$. 
Using \eqref{eq:phiEvanishing}, it is clear the vanishing conditions over all four edges are satisfied.
This motivates a more general definition.

\begin{definition*}
Let $(s_0,s_1)$ and $(t_0,t_1)$ be two pairs of coordinates which are arbitrary functions of some spatial variable in $\R^N$, $N=2,3$. Let $p_s$ be the order in the $(s_0,s_1)$ coordinates, and $p_t$ be the order in the $(t_0,t_1)$ coordinates. Then
\begin{equation}
    \phi_{ij}^\square(s_0,s_1,t_0,t_1)=\phi_i^\E(s_0,s_1)\phi_j^\E(t_0,t_1)\,,
\end{equation}
for $i=2,\ldots,p_s$ and $j=2,\ldots,p_t$. The gradients, understood in $\R^N$, are
\begin{equation}
	\begin{aligned}
    \nabla\phi_{ij}^\square(s_0,s_1,t_0,t_1)&=\phi_i^\E(s_0,s_1)\nabla\phi_j^\E(t_0,t_1)+\phi_j^\E(t_0,t_1)\nabla\phi_i^\E(s_0,s_1)\,.
	\end{aligned}
\end{equation}
\end{definition*}


Rewritten in terms of $\phi_{ij}^\square$, the general formulas for the bubbles and their gradients are,
\begin{equation}
    \phi_{ij}^\mathrm{f}(\xi)=\phi_{ij}^\square(\vec{\mu}_{01}(\xi_1),\vec{\mu}_{01}(\xi_2))\,,\qquad\quad
        \nabla\phi_{ij}^\mathrm{f}(\xi)=\nabla\phi_{ij}^\square(\vec{\mu}_{01}(\xi_1),\vec{\mu}_{01}(\xi_2))\,,
\end{equation}
where $i=2,\ldots,p$ and $j=2,\ldots,q$. 
There are a total of $(p-1)(q-1)$ such functions. 

\subsection{\texorpdfstring{$H(\mathrm{curl})$}{Hcurl} Shape Functions}
It will be clear that all shape functions will lie in $\mathcal{Q}^{p-1,q}\times\mathcal{Q}^{p,q-1}$. 
Moreover, after all functions are defined, a rigorous count will give $p(q+1)+(p+1)q$, which is precisely the dimension of $\mathcal{Q}^{p-1,q} \times\mathcal{Q}^{p,q-1}$, so that the shape functions span the desired space.

\subsubsection{\texorpdfstring{$H(\mathrm{curl})$}{Hcurl} Edges}
\label{sec:HcurledgesQuad}

First, take for instance edge 12.
As mentioned in \S\ref{sec:dimensionalhierarchy}, the tangential trace of the $H(\mathrm{curl})$ edge functions should be a 1D $L^2$ shape function.
From \eqref{eq:the1DL2edge}, the 1D $L^2$ edge functions (with coordinate $\xi_1$) are $[P_i](\vec{\mu}_{01}(\xi_1))$.
Meanwhile, note the tangential vector to edge 12 is $(1,0)=\nabla\mu_1(\xi_1)$.
When coupled with a blending factor, $\mu_0(\xi_2)$, representing a linear decay (like that of $H^1$), this suggests,
\begin{equation*}
    E_i^\mathrm{e}(\xi)=\mu_0(\xi_2)[P_i](\vec{\mu}_{01}(\xi_1))\nabla\mu_1(\xi_1)
    	=(1-\xi_2)P_i(\xi_1)\Big(\begin{smallmatrix}1\\[2pt]0\end{smallmatrix}\Big)\,,
\end{equation*}
for $i=0,\ldots,p-1$.
Next, the trace properties are checked.
For this, note that $(0,1)=\nabla\mu_1(\xi_2)$ is the tangent direction to the edges 23 and 14 (where $\xi_1=1$ and $\xi_1=0$ respectively). Hence,
\begin{align*}
    \mathrm{tr}(E_i^\mathrm{e}(\xi))|_{\xi_2=0}&=E_i^\mathrm{e}(\xi)|_{\vec{\mu}_{01}(\xi_2)=(1,0)}\cdot(v_2-v_1)
    		=1\cdot[P_i](\vec{\mu}_{01}(\xi_1))\cdot1=[P_i](\vec{\mu}_{01}(\xi_1))\,,\\
    \mathrm{tr}(E_i^\mathrm{e}(\xi))|_{\xi_1=1}&=E_i^\mathrm{e}(\xi)|_{\vec{\mu}_{01}(\xi_1)=(0,1)}\cdot(v_3-v_2)
    		=\mu_0(\xi_2)[P_i](0,1)\cdot0=0\,,\\
  	\mathrm{tr}(E_i^\mathrm{e}(\xi))|_{\xi_2=1}&=E_i^\mathrm{e}(\xi)|_{\vec{\mu}_{01}(\xi_2)=(0,1)}\cdot(v_3-v_4)
    		=0\cdot[P_i](\vec{\mu}_{01}(\xi_1))\cdot1=0\,,\\
    \mathrm{tr}(E_i^\mathrm{e}(\xi))|_{\xi_1=0}&=E_i^\mathrm{e}(\xi)|_{\vec{\mu}_{01}(\xi_1)=(1,0)}\cdot(v_4-v_1)
    		=\mu_0(\xi_2)[P_i](1,0)\cdot0=0\,.
\end{align*}
The trace properties are then satisfied. Inspired by first order Whitney functions, this motivates the following more general definition.


\begin{definition*}
Let $s_0$ and $s_1$ be arbitrary functions of some spatial variable in $\R^N$, with $N=2,3$. Denote by $p_s$ the order in the coordinate pair $(s_0,s_1)$. Then
\begin{equation}
    E^\E_i(s_0,s_1)=[P_i](s_0,s_1)(s_0\nabla s_1-s_1\nabla s_0)\,,\label{eq:Hcurledgefunctions}
\end{equation}
for $i=0,\ldots,p_s-1$, and where the gradients are understood in $\R^N$. The curls are
\begin{equation}
    \nabla\times E_i^\E(s_0,s_1)=(i+2)[P_i](s_0,s_1)\nabla s_0\times\nabla s_1\,.\label{eq:curlsEiE}
\end{equation}
\end{definition*}

Here, if $N=2$, the curl and cross product take the form described in \eqref{eq:2Dcurlandcross}. 
Note $E_i^\E$ involves the gradients of its entries, so that it is actually a differential operator assumed to be acting on a functional space (the entries $s_0,s_1$ are \textit{functions}). 
Hence, the use of the term ancillary \textit{operator} is perhaps more appropriate in this case.
The final expression for the curls is nontrivial (see Lemma \ref{lem:curlformula} below). 
Indeed, it is a very powerful result, since at first it is not evident that there should be no partial derivatives of $P_i$ in \eqref{eq:curlsEiE}. 
Fortunately that is the case. 
In fact, it is a requirement, since the $P_i$ are elements of $L^2$, and their derivatives do not exist in general.
The formula follows from the following lemma coupled with the fact that $[P_i](s_0,s_1)$ is a homogeneous polynomial of total order $i$ in $s_0$ and $s_1$.

\begin{lemma}
\label{lem:curlformula}
Let $\psi_i(s_0,s_1)\in\tilde{\mathcal{P}}^i(s_0,s_1)$ be a homogeneous polynomial of total order $i$ in $s_0$ and $s_1$ , where $s_0$ and $s_1$ are arbitrary functions of some spatial variable in $\R^N$, with $N=2,3$. Then
\begin{equation*}
    \nabla\times\Big(\psi_i(s_0,s_1)(s_0\nabla s_1-s_1\nabla s_0)\Big)=(i+2)\psi_i(s_0,s_1)\nabla s_0\times\nabla s_1\,.
\end{equation*}
\end{lemma}
\begin{proof}
Notice that
\begin{align*}
    \nabla\times(s_0\nabla s_1-s_1\nabla s_0)=\nabla s_0\times\nabla s_1+s_0\nabla\times\nabla s_1
        -\nabla s_1\times\nabla s_0-s_1\nabla\times\nabla s_0=2\nabla s_0\times\nabla s_1\,,
\end{align*}
due to $\nabla\times\nabla s_0=\nabla\times\nabla s_1=0$. Now, consider a monomial $s_0^as_1^b$, so that
\begin{align*}
    \nabla\times\Big(s_0^as_1^b(s_0\nabla s_1-s_1\nabla s_0)\Big)
        &=s_0^as_1^b\nabla\times(s_0\nabla s_1-s_1\nabla s_0)+\nabla(s_0^as_1^b)\times(s_0\nabla s_1-s_1\nabla s_0)\\
        &=2s_0^as_1^b\nabla s_0\times\nabla s_1
            +(as_0^{a-1}s_1^b\nabla s_0+bs_0^{a}s_1^{b-1}\nabla s_1)\times(s_0\nabla s_1-s_1\nabla s_0)\\
        &=2s_0^as_1^b\nabla s_0\times\nabla s_1+as_0^{a-1}s_1^bs_0\nabla s_0\times\nabla s_1
            -bs_0^{a}s_1^{b-1}s_1\nabla s_1\times\nabla s_0\\
        &=(2+a+b)s_0^as_1^b\nabla s_0\times\nabla s_1\,,
\end{align*}
where it is used that $\nabla s_0\times \nabla s_0=\nabla s_1\times \nabla s_1=0$. Then observe that any homogeneous polynomial $\psi_i(s_0,s_1)$ is composed of monomials of the form $s_0^as_1^b$ of fixed total order $a+b=i$. The result follows.
\end{proof}

Next, record the following important remark.

\begin{remark}
Let $\mu_0=1-\mu_1$, where $\mu_1$ is an arbitrary function of some spatial variable in $\R^N$, with $N=2,3$, and where $p$ is the order in the coordinates $(\mu_0,\mu_1)$. Then for all $i=0,\ldots,p-1$,
\begin{equation}
    E_i^\E(\mu_0,\mu_1)=P_i(\mu_1)\nabla\mu_1\,,\quad\qquad \nabla\times E_i^\E(\mu_0,\mu_1)=0\,.\label{eq:Hcurl1Dspecialcase}
\end{equation}
\end{remark}

With this new ancillary function in our toolset, the $H(\mathrm{curl})$ shape functions for edge 12 are written analogously to those in $H^1$, and the same logic of project$\,\to\,$evaluate$\,\to\,$blend applies,
\begin{equation*}
    E_i^\mathrm{e}(\xi)=\underbrace{\mu_0(\xi_2)}_{\text{blend}}
        \underbrace{E_i^\E(\underbrace{\vec{\mu}_{01}(\xi_1)}_{\text{project}})}_{\text{evaluate}}\,.
\end{equation*}

In general, the edge shape functions and their curls are
\begin{equation}
    E_i^\mathrm{e}(\xi)=\mu_c(\xi_b)E_i^\E(\vec{\mu}_{01}(\xi_a))\,,\qquad\quad
    \nabla\times E_i^\mathrm{e}(\xi)=\nabla\mu_c(\xi_b)\times E_i^\E(\vec{\mu}_{01}(\xi_a))\,,
    \label{eq:QuadHcurlEdge}
\end{equation}
where $i=0,\ldots,p_a-1$, $(a,b)=(1,2),(2,1)$ and $c=0,1$. There are $p_a$ functions for each edge, giving a total of $2p+2q$ edge shape functions.

\subsubsection{\texorpdfstring{$H(\mathrm{curl})$}{Hcurl} Face Bubbles}

In general, the idea is to consider the tensor product of the ancillary functions $E_i^\E$ and $\phi_j^\E$, evaluated at the 1D affine coordinate pairs $\vec{\mu}_{01}(\xi_1)$ and $\vec{\mu}_{01}(\xi_2)$. 
To cover all possibilities, one must consider both $\phi_j^\E(\vec{\mu}_{01}(\xi_2))E_i^\E(\vec{\mu}_{01}(\xi_1))$ and $\phi_j^\E(\vec{\mu}_{01}(\xi_1))E_i^\E(\vec{\mu}_{01}(\xi_2))$.
Clearly, both of these cases are actually generated by a single key operator $E_{ij}^\square$, which is defined generally next.

\begin{definition*}
Let $(s_0,s_1)$ and $(t_0,t_1)$ be two pairs of coordinates which are arbitrary functions of some spatial variable in $\R^N$, with $N=2,3$. Let $p_s$ be the order in the $(s_0,s_1)$ coordinates, and $p_t$ be the order in the $(t_0,t_1)$ coordinates. Then
\begin{equation}
    E_{ij}^{\square}(s_0,s_1,t_0,t_1)=\phi_j^\E(t_0,t_1)E_i^\E(s_0,s_1)\,,
\end{equation}
for $i=0,\ldots,p_s-1$ and $j=2,\ldots,p_t$. The curls, understood in $\R^N$, are
\begin{equation}
    \nabla\times E_{ij}^{\square}(s_0,s_1,t_0,t_1)=\phi_j^\E(t_0,t_1)\nabla\times E_i^\E(s_0,s_1)
    	+\nabla\phi_j^\E(t_0,t_1)\times E_i^\E(s_0,s_1)\,.
\end{equation}
\end{definition*}
Using \eqref{eq:phiEvanishing}, and proceeding as with the $H(\mathrm{curl})$ edge functions, it is clear that this ancillary operator, when evaluated at  $(\vec{\mu}_{01}(\xi_1),\vec{\mu}_{01}(\xi_2))$ or  $(\vec{\mu}_{01}(\xi_2),\vec{\mu}_{01}(\xi_1))$, satisfies the necessary vanishing trace at all edges. 
There are two families, which together comprise $p(q-1)+q(p-1)$ bubble functions.

\subparagraph{Family I:}
The shape functions for the first family and their corresponding curls are
\begin{equation}
    E_{ij}^{\mathrm{f}}(\xi)=E_{ij}^{\square}(\vec{\mu}_{01}(\xi_1),\vec{\mu}_{01}(\xi_2))\,,\qquad
    \nabla\times E_{ij}^{\mathrm{f}}(\xi)=\nabla\times E_{ij}^{\square}(\vec{\mu}_{01}(\xi_1),\vec{\mu}_{01}(\xi_2))\,,
\end{equation}
for $i=0,\ldots,p-1$ and $j=2,\ldots,q$.
There are $p(q-1)$ shape functions in this family.

\subparagraph{Family II:}
The shape functions for the second family and their corresponding curls are
\begin{equation}
    E_{ij}^{\mathrm{f}}(\xi)=E_{ij}^{\square}(\vec{\mu}_{01}(\xi_2),\vec{\mu}_{01}(\xi_1))\,,\qquad
    \nabla\times E_{ij}^{\mathrm{f}}(\xi)=\nabla\times E_{ij}^{\square}(\vec{\mu}_{01}(\xi_2),\vec{\mu}_{01}(\xi_1))\,,
\end{equation}
for $i=0,\ldots,q-1$ and $j=2,\ldots,p$. 
Notice the only difference with the first family is that the entries $\vec{\mu}_{01}(\xi_1)$ and $\vec{\mu}_{01}(\xi_2)$ are permuted (along with the associated orders $p$ and $q$).
Hence, there are $q(p-1)$ shape functions in this family.

\subsection{\texorpdfstring{$H(\mathrm{div})$}{Hdiv} Shape Functions}

By the definition of the $H(\mathrm{div})$ space (see \eqref{eq:Hdiv2Ddef}) in two dimensions, it is clear that it is isomorphic to $H(\mathrm{curl})$. 
Indeed, the $H(\mathrm{div})$ shape functions will be the rotation of the corresponding $H(\mathrm{curl})$ shape functions. 
More explicitly, given a shape function $E\in H(\mathrm{curl})$ and its curl $\nabla\times E\in L^2$, the corresponding $H(\mathrm{div})$ shape function with its divergence is
\begin{equation}
    V=\begin{pmatrix}0&1\\-1&0\end{pmatrix}E\,,\quad\qquad\nabla\cdot V=\nabla\times E\,.
\end{equation}
Note in this case the original polynomial space $Q^{p,q}=\mathcal{Q}^{p-1,q}\times\mathcal{Q}^{p,q-1}$ for $H(\mathrm{curl})$ simply becomes the $H(\mathrm{div})$ conforming space $V^{p,q}=\mathcal{Q}^{p,q-1}\times\mathcal{Q}^{p-1,q}$.

\subsection{\texorpdfstring{$L^2$}{L2} Shape Functions}

As expected, they are the tensor products of the 1D $L^2$ shape functions, and there are $pq$ such functions spanning $\mathcal{Q}^{p-1,q-1}$.

\subsubsection{\texorpdfstring{$L^2$}{L2} Face}

The coordinate free $L^2$ shape functions for the quadrilateral faces are
\begin{equation}
    \psi_{ij}^\mathrm{f}(\xi)=P_i(\mu_1(\xi_1))P_j(\mu_1(\xi_2))
    	=[P_i](\vec{\mu}_{01}(\xi_1))[P_j](\vec{\mu}_{01}(\xi_2))(\nabla\mu_1(\xi_1)\!\!\times\!\!\nabla\mu_1(\xi_2))\,,
   \label{eq:QuadL2Functions}
\end{equation}
for $i=0,\ldots,p-1$ and $j=0,\ldots,q-1$. There are $pq$ face functions. 
The factor $\nabla\mu_1(\xi_1)\!\times\!\nabla\mu_1(\xi_2)$ makes the expression coordinate free (with the $\xi$ coordinates it is $1$).

\subsection{Orientations}
\label{sec:QuadOrientations}

For 2D quadrilaterals, only edge orientations need to be considered to ensure compatibility.
However, in 3D elements, quadrilateral \textit{faces} will have the notion of orientation, and one will need to consider this to ensure full compatibility.
In view of this, first we will show how the concept of edge orientations, already explained in \S\ref{sec:edgeorientations}, applies to the quadrilateral element. We assume that section has been covered.
After that, the quadrilateral face orientations are introduced as a preview to the 3D elements.

\subsubsection{Edge Orientations}
\label{sec:QuadEdgeOrientations}

\begin{figure}[!ht]
\begin{center}
\includegraphics[scale=0.5]{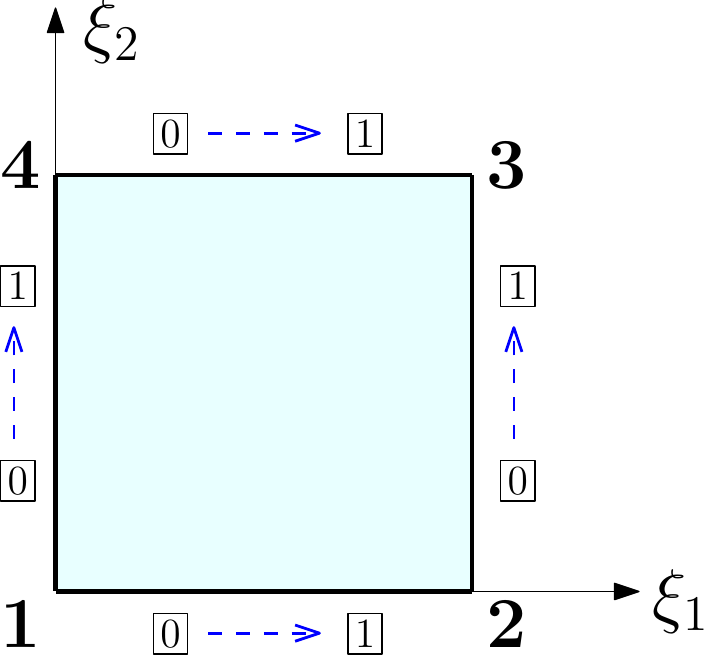}
\caption{Master quadrilateral with local edge orientations.}
\label{fig:MasterQuadOrientations}
\end{center}
\end{figure}

The master quadrilateral has a predefined \textit{local} orientation for each edge, which reperesents the $\oo=0$ case. 
These are illustrated in Figure \ref{fig:MasterQuadOrientations}.
They are \textit{our} choices for the local orientations.

Each edge has a local orientation described by the fixed ordering $\boxednum{0}\!\tdashto\boxednum{1}$.
This induces a master element \textit{local} edge vertex-ordering, which in turn determines a \textit{locally ordered} pair of affine coordinates, since, over a given edge, each master element vertex is linked to one affine coordinate.
For instance, over edge 12, $v_1$ is linked to $\mu_0(\xi_1)$ and $v_2$ linked to $\mu_1(\xi_1)$.
On this edge the induced master element \textit{local} ordering is $v_1\tdashto v_2$ (see Figure \ref{fig:MasterQuadOrientations}), meaning that the locally ordered pair is $\vec{\mu}_{01}(\xi_1)=(\mu_0(\xi_1),\mu_1(\xi_1))$.
The locally ordered pair represents the local coordinates, and it serves as the input of the edge local-to-global transformation $\sigma_\oo^\E$, which transforms them to a globally ordered pair (depending on the parameter $\oo$).
The globally ordered pair is then introduced into the edge ancillary functions of the edge shape functions, and the resulting functions are said to be \textit{orientation embedded} shape functions.
Thus, for example for edge 12 the orientation embedded $H^1$ edge functions are, 
\begin{equation*}
    \phi_i^\mathrm{e}(\xi)=\mu_0(\xi_2)\phi_i^\E\circ\sigma_\oo^\E(\vec{\mu}_{01}(\xi_1))
        =\begin{cases}
            \mu_0(\xi_2)\phi_i^\E\Big(\sigma_0^\E(\vec{\mu}_{01}(\xi_1))\Big)
            	=\mu_0(\xi_2)\phi_i^\E(\mu_0(\xi_1),\mu_1(\xi_1))\,\,\,\text{if }\oo=0\,,\\
            \mu_0(\xi_2)\phi_i^\E\Big(\sigma_1^\E(\vec{\mu}_{01}(\xi_1))\Big)
            	=\mu_0(\xi_2)\phi_i^{\e}(\mu_1(\xi_1),\mu_0(\xi_1))\,\,\,\text{if }\oo=1\,,
        \end{cases}
\end{equation*}
for $i=2,\ldots,p$.
This composition with $\sigma_\oo^\E$ naturally applies to all 2D edge functions in $H^1$ and $H(\mathrm{curl})$ and their differential forms. 
Hence, $\sigma_\oo^\E$ should be composed with $\phi_i^\E$, $\nabla\phi_i^\E$, $E_i^\E$ and $\nabla\times E_i^\E$ in \eqref{eq:QuadH1Edge} and \eqref{eq:QuadHcurlEdge}.

\begin{figure}[!ht]
\begin{center}
\includegraphics[scale=0.75]{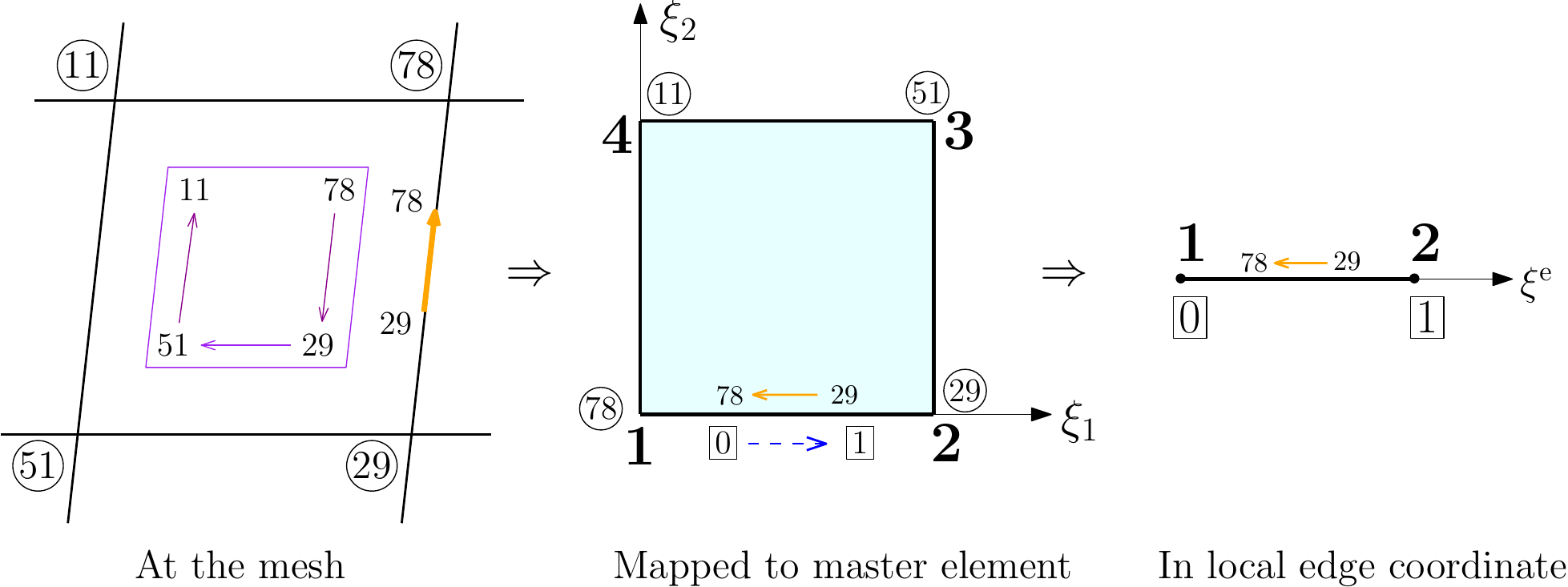}
\caption{A mesh with global edge orientations followed by transformations to master element and local edge coordinates.}
\label{fig:QuadEdgeOrientExample}
\end{center}
\end{figure}

Now, to see a more explicit example starting from the global mesh, consider Figure \ref{fig:QuadEdgeOrientExample}.
There, one can observe a quadrilateral in a mesh with vertices $11$, $29$, $51$ and $78$.
In 2D, Szabo's approach (implicitly) involves specifying a \textit{face vertex-ordering} at the mesh that determines the mapping to the master element, which has a fixed master element ordering $v_1\tto v_2\tto v_3\tto v_4$.
However, there is no independent edge vertex-ordering of each edge.
In this case, the face vertex-ordering at the mesh is shown to be $78\tto29\tto51\tto11$, and at the master element level it coincides with $v_1\tto v_2\tto v_3\tto v_4$.
The novelty here is that additionally the mesh edge with vertices $78$ and $29$ also has a \textit{global} orientation given by the edge vertex-ordering $29\tto78$.
This edge is mapped to the master element edge 12 (as induced by the face vertex-ordering), and receives the induced master element \textit{global} ordering $v_2\tto v_1$. 
Clearly, the \textit{local} orientation, given by the induced \textit{local} ordering $v_1\tdashto v_2$, does not coincide with the \textit{global} orientation, meaning that the orientation parameter is $\oo=1$ for this master element edge.
Therefore, the orientation embedded $H^1$ edge shape functions would specifically be,
\begin{equation*}
    \phi_i^\mathrm{e}(\xi)=\mu_0(\xi_2)\phi_i^{\e}\circ\sigma_\oo^\E(\vec{\mu}_{01}(\xi_1))
            	=\mu_0(\xi_2)\phi_i^\E(\mu_1(\xi_1),\mu_0(\xi_1))\,,
\end{equation*}
for $i=2,\ldots,p$. 
The neighboring element in the mesh (also sharing vertices $29$ and $78$) might have a different face vertex-ordering, but it has the \textit{same} edge vertex-ordering at the mesh (the global orientation). 
This can result in another orientation parameter at the master element edge of the neighboring element, but by construction, the fact remains that when mapped back to the global mesh, the shape functions will be fully compatible along that shared edge.  

\subsubsection{Quadrilateral Face Orientations Explained}
\label{sec:QuadFaceOrientations}

\begin{figure}[!ht]
\begin{center}
\includegraphics[scale=0.75]{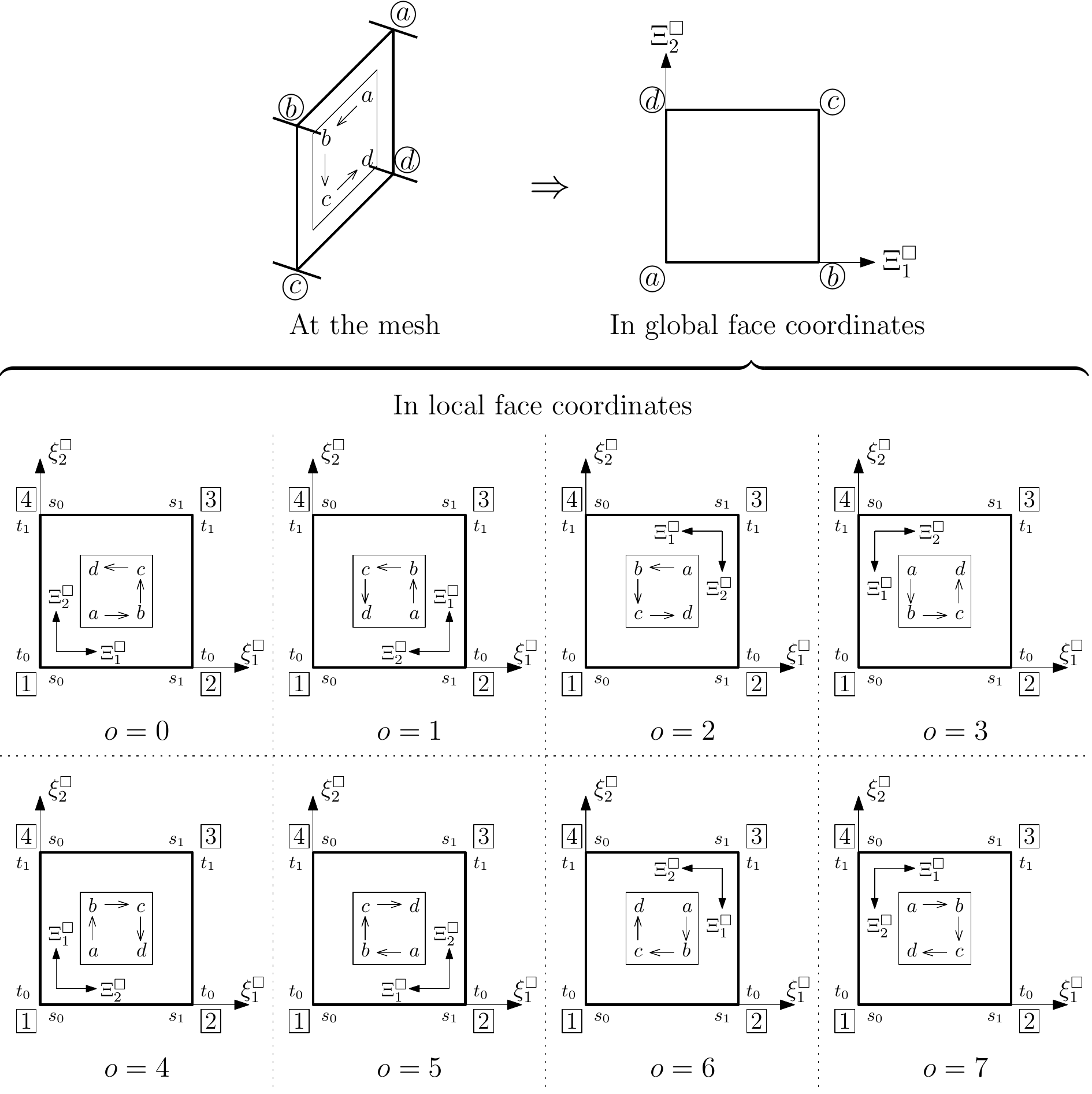}
\caption{Quadrilateral face orientations.}
\label{fig:orientationsquad}
\end{center}
\end{figure}

In 3D, as with edge orientations, each face at the mesh must be given its own \textit{global face orientation} to ensure full compatibility across the boundaries.
For quadrilaterals, this is represented by the global quadrilateral face coordinates $\Xi^\square=(\Xi_1^\square,\Xi_2^\square)$, or equivalently, by the \textit{global face vertex-ordering}.
For example, given a quadrilateral face in the mesh, a vertex-ordering of the form $a\tto b\tto c\tto d$ means the origin of $\Xi^\square$ is located at $a$, $\Xi_1^\square$ points from $a$ to $b$, and $\Xi_2^\square$ points from $a$ to $d$.
Meanwhile, at the master element, the mapped face has its own fixed \textit{local orientation}.
It is represented by the coordinates $\xi^\square=(\xi_1^\square,\xi_2^\square)$ or equivalently by the fixed \textit{local} ordering of the form $\boxednum{1}\!\tdashto\boxednum{2}\!\tdashto\boxednum{3}\!\tdashto\boxednum{4}$.
In general, the two systems of coordinates will not match, and this mismatch is represented by the orientation parameter $\oo$.
In fact, there are \textit{eight} possible orientations for quadrilateral faces, meaning $\oo=0,\ldots,7$.
These are all illustrated in Figure \ref{fig:orientationsquad}.


As with edges, the idea is to have a local-to-global transformation which depends on the orientation parameter $\oo$.
Here, the \textit{local} orientation is represented by the \textit{locally ordered} quadruple $(s_0,s_1,t_0,t_1)$ composed of the two pairs $(s_0,s_1)$ and $(t_0,t_1)$.
The \textit{first} pair, $(s_0,s_1)$, is a locally ordered pair corresponding to the \textit{first} local face coordinate $\xi_1^\square$ (which as an \textit{edge} coordinate has the local ordering $\boxednum{1}\!\tdashto\boxednum{2}$ or $\boxednum{4}\!\tdashto\boxednum{3}$).
The \textit{second} pair, $(t_0,t_1)$, is a locally ordered pair corresponding to the \textit{second} local face coordinate $\xi_2^\square$ (which as an \textit{edge} coordinate has the local ordering $\boxednum{1}\!\tdashto\boxednum{4}$ or $\boxednum{2}\!\tdashto\boxednum{3}$).
Meanwhile, the \textit{global} orientation is analogously represented by a \textit{globally ordered} quadruple composed of two pairs.
The first pair is a globally ordered pair corresponding to the first global face coordinate $\Xi_1^\square$, and similarly with the second pair.
For example, looking at Figure \ref{fig:orientationsquad}, when $\oo=1$, $\Xi_1^\square$ is associated to $(t_0,t_1)$, while $\Xi_2^\square$ is associated to $(s_1,s_0)$, so that the globally ordered quadruple is $(t_0,t_1,s_1,s_0)$.
This way, the local-to-global transformation is actually a permutation dependent on $\oo$ that can easily be determined by looking at Figure \ref{fig:orientationsquad}.

\begin{definition*}
Let $s_0$, $s_1$, $t_0$ and $t_1$ be arbitrary variables, and let $\oo=0,1,2,3,4,5,6,7$ be the quadrilateral face orientation parameter. 
The quadrilateral face orientation permutation function, $\sigma_\oo^\square$, is defined as
\begin{equation}
	\sigma_\oo^\square(s_0,s_1,t_0,t_1)=\begin{cases}
		\sigma_0^\square(s_0,s_1,t_0,t_1)=(s_0,s_1,t_0,t_1)&\quad\text{if  }\,\oo=0\,,\\
		\sigma_1^\square(s_0,s_1,t_0,t_1)=(t_0,t_1,s_1,s_0)&\quad\text{if  }\,\oo=1\,,\\
		\sigma_2^\square(s_0,s_1,t_0,t_1)=(s_1,s_0,t_1,t_0)&\quad\text{if  }\,\oo=2\,,\\
		\sigma_3^\square(s_0,s_1,t_0,t_1)=(t_1,t_0,s_0,s_1)&\quad\text{if  }\,\oo=3\,,\\
		\sigma_4^\square(s_0,s_1,t_0,t_1)=(t_0,t_1,s_0,s_1)&\quad\text{if  }\,\oo=4\,,\\
		\sigma_5^\square(s_0,s_1,t_0,t_1)=(s_1,s_0,t_0,t_1)&\quad\text{if  }\,\oo=5\,,\\
		\sigma_6^\square(s_0,s_1,t_0,t_1)=(t_1,t_0,s_1,s_0)&\quad\text{if  }\,\oo=6\,,\\
		\sigma_7^\square(s_0,s_1,t_0,t_1)=(s_0,s_1,t_1,t_0)&\quad\text{if  }\,\oo=7\,.\end{cases}\label{eq:orientQuadFace}
\end{equation}
\end{definition*}


As with edges, all that is required is to compose the \textit{quadrilateral face} ancillary functions and their differential form (those with superscript $\square$) with the local-to-global transformation given by $\sigma_\oo^\square$.
This should be done in all 3D shape functions associated to quadrilateral faces.
More concrete examples will be given in the 3D elements as the document progresses.

%% file: triangle.tex
\section{Triangle}
\label{sec:Tri}

\begin{figure}[!ht]
\begin{center}
\includegraphics[scale=0.5]{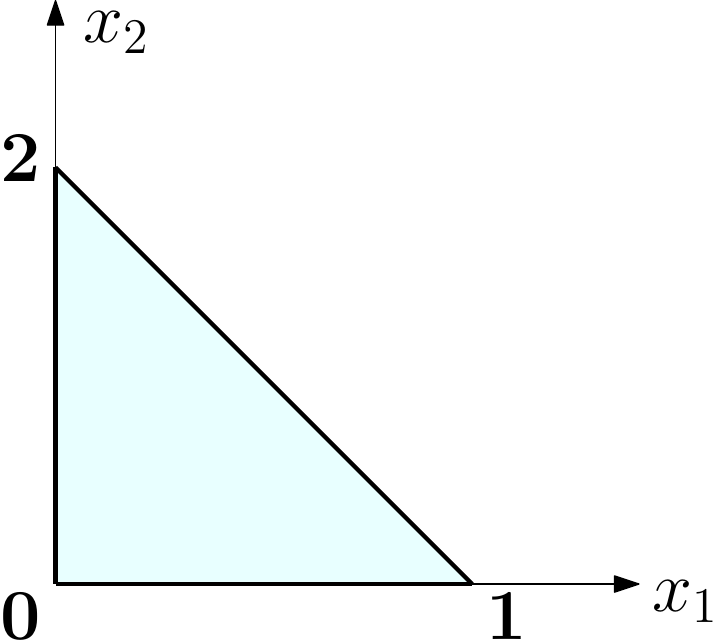}
\caption{Master triangle with numbered vertices.}
\label{fig:MasterTriangle}
\end{center}
\end{figure}

The triangle is the 2D simplex.
The master element for triangles in the $x=(x_1,x_2)$ space is the set $\{x\in\R^2:x_1>0,x_2>0,x_1+x_2<1\}$. 
It is illustrated in Figure \ref{fig:MasterTriangle}.

Denote vertex $a$ by $v_a$, so that $v_0=(0,0)$, $v_1=(1,0)$ and $v_2=(0,1)$.
As described in \S\ref{sec:affinecoordinates}, the 2D affine coordinates, $\nu_0$, $\nu_1$ and $\nu_2$, can be easily calculated for this master triangle:
\begin{equation}
	\nu_0(x)=1-x_1-x_2\,,\qquad
	\nu_1(x)=x_1\,,\qquad
	\nu_2(x)=x_2\,.
\end{equation}
Their gradients are
\begin{equation}
\nabla\nu_0(x)=\Big(\begin{smallmatrix}-1\\[2pt]-1\end{smallmatrix}\Big)\,,\qquad
\nabla\nu_1(x)=\Big(\begin{smallmatrix}1\\[2pt]0\end{smallmatrix}\Big)\,,\qquad
\nabla\nu_2(x)=\Big(\begin{smallmatrix}0\\[2pt]1\end{smallmatrix}\Big)\,.
\end{equation}

Like the quadrilateral and segment, the triangle exhibits a correspondence of its vertices and its affine coordinates.
Looking at the formula $x=\nu_0(x)v_0+\nu_1(x)v_1+\nu_2(x)v_2$ this relation is evident, in the sense that each vertex is linked to \textit{one} affine coordinate (its corresponding weight).
For example $v_1$ is linked to the affine coordinate $\nu_1(x)$, and indeed it takes the value $1$ when $x=v_1$ while it is zero at the other two vertices.

\subsubsection*{Exact Sequence}




As with the quadrilateral, the triangle will have 2D discrete polynomial exact sequences that represent the continuous exact sequence \eqref{eq:2DExactSeq} and its rotated analogue \eqref{eq:2DExactSeqRotated}. 
They are
%
%
\begin{equation}
\begin{alignedat}{4}
    &\mathcal{P}^p \xrightarrow{\,\,\nabla\,\,} &\mathcal{N}^p \xrightarrow{\nabla\times} && \mathcal{P}^{p-1} \,,\\
    &\mathcal{P}^p \xrightarrow{\mathrm{curl}\,} &\mathcal{RT}^p \xrightarrow{\,\nabla\cdot\,} &&\mathcal{P}^{p-1} \,,
    \label{eq:EStriangle}
\end{alignedat}
\end{equation}
where $\mathcal{P}^p =\mathcal{P}^p(x_1,x_2)$ is the space of polynomials of total order $p$.
The spaces $\mathcal{N}^p$ and $\mathcal{RT}^p$ are the N\'{e}d\'{e}lec and Raviart-Thomas spaces for simplices:
\begin{align}
	\label{eq:NedelecSpace}
	\mathcal{N}^p&=(\mathcal{P}^{p-1})^N\oplus\Big\{E\in(\tilde{\mathcal{P}}^{p})^N: x\cdot E(x)=0\,\text{ for all } x\in\R^N\Big\}\,,\\
	\label{eq:RaviartThomasSpace}
	\mathcal{RT}^p&=(\mathcal{P}^{p-1})^N\oplus\Big\{V\in(\tilde{\mathcal{P}}^{p})^N: V(x)=\phi(x)x
		\, \text{ with }\phi\in\tilde{\mathcal{P}}^{p-1}\,\text{ and }x\in\R^N\Big\} \,.
\end{align}
In the case of the triangle, the number of spatial dimensions is $N=2$.
Note the sequence has an overall drop in polynomial order of one. 
This makes it compatible with the construction presented for the quadrilateral.
Moreover, all of the spaces in the exact sequences above are invariant under affine transformations. 
This implies the exact sequence takes the same form for any given triangle (provided it is mapped via an affine transformation from the master triangle, which is always possible).

\subsection{\texorpdfstring{$H^1$}{H1} Shape Functions}

All shape functions defined here will lie in $\mathcal{P}^{p}$, which has dimension $\frac{1}{2}(p+2)(p+1)$. 
Moreover, a careful count of the linearly independent shape functions will coincide with that dimension, so that indeed the space is spanned.

\subsubsection{\texorpdfstring{$H^1$}{H1} Vertices}

In this case, for a given vertex, the associated shape function will be its related affine coordinate.
For example, $v_0$ is linked to $\nu_0(x)$, so its associated vertex function is simply 
\begin{equation*}
	\phi^\mathrm{v}(x)=\nu_0(x)\,.
\end{equation*}
As expected, it vanishes at the disjoint opposite edge 12 and its trace over the adjacent edges is a 1D $H^1$ vertex function associated to the vertex. 
Indeed, the traces are explicitly,
\begin{equation*}
	\nu_0(x)|_{x_2=0}=1-x_1=\mu_0(x_1)\,,\qquad\nu_0(x)|_{1-x_1-x_2=0}=0\,,\qquad\nu_0(x)|_{x_1=0}=1-x_2=\mu_0(x_2)\,. 
\end{equation*}
Lastly, the function decays linearly and is in the lowest order possible space, $\mathcal{P}^1$, so that it respects the hierarchy in $p$.

In general, the vertex functions and their gradient are,
\begin{equation}
    \phi^\mathrm{v}(x)=\nu_a(x)\,,\qquad\quad\nabla\phi^\mathrm{v}(x)=\nabla\nu_a(x)\,,
\end{equation}
for $a=0,1,2$.
There are a total of $3$ vertex functions (one for each vertex).
%
%

\subsubsection{\texorpdfstring{$H^1$}{H1} Edges}
\label{sec:H1edgesTri}

For the construction of the edge functions consider edge 01 as an example.
To satisfy compatibility with the quadrilateral, the shape functions should have $\phi_i^\E(\vec{\mu}_{01}(x_1))$ as the trace over edge 01, and be zero at the two other edges.
Additionally, they should be polynomial (they should be in $\mathcal{P}^p$).

Unfortunately, at first sight the construction is not trivial since many intuitive approaches lead to violating some of the required properties mentioned above.
This is in large part due to the fact that the triangle is \textit{not} a Cartesian product of lower dimensional elements.
Luckily, these issues are all solved by the process of homogenization, which is viewed as a particular extension of the lower dimensional edge functions $\phi_i^\E(\vec{\mu}_{01}(x_1))$.
The idea is to exploit the 2D \textit{triangle} affine coordinates associated to edge 01, which are $\nu_0$ and $\nu_1$ (those linked to the vertices $v_0$ and $v_1$ which compose the edge).
Indeed, let the shape functions for this edge be
\begin{equation*}
    \phi_i^\mathrm{e}(x)=\phi_i^\E(\vec{\nu}_{01}(x))=[L_i](\vec{\nu}_{01}(x))
    	=(\nu_0(x)+\nu_1(x))^i
    		[L_i]\Big(\textstyle{\frac{\nu_0(x)}{\nu_0(x)+\nu_1(x)}},\textstyle{\frac{\nu_1(x)}{\nu_0(x)+\nu_1(x)}}\Big)\,,
\end{equation*}
with $i=2,\ldots,p$.
Here, property \eqref{eq:ScalingProperty} was used.
Notice that $\nu_2=0$ is the equation for edge 01 and similarly with the other edges.
Hence, using the vanishing properties of $\phi_i^\E$ in \eqref{eq:phiEvanishing}, it follows that the trace properties are satisfied:
\begin{alignat*}{3}
    &\phi_i^\mathrm{e}(x)|_{x_2=0}&&=\phi_i^\E(\vec{\nu}_{01}(x))|_{\nu_2=0}=\phi_i^\E(\vec{\mu}_{01}(x_1))\,,\\
    &\phi_i^\mathrm{e}(x)|_{1-x_1-x_2=0}&&=\phi_i^\E(\vec{\nu}_{01}(x))|_{\nu_0=0}=\phi_i^\E(0,\mu_1(x_1))=0\,,\\
  	&\phi_i^\mathrm{e}(x)|_{x_1=0}&&=\phi_i^\E(\vec{\nu}_{01}(x))|_{\nu_1=0}=\phi_i^\E(\mu_0(x_2),0)=0\,.
\end{alignat*}
In this case, the decay of each shape function is represented by the nonlinear function $(\nu_0(x)+\nu_1(x))^i$, which comes hidden within the homogenization.
Additionally, $\phi_i^\E(\vec{\nu}_{01})\in\tilde{\mathcal{P}}^i(\nu_0,\nu_1)$ by homogenization while $\nu_0(x),\nu_1(x)\in\mathcal{P}^1(x)$, meaning that the edge functions $\phi_i^\E(\vec{\nu}_{01}(x))$ are in $\mathcal{P}^i(x)\subseteq\mathcal{P}^p(x)$ as required.

\begin{figure}[!ht]
\begin{center}
\includegraphics[scale=0.55]{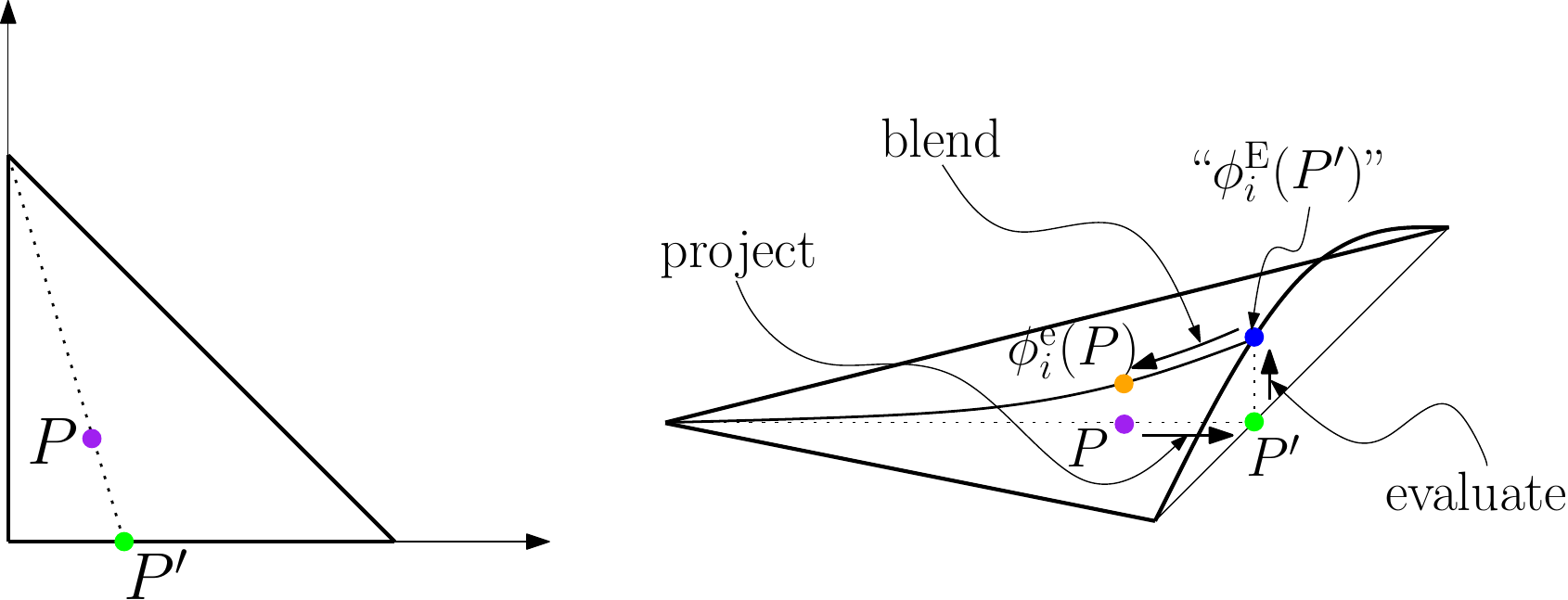}
\caption{Edge projection from $P$ to $P'$, and the logic project$\,\to\,$evaluate$\,\to\,$blend.}
\label{fig:TriangleProjection}
\end{center}
\end{figure}

As with the quadrilateral, there is a geometrical interpretation to these expressions, and again it follows the fundamental logic of project$\,\to\,$evaluate$\,\to\,$blend, which is marked below:
\begin{equation*}
    \phi_i^\mathrm{e}(x)=\phi_i^\E(\vec{\nu}_{01}(x))
    	=\underbrace{(\nu_0(x)+\nu_1(x))^i}_{\text{blend}}
    		\underbrace{\phi_i^\E\Big(\underbrace{\textstyle{\frac{\nu_0(x)}{\nu_0(x)+\nu_1(x)}},
    			\textstyle{\frac{\nu_1(x)}{\nu_0(x)+\nu_1(x)}}}_{\text{project}}\Big)}_{\text{evaluate}}\,.
\end{equation*}
The coordinates $(\frac{\nu_0}{\nu_0+\nu_1},\frac{\nu_1}{\nu_0+\nu_1})$ are projected 1D coordinates, because they sum to $1$ for all $x$. 
Note this is \textit{not} true for the coordinates $(\nu_0,\nu_1)$ in general (only over the edge itself).
As argued before for the quadrilateral, the projected coordinates represent a point over edge 01:
\begin{equation*}
	(x_1,x_2)\;\longmapsto\;(\textstyle{\frac{x_1}{1-x_2}},0)\,.
\end{equation*}
Geometrically, it consists of finding the intersection $P'=(\frac{x_1}{1-x_2},0)$ of the edge with the projecting line passing through the original point $P=(x_1,x_2)$ and the disjoint opposite vertex. 
This projection and the logic of the construction is illustrated in Figure \ref{fig:TriangleProjection}.


The complete list of edge functions with their gradients is
\begin{equation}
	\phi_i^\mathrm{e}(x)=\phi_i^\E(\vec{\nu}_{ab}(x))\,,\qquad\quad\nabla\phi_i^\mathrm{e}(x)=\nabla\phi_i^\E(\vec{\nu}_{ab}(x))\,,
	\label{eq:Triphigeneral}
\end{equation}
with $i=2,\ldots,p$, and $0\leq a<b\leq2$ (so $(a,b)=(0,1),(1,2),(0,2)$). As usual, there are a total of $p-1$ edge functions for every edge, for a total of $3(p-1)$ edge functions.

\subsubsection{\texorpdfstring{$H^1$}{H1} Face Bubbles}

These intuitively involve all three affine coordinates $(\nu_0,\nu_1,\nu_2)$.
The first natural idea is to use the existing edge shape functions and multiply by another function $\hat{\phi}_j$ which vanishes at the remaining edge.
That is,
\begin{equation*}
	\phi_{ij}^\mathrm{f}(x)=\phi^\E_i(\nu_0(x),\nu_1(x))\hat{\phi}_j(\nu_2(x))\,.
\end{equation*}
Here, $\hat{\phi}_j(y)$ should also be an $H^1$ function with domain $y\in[0,1]$ (since $0\leq\nu_2(x)\leq1$) and which vanishes at $y=0$, so that $\hat{\phi}_j(0)=0$. 
Notice that in the quadrilateral, when constructing $\phi_{ij}^\Box$ the choice was naturally $\hat{\phi}_j=L_j$ since it was required that $\hat{\phi}_j(0)=\hat{\phi}_j(1)=0$, so it should vanish at \textit{both} endpoints. 
For the triangle there is much more liberty for the choice of $\hat{\phi}_j$, and indeed it can be chosen such that it has many advantages. 
Following \citet{Beuchler_Schoeberl_06} it is chosen as a Jacobi polynomial, $\hat{\phi}_j=L_j^\alpha$, with $\alpha=2i$, where $i$ is the order of $\phi^\E_i$. This results in
\begin{equation*}
	\phi_{ij}^\mathrm{f}(x)=\phi^\E_i(\nu_0(x),\nu_1(x))L_j^{2i}(\nu_2(x))\,.
\end{equation*}

Thinking ahead to the tetrahedron, this definition can be easily generalized. Just as with $\phi_i^\E$, the generalization is nothing more than the homogenization of the previous formula, as written in \eqref{eq:homogproduct}.

\begin{definition*}
Let $s_0$, $s_1$ and $s_2$ be arbitrary functions of some spatial variable in $\R^N$, with $N=2,3$, and denote by $p_s$ the order in the coordinate triplet $(s_0,s_1,s_2)$. Then
\begin{equation}
    \phi_{ij}^\Tri(s_0,s_1,s_2)=\phi_i^\E(s_0,s_1)[L_j^{2i}](s_0+s_1,s_2)=[L_i,L_j^{2i}](s_0,s_1,s_2)\,,
\end{equation}
for $i\geq2$, $j\geq1$ and $n=i+j=3,\ldots,p_s$. 
The gradients, understood in $\R^N$, are
\begin{equation}
    \begin{aligned}
    \nabla\phi_{ij}^\Tri(s_0,s_1,s_2)&=[L_j^{2i}](s_0+s_1,s_2)\nabla\phi_i^\E(s_0,s_1)
    	+\phi_i^\E(s_0,s_1)\nabla[L_j^{2i}](s_0+s_1,s_2)\,,
    \end{aligned}
\end{equation}
where for any $\alpha$,
\begin{equation}	
	\begin{aligned}
  	\nabla[L_j^\alpha](s_0\!+\!s_1,s_2)&=[P_{j-1}^\alpha](s_0\!+\!s_1,s_2)\nabla s_2
  		+[R_{j-1}^\alpha](s_0\!+\!s_1,s_2)\nabla(s_0\!+\!s_1\!+\!s_2)\\
        &=P_{j-1}^\alpha(s_2;s_0\!+\!s_1\!+\!s_2)\nabla s_2
        		+R_{j-1}^\alpha(s_2;s_0\!+\!s_1\!+\!s_2)\nabla(s_0\!+\!s_1\!+\!s_2)\,.
  \end{aligned}\label{eq:DerivativeofLalpha}
\end{equation}
\end{definition*}

Note the indexing was shown as $i\geq2$, $j\geq1$ and $n=i+j=3,\ldots,p_s$. 
In many $hp$ codes it is useful to enforce hierarchy, so that the shape functions are organized by total order. 
That is, first list all order $3$ bubbles, then all order $4$ bubbles, and so on.
Hence when coding is in mind, it is useful to have an outer loop with numbering $n=3,\ldots,p_s$ and an inner loop indexing with either $i=2,\ldots,n-1$ (so $j=n-i$) or $j=1,\ldots,n-2$ (so $i=n-j$).

Regarding the vanishing properties of this new ancillary function, it suffices to rewrite \eqref{eq:LiLjvanishing} so that for any $s_0,s_1,s_2$, and all $i\geq2$, $j\geq1$,
\begin{equation}
	\phi_{ij}^\Tri(s_0,s_1,0)=\phi_{ij}^\Tri(s_0,0,s_2)=\phi_{ij}^\Tri(0,s_1,s_2)=0\,.\label{eq:phiTrivanishing}
\end{equation}

As expected, when the coordinates are 2D affine coordinates, the formulas for $\phi_{ij}^\Tri$ and its gradient are simplified.
For this, record the next remark.

\begin{remark}
Let $\nu_0=1-\nu_1-\nu_2$, where $\nu_1$ and $\nu_2$ are arbitrary functions of some spatial variable in $\R^N$, $N=2,3$, and where $p$ is the order in the coordinates $(\nu_0,\nu_1,\nu_2)$. Then for all $i\geq2$, $j\geq1$, and $n=i+j=3,\ldots,p$,
\begin{equation}
	\begin{aligned}
    \phi_{ij}^\Tri(\nu_0,\nu_1,\nu_2)&=\phi_i^\E(\nu_0,\nu_1)L_j^{2i}(\nu_2)\,,\\
    \nabla\phi_{ij}^\Tri(\nu_0,\nu_1,\nu_2)&=L_j^{2i}(\nu_2)\nabla\phi_i^\E(\nu_0,\nu_1)
    	+\phi_i^\E(\nu_0,\nu_1)P_{j}^{2i}(\nu_2)\nabla\nu_2\,.
	\end{aligned}
  \label{eq:H1Trispecialcase}
\end{equation}
\end{remark}



In general, the $H^1$ triangle face bubbles and their gradients are
\begin{equation}
	\phi_{ij}^\mathrm{f}(x)=\phi_{ij}^\Tri(\vec{\nu}_{012}(x))\,,\qquad\quad
		\nabla\phi_{ij}^\mathrm{f}(x)=\nabla\phi_{ij}^\Tri(\vec{\nu}_{012}(x))\,,
\label{eq:H1_CountingTriFace}
\end{equation}
where $i\geq2$, $j\geq1$ and $n=i+j=3,\ldots,p$. 
The vanishing conditions are satisfied by construction or simply by looking at \eqref{eq:phiTrivanishing}. 
There are a total of $\frac{1}{2}(p-1)(p-2)$ $H^1$ face bubbles for the triangle.



%
%

\subsection{\texorpdfstring{$H(\mathrm{curl})$}{Hcurl} Shape Functions}


The dimension of $\mathcal{N}^p$ in two dimensions is $p(p+2)$.
A careful count of the linearly independent conforming shape functions to be presented throughout this section will coincide with that dimension. 
Showing that the functions constructed are in $\mathcal{N}^p$ is nontrivial, but will follow from the next instrumental lemma.

\begin{lemma}
\label{lemma:curl}
Let $x\in\R^N$ for $N=2,3$, and $f_n\in\mathcal{P}^n(x)$ be any polynomial of total order $n$ in the coordinates $x=(x_1,\ldots,x_N)$. Given $s_0$ and $s_1$ coordinates in $\mathcal{P}^1(x)$ $($linear functions in $x$$)$, it follows that the N\'ed\'elec space of order $n+1$, $\mathcal{N}^{n+1}$, contains the function
\begin{equation*}
	f_n(\bcdot)\Big(s_0\nabla s_1-s_1\nabla s_0\Big)\in\mathcal{N}^{n+1}\,.
\end{equation*}
\end{lemma}
\begin{proof}
Recall the definition of the N\'ed\'elec space,
\begin{equation*}
	\mathcal{N}^p=(\mathcal{P}^{p-1})^N\oplus\Big\{E\in(\tilde{\mathcal{P}}^{p})^N: x\cdot E(x)=0\,\text{ for all } 
		x\in\R^N\Big\}\,.
\end{equation*}
The (affine) coordinates are linear functions in $x=(x_1,\ldots,x_N)$, so that
\begin{equation*}
	s_k(x)=a_k+b_k\cdot x\,,
\end{equation*}
for $a_k\in\R$, $b_k\in\R^N$ and $k=0,1$. As a result $\nabla s_k(x)=b_k$, and
\begin{equation*}
	s_0(x)\nabla s_1(x)-s_1(x)\nabla s_0(x)=\underbrace{(a_0b_1-a_1b_0)}_{=A}+\underbrace{((b_0\cdot x)b_1-(b_1\cdot x)b_0)}_{=B(x)}\,.
\end{equation*}
Clearly, $A\in(\mathcal{P}^0)^N=\R^N$, while $B\in(\tilde{\mathcal{P}}^1)^N$. Moreover,
\begin{equation*}
	x\cdot B(x)=x\cdot((b_0\cdot x)b_1-(b_1\cdot x)b_0)=(b_0\cdot x)(b_1\cdot x)-(b_1\cdot x)(b_0\cdot x)=0\,,
\end{equation*}
so that $B\in\{E\in(\tilde{\mathcal{P}}^{1})^N: x\cdot E(x)=0\}$, and $s_0\nabla s_1-s_1\nabla s_0\in\mathcal{N}^1$.

Now, $f_n\in\mathcal{P}^n=\mathcal{P}^{n-1}\oplus\tilde{\mathcal{P}}^n$, for $n\geq1$ can always be decoupled into $f_n=f_{n-1}+\tilde{f}_n$, where $f_{n-1}\in\mathcal{P}^{n-1}$ and $\tilde{f}_n\in\tilde{\mathcal{P}}^n$. As a result
\begin{equation*}
	f_n(x)(s_0(x)\nabla s_1(x)-s_1(x)\nabla s_0(x))=f_n(x)A+f_{n-1}(x)B(x)+\tilde{f}_n(x)B(x)\,,
\end{equation*}
where it is clear $f_{n}A+f_{n-1}B\in(\mathcal{P}^n)^N$ and $\tilde{f}_nB\in(\tilde{\mathcal{P}}^{n+1})^N$. Meanwhile, 
\begin{equation*}
x\cdot(\tilde{f}_n(x)B(x))=\tilde{f}_n(x)x\cdot B(x)=0\,.
\end{equation*}
Hence, $f_n(\bcdot)(s_0\nabla s_1-s_1\nabla s_0)\in\mathcal{N}^{n+1}$.
\end{proof}

\subsubsection{\texorpdfstring{$H(\mathrm{curl})$}{Hcurl} Edges}

Having seen what ocurred with $H^1$ edge functions due to the process of homogenization, it is wise to consider the general definition of $H(\mathrm{curl})$ edge functions in \eqref{eq:Hcurledgefunctions}, using 2D affine coordinates as entries. 
Indeed this will suffice.
For instance, for edge 01, the shape functions are
\begin{equation*}
	\begin{aligned}
		E_i^\mathrm{e}(x)&=E_i^\E(\vec{\nu}_{01}(x))=[P_i](\vec{\nu}_{01}(x))\Big(\nu_0(x)\nabla\nu_1(x)-\nu_1(x)\nabla\nu_0(x)\Big)\\
    	&=(\nu_0(x)+\nu_1(x))^i
    		[P_i]\Big(\textstyle{\frac{\nu_0(x)}{\nu_0(x)+\nu_1(x)}},\textstyle{\frac{\nu_1(x)}{\nu_0(x)+\nu_1(x)}}\Big)
    			E_0^\E(\vec{\nu}_{01}(x))\\
    	&=\underbrace{(\nu_0(x)+\nu_1(x))^{i+2}}_{\text{blend}}
    		\underbrace{E_i^\E\Big(\underbrace{\textstyle{\frac{\nu_0(x)}{\nu_0(x)+\nu_1(x)}},
    			\textstyle{\frac{\nu_1(x)}{\nu_0(x)+\nu_1(x)}}}_{\text{project}}\Big)}_{\text{evaluate}}\,,
	\end{aligned}
\end{equation*}
for $i=0,\ldots,p-1$. 
By Lemma \ref{lemma:curl} it easily follows $E_i^\mathrm{e}\in\mathcal{N}^{i+1}\subseteq\mathcal{N}^{p}$ as desired.
The tangential trace properties are also satisfied. 
To see this, it suffices to look at $E_0^\mathrm{e}$, since $E_i^\mathrm{e}(x)=[P_i](\vec{\nu}_{01}(x))E_0^\mathrm{e}(x)$. 
Its traces are
\begin{alignat*}{3}
    &\mathrm{tr}(E_0^\mathrm{e}(x))|_{x_2=0}&&=E_0^\E(\vec{\nu}_{01}(x))|_{\nu_2=0}\cdot(v_1-v_0)
    	=\nabla\nu_1(x)\cdot(v_1-v_0)=1\,,\\
    &\mathrm{tr}(E_0^\mathrm{e}(x))|_{1-x_1-x_2=0}&&=E_0^\E(\vec{\nu}_{01}(x))|_{\nu_0=0}\cdot(v_2-v_1)
    	=-\nu_1(x)\nabla\nu_0(x)\cdot(v_2-v_1)=0\,,\\
  	&\mathrm{tr}(E_0^\mathrm{e}(x))|_{x_1=0}&&=E_0^\E(\vec{\nu}_{01}(x))|_{\nu_1=0}\cdot(v_2-v_0)
  		=\nu_0(x)\nabla\nu_1(x)\cdot(v_2-v_0)=0\,.
\end{alignat*}
Here the vanishing traces follow from the fact that the gradient of a function is always perpendicular to the tangents of its isosurfaces. 
Hence $\nabla\nu_0(x)$ is perpendicular to the tangents of the set where $\nu_0(x)=0$, which is precisely where $v_2-v_1$ lies. 
Regarding the nonzero trace, the formula $\nabla\nu_1(x)\cdot(v_1-v_0)$ follows after replacing $\nu_0=1-\nu_1$, and noting that $\nabla(\nu_0(x)+\nu_1(x))\cdot(v_1-v_0)=-\nabla\nu_2(x)\cdot(v_1-v_0)=0$.

More generally, the edge functions and their curls are
\begin{equation}
	E_i^\mathrm{e}(x)=E_i^\E(\vec{\nu}_{ab}(x))\,,\qquad\quad\nabla\times E_i^\mathrm{e}(x)=\nabla\times E_i^\E(\vec{\nu}_{ab}(x))\,,
	\label{eq:TriEgeneral}
\end{equation}
with $i=0,\ldots,p-1$, and $0\leq a<b\leq2$ (so $(a,b)=(0,1),(1,2),(0,2)$). There are a total of $p$ edge functions for every given edge, giving a total of $3p$ edge functions.

\subsubsection{\texorpdfstring{$H(\mathrm{curl})$}{Hcurl} Face Bubbles}

The construction of $H(\mathrm{curl})$ face bubbles is parallel to that of $H^1$ bubbles. 
The idea is to multiply $H(\mathrm{curl})$ edge functions with an $H^1$ function vanishing at zero. 
Again, it is chosen according to \citet{Beuchler_Pillwein_Zaglmayr_13} as the Jacobi polynomial, $L_j^\alpha$, with $\alpha=2i+1$, where $i$ is the order of $E^\E_i$.
Like the quadrilateral, the triangle has two closely related families generated by the same ancillary operator defined below.

\begin{definition*}
Let $s_0$, $s_1$ and $s_2$ be arbitrary functions of some spatial variable in $\R^N$, with $N=2,3$, and denote by $p_s$ the order in the coordinate triplet $(s_0,s_1,s_2)$. Then
\begin{equation}
    E_{ij}^\Tri(s_0,s_1,s_2)=[L_j^{2i+1}](s_0+s_1,s_2)E_i^\E(s_0,s_1)\,,
\end{equation}
for $i\geq0$, $j\geq1$ and $n=i+j=1,\ldots,p_s-1$.
The curls, understood in $\R^N$, are
\begin{equation}
    \begin{aligned}
    \nabla\!\times\! E_{ij}^\Tri(s_0,s_1,s_2)&\!=\![L_j^{2i+1}](s_0\!+\!s_1,s_2)\nabla\!\times\! E_i^\E(s_0,s_1)
    	\!+\!\nabla[L_j^{2i+1}](s_0\!+\!s_1,s_2)\!\times\! E_i^\E(s_0,s_1)\,,
    \end{aligned}
\end{equation}
where $\nabla[L_j^{2i+1}](s_0+s_1,s_2)$ is computed from \eqref{eq:DerivativeofLalpha}.
\end{definition*}

As with the edge functions, it is clear that when evaluated in any permutation of $(\nu_0,\nu_1,\nu_2)$ the vanishing trace properties are satisfied.
The two families defined below give a total of $p(p-1)$ $H(\mathrm{curl})$ triangle face bubbles (each family has $\frac{1}{2}p(p-1)$).

\subparagraph{Family I:} 
The shape functions for the first family and their curl are
\begin{equation}
	E_{ij}^{\mathrm{f}}(x)=E_{ij}^\Tri(\vec{\nu}_{012}(x))\,,\qquad\quad
		\nabla\times E_{ij}^{\mathrm{f}}(x)=\nabla\times E_{ij}^\Tri(\vec{\nu}_{012}(x))\,,
\end{equation}
for $i\geq0$, $j\geq1$ and $n=i+j=1,\ldots,p-1$. 
By Lemma \ref{lemma:curl}, $E_{ij}^\mathrm{f}\in\mathcal{N}^{n+1}$, where $n=i+j$. 
There are $\frac{1}{2}p(p-1)$ bubble functions in this family.

\subparagraph{Family II:}
The shape functions for the second family and their curl are
\begin{equation}
	E_{ij}^{\mathrm{f}}(x)=E_{ij}^\Tri(\vec{\nu}_{120}(x))\,,\qquad\quad
		\nabla\times E_{ij}^{\mathrm{f}}(x)=\nabla\times E_{ij}^\Tri(\vec{\nu}_{120}(x))\,,
\end{equation}
for $i\geq0$, $j\geq1$ and $n=i+j=1,\ldots,p-1$.
The only difference with the first family is that the entries are permuted to $\vec{\nu}_{120}(x)=(\nu_1(x),\nu_2(x),\nu_0(x))$ instead of $\vec{\nu}_{012}(x)=(\nu_0(x),\nu_1(x),\nu_2(x))$.
Again, they lie in $E_{ij}^\mathrm{f}\in\mathcal{N}^{n+1}$, where $n=i+j$, and there are $\frac{1}{2}p(p-1)$ bubble functions in this family.

\subsection{\texorpdfstring{$H(\mathrm{div})$}{Hdiv} Shape Functions}
\label{sec:TriangleHdiv}

In two dimensions, the $H(\mathrm{div})$ space (see \eqref{eq:Hdiv2Ddef}) is isomorphic to $H(\mathrm{curl})$. 
Indeed, just like with the quadrilateral, given a shape function $E\in H(\mathrm{curl})$ and its curl $\nabla\times E\in L^2$, the corresponding $H(\mathrm{div})$ shape function with its divergence is
\begin{equation}
    V=\begin{pmatrix}0&1\\-1&0\end{pmatrix}E\,,\quad\qquad\nabla\cdot V=\nabla\times E\,.
\end{equation}
Although not immediate, note that in \textit{two} dimensions the original polynomial space $\mathcal{N}^p$ for $H(\mathrm{curl})$ simply becomes the $H(\mathrm{div})$ conforming space $\mathcal{RT}^p$ after the rotation, as required.

\subsection{\texorpdfstring{$L^2$}{L2} Shape Functions}

These span $\mathcal{P}^{p-1}$, so there should be $\frac{1}{2}p(p+1)$ linearly independent shape functions.

\subsubsection{\texorpdfstring{$L^2$}{L2} Face}

Again, carefully chosen Jacobi polynomials are present in this construction. 
The shape functions are
\begin{equation}
\begin{aligned}
	\psi_{ij}^\mathrm{f}(x)&=[P_i](\nu_0(x),\nu_1(x))[P_j^{2i+1}](\nu_0(x)+\nu_1(x),\nu_2(x))\\
		&=[P_i,P_j^{2i+1}](\vec{\nu}_{012}(x))(\nabla\nu_1(x)\!\!\times\!\!\nabla\nu_2(x))\,,
\end{aligned}
\end{equation}
where $i\geq0$, $j\geq0$ and $n=i+j=0,\ldots,p-1$.
Clearly the functions lie in $\mathcal{P}^{p-1}$ and there are $\frac{1}{2}p(p+1)$ such functions. 
The factor $\nabla\nu_1(x)\!\times\!\nabla\nu_2(x)$ makes the expression coordinate free (with the $x$ coordinates it is $1$).

%
%

\subsection{Orientations}
\label{sec:TriaOrientations}

Only edge orientations are required to ensure compatibility of 2D triangles.
However, in 3D elements, triangle \textit{faces} have the notion of orientation, and one needs to consider this to ensure full compatibility.
In \S\ref{sec:QuadEdgeOrientations} it was shown how to apply the principles introduced in \S\ref{sec:edgeorientations} to construct orientation embedded edge functions.
It is convenient to have read those sections.
In what follows, a brief illustration of how it analogously applies to the triangle is given in \S\ref{sec:TriaEdgeOrientations}.
Afterwards, in \S\ref{sec:TriaFaceOrientations}, the triangle face orientations are described.

\subsubsection{Edge Orientations}
\label{sec:TriaEdgeOrientations}

\begin{figure}[!ht]
\begin{center}
\includegraphics[scale=0.5]{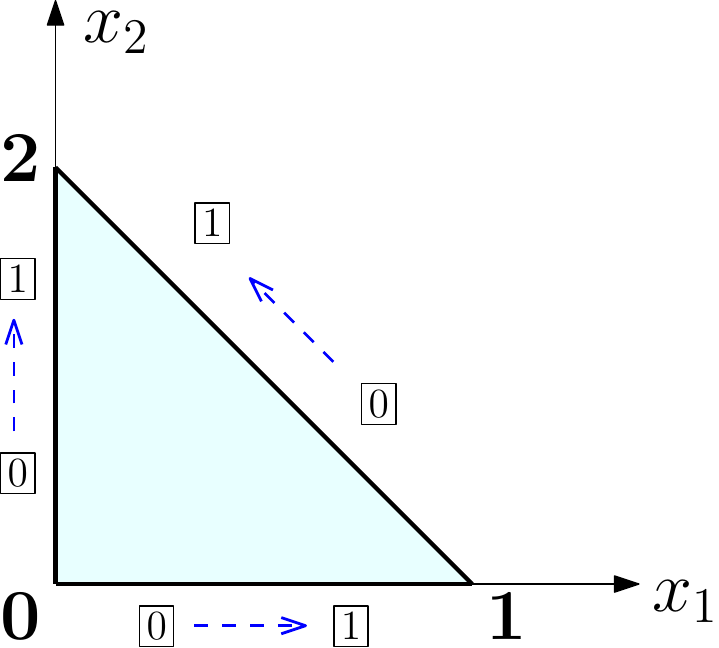}
\caption{Master triangle with local edge orientations.}
\label{fig:MasterTriOrientations}
\end{center}
\end{figure}

The master triangle has a predefined \textit{local} orientation for each edge, which reperesents the $\oo=0$ case. 
These are shown in Figure \ref{fig:MasterTriOrientations}.
They are \textit{our} choices for the local orientations.

As with the quadrilateral, the local orientations induce a master element \textit{local} edge vertex-ordering which then determines a \textit{locally ordered} pair of affine coordinates. 
In the triangle, the correspondence between vertices and affine coordinates is trivial, since all one needs to know is that the vertex $v_a$ is linked to the affine coordinate $\nu_a$, where $a=0,1,2$.
As an example take the local edge orientations shown in Figure \ref{fig:MasterTriOrientations}.
Then, it is clear that the induced \textit{local} edge vertex-orderings are $v_0\tdashto v_1$, $v_1\tdashto v_2$ and $v_0\tdashto v_2$ for edges 01, 12 and 02 respectively.
Therefore, the \textit{locally ordered} pair of coordinates are $(\nu_0,\nu_1)$, $(\nu_1,\nu_2)$ and $(\nu_0,\nu_2)$.
These are then transformed to a \textit{globally ordered} pair via the edge local-to-global transformation $\sigma_\oo^\E$, and inputted into the different ancillary functions to give the orientation embedded shape functions.
For example, for edge 01, the orientation embedded $H^1$ edge functions are
\begin{equation*}
    \phi_i^\mathrm{e}(x)=\phi_i^\E\circ\sigma_\oo^\E(\vec{\nu}_{01}(x))
        =\begin{cases}
            \phi_i^\E\Big(\sigma_0^\E(\vec{\nu}_{01}(x))\Big)
            	=\phi_i^\E(\vec{\nu}_{01}(x))\,\,\,\text{if }\oo=0\,,\\
            \phi_i^\E\Big(\sigma_1^\E(\vec{\nu}_{01}(x))\Big)
            	=\phi_i^{\e}(\vec{\nu}_{10}(x))\,\,\,\text{if }\oo=1\,,
        \end{cases}
\end{equation*}
for $i=2,\ldots,p$.
For a more detailed example which analogously applies to the triangle see \S\ref{sec:QuadEdgeOrientations}.

\subsubsection{Triangle Face Orientations Explained}
\label{sec:TriaFaceOrientations}


In 3D, each face at the mesh must be given its own \textit{global face orientation} to ensure full compatibility across the boundaries.
For triangles, this is represented by the global triangle face coordinates $\Xi^\Tri=(\Xi_1^\Tri,\Xi_2^\Tri)$, or equivalently, by the \textit{global face vertex-ordering}.
For instance, given a triangle face in the mesh, a vertex-ordering of the form $a\tto b\tto c$ means the origin of $\Xi^\Tri$ is located at $a$, $\Xi_1^\Tri$ points from $a$ to $b$, and $\Xi_2^\Tri$ points from $a$ to $c$.
Meanwhile, at the master element, the mapped face has its own fixed \textit{local orientation}.
It is represented by the coordinates $\xi^\Tri=(\xi_1^\Tri,\xi_2^\Tri)$ or equivalently by the fixed \textit{local} ordering of the form $\boxednum{0}\!\tdashto\boxednum{1}\!\tdashto\boxednum{2}$.
In general, the two systems of coordinates will not match, and this mismatch is represented by the orientation parameter $\oo$.
For triangle faces, there are \textit{six} possible orientations, meaning $\oo=0,\ldots,5$.
These are all illustrated in Figure \ref{fig:OrientationsTri}.

\begin{figure}[!ht]
\begin{center}
\includegraphics[scale=0.75]{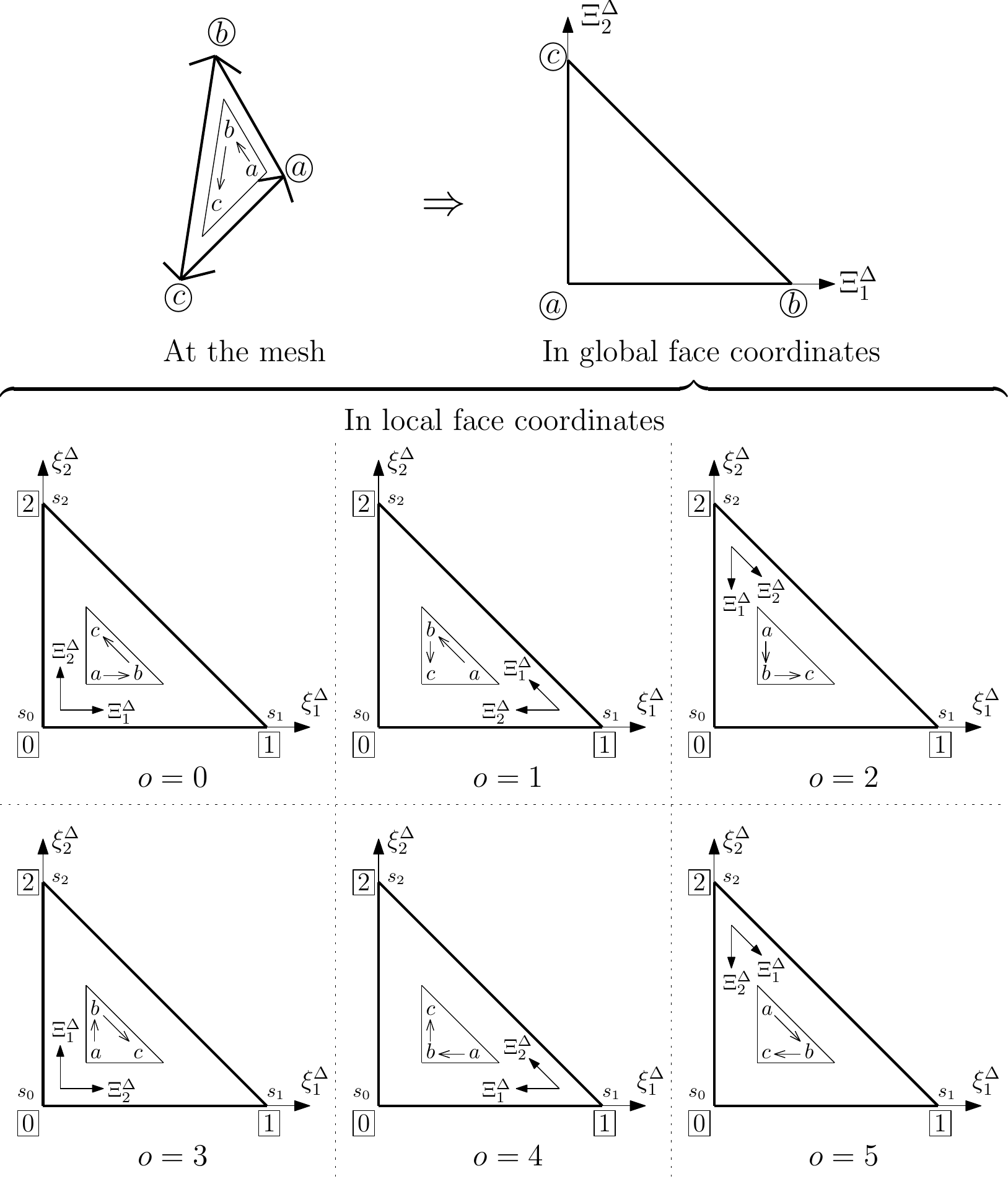}
\caption{Triangle face orientations.}
\label{fig:OrientationsTri}
\end{center}
\end{figure}

Like quadrilaterals and edges, one needs a local-to-global transformation dependent on $\oo$.
For this, the \textit{local} orientation is represented by the \textit{locally ordered} triplet $(s_0,s_1,s_2)$.
The \textit{global} orientation is analogously represented by a \textit{globally ordered} triplet naturally induced from the \textit{global face vertex-ordering}.
For example, looking at Figure \ref{fig:OrientationsTri}, when $\oo=1$, the global ordering $a\tto b\tto c$ corresponds to the ordering $\boxednum{1}\!\tto\boxednum{2}\!\tto\boxednum{0}$, which in turn coresponds to the globally ordered triplet $(s_1,s_2,s_0)$.
Hence, once again the local-to-global transformation can be represented by a permutation, $\sigma_\oo^\Tri$, dependent on $\oo$.
It can be determined by observing Figure \ref{fig:OrientationsTri}.
Subsequently, all that is required is to compose the \textit{triangle face} ancillary functions and their differential form (those with superscript $\Tri$) with the local-to-global transformation $\sigma_\oo^\Tri$.
This should be done in all 3D shape functions associated to triangle faces.
More concrete examples will be given later.
Finally, the permutation function, $\sigma_\oo^\Tri$, is precisely defined next.

\begin{definition*}
Let $s_0$, $s_1$ and $s_2$ be arbitrary variables, and let $\oo=0,1,2,3,4,5$ be the triangle face orientation parameter. 
The triangle face orientation permutation function, $\sigma_\oo^\Tri$, is defined as
\begin{equation}
	\sigma_\oo^\Tri(s_0,s_1,s_2)=\begin{cases}
		\sigma_0^\Tri(s_0,s_1,s_2)=(s_0,s_1,s_2)&\quad\text{if  }\,\oo=0\,,\\
		\sigma_1^\Tri(s_0,s_1,s_2)=(s_1,s_2,s_0)&\quad\text{if  }\,\oo=1\,,\\
		\sigma_2^\Tri(s_0,s_1,s_2)=(s_2,s_0,s_1)&\quad\text{if  }\,\oo=2\,,\\
		\sigma_3^\Tri(s_0,s_1,s_2)=(s_0,s_2,s_1)&\quad\text{if  }\,\oo=3\,,\\
		\sigma_4^\Tri(s_0,s_1,s_2)=(s_1,s_0,s_2)&\quad\text{if  }\,\oo=4\,,\\
		\sigma_5^\Tri(s_0,s_1,s_2)=(s_2,s_1,s_0)&\quad\text{if  }\,\oo=5\,.\end{cases}\label{eq:orientTriFace}
\end{equation}
\end{definition*}

%% file: hexa.tex
\section{Hexahedron}
\label{sec:Hexa}

\begin{figure}[!ht]
\begin{center}
\includegraphics[scale=0.5]{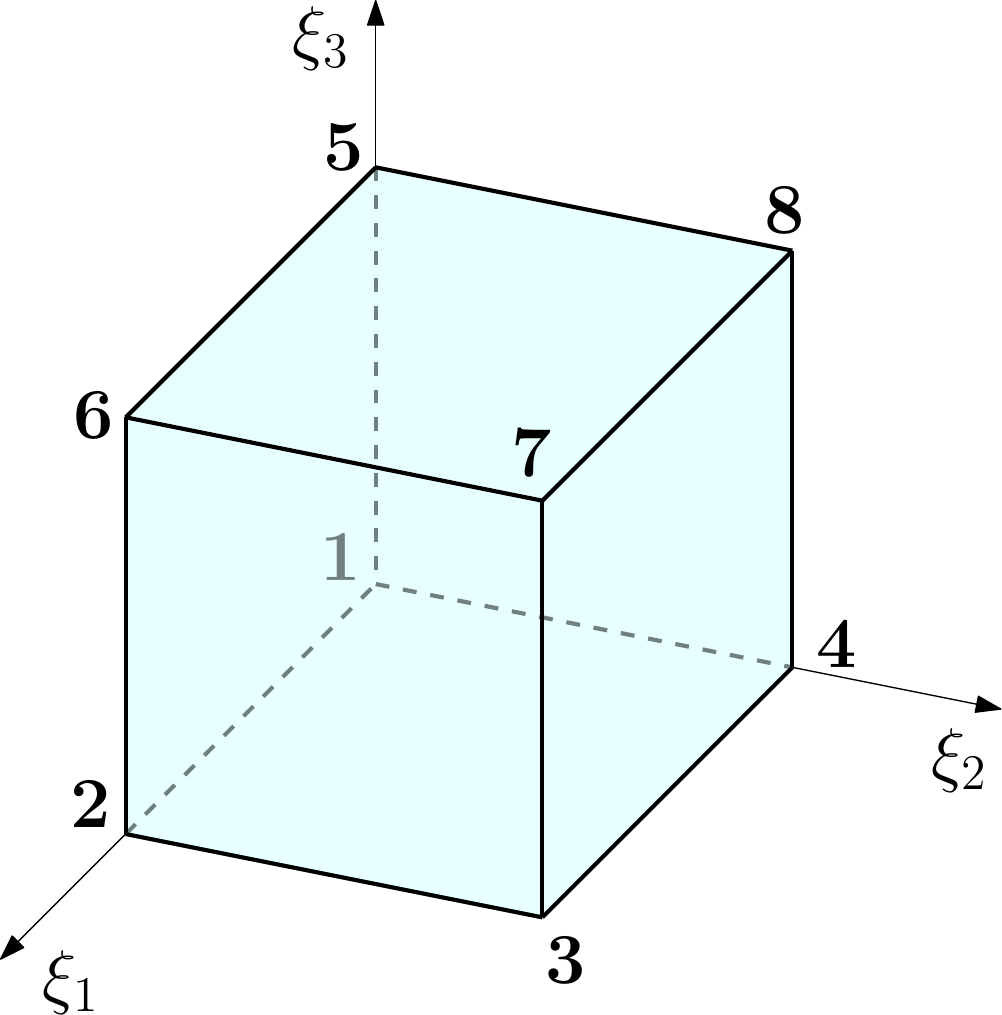}
\caption{Master hexahedron with numbered vertices.}
\label{fig:MasterHexa}
\end{center}
\end{figure}

The master element for hexahedra is $(0,1)^3$. 
It is shown in Figure \ref{fig:MasterHexa} in the $\xi=(\xi_1,\xi_2,\xi_3)$ space.
The master hexahedron is the Cartesian product of three segments.


There are \textit{three} pairs of 1D affine coordinates:
\begin{equation}
	\begin{alignedat}{4}
		\mu_0(\xi_1)&=1-\xi_1\,,\quad \mu_1(\xi_1)=\xi_1\,\qquad&\Rightarrow\qquad
			\nabla\mu_0(\xi_1)&=\bigg(\begin{smallmatrix}-1\\[2pt]0\\[2pt]0\end{smallmatrix}\bigg)\,,\quad
				\nabla\mu_1(\xi_1)=\bigg(\begin{smallmatrix}1\\[2pt]0\\[2pt]0\end{smallmatrix}\bigg)\,,\\
		\mu_0(\xi_2)&=1-\xi_2\,,\quad \mu_1(\xi_2)=\xi_2\,\qquad&\Rightarrow\qquad
			\nabla\mu_0(\xi_2)&=\bigg(\begin{smallmatrix}0\\[2pt]-1\\[2pt]0\end{smallmatrix}\bigg)\,,\quad
				\nabla\mu_1(\xi_2)=\bigg(\begin{smallmatrix}0\\[2pt]1\\[2pt]0\end{smallmatrix}\bigg)\,,\\
		\mu_0(\xi_3)&=1-\xi_3\,,\quad \mu_1(\xi_3)=\xi_3\,\qquad&\Rightarrow\qquad
			\nabla\mu_0(\xi_3)&=\bigg(\begin{smallmatrix}0\\[2pt]0\\[2pt]-1\end{smallmatrix}\bigg)\,,\quad
		\nabla\mu_1(\xi_3)=\bigg(\begin{smallmatrix}0\\[2pt]0\\[2pt]1\end{smallmatrix}\bigg)\,.
	\end{alignedat}
\end{equation}
These will be used explicitly or implicitly in the formulas that follow.

Just as with quadrilaterals, there are natural relationships between vertices, edges and faces, and the affine coordinates.
In fact, each vertex is linked to \textit{three} affine coordinates, each edge is linked to \textit{two} affine coordinates, and each face is linked to \textit{one} affine coordinate.
The linked affine coordinates take the value $1$ at the associated topological entity.
For example, vertex $1$, $v_1=(0,0,0)$, is linked to the affine coordinates $\mu_0(\xi_1)$, $\mu_0(\xi_2)$ and $\mu_0(\xi_3)$, edge 12 is linked to the affine coordinates $\mu_0(\xi_2)$ and $\mu_0(\xi_3)$, and face 1234 is linked to affine coordinate $\mu_0(\xi_3)$.

\subsubsection*{Exact Sequence}

Recall the 3D exact sequence for simply connected domains \eqref{eq:3D_exact_sequence}.
The corresponding discrete polynomial exact sequence is of the form 
\begin{equation}
	W^{p,q,r} \xrightarrow{\,\,\nabla\,\,} Q^{p,q,r} \xrightarrow{\nabla\times} V^{p,q,r} \xrightarrow{\nabla\cdot} Y^{p,q,r} \,,
\end{equation}
where the standard N\'{e}d\'{e}lec's spaces \cite{Nedelec80} of the first type for the hexahedron are utilized:
\begin{equation}
	\begin{aligned}
	W^{p,q,r} & = \mathcal{Q}^{p,q,r}= \mathcal{P}^p(\xi_1)\otimes \mathcal{P}^q (\xi_2)\otimes \mathcal{P}^r (\xi_3)\,,\\
	Q^{p,q,r} & = \mathcal{Q}^{p-1,q,r} \times\mathcal{Q}^{p,q-1,r}\times \mathcal{Q}^{p,q,r-1}\,,\\
	V^{p,q,r} & = \mathcal{Q}^{p,q-1,r-1} \times\mathcal{Q}^{p-1,q,r-1}\times \mathcal{Q}^{p-1,q-1,r}\,,\\
	Y^{p,q,r} & = \mathcal{Q}^{p-1,q-1,r-1}\,.
	\end{aligned}
\end{equation}
As with the quadrilateral, there is a natural anisotropy of the element, which has order $p$ in the $\xi_1$ direction, $q$ in the $\xi_2$ direction and $r$ in the $\xi_3$ direction. 
The hierarchy should be maintained in $p$, $q$ and $r$ separately.  
It will sometimes be convenient to refer to $p_a$ as the order in the $\xi_a$ direction, so that $p_1=p$, $p_2=q$ and $p_3=r$.

\subsection{\texorpdfstring{$H^1$}{H1} Shape Functions}
It will be clear that the $(p+1)(q+1)(r+1)$ shape functions defined in this section lie in $\mathcal{Q}^{p,q,r}$ and span the space.

The ideas in this section are the same as with the quadrilateral (see \S\ref{sec:Quad}) but in three dimensions.
This will simply translate to adding an extra blending function to account for the extra dimension.
Hence, the trace properties will not be analyzed in detail as they easily follow.

\subsubsection{\texorpdfstring{$H^1$}{H1} Vertices}
Without any delays, the vertex shape functions and their gradient are
\begin{equation}
	\begin{aligned}
		\phi^\mathrm{v}(\xi)&=\mu_a(\xi_1)\mu_b(\xi_2)\mu_c(\xi_3)\,,\\
		\nabla\phi^\mathrm{v}(\xi)&=\mu_a(\xi_1)\mu_b(\xi_2)\nabla\mu_c(\xi_3)+\mu_c(\xi_3)\mu_a(\xi_1)\nabla\mu_b(\xi_2)
			+\mu_b(\xi_2)\mu_c(\xi_3)\nabla\mu_a(\xi_1)\,,
	\end{aligned}
\end{equation}
for $a=0,1$, $b=0,1$ and $c=0,1$. 
There are a total of $8$ vertex functions (one for each vertex).

%

\subsubsection{\texorpdfstring{$H^1$}{H1} Edges}

Again, this is analogous to the quadrilateral case, but with an extra blending function. 
Take for example edge 12.
Then, the shape functions are
\begin{equation*}
	\phi_i^\mathrm{e}(\xi)=\underbrace{\mu_0(\xi_3)\mu_0(\xi_2)}_{\text{blend}}
		 \underbrace{\phi_i^\E(\underbrace{\vec{\mu}_{01}(\xi_1)}_{\text{project}})}_{\text{evaluate}}\,,
\end{equation*}
for $i=2,\ldots,p$. 
The projection being implied is:
\begin{equation*}
	(\xi_1,\xi_2,\xi_3)\;\longmapsto\;(\xi_1,\xi_2,0)\;\longmapsto\;(\xi_1,0,0)\,.
\end{equation*}
It is illustrated in Figure \ref{fig:HexaProjection}. 
It consists simply of finding the intersection $P''=(\xi_1,0,0)$ of the edge with the normal plane passing through the original point $P=(\xi_1,\xi_2,\xi_3)$.
Alternatively it can be interpreted in two steps.
First it is projected to the closest point $P'$ in an adjacent face, by using the normal to the face.
Once in the face, it is projected again to the desired edge using the traditional \textit{two} dimensional quadrilateral projection (see Figure \ref{fig:QuadProjection}).

\begin{figure}[!ht]
\begin{center}
\includegraphics[scale=0.6]{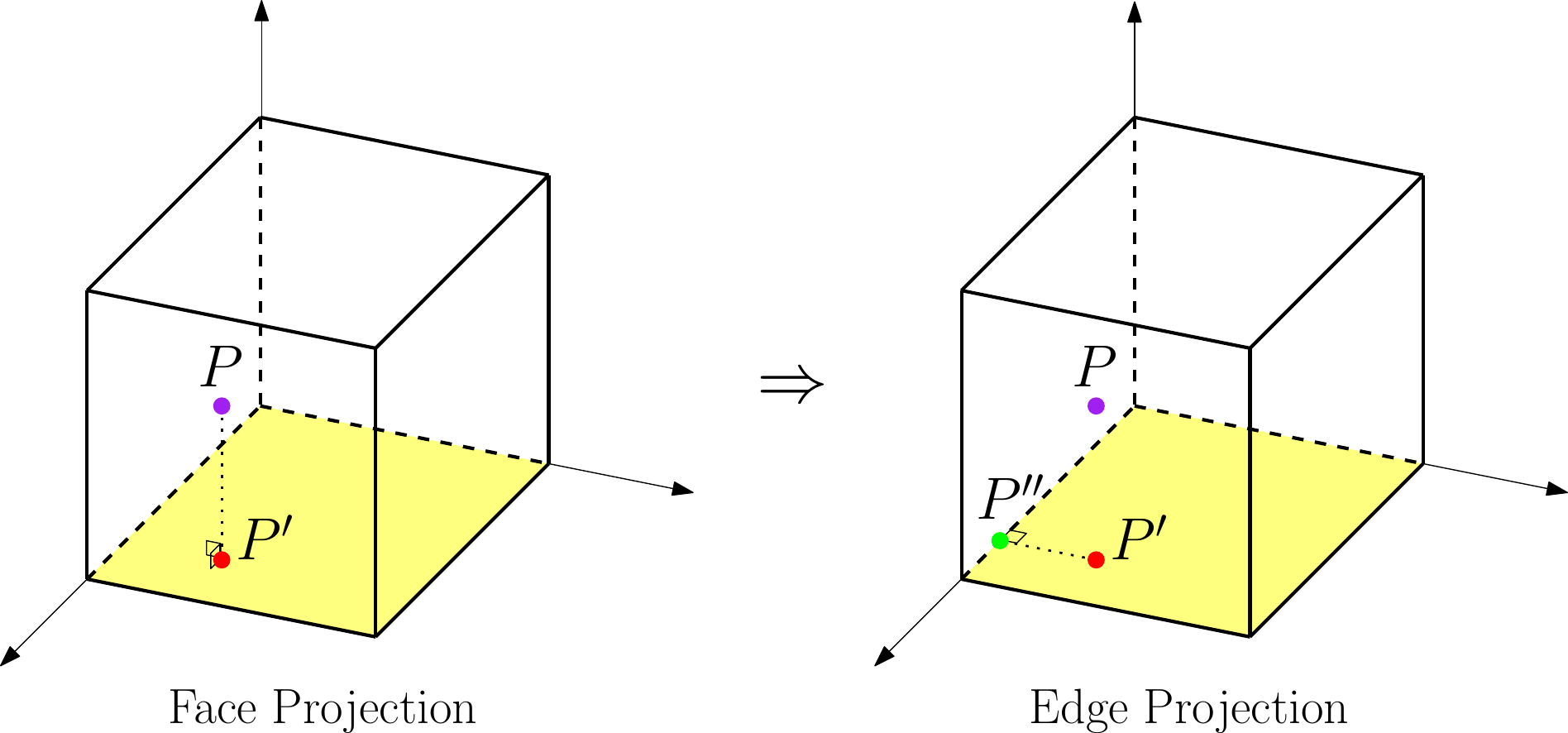}
\caption{Face projection from $P$ to $P'$ followed by an edge projection from $P'$ to $P''$.}
\label{fig:HexaProjection}
\end{center}
\end{figure}



The general edge shape functions and their gradient are
\begin{equation}
	\begin{aligned}
		\phi_i^\mathrm{e}(\xi)&=\mu_e(\xi_c)\mu_d(\xi_b)\phi_i^\E(\vec{\mu}_{01}(\xi_a))\,,\\
		\nabla\phi_i^\mathrm{e}(\xi)&=\mu_e(\xi_c)\mu_d(\xi_b)\nabla\phi_i^\E(\vec{\mu}_{01}(\xi_a))
			+\phi_i^\E(\vec{\mu}_{01}(\xi_a))\Big(\mu_e(\xi_c)\nabla\mu_d(\xi_b)+\mu_d(\xi_b)\nabla\mu_e(\xi_c)\Big)\,,
	\end{aligned}
	\label{eq:Hexaphigeneral}
\end{equation}
where $i=2,\ldots,p_a$, $(a,b,c)=(1,2,3),(2,3,1),(3,1,2)$, $d=0,1$ and $e=0,1$.
For example for edge 12, this would correspond to $(a,b,c)=(1,2,3)$, $d=0$, $e=0$ and $p_a=p$, for edge 23 it is $(a,b,c)=(2,3,1)$, $d=0$, $e=1$ and $p_a=q$, and so on. 
For each edge there are $p_a-1$ shape functions, leading to a total of $4(p-1)+4(q-1)+4(r-1)$ edge functions.


\subsubsection{\texorpdfstring{$H^1$}{H1} Faces}

Again, by adding as a blending factor the associated affine coordinate the construction becomes trivial. 
For example, for face 1234, the shape functions are
\begin{equation*}
	\phi_{ij}^\mathrm{f}(\xi)=\underbrace{\mu_0(\xi_3)}_{\text{blend}}
		\underbrace{\phi_{ij}^\square(\underbrace{\vec{\mu}_{01}(\xi_1),\vec{\mu}_{01}(\xi_2)}_{\text{project}})}_{\text{evaluate}}\,,
\end{equation*}
where $i=2,\ldots,p$ and $j=2,\ldots,q$. 
The projection is already illustrated in Figure \ref{fig:HexaProjection}, and simply consists of finding the intersection $P'=(\xi_1,\xi_2,0)$ of the normal to the face that passes through the original point $P=(\xi_1,\xi_2,\xi_3)$.
%

In general the face shape functions and their gradient are
\begin{equation}
	\begin{aligned}
		\phi_{ij}^\mathrm{f}(\xi)&=\mu_d(\xi_c)\phi_{ij}^\square(\vec{\mu}_{01}(\xi_a),\vec{\mu}_{01}(\xi_b))\,,\\
		\nabla\phi_{ij}^\mathrm{f}(\xi)&=\mu_d(\xi_c)\nabla\phi_{ij}^\square(\vec{\mu}_{01}(\xi_a),\vec{\mu}_{01}(\xi_b))
			+\phi_{ij}^\square(\vec{\mu}_{01}(\xi_a),\vec{\mu}_{01}(\xi_b))\nabla\mu_d(\xi_c)\,,
	\end{aligned}
	\label{eq:Hexaphifacegeneral}
\end{equation}
where $i=2,\ldots,p_a$, $j=2,\ldots,p_b$, $(a,b,c)=(1,2,3),(2,3,1),(3,1,2)$, and $d=0,1$.
For example for face 1234, this would correspond to $(a,b,c)=(1,2,3)$, $d=0$, $p_a=p$ and $p_b=q$, for face 1265 it is $(a,b,c)=(3,1,2)$, $d=0$, $p_a=r$ and $p_b=p$, and so on. 
For each face there are $(p_a-1)(p_b-1)$ shape functions, for a total of $2(p-1)(q-1)+2(q-1)(r-1)+2(r-1)(p-1)$ face functions.

\subsubsection{\texorpdfstring{$H^1$}{H1} Interior Bubbles}

These are constructed like the face functions, but by using edge bubbles $\phi_k^\E$ instead of the linear blending factor $\mu_d$. 
This will ensure the necessary vanishing trace properties.

The interior bubbles and their gradient are
\begin{equation}
	\begin{aligned}
		\phi_{ijk}^\mathrm{b}(\xi)&=L_i(\mu_1(\xi_1))L_j(\mu_1(\xi_2))L_k(\mu_1(\xi_3))
			=\phi_{ij}^\square(\vec{\mu}_{01}(\xi_1),\vec{\mu}_{01}(\xi_2))\phi_k^\E(\vec{\mu}_{01}(\xi_3))\,,\\
		\nabla\phi_{ijk}^\mathrm{b}(\xi)&=
			\phi_{ij}^\square(\vec{\mu}_{01}(\xi_1),\vec{\mu}_{01}(\xi_2))\nabla\phi_k^\E(\vec{\mu}_{01}(\xi_3))
				+\phi_k^\E(\vec{\mu}_{01}(\xi_3))\nabla\phi_{ij}^\square(\vec{\mu}_{01}(\xi_1),\vec{\mu}_{01}(\xi_2))\,,
	\end{aligned}
\end{equation}
for $i=2,\ldots,p$, $j=2,\ldots,q$ and $k=2,\ldots,r$. 
Clearly there will be $(p-1)(q-1)(r-1)$ interior bubbles.

\subsection{\texorpdfstring{$H(\mathrm{curl})$}{Hcurl} Shape Functions}

It will be clear that the $p(q+1)(r+1)+q(r+1)(p+1)+r(p+1)(q+1)$ linearly independent shape functions span $Q^{p,q,r}=\mathcal{Q}^{p-1,q,r}\times\mathcal{Q}^{p,q-1,r}\times\mathcal{Q}^{p,q,r-1}$ as required.

The ideas in this section are the same as with the quadrilateral but in three dimensions. 
This will simply translate to adding an extra blending function to account for the extra dimension.
The structure of projecting, evaluating and blending still holds in $H(\mathrm{curl})$, and the projections (and even the blending functions) are the same as those in $H^1$.
As with $H^1$, the analysis of the trace properties will be superfluous.

\subsubsection{\texorpdfstring{$H(\mathrm{curl})$}{Hcurl} Edges}
These will just be the quadrilateral edge functions with the extra blending factor. They are
\begin{equation}
	\begin{aligned}
		E_i^\mathrm{e}(\xi)&=\mu_e(\xi_c)\mu_d(\xi_b)E_i^\E(\vec{\mu}_{01}(\xi_a))\,,\\
		\nabla\times E_i^\mathrm{e}(\xi)&=
			\Big(\mu_e(\xi_c)\nabla\mu_d(\xi_b)+\mu_d(\xi_b)\nabla\mu_e(\xi_c)\Big)\times E_i^\E(\vec{\mu}_{01}(\xi_a))\,,
	\end{aligned}
\end{equation}
where $i=2,\ldots,p_a$, $(a,b,c)=(1,2,3),(2,3,1),(3,1,2)$, $d=0,1$ and $e=0,1$.
Notice the form is very similar to that of edge $H^1$ functions. 
For each edge there are $p_a$ shape functions, giving a total of $4p+4q+4r$ edge functions.

\subsubsection{\texorpdfstring{$H(\mathrm{curl})$}{Hcurl} Faces}

The pattern goes on, but this time with the two families.
There are a grand total of $2(p(q-1)+q(p-1))+2(q(r-1)+r(q-1))+2(r(p-1)+p(r-1))$ face shape functions.

\subparagraph{Family I:}
The shape functions and their curl are
\begin{equation}
	\begin{aligned}
		E_{ij}^\mathrm{f}(\xi)&=\mu_d(\xi_c)E_{ij}^\square(\vec{\mu}_{01}(\xi_a),\vec{\mu}_{01}(\xi_b))\,,\\
		\nabla\times E_{ij}^\mathrm{f}(\xi)&=\mu_d(\xi_c)\nabla\times E_{ij}^\square(\vec{\mu}_{01}(\xi_a),\vec{\mu}_{01}(\xi_b))
			+\nabla\mu_d(\xi_c)\times E_{ij}^\square(\vec{\mu}_{01}(\xi_a),\vec{\mu}_{01}(\xi_b))\,,
	\end{aligned}
\end{equation}
where $i=0,\ldots,p_a-1$, $j=2,\ldots,p_b$, $(a,b,c)=(1,2,3),(2,3,1),(3,1,2)$, and $d=0,1$.
For each face there are $p_a(p_b-1)$ shape functions in this family.

\subparagraph{Family II:}
The shape functions and their curl are
\begin{equation}
	\begin{aligned}
		E_{ij}^\mathrm{f}(\xi)&=\mu_d(\xi_c)E_{ij}^\square(\vec{\mu}_{01}(\xi_b),\vec{\mu}_{01}(\xi_a))\,,\\
		\nabla\times E_{ij}^\mathrm{f}(\xi)&=\mu_d(\xi_c)\nabla\times E_{ij}^\square(\vec{\mu}_{01}(\xi_b),\vec{\mu}_{01}(\xi_a))
			+\nabla\mu_d(\xi_c)\times E_{ij}^\square(\vec{\mu}_{01}(\xi_b),\vec{\mu}_{01}(\xi_a))\,,
	\end{aligned}
\end{equation}
where $i=0,\ldots,p_b-1$, $j=2,\ldots,p_a$, $(a,b,c)=(1,2,3),(2,3,1),(3,1,2)$, and $d=0,1$.
Again, recall the only difference with the first family is that the entries $(\vec{\mu}_{01}(\xi_a),\vec{\mu}_{01}(\xi_b))$ and their corresponding order $(p_a,p_b)$, are permuted.
For each face there are $p_b(p_a-1)$ shape functions in this family.

\subsubsection{\texorpdfstring{$H(\mathrm{curl})$}{Hcurl} Interior Bubbles}

These can be constructed by using $H^1$ edge bubbles $\phi_k^\E$ as blending functions instead of the linear blending $\mu_d$ in the expressions for the face shape functions.
All permutations of $(a,b,c)$ leading to linearly independent functions must be considered.
In the end, three famillies corresponding to the cyclic permutations of $(1,2,3)$ comprise the interior bubbles.

The interior functions and their curl are
\begin{equation}
	\begin{aligned}
		E_{ijk}^\mathrm{b}(\xi)&=\phi_k^\E(\vec{\mu}_{01}(\xi_c))E_{ij}^\square(\vec{\mu}_{01}(\xi_a),\vec{\mu}_{01}(\xi_b))\,,\\
		\nabla\times E_{ijk}^\mathrm{b}(\xi)&=\phi_k^\E(\vec{\mu}_{01}(\xi_c))\nabla\!\times\!
			E_{ij}^\square(\vec{\mu}_{01}(\xi_a),\vec{\mu}_{01}(\xi_b))
				+\nabla\phi_k^\E(\vec{\mu}_{01}(\xi_c))\!\times\! E_{ij}^\square(\vec{\mu}_{01}(\xi_a),\vec{\mu}_{01}(\xi_b))\,.
	\end{aligned}
\end{equation}
for $i=0,\ldots,p_a-1$, $j=2,\ldots,p_b$, $k=2,\ldots,p_c$, and $(a,b,c)=(1,2,3),(2,3,1),(3,1,2)$.
There will be a grand total of $p(q-1)(r-1)+q(r-1)(p-1)+r(p-1)(q-1)$ interior bubble functions.

\subsection{\texorpdfstring{$H(\mathrm{div})$}{Hdiv} Shape Functions}

It will be clear that all shape functions lie in $V^{p,q,r}=\mathcal{Q}^{p,q-1,r-1}\times\mathcal{Q}^{p-1,q,r-1}\times \mathcal{Q}^{p-1,q-1,r}$, which has dimension $rq(p+1)+pr(q+1)+qp(r+1)$.

This is the first time that the space $H(\mathrm{div})$ is tackled in 3D.
As expected, it requires of some analysis to develop the correct structure at first, but afterwards one can proceed very similarly as the previous spaces.

\subsubsection{\texorpdfstring{$H(\mathrm{div})$}{Hdiv} Faces}


First, recall from \S\ref{sec:dimensionalhierarchy} that the normal trace of the $H(\mathrm{div})$ face functions should be a 2D $L^2$ face function.
For the purposes of motivation, take for instance face 1234.
From \eqref{eq:QuadL2Functions}, the 2D $L^2$ face functions are of the form $[P_i](\vec{\mu}_{01}(\xi_1))[P_j](\vec{\mu}_{01}(\xi_2))$.
Meanwhile, the normal vector to face 1234 is $(0,0,1)=\nabla\mu_1(\xi_1)\times\nabla\mu_1(\xi_2)$.
When coupled with a blending factor, $\mu_0(\xi_3)$, representing a linear decay (like that of $H^1$), this suggests,
\begin{equation*}
    V_i^\mathrm{e}(\xi)=\mu_0(\xi_3)[P_i](\vec{\mu}_{01}(\xi_1))[P_j](\vec{\mu}_{01}(\xi_2))\,\nabla\mu_1(\xi_1)\times
    	\nabla\mu_1(\xi_2)
    		=\mu_0(\xi_3)E_i^\E(\vec{\mu}_{01}(\xi_1))\times E_j^\E(\vec{\mu}_{01}(\xi_2))\,,
\end{equation*}
for $i=0,\ldots,p-1$ and $j=0,\ldots,q-1$.
In fact, this expression makes a lot of sense, since a function normal to the face should be perpendicular to the two tangential $H(\mathrm{curl})$ edge functions. 
The cross product then seems like a natural idea.
Indeed, one can laboriously check that the desired normal trace properties are satified at all faces.
This motivates the definition of a new ancillary operator presented next.

%

\begin{definition*}
Let $(s_0,s_1)$ and $(t_0,t_1)$ be two pairs of coordinates which are arbitrary functions of some spatial variable in $\R^N$, with $N=3$. Let $p_s$ be the order in the $(s_0,s_1)$ coordinates, and $p_t$ be the order in the $(t_0,t_1)$ coordinates. Then
\begin{equation}
	V_{ij}^{\square}(s_0,s_1,t_0,t_1)=E_i^\E(s_0,s_1)\times E_j^\E(t_0,t_1)\,,
\end{equation}
for $i=0,\ldots,p_s-1$ and $j=0,\ldots,p_t-1$. Their divergence, understood in $\R^N$, is
\begin{equation}
	\nabla\cdot V_{ij}^{\square}(s_0,s_1,t_0,t_1)=E_j^\E(t_0,t_1)\cdot(\nabla\times E_i^\E(s_0,s_1))
		-E_i^\E(s_0,s_1)\cdot(\nabla\times E_j^\E(t_0,t_1)) \,.
\end{equation}
\end{definition*}

Record also the next useful remark.
\begin{remark}
Let $\mu_0^{(0)}=1-\mu_1^{(0)}$ and $\mu_0^{(1)}=1-\mu_1^{(1)}$, where $\mu_1^{(0)}$ and $\mu_1^{(1)}$ are arbitrary functions of some spatial variable in $\R^N$ with $N=3$, and where $p_{(0)}$ and $p_{(1)}$ are the orders in the coordinates $(\mu_0^{(0)},\mu_1^{(0)})$ and $(\mu_0^{(1)},\mu_1^{(1)})$ respectively. Then for all $i=0,\ldots,p_{(0)}-1$, and $j=0,\ldots,p_{(1)}-1$,
\begin{equation}
\begin{aligned}
	V_{ij}^\square(\mu_0^{(0)},\mu_1^{(0)},\mu_0^{(1)},\mu_1^{(1)})
		&=P_i(\mu_1^{(0)})P_j(\mu_1^{(1)})\nabla\mu_1^{(0)}\times\nabla\mu_1^{(1)}\,,\\
			\nabla\cdot V_{ij}^\square(\mu_0^{(0)},\mu_1^{(0)},\mu_0^{(1)},\mu_1^{(1)})&=0\,.
\end{aligned}
\label{eq:Vijsimplified}
\end{equation}
\end{remark}

Finally, the shape functions and their divergence are
\begin{equation}
	\begin{aligned}
		V_{ij}^\mathrm{f}(\xi)&=\mu_d(\xi_c)V_{ij}^{\square}(\vec{\mu}_{01}(\xi_a),\vec{\mu}_{01}(\xi_b))\,,\\
		\nabla\cdot V_{ij}^\mathrm{f}(\xi)&=\nabla\mu_d(\xi_c)\cdot V_{ij}^{\square}(\vec{\mu}_{01}(\xi_a),\vec{\mu}_{01}(\xi_b))\,,
	\end{aligned}	
\end{equation}
where $i=0,\ldots,p_a-1$, $j=0,\ldots,p_b-1$, $(a,b,c)=(1,2,3),(2,3,1),(3,1,2)$, and $d=0,1$.
There are $p_ap_b$ face functions for each face, leading to a total of $2pq+2qr+2rp$ face functions.

\subsubsection{\texorpdfstring{$H(\mathrm{div})$}{Hdiv} Interior Bubbles}
Using the same reasoning as with $H(\mathrm{curl})$ interior bubbles, there will essentially be three families of bubbles corresponding to the cyclic permutations of $(1,2,3)$.
The interior bubbles and their divergence are
\begin{equation}
	\begin{aligned}
		V_{ijk}^\mathrm{b}(\xi)&=\phi_k^\E(\vec{\mu}_{01}(\xi_c))V_{ij}^\square(\vec{\mu}_{01}(\xi_a),\vec{\mu}_{01}(\xi_b))\,,\\
		\nabla\cdot V_{ij}^\mathrm{f}(\xi)&=
			\nabla\phi_k^\E(\vec{\mu}_{01}(\xi_c))\cdot V_{ij}^\square(\vec{\mu}_{01}(\xi_a),\vec{\mu}_{01}(\xi_b))\,.
	\end{aligned}
\end{equation}
where $i=0,\ldots,p_a-1$, $j=0,\ldots,p_b-1$, $k=2,\ldots,p_c$, and $(a,b,c)=(1,2,3),(2,3,1),(3,1,2)$.
There are a grand total of $pq(r-1)+qr(p-1)+rp(q-1)$ bubbles.

\subsection{\texorpdfstring{$L^2$}{L2} Shape Functions}

As expected, they are the tensor products of the 1D $L^2$ shape functions, and there are $pqr$ such functions spanning $Y^{p,q,r}=\mathcal{Q}^{p-1,q-1,r-1}$.

\subsubsection{\texorpdfstring{$L^2$}{L2} Interior}

The coordinate free $L^2$ interior functions for the hexahedron are
\begin{equation}
    \psi_{ijk}^\mathrm{b}(\xi)=[P_i](\vec{\mu}_{01}(\xi_1))[P_j](\vec{\mu}_{01}(\xi_2))[P_k](\vec{\mu}_{01}(\xi_3))
    	(\nabla\mu_1(\xi_1)\!\!\times\!\!\nabla\mu_1(\xi_2))\!\cdot\!\nabla\mu_1(\xi_3)\,,
\end{equation}
for $i=0,\ldots,p-1$, $j=0,\ldots,q-1$ and $k=0,\ldots,r-1$. 
There are $pqr$ interior functions.
%

\subsection{Orientations}
\label{sec:HexaOrientations}

In 3D, both edges \textit{and} faces have orientations associated to them, and they need to be considered to ensure full compatibility of shape functions along adjacent elements.
Fortunately, these issues are handled almost effortlessly due to the structure of the formulas for the shape functions, and the local-to-global permutation functions, $\sigma_\oo^\E$, $\sigma_\oo^\square$ and $\sigma_\oo^\Tri$.
In what follows, it is assumed that \S\ref{sec:fulledgeorientations}, \S\ref{sec:QuadOrientations} and \S\ref{sec:TriaOrientations} have been covered.

To construct orientation embedded shape functions, the first step is to predefine a set of \textit{local} orientations for each edge and face at the master element level.
After they are defined, these remain fixed.
The next step is to find the associated \textit{locally ordered} tuples of affine coordinates representing those local orientations.
Once these are found, the orientation embedded shape functions are merely the usual edge and face functions, but with their respective ancillary operator being precomposed with the appropriate local-to-global permutation function.
The only ``burden'' is then to find the \textit{locally ordered} tuples.
This is shown next for the hexahedron.

\begin{figure}[!ht]
\begin{center}
\includegraphics[scale=0.5]{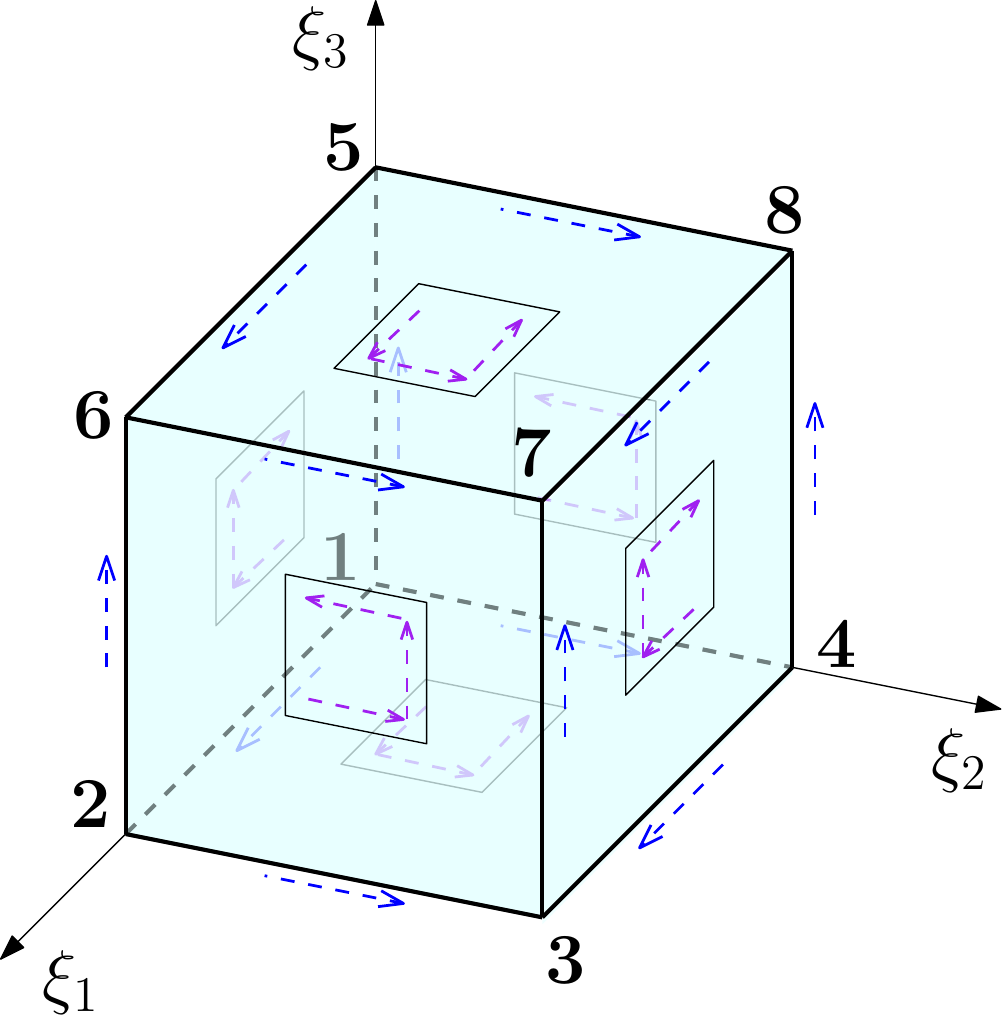}
\caption{Master hexahedron with numbered vertices \textit{and} local edge and face orientations.}
\label{fig:MasterHexaOrientations}
\end{center}
\end{figure}

Figure \ref{fig:MasterHexaOrientations} shows the master hexahedron along with a schematic representing all predefined \textit{local} edge and face orientations.
They represent the $\oo=0$ case.
They are \textit{our} choices for the local orientations (which in fact are the ``lexicographic'' orientations), but others may choose different local orientations to represent their $\oo=0$ case.

To find the locally ordered tuples, the key is being aware of the relationships between the vertices and the affine coordinates.
To illustrate this take as an example edge 12 and face 1234.

Edge 12 is composed of the vertices $v_1$ and $v_2$. 
Here, $v_1$ is linked to $\mu_0(\xi_1)$, $\mu_0(\xi_2)$,$\mu_0(\xi_3)$, while $v_2$ is linked to $\mu_1(\xi_1)$, $\mu_0(\xi_2)$ and $\mu_0(\xi_3)$.
The only difference between the the two vertices is that $v_1$ is linked to $\mu_0(\xi_1)$, while $v_2$ is linked to $\mu_1(\xi_1)$.
Now, the local orientation is represented by the local vertex-ordering $v_1\tdashto v_2$, so quite simply the locally ordered pair is $\vec{\mu}_{01}(\xi_1)=(\mu_0(\xi_1),\mu_0(\xi_1))$ (if the local ordering was $v_2\tdashto v_1$, then the pair would be $\vec{\mu}_{10}(\xi_1)$).
Hence, the orientation embedded edge 12 shape functions in $H^1$ with their gradient are
\begin{equation*}
	\begin{aligned}
		\phi_i^\mathrm{e}(\xi)&=\mu_0(\xi_3)\mu_0(\xi_2)\phi_i^\E(\sigma_\oo^\E(\vec{\mu}_{01}(\xi_1)))\,,\\
		\nabla\phi_i^\mathrm{e}(\xi)&=\mu_0(\xi_c)\mu_0(\xi_b)\nabla\phi_i^\E(\sigma_\oo^\E(\vec{\mu}_{01}(\xi_1)))
			+\phi_i^\E(\sigma_\oo^\E(\vec{\mu}_{01}(\xi_1)))\Big(\mu_0(\xi_3)\nabla\mu_0(\xi_2)+\mu_0(\xi_2)\nabla\mu_0(\xi_3)\Big)\,,
	\end{aligned}
\end{equation*}
where $i=2,\ldots,p$. 
The same applies to the $H(\mathrm{curl})$ edge 12 shape functions and their curl.
Clearly the approach is analogous with any other edge.

Face 1234 is composed of the vertices $v_1$, $v_2$, $v_3$ and $v_4$.
Here, the final goal is to find a locally ordered quadruple composed of two pairs.
The local vertex-ordering corresponding to the local orientation of face 1234 is $v_1\tdashto v_2\tdashto v_3\tdashto v_4$.
All one needs to do is to take the first two elements of the list, $v_1\tdashto v_2$, and the second and third components of the list, namely $v_2\tdashto v_3$.
The former will represent the \textit{first} pair in the quadruple, while the latter represent the \textit{second} pair in the quadruple.
Then one proceeds as if these where edges, so that $v_1\tdashto v_2$ is associated to $\vec{\mu}_{01}(\xi_1)$, while $v_2\tdashto v_3$ is associated to $\vec{\mu}_{01}(\xi_2)$.
Finally the locally ordered quadruple is then the ordered succession of these two pairs, $(\vec{\mu}_{01}(\xi_1),\vec{\mu}_{01}(\xi_2))$.
Hence, the orientation embedded face 1234 shape functions in $H^1$ with their gradient are
\begin{equation*}
	\begin{aligned}
		\phi_{ij}^\mathrm{f}(\xi)&=\mu_0(\xi_3)\phi_{ij}^\square(\sigma_\oo^\square(\vec{\mu}_{01}(\xi_1),\vec{\mu}_{01}(\xi_2)))\,,\\
		\nabla\phi_{ij}^\mathrm{f}(\xi)&=\mu_0(\xi_3)\nabla\phi_{ij}^\square(\sigma_\oo^\square(\vec{\mu}_{01}(\xi_1),\vec{\mu}_{01}(\xi_2)))
			+\phi_{ij}^\square(\sigma_\oo^\square(\vec{\mu}_{01}(\xi_1),\vec{\mu}_{01}(\xi_2)))\nabla\mu_0(\xi_3)\,,
	\end{aligned}
\end{equation*}
where $i=2,\ldots,p_a$, $j=2,\ldots,p_b$, where $(p_a,p_b)$ are the orders of the first and second coordinate pairs in the quadruple $\sigma_\oo^\square(\vec{\mu}_{01}(\xi_1),\vec{\mu}_{01}(\xi_2))$.
The same applies to the $H(\mathrm{curl})$ and $H(\mathrm{div})$ face 1234 shape functions and their differential forms.
Once again, the approach is analogous with any other face.

%% file: tet.tex
\section{Tetrahedron}
\label{sec:Tet}

\begin{figure}[!ht]
\begin{center}
\includegraphics[scale=0.5]{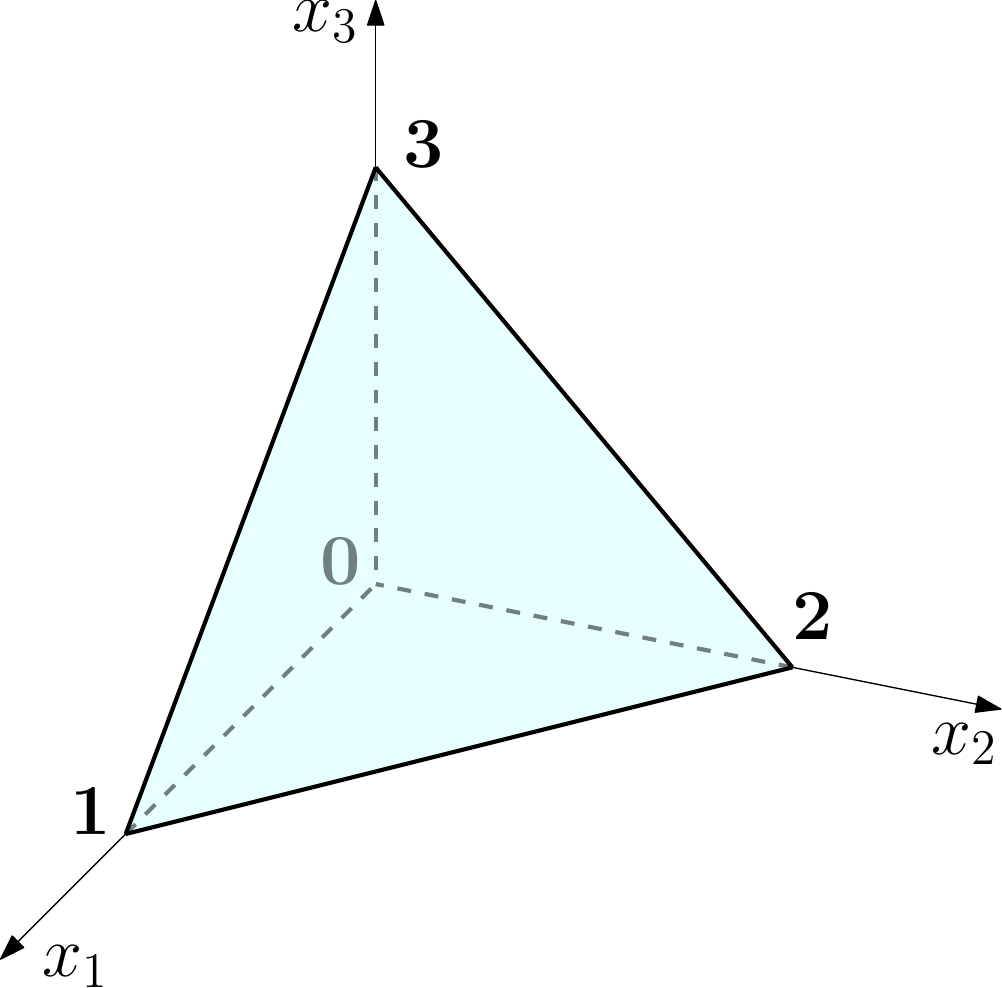}
\caption{Master tetrahedron with numbered vertices.}
\label{fig:MasterTet}
\end{center}
\end{figure}

The 3D simplex is the tetrahedron. 
The master element for tetrahedra is illustrated in Figure \ref{fig:MasterTet} in the $x=(x_1,x_2,x_3)$ space.
More precisely, it is the set $\{x\in\R^3:x_1>0,x_2>0,x_3>0,x_1+x_2+x_3<1\}$.

Denote vertex $a$ by $v_a$, so that $v_0=(0,0,0)$, $v_1=(1,0,0)$, $v_2=(0,1,0)$ and $v_3=(0,0,1)$.
As described in \S\ref{sec:affinecoordinates}, the 3D affine coordinates, $\lambda_0$, $\lambda_1$, $\lambda_2$ and $\lambda_3$, are easily calculated for this master tetrahedron:
\begin{equation}
	\lambda_0(x)=1-x_1-x_2-x_3\,,\qquad
	\lambda_1(x)=x_1\,,\qquad
	\lambda_2(x)=x_2\,,\qquad
	\lambda_3(x)=x_3\,.
\end{equation}
Their gradients are
\begin{equation}
	\nabla\lambda_0(x)=\bigg(\begin{smallmatrix}-1\\[2pt]-1\\[2pt]-1\end{smallmatrix}\bigg)\,,\qquad
	\nabla\lambda_1(x)=\bigg(\begin{smallmatrix}1\\[2pt]0\\[2pt]0\end{smallmatrix}\bigg)\,,\qquad
	\nabla\lambda_2(x)=\bigg(\begin{smallmatrix}0\\[2pt]1\\[2pt]0\end{smallmatrix}\bigg)\,,\qquad
	\nabla\lambda_3(x)=\bigg(\begin{smallmatrix}0\\[2pt]0\\[2pt]1\end{smallmatrix}\bigg)\,.
\end{equation}
These will be used explicitly or implicitly in what follows.

Just like the triangle and the segment, the tetrahedron has a very natural correspondence of its vertices and its affine coordinates.
Quite simply, each vertex $v_a$ is linked to the affine coordinate $\lambda_a$, for $a=0,1,2,3$.
Indeed, $\lambda_a$ takes the value $1$ at the associated vertex.

\subsubsection*{Exact Sequence}

As with the hexahedron, the tetrahedron will have a 3D discrete polynomial exact sequence that represents the continuous exact sequence \eqref{eq:3D_exact_sequence}. 
It is
\begin{equation}
\mathcal{P}^p\,\xrightarrow{\nabla}\,{\mathcal{N}}^p\,\xrightarrow{\nabla\times}\,{\mathcal{RT}}^p\,\xrightarrow{\nabla\cdot}
	\,\mathcal{P}^{p-1} \, ,
\end{equation}
where $\mathcal{P}^p =\mathcal{P}^p(x_1,x_2,x_3)$ is the space of polynomials of total order $p$.
Meanwhile, the N\'{e}d\'{e}lec and Raviart-Thomas spaces for the tetrahedron where already defined by \eqref{eq:NedelecSpace} and \eqref{eq:RaviartThomasSpace}, where $N=3$ in those definitions.

Like the triangle, the tetrahedron sequence has an overall drop in polynomial order of one, which makes it compatible with the construction of the hexahedron.
Also, as noted before, all of the spaces in the exact sequence are invariant under affine transformations.

\subsection{\texorpdfstring{$H^1$}{H1} Shape Functions}
%

It will be clear that all the $\frac{1}{6}(p+3)(p+2)(p+1)$ shape functions lie in $\mathcal{P}^{p}$ and span the space.

The ideas in this section are completely parallel to those presented for the triangle (see \S\ref{sec:Tri}) but in three dimensions.
Therefore, the trace properties will not be analyzed in detail, since they follow analogously.

\subsubsection{\texorpdfstring{$H^1$}{H1} Vertices}

The vertex shape functions and their gradients are simply the affine coordinates themselves,
\begin{equation}
	\phi^\mathrm{v}(x)= \lambda_a(x)\,,\qquad\quad
	\nabla\phi^\mathrm{v}(x)=\nabla\lambda_a(x)\,,
\end{equation}
for $a=0,1,2,3$.
There are a total of $4$ vertex functions (one for each vertex).


\subsubsection{\texorpdfstring{$H^1$}{H1}  Edges}

These are treated just like triangle edges. 
Hence, one can recur to $\phi_i^\E$ directly. 
Take for instance edge 01. In this case, the shape functions are simply
\begin{equation*}
      \phi_i^\mathrm{e}(x)=\phi_i^\E(\vec{\lambda}_{01}(x))
    	=\underbrace{(\lambda_0(x)+\lambda_1(x))^i}_{\text{blend}}
    		\underbrace{\phi_i^\E\Big(\underbrace{\textstyle{\frac{\lambda_0(x)}{\lambda_0(x)+\lambda_1(x)}},
    			\textstyle{\frac{\lambda_1(x)}{\lambda_0(x)+\lambda_1(x)}}}_{\text{project}}\Big)}_{\text{evaluate}}\,,
\end{equation*}
for $i=2,\ldots,p$. 
The projection being implied is
\begin{equation*}
	(x_1,x_2,x_3)\;\longmapsto\;(\textstyle{\frac{x_1}{1-x_3}},\textstyle{\frac{x_2}{1-x_3}},0)
		\;\longmapsto\;(\textstyle{\frac{x_1}{1-x_2-x_3}},0,0)\,.
\end{equation*}
It consists of finding the intersection $P''=(\frac{x_1}{1-x_2-x_3},0,0)$ of the edge with the projecting plane passing through the original point $P=(x_1,x_2,x_3)$ and the opposite nonadjacent edge.
It is illustrated in Figure \ref{fig:TetProjection}. 
Alternatively it can be interpreted in two steps.
First it is projected to a point $P'$ in an adjacent face, using the projecting line passing through $P$ and the disjoint vertex to the face.
Once in the face, it is projected again to the desired edge using the traditional \textit{two} dimensional triangle projection (see Figure \ref{fig:TriangleProjection}).

\begin{figure}[!ht]
\begin{center}
\includegraphics[scale=0.6]{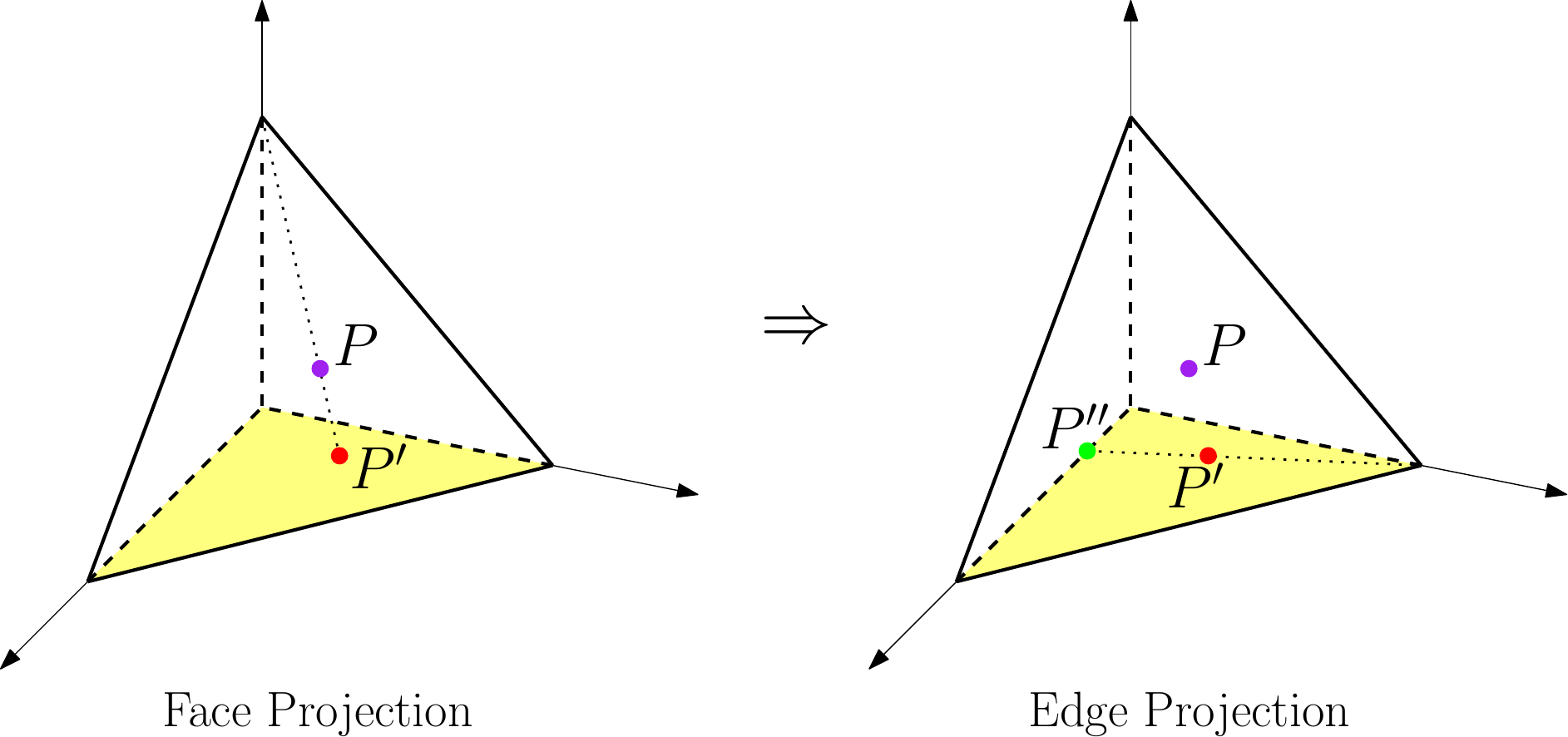}
\caption{Face projection from $P$ to $P'$ followed by an edge projection from $P'$ to $P''$.}
\label{fig:TetProjection}
\end{center}
\end{figure}

More generally, the edge functions and their gradients are
\begin{equation}
	\phi_i^\mathrm{e}(x)=\phi_i^\E(\vec{\lambda}_{ab}(x))\,,\qquad\quad
		\nabla\phi_i^\mathrm{e}(x)=\nabla\phi_i^\E(\vec{\lambda}_{ab}(x))\,,
	\label{eq:Tetphigeneral}
\end{equation}
with $i=2,\ldots,p$, $0\leq a<b\leq3$. 
There are a total of $p-1$ edge functions for every given edge, leading to a total of $6(p-1)$ edge functions.

\subsubsection{\texorpdfstring{$H^1$}{H1} Faces}

The construction of these shape functions follows simply by homogenizing the $H^1$ triangle face bubbles, since this will represent a polynomial extension preserving the desired vanishing properties.
This explains the definition of $\phi_{ij}^\Tri$ in terms of homogenized polynomials.
As an example, consider face 012, where the shape functions are
\begin{equation*}
	\phi_{ij}^\mathrm{f}(x)=\phi_{ij}^\Tri(\vec{\lambda}_{012}(x))=
		\underbrace{(\lambda_0(x)+\lambda_1(x)+\lambda_2(x))^{i+j}}_{\text{blend}}\underbrace{\phi_{ij}^\Tri
			\Big(\underbrace{\textstyle{\frac{1}{\lambda_0(x)+\lambda_1(x)+\lambda_2(x)}}
				\vec{\lambda}_{012}(x)}_{\text{project}}\Big)}_{\text{evaluate}}\,,
\end{equation*}
for $i\geq2$ and $j\geq1$.
The projection is already illustrated in Figure \ref{fig:TetProjection} and consists of finding the intersection $P'=(\frac{x_1}{1-x_3},\frac{x_2}{1-x_3},0)$ of the face with the projecting line passing through the original point $P=(x_1,x_2,x_3)$ and the opposite vertex to the face. 

The full collection of shape functions and their gradient is
\begin{equation}
	\phi_{ij}^\mathrm{f}(x)=\phi_{ij}^\Tri(\vec{\lambda}_{abc}(x))\,,\qquad\quad
			\nabla\phi_{ij}^\mathrm{f}(x)=\nabla\phi_{ij}^\Tri(\vec{\lambda}_{abc}(x))\,,
			\label{eq:H1Tetfaces}
\end{equation}
where $i\geq2$, $j\ge1$, $n=i+j=3,\ldots,p$, and $0\leq a<b<c\leq3$. 
There are $\frac{1}{2}(p-1)(p-2)$ shape functions for each face, leading to a total of $2(p-1)(p-2)$ face functions.

\subsubsection{\texorpdfstring{$H^1$}{H1} Interior Bubbles}

The tetrahedron bubbles are given by blending a face shape function with a polynomial of complementing order which vanishes on the remaining face. 
As with triangles, it is carefully chosen as a Jacobi polynomial $L_k^{2(i+j)}$. 

The interior functions and their gradient are
\begin{equation}
	\begin{aligned}
		\phi_{ijk}^\mathrm{b}(x)&=\phi_{ij}^\Tri(\vec{\lambda}_{012}(x))[L_k^{2(i+j)}](\vec{\mu}_{01}(\lambda_3(x)))\,,\\
			\nabla\phi_{ijk}^\mathrm{b}(x)&=[L_k^{2(i+j)}](\vec{\mu}_{01}(\lambda_3(x)))\nabla\phi_{ij}^\Tri(\vec{\lambda}_{012}(x))
				+\phi_{ij}^\Tri(\vec{\lambda}_{012}(x))\nabla[L_k^{2(i+j)}](\vec{\mu}_{01}(\lambda_3(x)))\,,
	\end{aligned}
\end{equation}
where $i\geq2$, $j\geq1$, $k\geq1$ and $n=i+j+k=4,\ldots,p$, and where $\vec{\mu}_{01}(\lambda_3(x))=(1-\lambda_3(x),\lambda_3(x))$.
There are $\frac{1}{6}(p-1)(p-2)(p-3)$ interior shape functions in total.

%
%
%
%
%

\subsection{\texorpdfstring{$H(\mathrm{curl})$}{Hcurl} Shape Functions}


The dimension of $\mathcal{N}^p$ in three dimensions is $\frac{1}{2}p(p+2)(p+3)$.
A careful count of the linearly independent shape functions to be presented throughout this section will coincide with that dimension. 
Showing that the functions constructed are in $\mathcal{N}^p$ follows from Lemma \ref{lemma:curl}.
The constructions are all analogous to those of the triangle and simply require of an extra extension which is naturally provided by homogenization.

\subsubsection{\texorpdfstring{$H(\mathrm{curl})$}{Hcurl} Edges}

These are just the same as in the triangle case, but using three dimensional affine coordinates for the homogenization.
For example, for edge 01, the shape functions are
\begin{equation*}
	\begin{aligned}
		E_i^\mathrm{e}(x)&=E_i^\E(\vec{\lambda}_{01}(x))=
			[P_i](\vec{\lambda}_{01}(x))\Big(\lambda_0(x)\nabla\lambda_1(x)-\lambda_1(x)\nabla\lambda_0(x)\Big)\\
    	&=(\lambda_0(x)+\lambda_1(x))^i
    [P_i]\Big(\textstyle{\frac{\lambda_0(x)}{\lambda_0(x)+\lambda_1(x)}},\textstyle{\frac{\lambda_1(x)}{\lambda_0(x)+\lambda_1(x)}}\Big)
    			E_0^\E(\vec{\lambda}_{01}(x))\\
    	&=\underbrace{(\lambda_0(x)+\lambda_1(x))^{i+2}}_{\text{blend}}
    		\underbrace{E_i^\E\Big(\underbrace{\textstyle{\frac{\lambda_0(x)}{\lambda_0(x)+\lambda_1(x)}},
    			\textstyle{\frac{\lambda_1(x)}{\lambda_0(x)+\lambda_1(x)}}}_{\text{project}}\Big)}_{\text{evaluate}}\,,
	\end{aligned}
\end{equation*}
for $i=0,\ldots,p-1$.
Regarding the traces, note that they are completely inherited from $E_0^\E(\vec{\lambda}_{01}(x))$, which is a Whitney function known to have the desired vanishing properties and being tracewise compatible with the lower dimensional triangle edge functions.
Therefore, all trace properties are satisfied, including the nonzero decay along the adjacent faces to the edge.

The edge functions with their curl are
\begin{equation}
	E_i^\mathrm{e}(x)=E_i^\E(\vec{\lambda}_{ab}(x))\,,\qquad\quad
		\nabla\times E_i^\mathrm{e}(x)=\nabla\times E_i^\E(\vec{\lambda}_{ab}(x))\,,
	\label{eq:TetEgeneral}
\end{equation}
for $i=0,\ldots,p-1$, and $0\leq a<b\leq3$. 
There are a total of $p$ edge functions for every given edge, for a total of $6p$ edge functions.

\subsubsection{\texorpdfstring{$H(\mathrm{curl})$}{Hcurl} Faces}


%

Like the triangle, the tetrahedron has two families of shape functions for every face.
The trace properties follow from those of the edge functions.
There is a grand total of $4p(p-1)$ face functions.

\subparagraph{Family I:} 
The shape functions and their curls are
\begin{equation}
	E_{ij}^{\mathrm{f}}(x)=E_{ij}^\Tri(\vec{\lambda}_{abc}(x))\,,\qquad\quad
		\nabla\times E_{ij}^{\mathrm{f}}(x)=\nabla\times E_{ij}^\Tri(\vec{\lambda}_{abc}(x))\,,
\end{equation}
for $i\geq0$, $j\geq1$, $n=i+j=1,\ldots,p-1$, and $0\leq a<b<c\leq3$. 
For every face, there are $\frac{1}{2}p(p-1)$ face functions in this family.

\subparagraph{Family II:}
The shape functions and their curls are
\begin{equation}
	E_{ij}^{\mathrm{f}}(x)=E_{ij}^\Tri(\vec{\lambda}_{bca}(x))\,,\qquad\quad
		\nabla\times E_{ij}^{\mathrm{f}}(x)=\nabla\times E_{ij}^\Tri(\vec{\lambda}_{bca}(x))\,,
\end{equation}
for $i\geq0$, $j\geq1$, $n=i+j=1,\ldots,p-1$, and $0\leq a<b<c\leq3$.
The only difference with the first family is that the entries are permuted to $\vec{\lambda}_{bca}(x)$ instead of $\vec{\lambda}_{abc}(x)$.
For every face, there are $\frac{1}{2}p(p-1)$ face functions in this family.

%

\subsubsection{\texorpdfstring{$H(\mathrm{curl})$}{Hcurl} Interior Bubbles}


The construction is completely analogous to that of $H^1$ in the sense that they are obtained by multiplying the face functions by the Jacobi polynomial $L_k^{2(i+j)}$.
One must attempt this for various possible permutations of the entries, but being careful to ensure that they are linearly independent.
Three families arise.

The interior bubbles and their curl are
\begin{equation}
	\begin{aligned}
		E_{ijk}^\mathrm{b}(x)&=[L_k^{2(i+j)}](\vec{\mu}_{01}(\lambda_d(x)))E_{ij}^\Tri(\vec{\lambda}_{abc}(x))\,,\\
		\nabla\!\!\times\! E_{ijk}^\mathrm{b}(x)&=[L_k^{2(i+j)}](\vec{\mu}_{01}(\lambda_d(x)))
			\nabla\!\!\times\! E_{ij}^\Tri(\vec{\lambda}_{abc}(x))
				\!+\!\nabla[L_k^{2(i+j)}](\vec{\mu}_{01}(\lambda_d(x)))\!\times\! E_{ij}^\Tri(\vec{\lambda}_{abc}(x))\,,
	\end{aligned}
\end{equation}
where $i\geq0$, $j\geq1$, $k\geq1$, $n=i+j+k=2,\ldots,p-1$ and $(a,b,c,d)=(0,1,2,3),(1,2,3,0),(2,3,0,1)$, and where $\vec{\mu}_{01}(\lambda_d(x))=(1-\lambda_d(x),\lambda_d(x))$.
There is a grand total of $\frac{1}{2}p(p-1)(p-2)$ interior shape functions.

\subsection{\texorpdfstring{$H(\mathrm{div})$}{Hdiv} Shape Functions}


The dimension of $\mathcal{RT}^p$ in three dimensions is $\frac{1}{2}p(p+1)(p+3)$.
A careful count of the linearly independent shape functions presented here will coincide with that dimension. 
Showing that the functions constructed are in $\mathcal{RT}^p$ is not immediate, but follows from the next lemma, which should be kept in mind.

\begin{lemma}
\label{lemma:div}
Let $x\in\R^N$ for $N=3$, and $f_n\in\mathcal{P}^n(x)$ be any polynomial of total order $n$ in the coordinates $x=(x_1,x_2,x_3)$. Given $s_0$, $s_1$ and $s_2$ affine coordinates in $\R^3$ $($or simply linear functions in $x$$)$, it follows that the Raviart-Thomas space of order $n+1$, $\mathcal{RT}^{n+1}$, contains the function 
\begin{equation*}
	f_n(\bcdot)\Big(s_0\nabla s_1\times\nabla s_2+s_1\nabla s_2\times\nabla s_0+s_2\nabla s_0\times\nabla s_1\Big)\in\mathcal{RT}^{n+1}\,.
\end{equation*}
\end{lemma}
\begin{proof}
Recall the definition of the Raviart-Thomas space in three dimensions,
\begin{equation*}
	\mathcal{RT}^p=(\mathcal{P}^{p-1})^3\oplus\Big\{V\in(\tilde{\mathcal{P}}^{p})^3:V(x)=\phi(x)x
		=\phi(x)\Big(\begin{smallmatrix}x_1\\x_2\\x_3\end{smallmatrix}\Big) \, \text{ with }\phi\in\tilde{\mathcal{P}}^{p-1}\Big\}\,.
\end{equation*}
Affine coordinates are linear functions in $x=(x_1,x_2,x_3)$, so that
\begin{equation*}
	s_k(x)=a_k+b_k\cdot x\,,
\end{equation*}
for $a_k\in\R$, $b_k\in\R^3$ and $k=0,1,2$. Then $\nabla s_k(x)=b_k$ and
\begin{align*}
	V(x)&=s_0(x)\nabla s_1(x)\times\nabla s_2(x)+s_1(x)\nabla s_2(x)\times\nabla s_0(x)+s_2(x)\nabla s_0(x)\times\nabla s_1(x)\\
		&=(a_0+b_0\cdot x)(b_1\times b_2)+(a_1+b_1\cdot x)(b_2\times b_0)+(a_2+b_2\cdot x)(b_0\times b_1)\\
		&=\underbrace{\Big(a_0(b_1\times b_2)+a_1(b_2\times b_0)+a_2(b_0\times b_1)\Big)}_{=A}
			+\underbrace{\Big(b_0\cdot(b_1\times b_2)\Big)x}_{=B(x)}\,,			
\end{align*}
where the last term follows from various identities. Clearly, $A\in(\mathcal{P}^0)^3=\R^3$ and $b_0\cdot(b_1\times b_2)\in\tilde{\mathcal{P}}^0=\R$, so that $B\in\{V\in(\tilde{\mathcal{P}}^{1})^3: V(x)=\phi(x)x\}$. Hence, $V\in\mathcal{RT}^1$.

Now, $f_n\in\mathcal{P}^n=\mathcal{P}^{n-1}\oplus\tilde{\mathcal{P}}^n$, for $n\geq1$ can always be decoupled into $f_n=f_{n-1}+\tilde{f}_n$, where $f_{n-1}\in\mathcal{P}^{n-1}$ and $\tilde{f}_n\in\tilde{\mathcal{P}}^n$. As a result
\begin{equation*}
	f_n(x)V(x)=f_n(x)A+f_{n-1}(x)B(x)+\tilde{f}_n(x)B(x)\,,
\end{equation*}
where it is clear $f_{n}A+f_{n-1}B\in(\mathcal{P}^n)^3$ and $\tilde{f}_nB\in\{V\in(\tilde{\mathcal{P}}^{n+1})^3: V(x)=\phi(x)x\}$. Therefore, $f_nV\in\mathcal{RT}^{n+1}$.
\end{proof}

\subsubsection{\texorpdfstring{$H(\mathrm{div})$}{Hdiv} Faces}

The general formula for these functions is motivated by the well known first order Whitney form for $H(\mathrm{div})$, along with the fact that the normal trace of the faces should span the two dimensional $L^2$ space. The general definition is presented next.

\begin{definition*}
Let $s_0$, $s_1$ and $s_2$ be arbitrary functions of some spatial variable in $\R^N$, with $N=3$, and denote by $p$ the order in the coordinate triplet $(s_0,s_1,s_2)$. Then
\begin{equation}
	V_{ij}^{\Tri}(s_0,s_1,s_2)=[P_i,P_j^{2i+1}](s_0,s_1,s_2)\Big(s_0\nabla s_1\times\nabla s_2
		+s_1\nabla s_2\times\nabla s_0+s_2\nabla s_0\times\nabla s_1\Big)\,,
\end{equation}
for $i=n-j$, $j=0,\ldots,n$ and $n=0,\ldots,p-1$ $($or equivalently $i\geq0$, $j\geq0$ and $n=i+j=0,\ldots,p-1$$)$. The divergence is
\begin{equation}
	\nabla\cdot V_{ij}^{\Tri}(s_0,s_1,s_2)=(i+j+3)[P_i,P_j^{2i+1}](s_0,s_1,s_2)\nabla s_0\cdot(\nabla s_1\times\nabla s_2) \,.
	\label{eq:VdeltaDiv}
\end{equation}
\end{definition*}

The formula for the divergence follows from the following lemma, because $[P_i,P_j^{2i+1}](s_0,s_1,s_2)$ is a homogeneous polynomial of total order $n=i+j$ in $s_0$, $s_1$ and $s_2$.

\begin{lemma}
\label{lem:divformula}
Let $\psi_n(s_0,s_1,s_2)\in\tilde{\mathcal{P}}^n(s_0,s_1,s_2)$ be a homogeneous polynomial of total order $n$ in $s_0$, $s_1$ and $s_2$, where $s_0$, $s_1$ and $s_2$ are arbitrary functions of some spatial variable in $\R^N$, with $N=3$. Then
\begin{equation*}
    \nabla\cdot\Big(\psi_n(s_0,s_1,s_2)(s_0\nabla s_1\times\nabla s_2
			+s_1\nabla s_2\times\nabla s_0+s_2\nabla s_0\times\nabla s_1)\Big)
				\!=\!(n+3)\psi_n(s_0,s_1,s_2)\nabla s_0\cdot(\nabla s_1\times\nabla s_2)\,.
\end{equation*}
\end{lemma}
\begin{proof}
Let $V_{00}^{\Tri}(s_0,s_1,s_2)=(s_0\nabla s_1\times\nabla s_2
			+s_1\nabla s_2\times\nabla s_0+s_2\nabla s_0\times\nabla s_1)$. First notice that
\begin{align*}
	\nabla\cdot \Big(s_0(\nabla s_1\times\nabla s_2)\Big)
		&=\nabla s_0\cdot(\nabla s_1\times\nabla s_2)
			+s_0\Big(\nabla s_2\cdot\nabla\times\nabla s_1-\nabla s_1\cdot\nabla\times\nabla s_2\Big)\\
		&=\nabla s_0\cdot(\nabla s_1\times\nabla s_2)\,,
\end{align*}
and similarly with $\nabla\cdot(s_1(\nabla s_2\times\nabla s_0))$ and $\nabla\cdot(s_2(\nabla s_1\times\nabla s_2))$. All of them result in a scalar triple product, which is invariant to cyclic permutations. It follows
\begin{equation*}
	\nabla\cdot V_{00}^{\Tri}(s_0,s_1,s_2)=3\nabla s_0\cdot(\nabla s_1\times\nabla s_2)\,.
\end{equation*}
Now, consider a monomial $s_0^as_1^bs_2^c$. Then
\begin{align*}
	\nabla(s_0^as_1^bs_2^c)\cdot V_{00}^{\Tri}(s_0,s_1,s_2)
		&=(as_0^{a-1}s_1^bs_2^c\nabla s_0+bs_0^as_1^{b-1}s_2^c\nabla s_1+cs_0^as_1^bs_2^{c-1}\nabla s_2)
			\cdot V_{00}^{\Tri}(s_0,s_1,s_2)\\
		&=as_0^{a-1}s_1^bs_2^cs_0\nabla s_0\cdot(\nabla s_1\times\nabla s_2)
			+bs_0^as_1^{b-1}s_2^cs_1\nabla s_1\cdot(\nabla s_2\times\nabla s_0)\\
				&\qquad\qquad\qquad\qquad\qquad\qquad\qquad\qquad+cs_0^as_1^bs_2^{c-1}s_2\nabla s_2\cdot(\nabla s_0\times\nabla s_1)\\
		&=(a+b+c)s_0^as_1^bs_2^c\nabla s_0\cdot(\nabla s_1\times\nabla s_2)\,.
\end{align*}
With these last two results it follows
\begin{equation*}
	\nabla\cdot\Big(s_0^as_1^bs_2^cV_{00}^{\Tri}(s_0,s_1,s_2)\Big)=(a+b+c+3)s_0^as_1^bs_2^c\nabla s_0\cdot(\nabla s_1\times\nabla s_2)\,.
\end{equation*}
Then observe that any homogeneous polynomial $\psi_n(s_0,s_1,s_2)$ is composed of monomials of the form $s_0^as_1^bs_2^c$ of \textit{fixed} total order $a+b+c=n$. The result immediately follows.
\end{proof}

The result in \eqref{eq:VdeltaDiv} is quite remarkable in the sense that, as required by the construction, there are no derivatives of $[P_i,P_j^{2i+1}](s_0,s_1,s_2)$ in the expression for the divergence. 
Although not used in this section, record the following useful remark which will be exploited when dealing with the prism and pyramid elements.
\begin{remark}
Let $\nu_0=1-\nu_1-\nu_2$, where $\nu_1$ and $\nu_2$ are arbitrary functions of some spatial variable in $\R^N$ with $N=3$, and where $p$ is the order in the coordinates $(\nu_0,\nu_1,\nu_2)$. Then for all $i\geq0$, $j\geq0$ and $n=i+j=0,\ldots,p-1$,
\begin{equation}
    V_{ij}^\Tri(\nu_0,\nu_1,\nu_2)=[P_i,P_j^{2i+1}](\nu_0,\nu_1,\nu_2)\nabla\nu_1\times\nabla\nu_2\,,
    	\quad\qquad \nabla\cdot V_{ij}^\Tri(\nu_0,\nu_1,\nu_2)=0\,.\label{eq:HdivtriangleRemark}
\end{equation}
\end{remark}

These functions are still represented by the same projection, and follow the logic of projecting, evaluating and blending.
For example, take face 012, so the shape functions are
\begin{equation*}
	V_{ij}^\mathrm{f}(x)=V_{ij}^\Tri(\vec{\lambda}_{012}(x))=
		\underbrace{(\lambda_0(x)+\lambda_1(x)+\lambda_2(x))^{i+j+3}}_{\text{blend}}\underbrace{V_{ij}^\Tri
			\Big(\underbrace{\textstyle{\frac{1}{\lambda_0(x)+\lambda_1(x)+\lambda_2(x)}}
				\vec{\lambda}_{012}(x)}_{\text{project}}\Big)}_{\text{evaluate}}\,,
\end{equation*}
for $i\geq0$ and $j\geq0$.
It can be readily checked that they satisfy the required vanishing properties by focusing on the lowest order function $V_{00}^\mathrm{f}(x)$.
Moreover, by Lemma \ref{lemma:div} it is clear $V_{ij}^\mathrm{f}\in\mathcal{RT}^{n+1}$ for $n=i+j$. 

More generally, the shape functions and their divergence are
\begin{equation}
	V_{ij}^\mathrm{f}(x)=V_{ij}^\Tri(\vec{\lambda}_{abc}(x))\,,\qquad\quad
		\nabla\cdot V_{ij}^\mathrm{f}(x)=\nabla \cdot V_{ij}^\Tri(\vec{\lambda}_{abc}(x))\,,
\end{equation}
where $i\geq0$, $j\geq0$, $n=i+j=0,\ldots,p-1$, and $0\leq a<b<c\leq3$. 
Clearly, for each face there are $\frac{1}{2}p(p+1)$ functions, giving a total of $2p(p+1)$ face functions.

\subsubsection{\texorpdfstring{$H(\mathrm{div})$}{Hdiv} Interior Bubbles}

The construction is just like in $H^1$ and $H(\mathrm{curl})$, and there are three resulting families to consider.
The polynomial chosen to vanish is the Jacobi polynomial, $L_k^{2(i+j+1)}$. 

The interior functions with their divergence are
\begin{equation}
	\begin{aligned}
		V_{ijk}^\mathrm{b}(x)&=[L_k^{2(i\!+\!j\!+\!1)}](\vec{\mu}_{01}(\lambda_d(x)))V_{ij}^\Tri(\vec{\lambda}_{abc}(x))\,,\\
		\nabla\!\!\cdot\! V_{ijk}^\mathrm{b}(x)&=[L_k^{2(i\!+\!j\!+\!1)}](\vec{\mu}_{01}(\lambda_d(x)))
			\nabla\!\!\cdot\! V_{ij}^\Tri(\vec{\lambda}_{abc}(x))
				\!+\!\nabla[L_k^{2(i\!+\!j\!+\!1)}](\vec{\mu}_{01}(\lambda_d(x)))\!\cdot\! V_{ij}^\Tri(\vec{\lambda}_{abc}(x))\,,
	\end{aligned}
\end{equation}
where $i\geq0$, $j\geq0$, $k\geq1$, $n=i+j+k=1,\ldots,p-1$ and $(a,b,c,d)=(0,1,2,3),(1,2,3,0),(2,3,0,1)$, and where $\vec{\mu}_{01}(\lambda_d(x))=(1-\lambda_d(x),\lambda_d(x))$.
There is a grand total of $\frac{1}{2}p(p-1)(p+1)$ interior shape functions.

\subsection{\texorpdfstring{$L^2$}{L2} Shape Functions}
The $\frac{1}{6}p(p+2)(p+1)$ shape functions will clearly span $\mathcal{P}^{p-1}$.

\subsubsection{\texorpdfstring{$L^2$}{L2} Interior}

Similar to the construction of $L^2$ triangle functions, the $L^2$ interior functions for the tetrahedron are
\begin{equation}
	\begin{aligned}
		\psi_{ijk}^\mathrm{b}(x)&=[P_i](\lambda_0(x),\lambda_1(x))[P_j^{2i+1}](\lambda_0(x)\!+\!\lambda_1(x),\lambda_2(x))
			[P_k^{2(i+j+1)}](1\!-\!\lambda_3(x),\lambda_3(x))\\
		&=[P_i,P_j^{2i+1}](\vec{\lambda}_{012}(x))[P_k^{2(i+j+1)}](\vec{\mu}_{01}(\lambda_3(x)))
			(\nabla\lambda_1(x)\!\!\times\!\!\nabla\lambda_2(x))\!\cdot\!\nabla\lambda_3(x)\,,
	\end{aligned}
\end{equation}
where $i\geq0$, $j\geq0$, $k\geq0$, and $n=i+j+k=0,\ldots,p-1$, and where $\vec{\mu}_{01}(\lambda_3(x))=(1-\lambda_3(x),\lambda_3(x))$. There is a total of $\frac{1}{6}p(p+2)(p+1)$ interior functions.

%
%

\subsection{Orientations}
\label{sec:TetOrientations}

\begin{figure}[!ht]
\begin{center}
\includegraphics[scale=0.5]{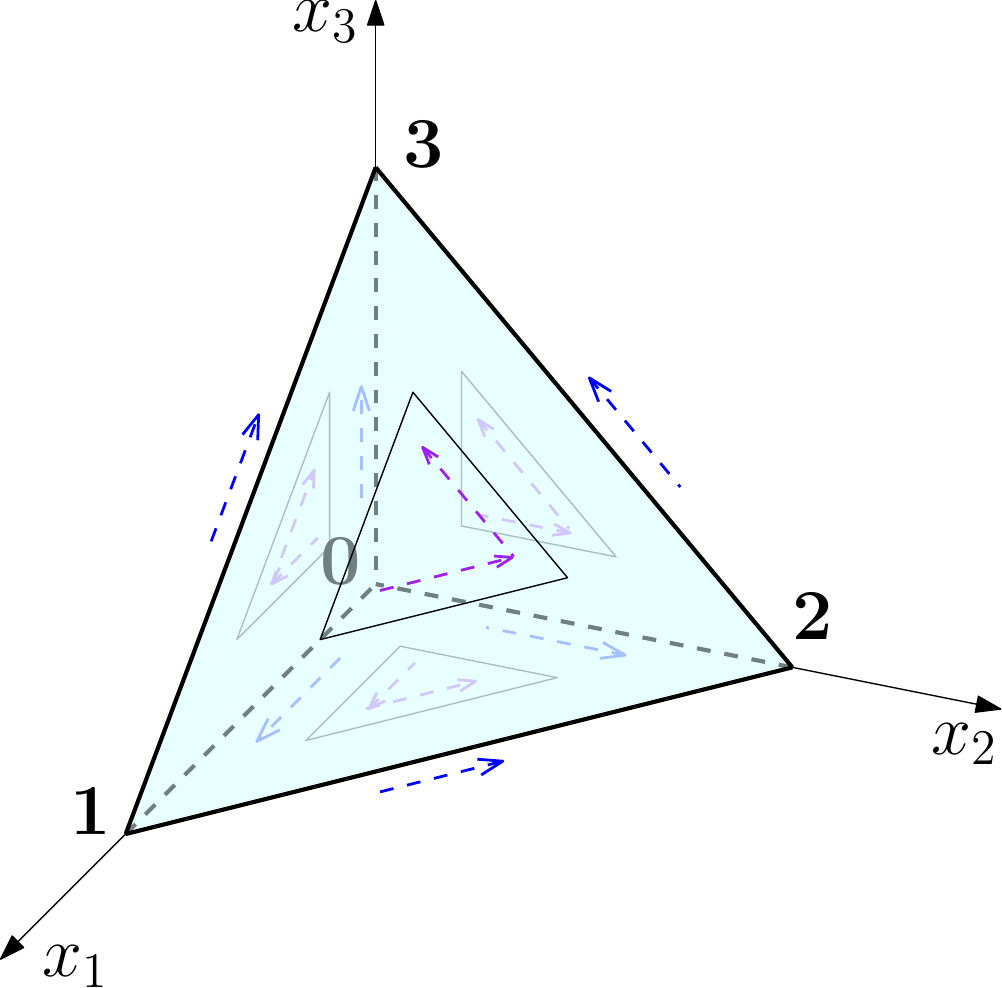}
\caption{Master tetrahedron with numbered vertices \textit{and} local edge and face orientations.}
\label{fig:MasterTetOrientations}
\end{center}
\end{figure}

To construct orientation embedded shape functions for the tetrahedron, it is recommended to have read at least the first part of \S\ref{sec:HexaOrientations}.
The predefined \textit{local} edge and face orientations for the tetrahedron are illustrated in Figure \ref{fig:MasterTetOrientations}.
They represent the $\oo=0$ case.
The task at hand is to find the associated \textit{locally ordered} tuples of affine coordinates representing those local orientations.
As examples take edge 01 and face 012.

Edge 01 is composed of the vertices $v_0$ and $v_1$, which are linked uniquely to $\lambda_0$ and $\lambda_1$ respectively.
The local orientation for edge 01 is represented by the local vertex-ordering $v_0\tdashto v_1$.
As a result, the locally ordered pair for edge 01 is $\vec{\lambda}_{01}=(\lambda_0,\lambda_1)$.
It follows that the orientation embedded edge 01 shape functions in $H^1$ with their gradient are
\begin{equation*}
	\phi_i^\mathrm{e}(x)=\phi_i^\E(\sigma_\oo^\E(\vec{\lambda}_{01}(x)))\,,\qquad\quad
		\nabla\phi_i^\mathrm{e}(x)=\nabla\phi_i^\E(\sigma_\oo^\E(\vec{\lambda}_{01}(x)))\,,
\end{equation*}
with $i=2,\ldots,p$.
The same applies to the $H(\mathrm{curl})$ edge 01 shape functions and their curl.
The approach is analogous with any other edge.

Face 012 is composed of the vertices $v_0$, $v_1$ and $v_2$, which are linked to $\lambda_0$, $\lambda_1$ and $\lambda_2$ respectively.
The local orientation for the face is represented by the local vertex-ordering $v_0\tdashto v_1\tdashto v_2$, and as a result the locally ordered triplet for face 012 is $\vec{\lambda}_{012}=(\lambda_0,\lambda_1,\lambda_2)$.
Hence, for example the $H^1$ orientation embedded face 012 shape functions with their gradient are
\begin{equation*}
	\phi_{ij}^\mathrm{f}(x)=\phi_{ij}^\Tri(\sigma_\oo^\Tri(\vec{\lambda}_{012}(x)))\,,\qquad\quad
		\nabla\phi_{ij}^\mathrm{f}(x)=\nabla\phi_{ij}^\Tri(\sigma_\oo^\Tri(\vec{\lambda}_{012}(x)))\,,
\end{equation*}
with $i\geq2$, $j\geq1$ and $n=i+j=3,\ldots,p$.
The same applies to the $H(\mathrm{curl})$ and $H(\mathrm{div})$ face 012 shape functions and their differential forms.
Naturally, the approach is analogous with any other face.

%% file: prism.tex
\section{Prism}
\label{sec:Prism}

\begin{figure}[!ht]
\begin{center}
\includegraphics[scale=0.5]{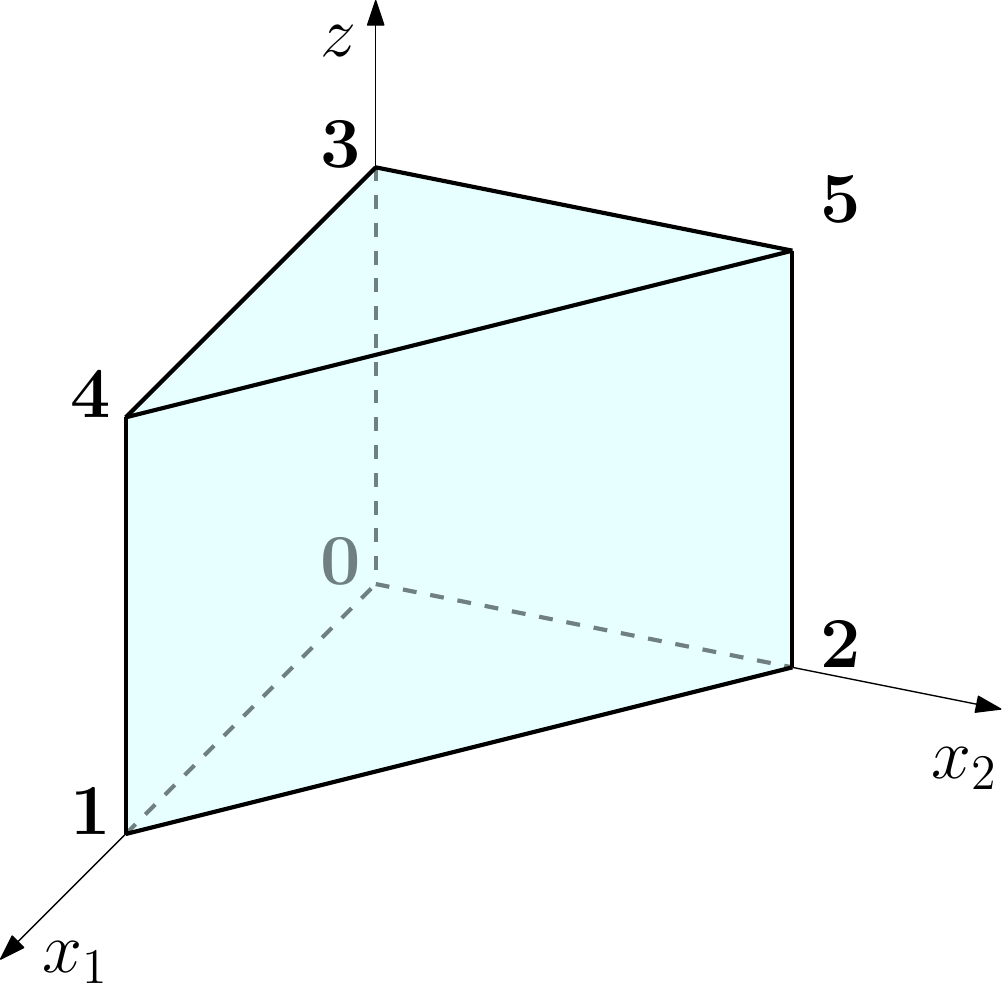}
\caption{Master prism with numbered vertices.}
\label{fig:MasterPrism}
\end{center}
\end{figure}

The master prism is shown in Figure \ref{fig:MasterPrism} in the $(x,z)=(x_1,x_2,z)$ space.
It is the Cartesian product of a triangle and a segment.
More specifically, it is $\{x\in\R^2:x_1>0,x_2>0,x_1+x_2<1\}\times(0,1)$.
Here, $x$ represents the triangle element coordinates, while $z$ represents the 1D segment element coordinate.


The prism has both quadrilateral and triangle faces. 
Similarly, there are two types of edges. 
Those edges which are adjacent to both a quadrilateral face and a triangle face are called \textit{mixed edges}, while those edges only shared by quadrilateral faces are simply called \textit{quadrilateral edges}. 
These distinctions are important, and the form of the shape functions will differ for the different types of edges and faces.

Due to the Cartesian product structure, it is natural to consider the 2D affine coordinates for the triangle (dependent on $x$) and the 1D affine coordinates for the segment (dependent on $z$). 
For this master element they are
\begin{equation}
	\begin{gathered}
		\nu_0(x)=1-x_1-x_2\,,\qquad\nu_1(x)=x_1\,,\qquad\nu_2(x)=x_2\,,\\
		\mu_0(z)=1-z\,,\qquad\quad\mu_1(z)=z\,.
	\end{gathered}
\end{equation}
Their gradients in 3D are
\begin{equation}
	\begin{gathered}
		\nabla\nu_0(x)=\bigg(\begin{smallmatrix}-1\\[2pt]-1\\[2pt]0\end{smallmatrix}\bigg)\,,\qquad
			\nabla\nu_1(x)=\bigg(\begin{smallmatrix}1\\[2pt]0\\[2pt]0\end{smallmatrix}\bigg)\,,\qquad
				\nabla\nu_2(x)=\bigg(\begin{smallmatrix}0\\[2pt]1\\[2pt]0\end{smallmatrix}\bigg)\,,\\
		\nabla\mu_0(z)=\bigg(\begin{smallmatrix}0\\[2pt]0\\[2pt]-1\end{smallmatrix}\bigg)\,,\qquad\quad
			\nabla\mu_1(z)=\bigg(\begin{smallmatrix}0\\[2pt]0\\[2pt]1\end{smallmatrix}\bigg)\,.
	\end{gathered}
\end{equation}
These are important and are used explicitly or implicitly in many of the oncoming calculations.

As usual, there are natural relationships between vertices, edges and faces, and the affine coordinates.
The related coordinates are those which take the value $1$ at the given topological entity.
In fact, for the prism each vertex is linked to \textit{two} affine coordinates, one of them a 2D affine coordinate and the other a 1D affine coordinate.
Moreover, each edge is linked to \textit{one} affine coordinate.
For mixed edges it is a 1D affine coordinate, while for quadrilateral edges it is a 2D affine coordinate.
Lastly, the triangle faces are also linked to \textit{one} 1D affine coordinate.
For example, vertex $0$, $v_0=(0,0,0)$, is linked to the affine coordinates $\nu_0(x)$ and $\mu_0(z)$.
Meanwhile, mixed edge 01 is linked to the affine coordinate $\mu_0(z)$, while quadrilateral edge 03 is linked to the affine coordinate $\nu_0(x)$.
Finally, face 012 is linked to $\mu_0(z)$.

\subsubsection*{Exact Sequence}

%
The product structure of the prism suggests that the discrete polynomial exact sequence approximating \eqref{eq:3D_exact_sequence} is intimately related to the discrete sequences for the triangle and segment (see \eqref{eq:EStriangle} and \eqref{eq:ESsegment}).
Indeed, this is the case. 
The sequence is of the form
\begin{equation}
	W^{p,q} \xrightarrow{\,\,\nabla\,\,} Q^{p,q} \xrightarrow{\nabla\times} V^{p,q} \xrightarrow{\nabla\cdot} Y^{p,q} \,,
\end{equation}
where
\begin{equation}
	\begin{aligned}
	W^{p,q}&=\mathcal{P}^p\otimes\mathcal{P}^q=\mathcal{P}^p(x_1,x_2)\otimes\mathcal{P}^q(z)\,,\\
	Q^{p,q}&=(\mathcal{N}^p\otimes\mathcal{P}^q)\times(\mathcal{P}^p\otimes\mathcal{P}^{q-1})\,,\\
	V^{p,q}&=(\mathcal{RT}^p\otimes\mathcal{P}^{q-1})\times(\mathcal{P}^{p-1}\otimes\mathcal{P}^q)\,,\\
	Y^{p,q}&=\mathcal{P}^{p-1}\otimes\mathcal{P}^{q-1}\,.
	\end{aligned}
\end{equation}
Here, the spaces $\mathcal{N}^p$ and $\mathcal{RT}^p$ correspond to the $N=2$ case (see \eqref{eq:NedelecSpace} and	\eqref{eq:RaviartThomasSpace}), so that the space $\mathcal{N}^p\otimes\mathcal{P}^q=\{\phi E:\phi\in\mathcal{P}^q,E\in\mathcal{N}^p\}$ has two components and dimension $p(p+2)(q+1)$. The same applies to $\mathcal{RT}^p\otimes\mathcal{P}^{q-1}$ which has dimension $p(p+2)q$.

\subsection{\texorpdfstring{$H^1$}{H1} Shape Functions}
It will be clear that all shape functions lie in $\mathcal{P}^p\otimes\mathcal{P}^q$, which has dimension $\frac{1}{2}(p+2)(p+1)(q+1)$. 
Moreover, a judicious count of the shape functions constructed in this section coincides with that dimension, ensuring that the span of these is precisely $\mathcal{P}^{p}\otimes\mathcal{P}^q$.

Notice that for the triangle there are $H^1$ vertex, edge \textit{and} face shape functions, while for the segment there are $H^1$ vertex and edge functions. 
The six possible tensor products of these will precisely give the $H^1$ shape functions for the prism.
The vanishing properties are naturally inherited from each of the components of the tensor product structure of the shape functions, so they follow easily.

\subsubsection {\texorpdfstring{$H^1$}{H1} Vertices}
\label{sec:PrismH1vertices}

As usual with Cartesian product structures, the vertex functions are simply the product of the affine coordinates associated to the given vertex. 
Hence, they are the tensor product of lower dimensional vertex functions which inherit all the desired vanishing properties and the decays along adjacent faces to the vertex.

The vertex shape functions and their gradients are
\begin{equation} 
	\begin{aligned}	
		\phi^\mathrm{v}(x,z)&=\nu_a(x)\mu_c(z)\,,\\
		\nabla\phi^\mathrm{v}(x,z)&=\mu_c(z)\nabla\nu_a(x)+\nu_a(x)\nabla\mu_c(z)\,,
	\end{aligned}		
\end{equation}
where $a=0,1,2$ and $c=0,1$. 
Notice $c=0$ represents the ``bottom'' face (vertices 0, 1 and 2), while $c=1$ represents the ``top'' face (vertices 3, 4 and 5).
Also $\nu_a$ represents vertices $a$ and $a+3$. 
There is a total of $6$ vertex functions (one for each vertex).



\subsubsection {\texorpdfstring{$H^1$}{H1} Edges}
 
\paragraph{\texorpdfstring{$H^1$}{H1} Mixed Edges.} 
These functions are the tensor product of $H^1$ triangle edge functions with 1D $H^1$ vertex shape functions. 
For instance, take the edge 01, which is in the bottom triangle face (associated to $\mu_0(z)$). 
The shape functions then take the form
\begin{equation*}
    \phi_i^\mathrm{e}(x)=\mu_0(z)\phi_i^\E(\vec{\nu}_{01}(x))
    	=\underbrace{\mu_0(z)(\nu_0(x)+\nu_1(x))^i}_{\text{blend}}
    		\underbrace{\phi_i^\E\Big(\underbrace{\textstyle{\frac{\nu_0(x)}{\nu_0(x)+\nu_1(x)}},
    			\textstyle{\frac{\nu_1(x)}{\nu_0(x)+\nu_1(x)}}}_{\text{project}}\Big)}_{\text{evaluate}}\,,
\end{equation*}
for $i=2,\ldots,p$.
The trace properties are naturally inherited along the edge, its adjacent faces (including the nonlinear decay in the triangle face when $\mu_0(z)=1$) and all the other faces where it is supposed to vanish.
Like the triangle, the projection is of the form
\begin{equation*}
	(x_1,x_2,z)\;\longmapsto\;\begin{matrix}(x_1,x_2,0)\\(\frac{x_1}{1-x_2},0,z)\end{matrix}
		\;\longmapsto\;(\textstyle{\frac{x_1}{1-x_2}},0,0)\,.
\end{equation*}
It consists of finding the intersection $P''=(\frac{x_1}{1-x_2},0,0)$ of the edge with the plane passing through the original point $P=(x_1,x_2,x_3)$ and the opposite disjoint quadrilateral edge. 
Alternatively, it can be interpreted as a two step projection, where it is first projected to an adjacent face, and then projected to the desired edge. 
This projection is shown in Figures \ref{fig:PrismProjectionTriangle} and \ref{fig:PrismProjectionQuad}.

\begin{figure}[!ht]
\begin{center}
\includegraphics[scale=0.6]{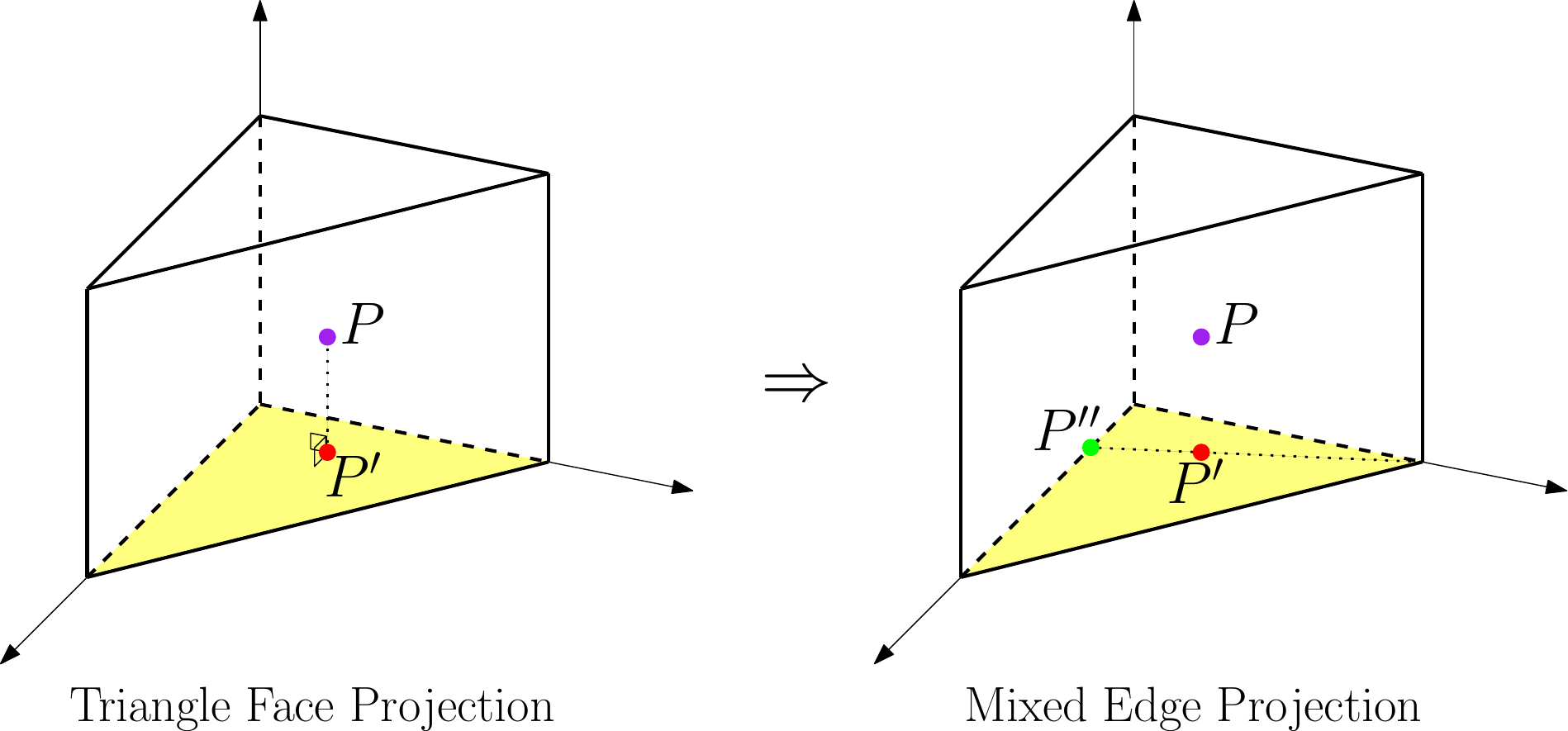}
\caption{Triangle face projection from $P$ to $P'$ followed by a mixed edge projection from $P'$ to $P''$.}
\label{fig:PrismProjectionTriangle}
\end{center}
\end{figure}

The shape functions with their gradient are
\begin{equation}
	\begin{aligned}	
		\phi_i^\mathrm{e}(x,z)&=\mu_c(z)\phi_i^\E(\vec{\nu}_{ab}(x))\,,\\
		\nabla\phi_i^\mathrm{e}(x,z)&=\mu_c(z)\nabla\phi_i^\E(\vec{\nu}_{ab}(x))+\phi_i^\E(\vec{\nu}_{ab}(x))\nabla\mu_c(z)\,,
	\end{aligned}\label{eq:PrismH1mixededges}
\end{equation}
where $i=2,\ldots,p$, $0\leq a<b\leq2$, and $c=0,1$. 
There are $p-1$ shape functions for each mixed edge, for a total of $6(p-1)$ mixed edge functions.

\paragraph{\texorpdfstring{$H^1$}{H1} Quadrilateral Edges.}
These are the tensor product of $H^1$ triangle vertex shape functions with the 1D $H^1$ edge functions. 
For edge 03, the shape functions are
\begin{equation*}
	\phi_i^\mathrm{e}(x,z)=\underbrace{\nu_0(x)}_{\text{blend}}
		\underbrace{\phi_i^\E(\underbrace{\mu_0(z),\mu_1(z)}_{\text{project}})}_{\text{evaluate}}\,,
\end{equation*}
for $i=2,\ldots,p$.
As expected, there is a linear edge blending towards both of the adjacent quadrilateral faces, given by $\nu_0(x)$, while all other trace properties are also inherited.
The implied projection is simply
\begin{equation*}
	(x_1,x_2,z)\;\longmapsto\;(\textstyle{\frac{x_1}{1-x_2}},0,z)\;\longmapsto\;(0,0,z)\,.
\end{equation*}
It consists of finding the intersection $P'''=(0,0,z)$ of the edge with the normal plane passing through the original point $P=(x_1,x_2,x_3)$. 
Alternaltively it can be interpreted as a two step projection.
This is illustrated in Figure \ref{fig:PrismProjectionQuad}.

\begin{figure}[!ht]
\begin{center}
\includegraphics[scale=0.6]{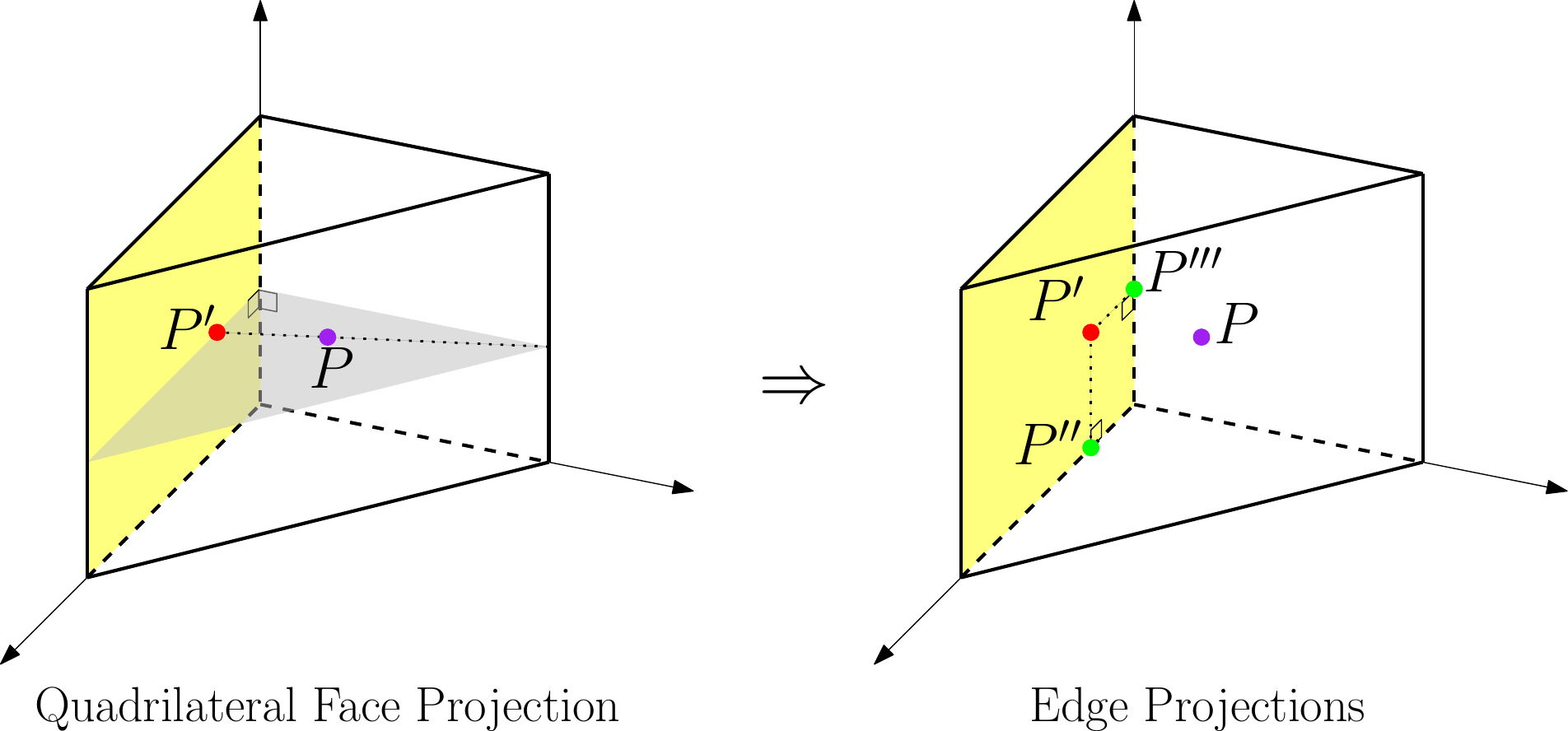}
\caption{Quadrilateral face projection from $P$ to $P'$ followed by a mixed edge projection from $P'$ to $P''$ and a quadrilateral edge projection from $P'$ to $P'''$.}
\label{fig:PrismProjectionQuad}
\end{center}
\end{figure}

The shape functions and their gradient are
\begin{equation}
	\begin{aligned}	
		\phi_i^\mathrm{e}(x,z)&=\nu_a(x)\phi_i^\E(\vec{\mu}_{01}(z))\,,\\
		\nabla\phi_i^\mathrm{e}(x,z)&=\nu_a(x)\nabla\phi_i^\E(\vec{\mu}_{01}(z))+\phi_i^\E(\vec{\mu}_{01}(z))\nabla\nu_a(x)\,,
	\end{aligned}\label{eq:PrismH1quadedges}
\end{equation}
where $i=2,\ldots,q$ and $a=0,1,2$. 
There are $q-1$ shape functions for each quadrilateral edge, for a total of $3(q-1)$ quadrilateral edge functions.


%

\subsubsection{\texorpdfstring{$H^1$}{H1} Faces} 

\paragraph{\texorpdfstring{$H^1$}{H1} Triangle Faces.} 
These are tensor products of $H^1$ triangle face bubbles and 1D $H^1$ vertex shape functions. 
For instance, for face 012 the shape functions are
\begin{equation*}
	\phi_{ij}^\mathrm{f}(x,z)=\underbrace{\mu_0(z)}_{\text{blend}}
		\underbrace{\phi_{ij}^\Tri(\underbrace{\nu_0(x),\nu_1(x),\nu_2(x)}_{\text{project}})}_{\text{evaluate}}\,,
\end{equation*}
for $i\geq2$ and $j\geq1$.
The trace properties are trivially inherited from each of the components.
The triangle face projection is illustrated in Figure \ref{fig:PrismProjectionTriangle} and consists of finding the intersection $P'=(x_1,x_2,0)$ of the face with the normal line passing through the original point $P=(x_1,x_2,x_3)$.

In general, the shape functions and their gradient are
\begin{equation}
	\begin{aligned}	
		\phi_{ij}^\mathrm{f}(x,z)&=\mu_c(z)\phi_{ij}^\Tri(\vec{\nu}_{012}(x))\,,\\
		\nabla\phi_{ij}^\mathrm{f}(x,z)&=\mu_c(z)\nabla\phi_{ij}^\Tri(\vec{\nu}_{012}(x))
			+\phi_{ij}^\Tri(\vec{\nu}_{012}(x))\nabla\mu_c(z)\,,
	\end{aligned}\label{eq:PrismH1TriFace}
\end{equation}
where $i\geq2$, $j\geq1$, $n=i+j=3,\ldots,p$, and $c=0,1$. 
As with any $H^1$ triangle face, there are $\frac{1}{2}(p-1)(p-2)$ shape functions per face, for a total of $(p-1)(p-2)$ triangle face  functions.



\paragraph{\texorpdfstring{$H^1$}{H1} Quadrilateral Faces.} 
The quadrilateral face functions are simply the product $H^1$ triangle edge functions and 1D $H^1$ edge functions, so they conveniently fall into the general definition of $\phi_{ij}^\square$.
For face 0143, they take the form
\begin{equation*}
	\phi_{ij}^\mathrm{f}(x,z)=\phi_{ij}^\square(\vec{\nu}_{01}(x),\vec{\mu}_{01}(z))
		=\underbrace{(\nu_0(x)+\nu_1(x))^i}_{\text{blend}}
			\underbrace{\phi_{ij}^\square\Big(\underbrace{\textstyle{\frac{1}{\nu_0(x)+\nu_1(x)}}
				\vec{\nu}_{01}(x),\vec{\mu}_{01}(z)}_{\text{project}}\Big)}_{\text{evaluate}}\,,
\end{equation*}
where $i=2,\ldots,p$ and $j=2,\ldots,q$.
Clearly, the desired trace properties are inherited.
The implied projection, already illustrated in Figure \ref{fig:PrismProjectionQuad}, consists of finding the intersection $P'=(\frac{x_1}{1-x_2},0,z)$ of the face with the projecting line lying in the horizontal plane and passing through the original point $P=(x_1,x_2,x_3)$ and the opposite disjoint quadrilateral edge. 

The shape functions and their gradient are
\begin{equation}
		\phi_{ij}^\mathrm{f}(x,z)=\phi_{ij}^\square(\vec{\nu}_{ab}(x),\vec{\mu}_{01}(z))\,,\qquad\quad
		\nabla\phi_{ij}^\mathrm{f}(x,z)=\nabla\phi_{ij}^\square(\vec{\nu}_{ab}(x),\vec{\mu}_{01}(z))\,,
\end{equation}
where $i=2,\ldots,p$, $j=2,\ldots,q$, and $0\leq a<b\leq2$. 
There are $(p-1)(q-1)$ shape functions per face, leading to a total of $3(p-1)(q-1)$ quadrilateral face functions.



\subsubsection{\texorpdfstring{$H^1$}{H1} Interior Bubbles} 
The $H^1$ bubble functions are the tensor product of the $H^1$ triangle face functions and 1D $H^1$ edge functions.
Due to their structure, the zero trace properties are trivially satisfied along the whole boundary.

The interior bubbles and their gradient are
\begin{equation}
	\begin{aligned}	
		\phi_{ijk}^\mathrm{b}(x,z)&=\phi_{ij}^\Tri(\vec{\nu}_{012}(x))\phi_k^\E(\vec{\mu}_{01}(z))\,,\\
		\nabla\phi_{ijk}^\mathrm{b}(x,z)&=\phi_k^\E(\vec{\mu}_{01}(z))\nabla\phi_{ij}^\Tri(\vec{\nu}_{012}(x))
			+\phi_{ij}^\Tri(\vec{\nu}_{012}(x))\nabla\phi_k^\E(\vec{\mu}_{01}(z))\,,
	\end{aligned}
\end{equation}
where $i\geq2$, $j\geq1$, $n=i+j=3,\ldots,p$, and $k=2,\ldots,q$. 
There are a total of $\frac{1}{2}(p-1)(p-2)(q-1)$ bubble shape functions in $H^1$.



\subsection{\texorpdfstring{$H(\mathrm{curl})$}{Hcurl} Shape Functions}

The linearly independent shape functions presented here are shown to belong and span the $H(\mathrm{curl})$ conforming space $(\mathcal{N}^p\otimes\mathcal{P}^q)\times(\mathcal{P}^p\otimes\mathcal{P}^{q-1})$, which has dimension $p(p+2)(q+1)+\frac{1}{2}(p+2)(p+1)q$.


These shape functions, as expected, are composed of combinations of $H^1$ and $H(\mathrm{curl})$ components. 
Intuitively, they involve the affine coordinates and at least $E_i^\E$, $E_{ij}^\Tri$ and $E_{ij}^\square$.
All shape functions continue to respect the logic of projecting, evaluating and blending, and for a given topological entity, the projections are the same as those in $H^1$ (see Figures \ref{fig:PrismProjectionTriangle} and \ref{fig:PrismProjectionQuad}).


\subsubsection{\texorpdfstring{$H(\mathrm{curl})$}{Hcurl} Edges}

\paragraph{\texorpdfstring{$H(\mathrm{curl})$}{Hcurl} Mixed Edges.} 
These are tensor products of triangle $H(\text{curl})$ edge functions and 1D $H^1$ vertex functions.
For example, take edge 01. 
Note that $E_i^\E(\vec{\nu}_{01}(x))$ in three dimensions is a three component vector whose last component is zero, since it is independent of the $z$ coordinate. 
Indeed, if considered in \textit{two} dimensions, it is just the well defined triangle $H(\text{curl})$ edge function. 
As such, it is an element of $\mathcal{N}^p$ (in two dimensions). 
Due to this important fact, the first two components of $E_i^\E$ are denoted by $(E_i^\E(\vec{\nu}_{01}(x)))_x\in\mathcal{N}^p$. 
For this edge, the shape functions are
\begin{equation*}
	\begin{aligned}
		E_i^\mathrm{e}(x,z)&=\mu_0(z)E_i^\E(\vec{\nu}_{01}(x))
				=\begin{pmatrix}\mu_0(z)(E_i^\E(\vec{\nu}_{01}(x)))_x\\0\end{pmatrix}
					\in(\mathcal{N}^p\otimes\mathcal{P}^1)\times\{0\}\\
		  &=\underbrace{\mu_0(z)(\nu_0(x)+\nu_1(x))^{i+2}}_{\text{blend}}
    		\underbrace{E_i^\E\Big(\underbrace{\textstyle{\frac{1}{\nu_0(x)+\nu_1(x)}\vec{\nu}_{01}(x)}}_{\text{project}}
    			\Big)}_{\text{evaluate}}\,,
	\end{aligned}
\end{equation*}
for $i=0,\ldots,p-1$.
Hence, $E_i^\mathrm{e}\in(\mathcal{N}^p\otimes\mathcal{P}^q)\times(\mathcal{P}^p\otimes\mathcal{P}^{q-1})$.
For the trace properties, note that over faces 1254 and 0253 both tangential components are zero, since the $z$ component is zero and the tangent to edges 12 and 02 is zero by inheritance of $E_i^\E(\vec{\nu}_{01}(x))$.
In the top face, the tangent is also zero because $\mu_0(z)$ is zero there.
Finally, by construction, over face 012 its tangent is precisely the triangle $H(\text{curl})$ edge shape function $(E_i^\E(\vec{\nu}_{01}(x)))_x$, while over face 0143 its tangent behaves like a quadrilateral $H(\text{curl})$ edge function $\mu_0(z)(E_i^\E(\vec{\mu}_{01}(x_1)))_x$.

Now, the general shape functions and their curl are
\begin{equation}
	\begin{aligned}	
		E_i^\mathrm{e}(x,z)&=\mu_c(z)E_i^\E(\vec{\nu}_{ab}(x))\,,\\
		\nabla\times E_i^\mathrm{e}(x,z)&=\mu_c(z)\nabla\times E_i^\E(\vec{\nu}_{ab}(x))
			+\nabla\mu_c(z)\times E_i^\E(\vec{\nu}_{ab}(x))\,,
	\end{aligned}
\end{equation}
with $i=0,\ldots,p-1$, $0\leq a<b\leq2$ and $c=0,1$. 
There are $p$ shape functions for each mixed edge, giving a total of $6p$ mixed edge functions.

\paragraph{\texorpdfstring{$H(\mathrm{curl})$}{Hcurl} Quadrilateral Edges.} 
These are tensor products of triangle $H^1$ vertex shape functions and $E_i^\E(\vec{\mu}_{01}(z))$, which act as 1D $H(\mathrm{curl})$ functions (even though these do not formally exist for the segment element).
For instance, take edge 03. 
Now, $E_i^\E(\vec{\mu}_{01}(z))$ in three dimensions is a three component vector whose first two components are zero, since it is only dependent on the $z$ coordinate. 
Indeed, if considered in \textit{one} dimension, it is just a segment $L^2$ function belonging to $\mathcal{P}^{q-1}$. 
This last component of $E_i^\E$ is denoted by $(E_i^\E(\vec{\mu}_{01}(z)))_z\in\mathcal{P}^{q-1}$. 
For this edge, the shape functions are
\begin{equation*}
	E_i^\mathrm{e}(x,z)=\underbrace{\nu_0(x)}_{\text{blend}}
			\underbrace{E_i^\E(\underbrace{\vec{\mu}_{01}(z)}_{\text{project}})}_{\text{evaluate}}
		=\begin{pmatrix}0\\0\\\nu_0(x)(E_i^\E(\vec{\mu}_{01}(z)))_{z}\end{pmatrix}
			\in\{0\}\times\{0\}\times(\mathcal{P}^1\otimes\mathcal{P}^{q-1})\,,
\end{equation*}
for $i=0,\ldots,q-1$.
Clearly $E_i^\mathrm{e}\in(\mathcal{N}^p\otimes\mathcal{P}^q)\times(\mathcal{P}^p\otimes\mathcal{P}^{q-1})$.
The trace properties follow easily for the triangle faces, since the vector field points normal to those faces, while in the nonadjacent quadrilateral face 1254 it is zero due to $\nu_0(x)$ being zero.
In the adjacent faces, $E_i^\E(\vec{\mu}_{01}(z))$ is unaffected by the restictions (it is already tangent to the faces and independent of $x$), so the tangential trace is a quadrilateral $H(\text{curl})$ edge function simply because the restriction of $\nu_0(x)$ is a linear blending function over that face.

In view of \eqref{eq:Hcurl1Dspecialcase}, the shape functions and their curl are
\begin{equation}
		E_i^\mathrm{e}(x,z)=\nu_a(x)E_i^\E(\vec{\mu}_{01}(z))\,,\qquad\quad
		\nabla\times E_i^\mathrm{e}(x,z)=\nabla\nu_a(x)\times E_i^\E(\vec{\mu}_{01}(z))\,,
\end{equation}
with $i=0,\ldots,q-1$, and $a=0,1,2$. 
There are $q$ shape functions for each quadrilateral edge, for a total of $3q$ quadrilateral edge functions.

\subsubsection{\texorpdfstring{$H(\mathrm{curl})$}{Hcurl} Faces}

\paragraph{\texorpdfstring{$H(\mathrm{curl})$}{Hcurl} Triangle Faces.} 
These are tensor products of triangle $H(\text{curl})$ face functions and $H^1$ vertex shape functions. 
As usual, there are two families. 
Proceeding as with $H(\text{curl})$ mixed edge functions, it follows the shape functions lie in $(\mathcal{N}^p\otimes\mathcal{P}^q)\times(\mathcal{P}^p\otimes\mathcal{P}^{q-1})$ and that the trace properties are satisfied.
As accustomed, there are $p(p-1)$ functions per triangle face, and a grand total of $2p(p-1)$ triangle face functions.

\subparagraph{Family I:} 
The shape functions and their curl are
\begin{equation}
	\begin{aligned}
		E_{ij}^{\mathrm{f}}(x,z)&=\mu_c(z)E_{ij}^\Tri(\vec{\nu}_{012}(x))\,,\\
		\nabla\times E_{ij}^{\mathrm{f}}(x,z)&=\mu_c(z)\nabla\times E_{ij}^\Tri(\vec{\nu}_{012}(x))
			+\nabla\mu_c(z)\times E_{ij}^\Tri(\vec{\nu}_{012}(x))\,,
	\end{aligned}
\end{equation}
for $i\geq0$, $j\geq1$, $n=i+j=1,\ldots,p-1$, and $c=0,1$. 
For every face, there are $\frac{1}{2}p(p-1)$ functions in this family.

\subparagraph{Family II:}
The shape functions and their curl are
\begin{equation}
	\begin{aligned}
		E_{ij}^{\mathrm{f}}(x,z)&=\mu_c(z)E_{ij}^\Tri(\vec{\nu}_{120}(x))\,,\\
		\nabla\times E_{ij}^{\mathrm{f}}(x,z)&=\mu_c(z)\nabla\times E_{ij}^\Tri(\vec{\nu}_{120}(x))
			+\nabla\mu_c(z)\times E_{ij}^\Tri(\vec{\nu}_{120}(x))\,,
	\end{aligned}
\end{equation}
for $i\geq0$, $j\geq1$, $n=i+j=1,\ldots,p-1$, and $c=0,1$.
Note the fact that the entries are $\vec{\nu}_{120}(x)$ as opposed to $\vec{\nu}_{012}(x)$.
For every face, there are $\frac{1}{2}p(p-1)$ functions in this family.

\paragraph{\texorpdfstring{$H(\mathrm{curl})$}{Hcurl} Quadrilateral Faces.} 
These are tensor products of $H(\text{curl})$ triangle edge functions and $H^1$ 1D edge shape functions, and viceversa, so that it is clear there are two families. 
Both fall naturally into the general definition of $E_{ij}^\square$.
There are a total of $p(q-1)+q(p-1)$ functions per quadrilateral face, for a grand total of $3(p(q-1)+q(p-1))$ quadrilateral face functions.

\subparagraph{Family I:}
To begin, take for example face 0143. 
The shape functions for the first family are of the form $E_{ij}^\square(\vec{\nu}_{01}(x),\vec{\mu}_{01}(z))$, so proceeding exactly as with the \textit{mixed} edges it is easily shown that it lies in $(\mathcal{N}^p\otimes\mathcal{P}^q)\times\{0\}$ and more importantly that the trace properties hold.

The general shape functions and their curl are
\begin{equation}
		E_{ij}^{\mathrm{f}}(x,z)=E_{ij}^\square(\vec{\nu}_{ab}(x),\vec{\mu}_{01}(z))\,,\qquad\quad
		\nabla\times E_{ij}^{\mathrm{f}}(x,z)=\nabla\times E_{ij}^\square(\vec{\nu}_{ab}(x),\vec{\mu}_{01}(z))\,,
\end{equation}
for $i=0,\ldots,p-1$, $j=2,\ldots,q$, and $0\leq a<b\leq2$. 
There are $p(q-1)$ shape functions in this family per quadrilateral face.

\subparagraph{Family II:}
Again, taking face 0143 as an example, the second family of shape functions has the form $E_{ij}^\square(\vec{\mu}_{01}(z),\vec{\nu}_{01}(x))$.
This time, proceeding as with \textit{quadrilateral} edges, one can show that the trace properties hold and that the functions lie in $\{0\}\times\{0\}\times(\mathcal{P}^p\otimes\mathcal{P}^{q-1})$.

The general shape functions and their curl are
\begin{equation}
		E_{ij}^{\mathrm{f}}(x,z)=E_{ij}^\square(\vec{\mu}_{01}(z),\vec{\nu}_{ab}(x))\,,\qquad\quad
		\nabla\times E_{ij}^{\mathrm{f}}(x,z)=\nabla\times E_{ij}^\square(\vec{\mu}_{01}(z),\vec{\nu}_{ab}(x))\,,
\end{equation}
for $i=0,\ldots,q-1$, $j=2,\ldots,p$, and $0\leq a<b\leq2$. 
Note the entries are permuted with respect to the first family.
There are $q(p-1)$ shape functions in this family per quadrilateral face.

\subsubsection{\texorpdfstring{$H(\mathrm{curl})$}{Hcurl} Interior Bubbles}

There are three families of interior bubble functions.
They involve tensor products of $\phi_i^\E$, $\phi_{ij}^\Tri$, $E_i^\E$ and  $E_{ij}^\Tri$.
The first two families have elements lying in $(\mathcal{N}^p\otimes\mathcal{P}^q)\times\{0\}$, while the last family has elements in $\{0\}\times\{0\}\times(\mathcal{P}^p\otimes\mathcal{P}^{q-1})$.
The trace properties are also satisfied by using similar arguments to those used for the face and edge functions. 
There is a grand total of $p(p-1)(q-1)+\frac{1}{2}(p-1)(p-2)q$ interior bubble functions.


\subparagraph{Family I:}
The shape functions and their curl are
\begin{equation}
	\begin{aligned}	
		E_{ijk}^\mathrm{b}(x,z)&=\phi_k^\E(\vec{\mu}_{01}(z))E_{ij}^\Tri(\vec{\nu}_{012}(x))\,,\\
		\nabla\times E_{ijk}^\mathrm{b}(x,z)&=\phi_k^\E(\vec{\mu}_{01}(z))\nabla\times E_{ij}^\Tri(\vec{\nu}_{012}(x))
			+\nabla\phi_k^\E(\vec{\mu}_{01}(z))\times E_{ij}^\Tri(\vec{\nu}_{012}(x))\,,
	\end{aligned}
\end{equation}
with $i\geq0$, $j\geq1$, $n=i+j=1,\ldots,p-1$, and $k=2,\ldots,q$. 
The family has $\frac{1}{2}p(p-1)(q-1)$ functions.

\subparagraph{Family II:}
The shape functions and their curl are
\begin{equation}
	\begin{aligned}	
		E_{ijk}^\mathrm{b}(x,z)&=\phi_k^\E(\vec{\mu}_{01}(z))E_{ij}^\Tri(\vec{\nu}_{120}(x))\,,\\
		\nabla\times E_{ijk}^\mathrm{b}(x,z)&=\phi_k^\E(\vec{\mu}_{01}(z))\nabla\times E_{ij}^\Tri(\vec{\nu}_{120}(x))
			+\nabla\phi_k^\E(\vec{\mu}_{01}(z))\times E_{ij}^\Tri(\vec{\nu}_{120}(x))\,,
	\end{aligned}
\end{equation}
with $i\geq0$, $j\geq1$, $n=i+j=1,\ldots,p-1$, and $k=2,\ldots,q$. 
The only difference with the first family is that the permutation $\vec{\nu}_{120}(x)$ is used instead of $\vec{\nu}_{012}(x)$.
The family has $\frac{1}{2}p(p-1)(q-1)$ functions.

\subparagraph{Family III:}
Due to \eqref{eq:Hcurl1Dspecialcase}, the shape functions and their curl are
\begin{equation}
	\begin{aligned}	
		E_{ijk}^\mathrm{b}(x,z)&=\phi_{ij}^\Tri(\vec{\nu}_{012}(x))E_k^\E(\vec{\mu}_{01}(z))\,,\\
		\nabla\times E_{ijk}^\mathrm{b}(x,z)&=\nabla\phi_{ij}^\Tri(\vec{\nu}_{012}(x))\times E_k^\E(\vec{\mu}_{01}(z))\,,
	\end{aligned}
\end{equation}
with $i\geq2$, $j\geq1$, $n=i+j=3,\ldots,p$, and $k=0,\ldots,q-1$. 
The family has $\frac{1}{2}(p-1)(p-2)q$ functions.

\subsection{\texorpdfstring{$H(\mathrm{div})$}{Hdiv} Shape Functions}
The linearly independent shape functions presented here are shown to belong and span the $H(\mathrm{div})$ conforming space $(\mathcal{RT}^p\otimes\mathcal{P}^{q-1})\times(\mathcal{P}^{p-1}\otimes\mathcal{P}^q)$, which has dimension $p(p+2)q+\frac{1}{2}p(p+1)(q+1)$.

As expected, the shape functions are composed of combinations of the lower dimensional $H^1$, $H(\mathrm{curl})$ and $H(\mathrm{div})$ components. 
Intuitively, they involve the affine coordinates and at least $V_{ij}^\Tri$ and $V_{ij}^\square$.
Again, the shape functions respect the logic of projecting, evaluating and blending, and for a given face, the projections are the same as those in $H^1$ (see Figures \ref{fig:PrismProjectionTriangle} and \ref{fig:PrismProjectionQuad}).




\subsubsection{\texorpdfstring{$H(\mathrm{div})$}{Hdiv} Faces}

\paragraph{\texorpdfstring{$H(\mathrm{div})$}{Hdiv} Triangle Faces.} 
These are the tensor product of $H(\mathrm{div})$ triangle face functions and 1D $H^1$ vertex functions. 
The $H(\mathrm{div})$ triangle functions are of the form $V_{ij}^\Tri(\vec{\nu}_{012}(x))$, so by \eqref{eq:HdivtriangleRemark}, it follows that their first two components are zero. 
The last component, which corresponds to the normal trace, by construction is an $L^2$ triangle face function $[P_i,P_j^\alpha](\vec{\nu}_{012}(x))$, which is known to lie in $\mathcal{P}^{p-1}$.
For face 012, the shape functions have the form $\mu_0(z)V_{ij}^\Tri(\vec{\nu}_{012}(x))$, meaning that they lie in $\{0\}\times\{0\}\times(\mathcal{P}^{p-1}\otimes\mathcal{P}^1)$.
The trace properties also follow since the function points normal to the triangle faces, meaning that they are tangent to the quadrilateral faces, an as a result have zero normal trace along those faces.
Meanwhile at the opposite triangle face 345, the function is also zero because $\mu_0(z)$ is zero there.
Lastly, at the face itself, $\mu_0(z)$ is unity, so along the face the normal component is precisely an $L^2$ triangle face function.

Due to \eqref{eq:HdivtriangleRemark}, the triangle shape functions and their divergence are
\begin{equation}
		V_{ij}^{\mathrm{f}}(x,z)=\mu_c(z)V_{ij}^\Tri(\vec{\nu}_{012}(x))\,,\qquad\quad
		\nabla\cdot V_{ij}^{\mathrm{f}}(x,z)=\nabla\mu_c(z)\cdot V_{ij}^\Tri(\vec{\nu}_{012}(x))\,,
\end{equation}
where $i\geq0$, $j\geq0$, $n=i+j=0,\ldots,p-1$ and $c=0,1$. There are $\frac{1}{2}p(p+1)$ shape functions per triangle face, leading to a total of $p(p+1)$ triangle face functions.

\paragraph{\texorpdfstring{$H(\mathrm{div})$}{Hdiv} Quadrilateral Faces.} 
These are cross products of $H(\mathrm{curl})$ edge functions. 
They naturally fall into the definition of $V_{ij}^\square$.
For instance, take face 0143.
Now, $E_i^\E(\vec{\nu}_{01}(x))$ in 3D is a three component vector whose last component is zero, since it is independent of the $z$ coordinate. 
Similarly, $E_i^\E(\vec{\mu}_{01}(z))$ in 3D is a three component vector whose first two components are zero, since it is only dependent on the $z$ coordinate.
It then makes sense to speak of $(E_i^\E(\vec{\nu}_{01}(x)))_x\in\mathcal{N}^p$ and $(E_i^\E(\vec{\mu}_{01}(z)))_z\in\mathcal{P}^{q-1}$. 
Moreover, as discussed in \S\ref{sec:TriangleHdiv}, $(\begin{smallmatrix}0&1\\-1&0\end{smallmatrix})(E_i^\E(\vec{\nu}_{01}(x)))_x\in\mathcal{RT}^p$.
For this face, the shape functions can be written as
\begin{equation*}
	\begin{aligned}
	V_{ij}^\mathrm{f}(x,z)&=V_{ij}^\square(\vec{\nu}_{01}(x),\vec{\mu}_{01}(z))
		=\begin{pmatrix}(E_i^\E(\vec{\nu}_{01}(x)))_x\\0\end{pmatrix}
		\times \begin{pmatrix}0\\(E_i^\E(\vec{\mu}_{01}(z)))_z\end{pmatrix}\\
				&=\begin{pmatrix}(E_i^\E(\vec{\mu}))_z
					\Big(\begin{smallmatrix}0&1\\-1&0\end{smallmatrix}\Big)(E_i^\E(\vec{\nu}_{01}))_x\\0\end{pmatrix}
					\in(\mathcal{RT}^p\otimes\mathcal{P}^{q-1})\times\{0\}\,,
	\end{aligned}
\end{equation*}
for $i\geq0$ and $j\geq0$. 
Hence, $V_{ij}^\mathrm{f}(x,z)\in(\mathcal{RT}^p\otimes\mathcal{P}^{q-1})\times(\mathcal{P}^{p-1}\otimes\mathcal{P}^q)$.
The trace properties are also satisfied because $E_i^\E(\vec{\mu}_{01}(z))$ is normal to the triangle faces, so $V_{ij}^\mathrm{f}(x,z)$ must be tangential, and as a result has zero normal trace.
At the quadrilateral faces one only has to look at the tangential component of $E_i^\E(\vec{\nu}_{01}(x))$ along the mixed edges, and it follows that the normal traces of $V_{ij}^\mathrm{f}(x,z)$ to the faces 1254 and 0253 are zero, while at face 0143 it takes the form of an $L^2$ quadrilateral face function.

The quadrilateral face functions and their divergence are
\begin{equation}
		V_{ij}^{\mathrm{f}}(x,z)=V_{ij}^\square(\vec{\nu}_{ab}(x),\vec{\mu}_{01}(z))\,,\qquad\quad
		\nabla\cdot V_{ij}^{\mathrm{f}}(x,z)=\nabla\cdot V_{ij}^\square(\vec{\nu}_{ab}(x),\vec{\mu}_{01}(z))\,,
\end{equation}
where $i=0,\ldots,p-1$, $j=0,\ldots,q-1$, and $0\leq a<b\leq2$. 
There are $pq$ shape functions per quadrilateral face, for a total of $3pq$ quadrilateral face functions.

\subsubsection{\texorpdfstring{$H(\mathrm{div})$}{Hdiv} Interior Bubbles}

There are three families of $H(\mathrm{div})$ bubble functions. 
Two of them are closely related to $V_{ij}^\square$ and have elements in $(\mathcal{RT}^p\otimes\mathcal{P}^{q-1})\times\{0\}$, while the third family is related to $V_{ij}^\Tri$ an has elements in the space $\{0\}\times\{0\}\times(\mathcal{P}^{p-1}\otimes\mathcal{P}^q)$.
Naturally, the trace properties are satisfied by using similar arguments to those used for the face functions. 
There is a grand total of $p(p-1)q+\frac{1}{2}p(p+1)(q-1)$ interior bubble functions.


\subparagraph{Family I:}
Using \eqref{eq:Hcurl1Dspecialcase}, the shape functions and their divergence are
\begin{equation}
	\begin{aligned}	
		V_{ijk}^\mathrm{b}(x,z)&=E_{ij}^\Tri(\vec{\nu}_{012}(x))\times E_k^\E(\vec{\mu}_{01}(z))\,,\\
		\nabla\cdot V_{ijk}^\mathrm{b}(x,z)&=E_k^\E(\vec{\mu}_{01}(z))\cdot\Big(\nabla\times E_{ij}^\Tri(\vec{\nu}_{012}(x))\Big)\,,
	\end{aligned}
\end{equation}
where $i\geq0$, $j\geq1$, $n=i+j=1,\ldots,p-1$, and $k=0,\ldots,q-1$. 
The family has $\frac{1}{2}p(p-1)q$ functions.

\subparagraph{Family II:}
Using \eqref{eq:Hcurl1Dspecialcase}, the shape functions and their divergence are
\begin{equation}
	\begin{aligned}	
		V_{ijk}^\mathrm{b}(x,z)&=E_{ij}^\Tri(\vec{\nu}_{120}(x))\times E_k^\E(\vec{\mu}_{01}(z))\,,\\
		\nabla\cdot V_{ijk}^\mathrm{b}(x,z)&=E_k^\E(\vec{\mu}_{01}(z))\cdot\Big(\nabla\times E_{ij}^\Tri(\vec{\nu}_{120}(x))\Big)\,,
	\end{aligned}
\end{equation}
where $i\geq0$, $j\geq1$, $n=i+j=1,\ldots,p-1$, and $k=0,\ldots,q-1$.
The only difference with the first family is that the permutation $\vec{\nu}_{120}(x)$ is used instead of $\vec{\nu}_{012}(x)$.
The family has $\frac{1}{2}p(p-1)q$ functions.

\subparagraph{Family III:}
Due to \eqref{eq:HdivtriangleRemark}, the shape functions and their divergence are
\begin{equation}
	\begin{aligned}	
		V_{ijk}^\mathrm{b}(x,z)&=\phi_k^\E(\vec{\mu}_{01}(z))V_{ij}^\Tri(\vec{\nu}_{012}(x))\,,\\
		\nabla\cdot V_{ijk}^\mathrm{b}(x,z)&=\nabla\phi_k^\E(\vec{\mu}_{01}(z))\cdot V_{ij}^\Tri(\vec{\nu}_{012}(x))\,,
	\end{aligned}
\end{equation}
with $i\geq0$, $j\geq0$, $n=i+j=0,\ldots,p-1$, and $k=2,\ldots,q$. 
The family has $\frac{1}{2}p(p+1)(q-1)$ functions.

\subsection{\texorpdfstring{$L^2$}{L2} Shape Functions}
In this case, the space $\mathcal{P}^{p-1}\otimes\mathcal{P}^q$ is spanned by the $\frac{1}{2}p(p+1)q$ linearly independent shape functions.

\subsubsection{\texorpdfstring{$L^2$}{L2} Interior}

These are the tensor products of $L^2$ triangle functions and $L^2$ edge functions. 
They are
\begin{equation}
 \begin{aligned}
	 \psi_{ijk}^\mathrm{b}(x,z)&=[P_i](\nu_0(x),\nu_1(x))[P_j^{2i+1}](\nu_0(x)+\nu_1(x),\nu_2(x))[P_k](\mu_0(z),\mu_1(z))\\
		 &=[P_i,P_j^{2i+1}](\vec{\nu}_{012}(x))[P_k](\vec{\mu}_{01}(z))
		 	(\nabla\nu_1(x)\!\!\times\!\!\nabla\nu_2(x))\!\cdot\!\nabla\mu_1(z)\,,
	\end{aligned}
\end{equation}
where $i\geq0$, $j\geq0$, $n=i+j=0,\ldots,p-1$ and $k=0,\ldots,q-1$. 
There is a total of $\frac{1}{2}p(p+1)q$ interior functions.



\subsection{Orientations}
\label{sec:PrismOrientations}

\begin{figure}[!ht]
\begin{center}
\includegraphics[scale=0.5]{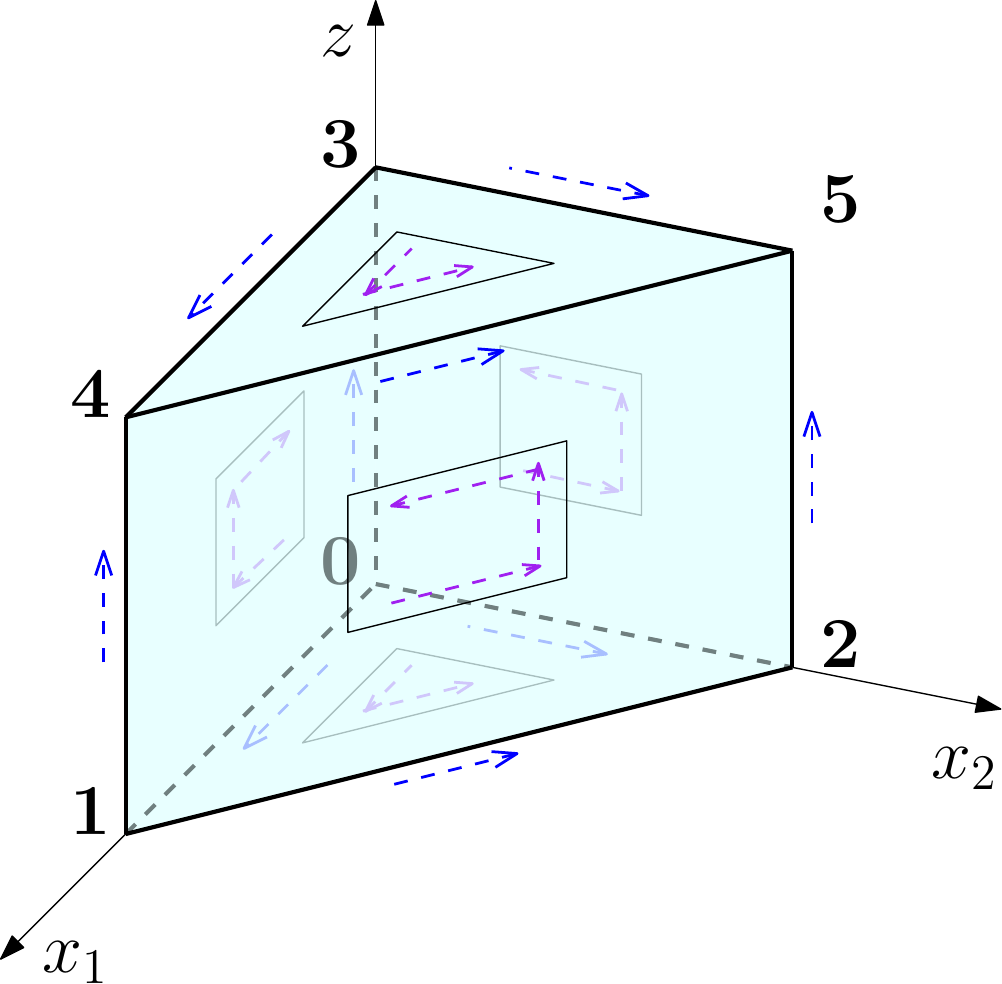}
\caption{Master prism with numbered vertices \textit{and} local edge and face orientations.}
\label{fig:MasterPrismOrientations}
\end{center}
\end{figure}

To construct orientation embedded shape functions for the prism, it is recommended to have read \S\ref{sec:HexaOrientations} and \S\ref{sec:TetOrientations}.
The predefined \textit{local} edge and face orientations for the prism are illustrated in Figure \ref{fig:MasterPrismOrientations}.
They represent the $\oo=0$ case.
The task at hand is to find the associated \textit{locally ordered} tuples of affine coordinates representing those local orientations.
As usual, the key is being aware of the relationships between the vertices and the affine coordinates.
As examples take edges 01 and 03, and faces 012 and 0143.

For mixed edge 01, the vertices are $v_0$ and $v_1$, which are linked to $(\nu_0(x),\mu_0(z))$ and $(\nu_1(x),\mu_0(z))$ respectively.
The only difference is that $v_0$ is linked to $\nu_0(x)$, while $v_1$ is linked to $\nu_1(x)$.
The local orientation for edge 01 is represented by the local vertex-ordering $v_0\tdashto v_1$.
Therefore, the locally ordered pair for edge 01 is $\vec{\nu}_{01}(x)=(\nu_0(x),\nu_1(x))$.
The orientation embedded shape functions for mixed edge 01 are simply the usual shape functions, but with their respective ancillary operator and differential form (that is $\phi_i^\E$, $E_i^\E$, $\nabla\phi_i^\E$ and $\nabla\times E_i^\E$) being precomposed with $\sigma_\oo^\E$ and evaluated at the locally ordered pair.

For quadrilalteral edge 03, composed of vertices $v_0$ and $v_3$, the differing affine coordinates are $\mu_0(z)$ and $\mu_1(z)$ respectively.
Since the local vertex-ordering is $v_0\tdashto v_3$, it follows the locally ordered pair for this edge is $\vec{\mu}_{01}(z)$.
Then, the orientation embedded shape functions are constructed like those of mixed edge 01.
That is, precomposing the ancillary operators with $\sigma_\oo^\E$ and evaluating at the locally ordered pair.

For triangle face 012, composed of vertices $v_0$, $v_1$ and $v_2$, the differing affine coordinates are $\nu_0(x)$, $\nu_1(x)$ and $\nu_2(x)$ respectively.
The local vertex-ordering is $v_0\tdashto v_1\tdashto v_2$, so the locally ordered triplet for this edge is $\vec{\nu}_{012}(x)$.
Then, the orientation embedded shape functions are the usual shape functions but with the ancillary operators ($\phi_{ij}^\Tri$, $E_{ij}^\Tri$, $V_{ij}^\Tri$, $\nabla\phi_{ij}^\Tri$, $\nabla\times E_{ij}^\Tri$ and $\nabla\cdot V_{ij}^\Tri$) precomposed with $\sigma_\oo^\Tri$ and evaluated at the locally oriented triplet.

Finally, quadrilateral face 0143 has local vertex-ordering $v_0\tdashto v_1\tdashto v_4\tdashto v_3$, so one only needs to look at $v_0\tdashto v_1$ and $v_1\tdashto v_4$ as if they were edges.
This leads to the locally ordered pairs $\vec{\nu}_{01}(x)$ and $\vec{\mu}_{01}(z)$ respectively, so the locally ordered quadruple is $(\vec{\nu}_{01}(x),\vec{\mu}_{01}(z))$.
Again, the orientation embedded shape functions are simply the shape functions but with the ancillary operators ($\phi_{ij}^\square$, $E_{ij}^\square$, $V_{ij}^\square$, $\nabla\phi_{ij}^\square$, $\nabla\times E_{ij}^\square$ and $\nabla\cdot V_{ij}^\square$) precomposed with $\sigma_\oo^\square$ and evaluated at the locally oriented quadruple.

%% file: pyramid.tex
\section{Pyramid}
\label{sec:Pyramid}

\begin{figure}[!ht]
\begin{center}
\includegraphics[scale=0.5]{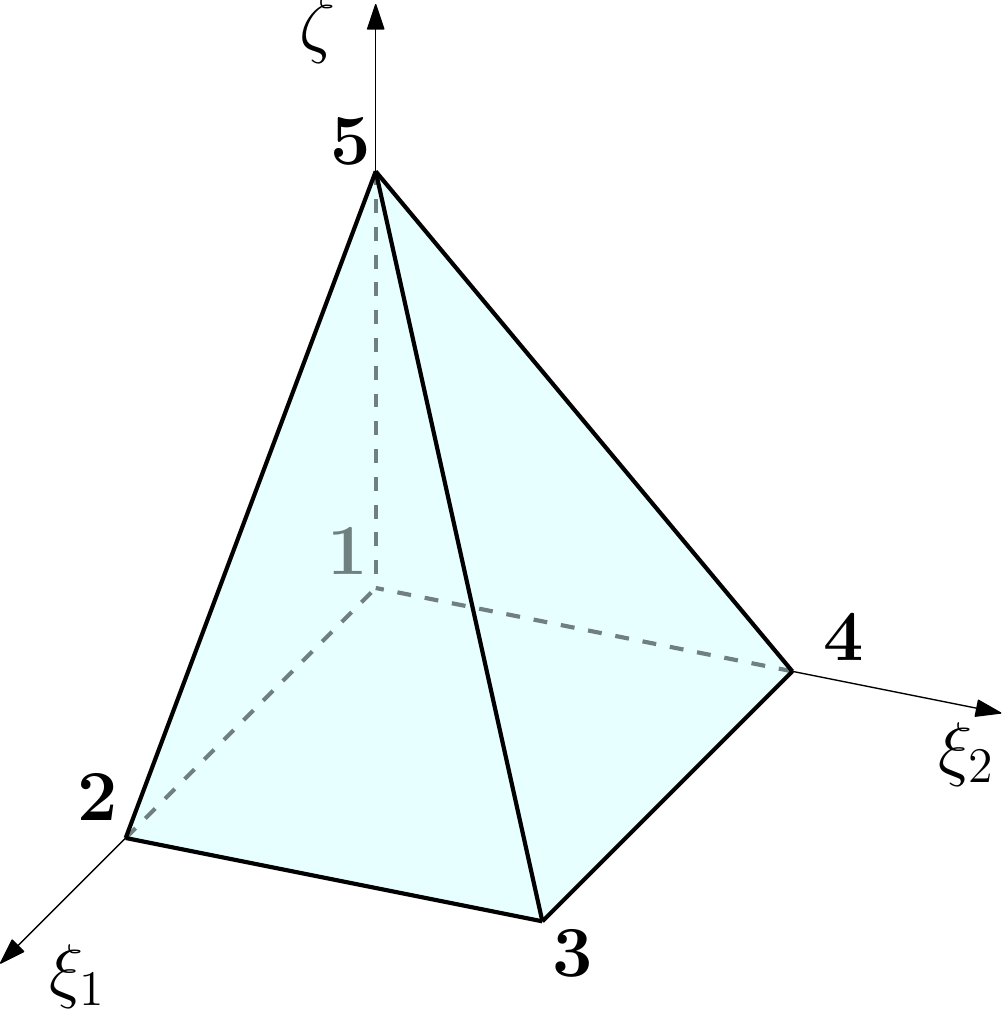}
\caption{Master pyramid with numbered vertices.}
\label{fig:MasterPyramid}
\end{center}
\end{figure}

The master pyramid is shown in Figure \ref{fig:MasterPyramid} in the $(\xi,\zeta)=(\xi_1,\xi_2,\zeta)$ space.
More specifically, the definition is $\{(\xi_1,\xi_2,\zeta)\in\R^3:\xi_1>0,\xi_2>0,\zeta>0,\xi_1+\zeta<1,\xi_2+\zeta<1\}$.

Clearly, the pyramid is neither a simplex nor a Cartesian product, but it captures features of both quadrilaterals and triangles.
Indeed, it has both quadrilateral and triangle faces. 
Similarly, it has two types of edges. 
Those edges which are adjacent to both a quadrilateral face and a triangle face are called \textit{mixed edges}, while those edges only shared by triangle faces alone are called \textit{triangle edges}. 
These distinctions are fundamental, and the form of the shape functions will differ for the different types of edges and faces.

Due to the virtually unknown structure of the pyramid, at first it seems almost like an insurmountable task to be able to find representative functions that resemble affine coordinates for this 3D element. 
Surprisingly, there is in fact such a set (or rather set\textit{s}) of coordinates.
However, to reach that point, it is better to start by elementary means.
With this in mind, the idea is to separately analyze the affine coordinates of the quadrilateral face and the triangle faces.

In truth, nothing is getting in the way of explicitly computing the 2D triangle affine coordinates of each of the four triangle faces as described in \S\ref{sec:affinecoordinates}.
There turns out to be two independent sets of such coordinates which are
\begin{equation}
	\begin{gathered}
		\nu_0(\xi_1,\zeta)=1-\xi_1-\zeta\,,\qquad\nu_1(\xi_1,\zeta)=\xi_1\,,\qquad\nu_2(\zeta)=\zeta\,,\\
		\nu_0(\xi_2,\zeta)=1-\xi_2-\zeta\,,\qquad\nu_1(\xi_2,\zeta)=\xi_2\,,\qquad\nu_2(\zeta)=\zeta\,.
	\end{gathered}
	\label{eq:PyramidTriCoord}
\end{equation}
Indeed, the triplet $\vec{\nu}_{012}(\xi_1,\zeta)$ represents triangle faces 125 and 435, while the triplet $\vec{\nu}_{012}(\xi_2,\zeta)$ represents triangle faces 145 and 235.
Their gradient is
\begin{equation}
	\begin{gathered}
		\nabla\nu_0(\xi_1,\zeta)=\bigg(\begin{smallmatrix}-1\\[2pt]0\\[2pt]-1\end{smallmatrix}\bigg)\,,\qquad
			\nabla\nu_1(\xi_1,\zeta)=\bigg(\begin{smallmatrix}1\\[2pt]0\\[2pt]0\end{smallmatrix}\bigg)\,,\qquad
				\nabla\nu_2(\zeta)=\bigg(\begin{smallmatrix}0\\[2pt]0\\[2pt]1\end{smallmatrix}\bigg)\,,\\
		\nabla\nu_0(\xi_2,\zeta)=\bigg(\begin{smallmatrix}0\\[2pt]-1\\[2pt]-1\end{smallmatrix}\bigg)\,,\qquad
			\nabla\nu_1(\xi_2,\zeta)=\bigg(\begin{smallmatrix}0\\[2pt]1\\[2pt]0\end{smallmatrix}\bigg)\,,\qquad
				\nabla\nu_2(\zeta)=\bigg(\begin{smallmatrix}0\\[2pt]0\\[2pt]1\end{smallmatrix}\bigg)\,.
	\end{gathered}
	\label{eq:PyramidTriCoordGrad}
\end{equation}

Now, the quadrilateral face can undergo a similar treatment, resulting in the standard two sets of 1D affine coordinates,
\begin{equation*}
	\begin{gathered}
		\mu_0(\xi_1)=1-\xi_1\,,\quad\qquad\mu_1(\xi_1)=\xi_1\,,\\
		\mu_0(\xi_2)=1-\xi_2\,,\quad\qquad\mu_1(\xi_2)=\xi_2\,.
	\end{gathered}
\end{equation*}
However, these are convenient only when restricted to the 2D quadrilateral face, and not in 3D.
The reason is that they do \textit{not} act as blending functions to the \textit{faces}.
For example, $\mu_0(\xi_2)$ is unity at face 125, but it does \textit{not} vanish at the opposite face, which is face 435.
This inconvenience does not occur with the hexahedron or prism due to the Cartesian product structure of those elements.
Despite this setback, it is possible to fix this by considering \textit{scaled} coordinates which additionally depend on $\zeta$.
The sets of quadrilateral scaled 1D affine coordinates are
\begin{equation}
	\begin{gathered}
		\mu_0(\sxi)=1-\sxi\,,\quad\qquad\mu_1(\sxi)=\sxi\,,\\
		\mu_0(\sxii)=1-\sxii\,,\quad\qquad\mu_1(\sxii)=\sxii\,.
	\end{gathered}
	\label{eq:PyramidQuadCoord}
\end{equation}
These can be readily checked to act as \textit{face} blending functions between the opposite triangular faces, which is precisely what was desired.
Moreover, when restricted to the quadrilateral face, so $\zeta=0$, they coincide with the usual sets of affine coordinates for the 2D quadrilateral faces.
Their gradient is
\begin{equation}
	\begin{gathered}
		\nabla\mu_0(\sxi)=\frac{1}{(1-\zeta)^2}\bigg(\begin{smallmatrix}-(1-\zeta)\\[2pt]0\\[2pt]-\xi_1\end{smallmatrix}\bigg)\,,\quad\qquad
    	\nabla\mu_1(\sxi)=\frac{1}{(1-\zeta)^2}\bigg(\begin{smallmatrix}(1-\zeta)\\[2pt]0\\[2pt]\xi_1\end{smallmatrix}\bigg)\,,\\
		\nabla\mu_0(\sxii)=\frac{1}{(1-\zeta)^2}\bigg(\begin{smallmatrix}0\\[2pt]-(1-\zeta)\\[2pt]-\xi_2\end{smallmatrix}\bigg)\,,\quad\qquad
    	\nabla\mu_1(\sxii)=\frac{1}{(1-\zeta)^2}\bigg(\begin{smallmatrix}0\\[2pt](1-\zeta)\\[2pt]\xi_2\end{smallmatrix}\bigg)\,.
	\end{gathered}
	\label{eq:PyramidQuadCoordGrad}
\end{equation}

Lastly, the 1D affine coordinates associated to the nonquadrilateral (top) vertex and the perpendicularly projected point to the quadrilateral face are
\begin{equation}
		\mu_0(\zeta)=1-\zeta\,,\quad\qquad\mu_1(\zeta)=\zeta\,.
		\label{eq:PyramidZCoord}
\end{equation}
Their gradient is
\begin{equation}
		\nabla\mu_0(\zeta)=\bigg(\begin{smallmatrix}0\\[2pt]0\\[2pt]-1\end{smallmatrix}\bigg)\,,\quad\qquad
			\nabla\mu_1(\zeta)=\bigg(\begin{smallmatrix}0\\[2pt]0\\[2pt]1\end{smallmatrix}\bigg)\,.
		\label{eq:PyramidZCoordGrad}			
\end{equation}

With these tools in the arsenal, it is possible to find the desired 3D affine-like coordinates.
The first key observation is that each vertex in the quadrilateral face is associated to \textit{four} lower dimensional affine coordinates.
The associated coordinates are those which take the value $1$ at the given vertex.
For example, vertex $v_1=(0,0,0)$ is linked to the coordinates $\nu_0(\xi_1,\zeta)$, $\nu_0(\xi_2,\zeta)$, $\mu_0(\sxi)$ and $\mu_0(\sxii)$.
To find a global coordinate associated to any vertex, the idea is to combine these components such that they vanish at all disjoint edges and faces.
One possibility is to consider the product of all four coordinates.
However, this gives a high order function, which is somewhat inconsistent with what one would expect.
Hence, the global coordinate should look as ``simple'' as possible.
Fortunately, there is such a coordinate, which in fact has a dual interpretation with respect to its associated coordinates.
It is the product of a 1D scaled affine coordinate and the complementing 2D affine coordinate.
For vertex $v_0$, it would either be $\mu_0(\sxi)\nu_0(\xi_2,\zeta)$ or $\mu_0(\sxii)\nu_0(\xi_1,\zeta)$. 
These two interpretations coincide and define the pyramid affine-related coordinates.
For the nonquadrilateral vertex, $v_5=(0,0,1)$, there is an already existing affine-related coordinate which is merely $\nu_2(\zeta)=\mu_1(\zeta)$.
In summary, the pyramid affine-related coordinates are
\begin{equation}
	\begin{alignedat}{5}
		&\lambda_1(\xi,\zeta)&&=\mu_0(\sxi)\nu_0(\xi_2,\zeta)&&=\mu_0(\sxii)\nu_0(\xi_1,\zeta)
			&&=\textstyle{\frac{(1-\xi_1-\zeta)(1-\xi_2-\zeta)}{1-\zeta}}\,,\\
		&\lambda_2(\xi,\zeta)&&=\mu_1(\sxi)\nu_0(\xi_2,\zeta)&&=\mu_0(\sxii)\nu_1(\xi_1,\zeta)
			&&=\textstyle{\frac{\xi_1(1-\xi_2-\zeta)}{1-\zeta}}\,,\\
		&\lambda_3(\xi,\zeta)&&=\mu_1(\sxi)\nu_1(\xi_2,\zeta)&&=\mu_1(\sxii)\nu_1(\xi_1,\zeta)
			&&=\textstyle{\frac{\xi_1\xi_2}{1-\zeta}}\,,\\
		&\lambda_4(\xi,\zeta)&&=\mu_0(\sxi)\nu_1(\xi_2,\zeta)&&=\mu_1(\sxii)\nu_0(\xi_1,\zeta)
			&&=\textstyle{\frac{(1-\xi_1-\zeta)\xi_2}{1-\zeta}}\,,\\
		&\lambda_5(\zeta)&&=\nu_2(\zeta)&&=\mu_1(\zeta)&&=\zeta\,.
	\end{alignedat}
	\label{eq:PyramidAffineCoord}
\end{equation}
Their gradient is
\begin{equation}
	\begin{gathered}
    \nabla\lambda_1(\xi,\zeta)=\begin{pmatrix}\frac{-(1-\xi_2-\zeta)}{1-\zeta}\\\frac{-(1 -\xi_1-\zeta)}{1-\zeta}\\
        \frac{\xi_1\xi_2}{(1-\zeta)^2}-1\end{pmatrix}\,,\quad
    \nabla\lambda_2(\xi,\zeta)=\begin{pmatrix}\frac{(1-\xi_2-\zeta)}{1-\zeta}\\\frac{-\xi_1}{1-\zeta}\\
        \frac{-\xi_1\xi_2}{(1-\zeta)^2}\end{pmatrix}\,,\quad
    \nabla\lambda_3(\xi,\zeta)=\begin{pmatrix}\frac{\xi_2}{1-\zeta}\\\frac{\xi_1}{1-\zeta}\\
        \frac{\xi_1\xi_2}{(1-\zeta)^2}\end{pmatrix}\,,\\
    \nabla\lambda_4(\xi,\zeta)=\begin{pmatrix}\frac{-\xi_2}{1-\zeta}\\\frac{(1-\xi_1-\zeta)}{1-\zeta}\\
        \frac{-\xi_1\xi_2}{(1-\zeta)^2}\end{pmatrix}\,,\quad
    \nabla\lambda_5(\zeta)=\begin{pmatrix}0\\0\\1\end{pmatrix}\,.
	\end{gathered}
	\label{eq:PyramidAffineCoordGrad}
\end{equation}

Apart from being products of lower dimensional affine coordinates, the pyramid affine-related coordinates truly do behave in many ways like 3D affine coordinates.
Firstly, notice that by construction the traces over adjacent faces and edges are the corresponding vertex functions of those lower dimensional topological entities.
For example, the trace of $\lambda_1$ over faces 125 and 145 is a 2D triangle affine coordinate associated to that vertex, while that of face 1234 is a bilinear quadrilateral vertex function.
Secondly, note that every $(\xi,\zeta)$ in the pyramid can be expressed as a convex combination of the vertices with the affine coordinates being the weights,
\begin{equation}
	\bigg(\begin{smallmatrix}\xi_1\\[2pt]\xi_2\\[2pt]\zeta\end{smallmatrix}\bigg)=\sum_{a=1}^5 \lambda_a(\xi,\zeta)v_a\,,
		\qquad\quad\text{ with }\quad\sum_{a=1}^5 \lambda_a(\xi,\zeta)=1\,\;\text{ and }\; 0\leq\lambda_a(\xi,\zeta)\leq1\,,
\end{equation}
and where $v_a$ are the coordinates of vertex $a$. 
The main difference with the legitimate simplex affine coordinates radicates in the fact the the pyramid affine-related coordinates are not \textit{defined} by the properties above (see \S\ref{sec:affinecoordinates}).
Indeed, even though they have polynomial traces at the boundary, they involve rational polynomials in the interior, and this is an inherently new property.
Nevertheless, for many practical purposes, they can be thought of as affine coordinates, and from now on will be referred to as \textit{the} pyramid affine coordinates.

An important remark is that all the results associated to the definitions of the ancillary functions were proved in a very general setting that encompasses the pyramid affine coordinates and the fact they can be rational. 
In particular, the proofs of Lemmas \ref{lem:curlformula} and \ref{lem:divformula} hold.

Note that all the affine coordinates illustrated can be computed for pyramids with a parallelogram base.
In fact, it is very easy to make these calculations for pyramids with an arbitrarily placed top vertex and whose rectangular base is normal to the vertical $\zeta$ direction and aligned with the $\xi$ coordinates.
This assertion includes any of the typical master pyramids found in the literature.
With the affine coordinates computed, it is just a matter of substituting them (and their gradient) into the expressions for the shape functions to be presented throughout this section, so that in fact these expressions are independent of the choice of the master pyramid.
Hopefully, this motivates other researchers to communicate their results in terms of affine coordinates as well.

Finally, by construction, there are natural relationships between the topological entities and the different types of affine coordinates defined.
The related affine coordinates are those which take the value $1$ at the prescribed topological entity.
The top vertex, $v_5$, is linked to $\nu_2(\zeta)=\mu_1(\zeta)=\lambda_5(\zeta)$.
The quadrilateral vertices are each associated to \textit{two} 1D scaled affine coordinates, \textit{two} 2D triangle affine coordinates and \textit{one} 3D pyramid affine coordinate.
Meanwhile, triangle edges are linked to \textit{two} 1D scaled affine coordinates, while mixed edges are associated to \textit{one} 1D scaled affine coordinate and the vertical 1D affine coordinate $\mu_0(\zeta)$.
Lastly, triangle faces are linked to \textit{one} 1D scaled affine coordinate, while the quadrilateral face is linked to the vertical 1D affine coordinate $\mu_0(\zeta)$.
As usual, these associated affine coordinates can act as natural blending functions.

\subsubsection*{Exact Sequence}

It should be clear by now that the pyramid has a fundamentally different structure than the previous elements, and one would expect this to have an impact on the discrete spaces that attempt to approximate the energy spaces in \eqref{eq:3D_exact_sequence}.

Firstly, note that an absolute requirement is that the trace of the spaces over the faces span the lower dimensional discrete polynomial spaces for the triangle and quadrilateral respectively.
This is what ensures that the shape functions are compatible over adjacent elements.
However, any attempt at finding a three dimensional polynomial space satisying those properties is futile, since one can find counterexamples mathematically showing that this task is impossible.

Hence, the use of \textit{rational} polynomial spaces is the next natural step.
This issue already arised, at least intuitively, while analyzing the desired properties of affine coordinates, because the use of \textit{scaled} coordinates was required.
Nevertheless, dealing with rational polynomial spaces is difficult, and finding finite dimensional higher order spaces satisfying all the desired trace, exact sequence and approximability properties is a far from trivial task.
In fact, only until recently did such constructions started to appear in the literature.
In the context of this work, perhaps the best suited set of such spaces is that proposed by \citet{Nigam_Phillips_11}, which is consistent with the ``natural'' first order spaces analyzed first by \citet{Hiptmair99}.

Respecting the notation of \citet{Nigam_Phillips_11}, the discrete rational polynomial spaces approximating \eqref{eq:3D_exact_sequence} are,
\begin{equation}
	\mathcal{U}^{(0),p} \xrightarrow{\,\,\nabla\,\,} \mathcal{U}^{(1),p} \xrightarrow{\nabla\times} 
	\mathcal{U}^{(2),p} \xrightarrow{\nabla\cdot} \mathcal{U}^{(3),p} \,,
	\label{eq:pyramidsequence}
\end{equation}
where the $m$ in $\mathcal{U}^{(m),p}$ corresponds to the order of the differential form in 3D, so that the elements in $\mathcal{U}^{(0),p}$ are $0$-forms, and so on.
The precise definitions of these spaces are somewhat technical and will be postponed to Appendix \ref{app:pyrappendix}.
In fact, the proofs that the shape funtions lie in the desired space are also technical and inconveniently load the readibility of the document, so they are presented in Appendix \ref{app:pyrappendix} as well. 
This by no means implies that the spaces are not important and do not play a role in the construction.
In fact, quite the opposite.
The spaces are so well suited to the pyramid, that most of the time they impose little restrictions on the intuitive constructions presented here.
Hence, in many ways, despite looking complicated, they are ``natural''.

Finally, it is worth emphasizing that the goal in this section (and in general in this work) is to motivate the construction of the shape functions through geometrical arguments (via the affine coordinates defined before) combined with the carefully chosen ancillary operators defined throughout the document.
This approach leads to shape functions satisfying the desired trace properties and which either are in the desired space or can be naturally tweaked to lie in the space.
The notable exception is that of the $H(\mathrm{div})$ triangle faces, in which the space truly plays a nontrivial role and forces to consider a more intricate yet consistent construction.

\subsection{\texorpdfstring{$H^1$}{H1} Shape Functions}

The dimension of the space $\mathcal{U}^{(0),p}$ is $p^3+3p+1$.
The number of linearly independent shape functions will coincide with that dimension.

\subsubsection {\texorpdfstring{$H^1$}{H1} Vertices}

The vertex shape functions will be precisely the associated 3D pyramid affine coordinates.
Indeed, take for example vertex $v_1$, so that the vertex function is
\begin{equation*}
	\phi^\mathrm{v}(\xieta)=\lambda_1(\xieta)\,.
\end{equation*}
The trace properties are satisfied by construction and are shown explicitly next,
\begin{equation*}
	\begin{alignedat}{5}
    &\phi^\mathrm{v}(\xieta)|_{\xi_2=0}&&=\lambda_1(\xieta)|_{\mu_0(\frac{\xi_1}{1-\zeta})=1}
    	&&=\mu_0(\sxi)\nu_0(\xi_2,\zeta)|_{\mu_0(\frac{\xi_1}{1-\zeta})=1}&&=\nu_0(\xi_2,\zeta)\,,\\
    &\phi^\mathrm{v}(\xieta)|_{\xi_1=1-\zeta}&&=\lambda_1(\xieta)|_{\mu_0(\frac{\xi_2}{1-\zeta})=0}
    	&&=\mu_0(\sxii)\nu_0(\xi_1,\zeta)|_{\mu_0(\frac{\xi_2}{1-\zeta})=0}&&=0\,,\\
  	&\phi^\mathrm{v}(\xieta)|_{\xi_2=1-\zeta}&&=\lambda_1(\xieta)|_{\mu_0(\frac{\xi_1}{1-\zeta})=0}
    	&&=\mu_0(\sxi)\nu_0(\xi_2,\zeta)|_{\mu_0(\frac{\xi_1}{1-\zeta})=0}&&=0\,,\\
    &\phi^\mathrm{v}(\xieta)|_{\xi_2=0}&&=\lambda_1(\xieta)|_{\mu_0(\frac{\xi_2}{1-\zeta})=1}
    	&&=\mu_0(\sxii)\nu_0(\xi_1,\zeta)|_{\mu_0(\frac{\xi_2}{1-\zeta})=1}&&=\nu_0(\xi_1,\zeta)\,,\\
    &\phi^\mathrm{v}(\xieta)|_{\zeta=0}&&=\lambda_1(\xieta)|_{\zeta=0}
    	&&=\mu_0(\xi_1)\mu(\xi_2)\,.&&\\
	\end{alignedat}
\end{equation*}
The function is also in the lowest order space $\mathcal{U}^{(0),1}$.
Similar arguments apply to all other quadrilateral vertices and the top vertex as well.

More generally, the vertex functions and their gradient are,
\begin{equation}
    \phi^\mathrm{v}(\xieta)=\lambda_a(\xieta)\,,\qquad\quad\nabla\phi^\mathrm{v}(\xieta)=\nabla\lambda_a(\xieta)\,,
    \label{eq:PyrH1Vertex}
\end{equation}
for $a=1,2,3,4,5$.
There are a total of $5$ vertex functions (one for each vertex).

\subsubsection{\texorpdfstring{$H^1$}{H1} Edges}

\paragraph{\texorpdfstring{$H^1$}{H1} Mixed Edges.} 
Take for example mixed edge 12.
The first naive approach is to use the 3D pyramid affine coordinates directly on $\phi_i^\E$, which gives
\begin{equation*}
	\phi_i^\mathrm{e}(\xieta)=\phi_i^\E(\vec{\lambda}_{12}(\xieta))
		=\phi_i^\E\Big(\mu_0(\sxii)\nu_0(\xi_1,\zeta),\mu_0(\sxii)\nu_1(\xi_1,\zeta)\Big)
			=\mu_0(\sxii)^i\phi_i^\E(\vec{\nu}_{01}(\xi_1,\zeta))\,,
\end{equation*}
for $i=2,\ldots,p$.
This attempt almost works because it is in the correct space, satisfies the vanishing conditions, and even has the right form at the edge itself.
Indeed, at triangle faces 235 and 145, $\nu_0(\xi_1,\zeta)=0$ and $\nu_1(\xi_1,\zeta)=0$ respectively, while $\mu_0(\sxii)=0$ at face 435.
However, the nonzero trace over the adjacent quadrilateral face is not of the correct form, since it blends nonlinearly with the factor $\mu_0(\xi_2)^i$ instead of linearly like $\mu_0(\xi_2)$.
Therefore the function violates dimensional hierarchy and does not work for our purposes.
Nevertheless this analysis ellucidates how to fix the issue.
The idea is to have the factor $\mu_0(\sxii)$ separated as a blending factor, so that the effects of $\mu_0(\sxii)$ are essentially separated from those of $\vec{\nu}_{01}(\xi_1,\zeta)$ in $\vec{\lambda}_{12}(\xieta)$.
Hence, the shape functions for this edge are
\begin{equation*}
	\phi_i^\mathrm{e}(\xieta)=\mu_0(\sxii)\phi_i^\E(\vec{\nu}_{01}(\xi_1,\zeta))=
		\underbrace{\mu_0(\sxii)(\nu_0(\xi_1,\zeta)+\nu_1(\xi_1,\zeta))^i}_{\text{blend}}
    	\underbrace{\phi_i^\E\Big(\underbrace{\textstyle{\frac{1}{\nu_0(\xi_1,\zeta)+\nu_1(\xi_1,\zeta)}}
    		\vec{\nu}_{01}(\xi_1,\zeta)}_{\text{project}}\Big)}_{\text{evaluate}}\,,
\end{equation*}
for $i=2,\ldots,p$.
As with the previous candidate all vanishing properties are satisfied, but this time the nonzero trace properties are also easily seen to hold.
The projection being implied is
\begin{equation*}
	(\xi_1,\xi_2,\zeta)\;\longmapsto\;\begin{matrix}(\xi_1,0,\zeta)\\(\sxi,\sxii,0)\end{matrix}
		\;\longmapsto\;(\sxi,0,0)\,.
\end{equation*}
It is a two step projection, where the first step is to project to an adjacent face and the second is to project along that face to the given edge via the standard 2D edge projections (see Figures \ref{fig:QuadProjection} and \ref{fig:TriangleProjection}).
If the face projection is chosen as the triangle, then the projection at play is called the \textit{horizontal} triangle face projection and consists of finding the intersection $P'=(\xi_1,0,\zeta)$ of the face with the projecting line parallel to the $\xi_2$ direction and passing through the original point $P=(\xi_1,\xi_2,\zeta)$. 
This is shown in Figure \ref{fig:PyramidProjectionHorizontalTriangle}.
If the face projection is chosen as the quadrilateral, then the projection is simply the intersection $P'=(\sxi,\sxii,0)$ of the face with the projecting line passing through the top vertex and the original point $P=(\xi_1,\xi_2,\zeta)$. 
This is shown in Figure \ref{fig:PyramidProjectionQuad}.

\begin{figure}[!ht]
\begin{center}
\includegraphics[scale=0.6]{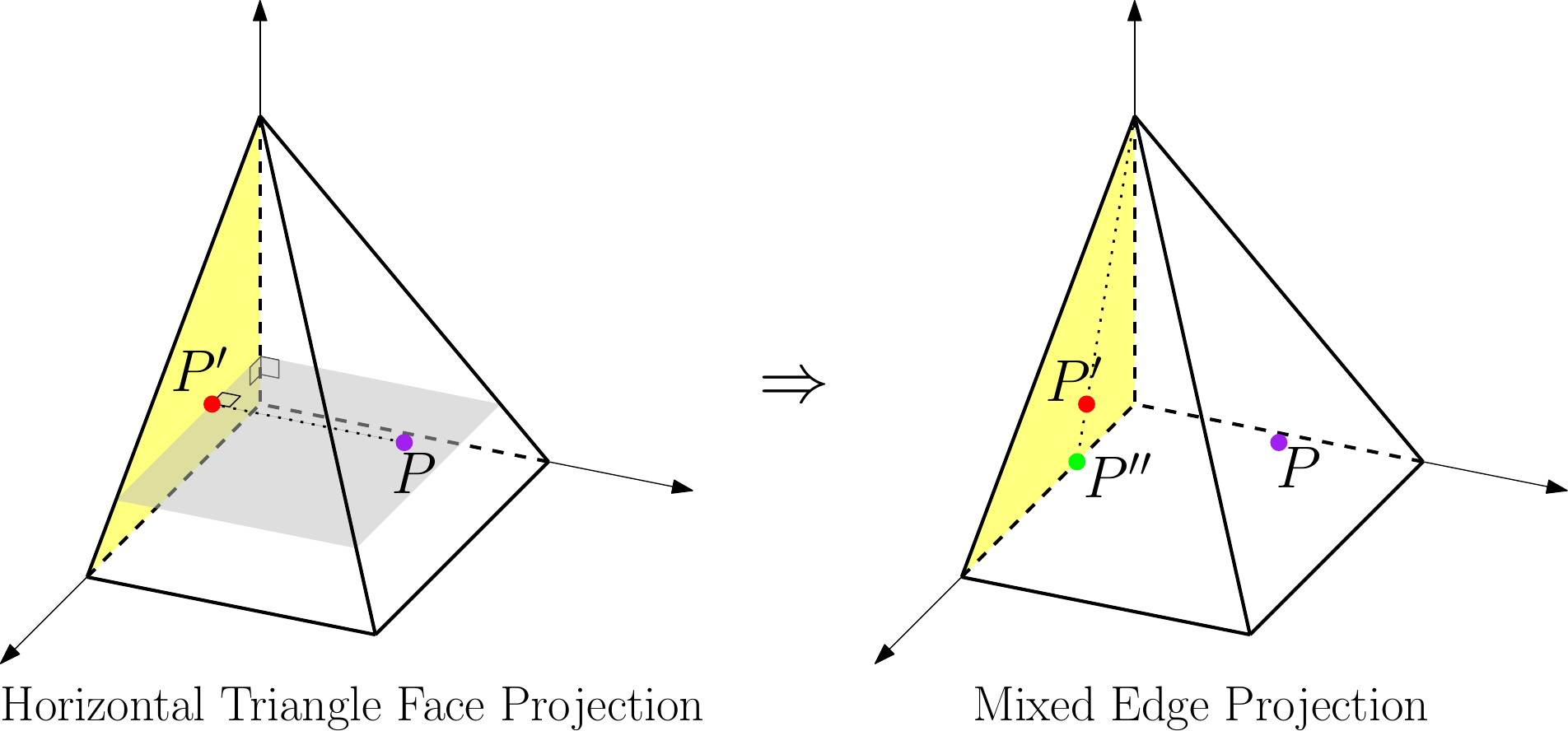}
\caption{Horizontal triangle face projection from $P$ to $P'$ followed by a mixed edge projection from $P'$ to $P''$.}
\label{fig:PyramidProjectionHorizontalTriangle}
\end{center}
\end{figure}

In general, the shape functions and their gradient are
\begin{equation}
	\begin{aligned}
		\phi_i^\mathrm{e}(\xieta)&=\mu_c(\sxib)\phi_i^\E(\vec{\nu}_{01}(\xi_a,\zeta))\,,\\
    	\nabla\phi_i^\mathrm{e}(\xieta)&=\mu_c(\sxib)\nabla\phi_i^\E(\vec{\nu}_{01}(\xi_a,\zeta))
        +\phi_i^\E(\vec{\nu}_{01}(\xi_a,\zeta))\nabla\mu_c(\sxib)\,,	
	\end{aligned}
	\label{eq:PyrH1MixedEdge}
\end{equation}
where $i=2,\ldots,p$, $(a,b)=(1,2),(2,1)$ and $c=0,1$.
There are $p-1$ edge function for each edge, for a total of $4(p-1)$ mixed edge functions.

\paragraph{\texorpdfstring{$H^1$}{H1} Triangle Edges.}
For instance, take triangle edge 15.
Again, the naive approach is to use the 3D pyramid affine coordinates on $\phi_i^\E$, leading to the shape functions,
\begin{equation*}
	\phi^\mathrm{e}(\xieta)=\phi_i^\E(\vec{\lambda}_{15}(\xieta))=
		\underbrace{(\lambda_1(\xieta)+\lambda_5(\zeta))^i}_{\text{blend}}
    	\underbrace{\phi_i^\E\Big(\underbrace{\textstyle{\frac{1}{\lambda_1(\xieta)+\lambda_5(\zeta)}}
    		\vec{\lambda}_{15}(\xieta)}_{\text{project}}\Big)}_{\text{evaluate}}\,,
\end{equation*}
for $i=2,\ldots,p$.
In this case it works perfectly well, with the trace properties being satisfied.
Indeed, $\lambda_1(\xieta)=0$ over faces 235 and 435, while $\lambda_5(\xieta)=0$ over the quadrilateral face.
Moreover the restriction of $\vec{\lambda}_{15}(\xieta)$ over the faces 125 and 145 gives $\vec{\nu}_{02}(\xi_1,\zeta)$ and $\vec{\nu}_{02}(\xi_2,\zeta)$ respectively, so the nonzero traces are the appropriate triangle traces.
The projection being implied here is highly nontrivial. 
It is a two step projection given by
\begin{equation*}
	(\xi_1,\xi_2,\zeta)\;\longmapsto\;\Big(\textstyle{\frac{\xi_1(1-\xi_2-\zeta)}{(1-\xi_2)(1-\zeta)}},
		0,\textstyle{\frac{\zeta}{1-\xi_2}}\Big)\;\longmapsto\;
			\Big(0,0,\textstyle{\frac{\zeta(1-\zeta)}{(1-\xi_1-\zeta)(1-\xi_2-\zeta)+\zeta(1-\zeta)}}\Big)\,.
\end{equation*}
The first step is called an \textit{oblique} triangle face projection and consists of running a plane through the original point $P=(\xi_1,\xi_2,\zeta)$ and the opposite bottom edge to the face (edge 43), followed by finding the intersection of this plane with the planes passing through the other two adjacent triangular faces (faces 235 and 415 with equations $\xi_1=1-\zeta$ and $\xi_1=0$ respectively).
Call this intersection $C=(0,-(\frac{1-\xi_2-\zeta}{\zeta}),1)$.
Finally, the intersection of face 125 with the projecting line from the original point $P$ to the intersection $C$ is found and labeled as $P'=(\frac{\xi_1(1-\xi_2-\zeta)}{(1-\xi_2)(1-\zeta)},0,\frac{\zeta}{1-\xi_2})$.
This projection is illustrated in Figure \ref{fig:PyramidProjectionObliqueTriangle}.
The final step is simply to project as usual along the 2D triangle face to the point $P''=(0,0,\frac{\zeta(1-\zeta)}{(1-\xi_1-\zeta)(1-\xi_2-\zeta)+\zeta(1-\zeta)})$.

\begin{figure}[!ht]
\begin{center}
\includegraphics[scale=0.6]{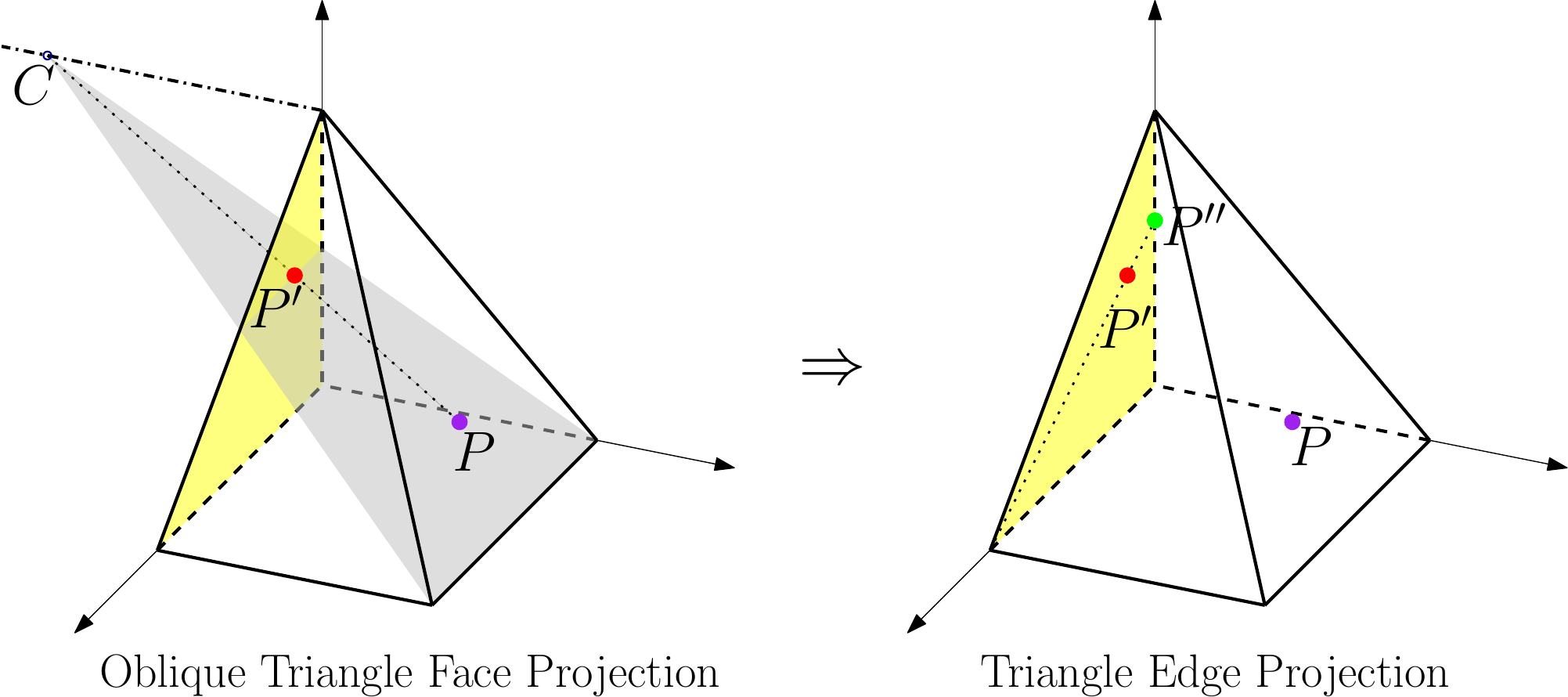}
\caption{Oblique triangle face projection from $P$ to $P'$ followed by a triangle edge projection from $P'$ to $P''$.}
\label{fig:PyramidProjectionObliqueTriangle}
\end{center}
\end{figure}

In general, the shape functions and their gradient are
\begin{equation}
	\phi_i^\mathrm{e}(\xieta)=\phi_i^\E(\vec{\lambda}_{a5}(\xieta))\,,\qquad\quad
	\nabla\phi_i^\mathrm{e}(\xieta)=\nabla\phi_i^\E(\vec{\lambda}_{a5}(\xieta))\,,
	\label{eq:PyrH1TriaEdge}
\end{equation}
for $i=2,\ldots,p$ and $a=1,2,3,4$.
There are $p-1$ edge functions for each edge, giving a total of $4(p-1)$ triangle edge functions.

\subsubsection{\texorpdfstring{$H^1$}{H1} Faces}

\begin{figure}[!ht]
\begin{center}
\includegraphics[scale=0.6]{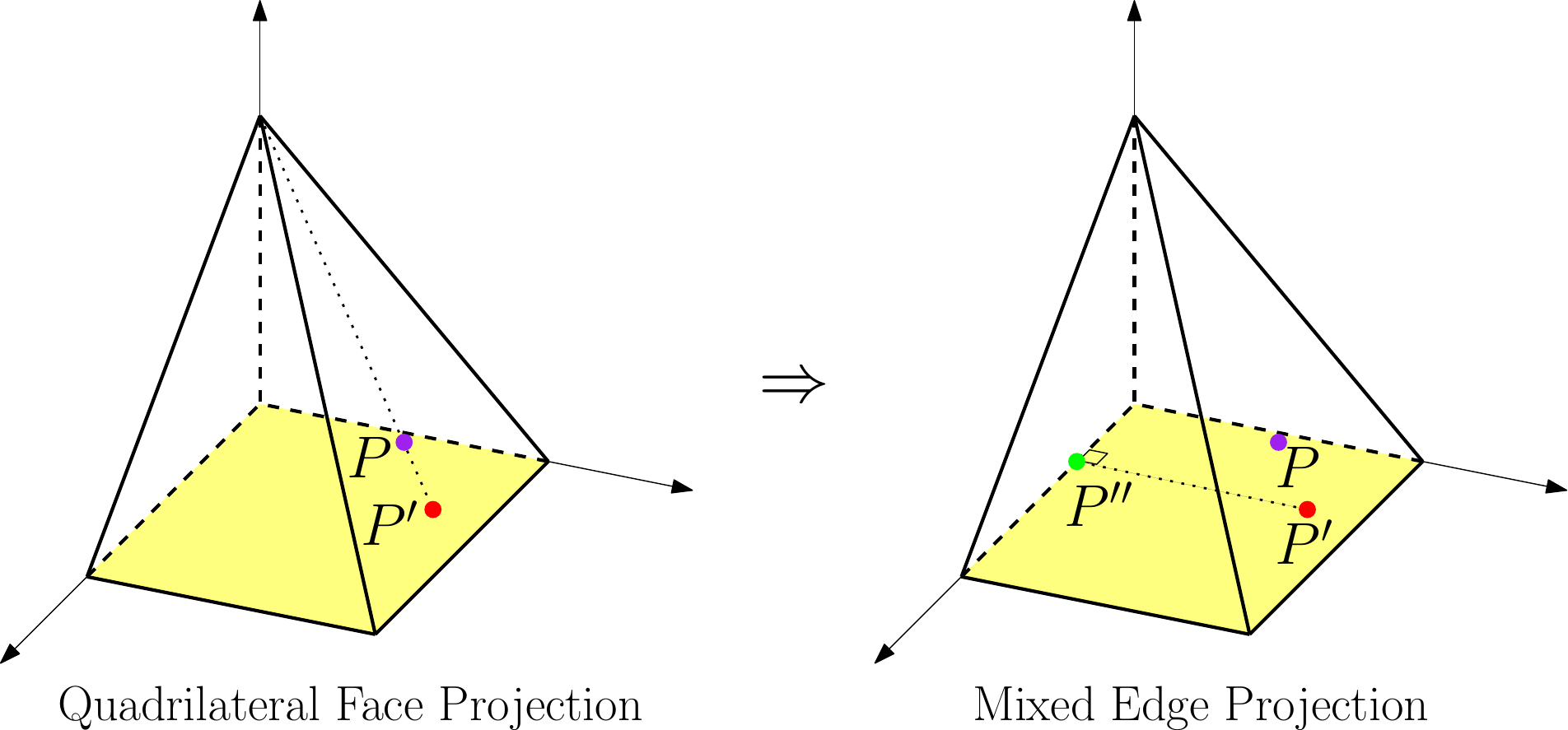}
\caption{Quadrilateral face projection from $P$ to $P'$ followed by a mixed edge projection from $P'$ to $P''$.}
\label{fig:PyramidProjectionQuad}
\end{center}
\end{figure}

\paragraph{\texorpdfstring{$H^1$}{H1} Quadrilateral Face.} 
It was already mentioned that the quadrilateral face projection, illustrated in Figure \ref{fig:PyramidProjectionQuad}, takes an arbitrary point $P=(\xi_1,\xi_2,\zeta)$ to the point $P'=(\sxi,\sxii,0)$ along the face. This projected point is actually represented by the affine coordinate quadruple $(\vec{\mu}_{01}(\sxi),\vec{\mu}_{01}(\sxii))$. 
Hence, the natural choice is to use the quadruple with the ancillary function $\phi_{ij}^\square$. 
This already satisfies all the necessary trace properties, except at the top vertex itself, where there might be a singularity.
This is corrected by adding a factor of $\mu_0(\zeta)$, which also ensures the function is in the correct space.

The shape functions and their gradient are
\begin{equation}
	\begin{aligned}
		\phi_{ij}^\mathrm{f}(\xieta)&=\mu_0(\zeta)\phi_{ij}^\square\Big(\vec{\mu}_{01}(\sxi),\vec{\mu}_{01}(\sxii)\Big)\,,\\
    	\nabla\phi_{ij}^\mathrm{f}(\xieta)&=\mu_0(\zeta)\nabla\phi_{ij}^\square\Big(\vec{\mu}_{01}(\sxi),\vec{\mu}_{01}(\sxii)\Big)
        +\phi_{ij}^\square\Big(\vec{\mu}_{01}(\sxi),\vec{\mu}_{01}(\sxii)\Big)\nabla\mu_0(\zeta)\,,	
	\end{aligned}
	\label{eq:PyrH1QuadFace}
\end{equation}
where $i=2,\ldots,p$ and $j=2,\ldots,p$. 
Naturally, there are $(p-1)^2$ shape functions for the quadrilateral face.

\paragraph{\texorpdfstring{$H^1$}{H1} Triangle Faces.} 
Similar to the mixed edges, there are two possibilities.
Obviously they both involve $\phi_{ij}^\Tri$.
Take for example triangle face 125.
The first alternative is to use the 3D pyramid affine coordinates directly, yielding as a result
\begin{equation*}
	\phi_{ij}^\mathrm{f}(\xieta)=\phi_{ij}^\Tri(\vec{\lambda}_{125}(\xieta))=
		\underbrace{(\lambda_1(\xieta)+\lambda_2(\xieta)+\lambda_5(\zeta))^{i+j}}_{\text{blend}}
    	\underbrace{\phi_{ij}^\Tri\Big(\underbrace{\textstyle{\frac{1}{\lambda_1(\xieta)+\lambda_2(\xieta)+\lambda_5(\zeta)}}
    		\vec{\lambda}_{125}(\xieta)}_{\text{project}}\Big)}_{\text{evaluate}}\,,
\end{equation*}
where $i\geq2$ and $j\geq1$.
In this case the projection implied is precisely the oblique triangle face projection illustrated in Figure \ref{fig:PyramidProjectionObliqueTriangle}.
This function lies in the correct space and is easily seen to satisfy the necessary trace properties (see \eqref{eq:phiTrivanishing}).
Hence, it is a perfectly valid candidate.

A second candidate relies in the same approach taken for the mixed edges, in which the effects of the components of $\vec{\lambda}_{12}(\xieta)$ are separated. 
In that case, the functions are
\begin{equation*}
		\phi_{ij}^\mathrm{f}(\xieta)=\underbrace{\mu_0(\sxii)}_{\text{blend}}
    	\underbrace{\phi_{ij}^\Tri\Big(\underbrace{\vec{\nu}_{012}(\xi_1,\zeta)}_{\text{project}}\Big)}_{\text{evaluate}}\,,
\end{equation*}
for $i\geq2$ and $j\geq1$.
Here, the projection implied is the horizontal triangle face projection shown in Figure \ref{fig:PyramidProjectionHorizontalTriangle}.
Again, the function is in the correct space and satisfies the required trace properties, so it is also a valid candidate.

The second alternative is chosen, so the general shape functions and their gradient are
\begin{equation}
	\begin{aligned}
		\phi_{ij}^\mathrm{f}(\xieta)&=\mu_c(\sxib)\phi_{ij}^\Tri(\vec{\nu}_{012}(\xi_a,\zeta))\,,\\
    	\nabla\phi_{ij}^\mathrm{f}(\xieta)&=\mu_c(\sxib)\nabla\phi_{ij}^\Tri(\vec{\nu}_{012}(\xi_a,\zeta))
        +\phi_{ij}^\Tri(\vec{\nu}_{012}(\xi_a,\zeta))\nabla\mu_c(\sxib)\,,	
	\end{aligned}
	\label{eq:PyrH1TriaFace}
\end{equation}
where $i\geq2$, $j\geq1$, $n=i+j=3,\ldots,p$, $(a,b)=(1,2),(2,1)$ and $c=0,1$.
As with all $H^1$ triangle edges, there are $\frac{1}{2}(p-1)(p-2)$ face functions for each face, for a total of $2(p-1)(p-2)$ triangle face functions.

\subsubsection{\texorpdfstring{$H^1$}{H1} Interior Bubbles}

The interior bubble functions resemble closely the case of the hexahedron bubbles, and they are deduced from the quadrilateral face functions, where the factor $\phi_k^\E(\vec{\mu}_{01}(\zeta))$ is used instead of $\mu_0(\zeta)$.
This ensures all the vanishing properties are satisfied.

The bubble functions and their gradient are
\begin{equation}
	\begin{aligned}
		\phi_{ijk}^\mathrm{b}(\xieta)&\!=\!
			\phi_{ij}^\square\Big(\vec{\mu}_{01}(\sxi),\vec{\mu}_{01}(\sxii)\Big)\phi_k^\E(\vec{\mu}_{01}(\zeta))\,,\\
    		\nabla\phi_{ijk}^\mathrm{b}(\xieta)&\!=\!\phi_k^\E(\vec{\mu}_{01}(\zeta))
    			\nabla\phi_{ij}^\square\Big(\vec{\mu}_{01}(\sxi),\vec{\mu}_{01}(\sxii)\Big)
        		\!+\!\phi_{ij}^\square\Big(\vec{\mu}_{01}(\sxi),\vec{\mu}_{01}(\sxii)\Big)\nabla\phi_k^\E(\vec{\mu}_{01}(\zeta))\,,	
	\end{aligned}
	\label{eq:PyrH1Interior}
\end{equation}
where $i=2,\ldots,p$, $j=2,\ldots,p$ and $k=2,\ldots,p$.
Clearly there is a total of $(p-1)^3$ interior bubble functions.

\subsection{\texorpdfstring{$H(\mathrm{curl})$}{Hcurl} Shape Functions}

The dimension of the space $\mathcal{U}^{(1),p}$ is $3p^3+5p$.
The same number of shape functions will span the space.

The construction of the shape functions for $H(\mathrm{curl})$ is completely parallel to that of $H^1$, and they involve the same underlying projections.

\subsubsection{\texorpdfstring{$H(\mathrm{curl})$}{Hcurl} Edges}

\paragraph{\texorpdfstring{$H(\mathrm{curl})$}{Hcurl} Mixed Edges.} 
Take for example mixed edge 12.
As in $H^1$, using $\mu_0(\sxii)$ and $\vec{\nu}_{01}(\xi_1,\zeta)$ separately instead of $\vec{\lambda}_{12}(\xieta)$, it follows the edge shape functions are
\begin{equation*}
	E_i^\mathrm{e}(\xieta)=\mu_0(\sxii)E_i^\E(\vec{\nu}_{01}(\xi_1,\zeta))=
		\underbrace{\mu_0(\sxii)(\nu_0(\xi_1,\zeta)\!+\!\nu_1(\xi_1,\zeta))^{i+2}}_{\text{blend}}
    	\underbrace{E_i^\E\Big(\underbrace{\textstyle{\frac{1}{\nu_0(\xi_1,\zeta)+\nu_1(\xi_1,\zeta)}}
    		\vec{\nu}_{01}(\xi_1,\zeta)}_{\text{project}}\Big)}_{\text{evaluate}}\,,
\end{equation*}
for $i=0,\ldots,p-1$.
From the $H(\mathrm{curl})$ edge triangle functions, it follows that along the edges 15 and 25 the tangential component of $E_i^\E(\vec{\nu}_{01}(\xi_1,\zeta))$ vanishes, and due to its independence from $\xi_2$ it immediately follows that the same is true for the faces 235 and 145.
Along face 435, it holds that $\mu_0(\sxii)=0$ so that it also vanishes there.
Along face 125 $\mu_0(\sxii)=1$, and the function becomes $E_i^\E(\vec{\nu}_{01}(\xi_1,\zeta))$, which as desired is the triangle 2D trace for the edge functions.
Finally, at the quadrilateral face, $\mu_0(\sxii)=\mu_0(\xi_2)$, while the tangential component of $E_i^\E(\vec{\nu}_{01}(\xi_1,\zeta))$ is the corresponding segment 1D $L^2$ edge function, meaning that the trace along this face is $\mu_0(\xi_2)E_i^\E(\vec{\mu}_{01}(\xi_1))$, as required.
Hence, all trace properties hold.
The projection implied is the mixed edge projection depicted in Figures \ref{fig:PyramidProjectionHorizontalTriangle} and \ref{fig:PyramidProjectionQuad}.
Lastly, the shape functions are in the correct space.

The shape functions and their curl are
\begin{equation}
	\begin{aligned}
		E_i^\mathrm{e}(\xieta)&=\mu_c(\sxib)E_i^\E(\vec{\nu}_{01}(\xi_a,\zeta))\,,\\
    	\nabla\times E_i^\mathrm{e}(\xieta)&=\mu_c(\sxib)\nabla\times E_i^\E(\vec{\nu}_{01}(\xi_a,\zeta))
        +\nabla\mu_c(\sxib)\times E_i^\E(\vec{\nu}_{01}(\xi_a,\zeta))\,,	
	\end{aligned}
	\label{eq:PyrHcurlMixedEdge}
\end{equation}
where $i=0,\ldots,p-1$, $(a,b)=(1,2),(2,1)$ and $c=0,1$.
There are $p$ edge functions for each edge, for a total of $4p$ mixed edge functions.

\paragraph{\texorpdfstring{$H(\mathrm{curl})$}{Hcurl} Triangle Edges.}
For instance, take triangle edge 15.
Like in $H^1$, one can directly use the 3D pyramid affine coordinates on $E_i^\E$.
The resulting shape functions are
\begin{equation*}
	E^\mathrm{e}(\xieta)=E_i^\E(\vec{\lambda}_{15}(\xieta))=
		\underbrace{(\lambda_1(\xieta)+\lambda_5(\zeta))^{i+2}}_{\text{blend}}
    	\underbrace{E_i^\E\Big(\underbrace{\textstyle{\frac{1}{\lambda_1(\xieta)+\lambda_5(\zeta)}}
    		\vec{\lambda}_{15}(\xieta)}_{\text{project}}\Big)}_{\text{evaluate}}\,,
\end{equation*}
for $i=0,\ldots,p-1$.
To argue the nonzero traces have the correct form, take for example face 125, and the decoupling $\lambda_1(\xieta)=\mu_0(\sxii)\nu_0(\xi_1,\zeta)$.
Then, using that $\nu_2(\zeta)=\lambda_5(\zeta)$, the lowest order element is
\begin{equation*}
\begin{aligned}
	E_0^\E(\vec{\lambda}_{15}(\xieta))
		&=\mu_0(\sxii)\nu_0(\xi_1,\zeta)\nabla\lambda_5(\zeta)-\lambda_5(\zeta)\nabla\Big(\mu_0(\sxii)\nu_0(\xi_1,\zeta)\Big)\\
			&=\mu_0(\sxii)E_0^\E(\vec{\nu}_{02}(\xi_1,\zeta))-\nu_0(\xi_1,\zeta)\nu_2(\zeta)\nabla\mu_0(\sxii)\,.
\end{aligned}
\end{equation*}
When evaluated at face 125, $\mu_0(\sxii)=1$, and as a result $\nabla\mu_0(\sxii)$ is orthogonal to the face (it is an isosurface), so that the tangential component at the face is precisely the nonzero components of $E_0^\E(\vec{\nu}_{02}(\xi_1,\zeta))$.
Hence, its nonzero trace on the face is a triangle 2D $H(\mathrm{curl})$ edge function, as expected.
Using the same argument but at face 435, where $\mu_0(\sxii)=0$, this time the tangential component vanishes completely.  
Symmetric arguments apply to faces 145 and 235.
Finally, at the quadrilateral face, where $\lambda_5(\zeta)=0$, the lowest order function takes the form $\lambda_1(\xieta)\nabla\lambda_5(\zeta)$ which is normal to the face (it is an isosurface of $\lambda_5(\zeta)$), so the tangential component is zero.
These arguments confirm that the trace properties hold.
Lastly, the projection is a triangle edge projection as illustrated in Figure \ref{fig:PyramidProjectionObliqueTriangle}.

The shape functions and their curl are
\begin{equation}
	E_i^\mathrm{e}(\xieta)=E_i^\E(\vec{\lambda}_{a5}(\xieta))\,,\qquad\quad
	\nabla\times E_i^\mathrm{e}(\xieta)=\nabla \times E_i^\E(\vec{\lambda}_{a5}(\xieta))\,,
	\label{eq:PyrHcurlTriaEdge}
\end{equation}
for $i=0,\ldots,p-1$ and $a=1,2,3,4$.
There are $p$ edge functions for each edge, for a total of $4p$ triangle edge functions.

\subsubsection{\texorpdfstring{$H(\mathrm{curl})$}{Hcurl} Faces}

\paragraph{\texorpdfstring{$H(\mathrm{curl})$}{Hcurl} Quadrilateral Face.} 
As expected, the projection implied in these expressions will be the same as that of $H^1$. 
It is a quadrilateral face projection as depicted in Figure \ref{fig:PyramidProjectionQuad}.
However, in this case there will be two families.
Both families will easily satisfy the vanishing trace properties by use of \eqref{eq:phiEvanishing}, \eqref{eq:Hcurl1Dspecialcase} and that $\nabla\mu_c(\sxia)$ is orthogonal to the triangle faces where $\mu_c(\sxia)=1$.
The only difference with $H^1$ radicates in the use of the higher order blending function $\mu_0(\zeta)^2$, which is used in order to be in the correct space. 
There is a grand total of $2p(p-1)$ quadrilateral face functions.

\subparagraph{Family I:}
The shape functions and their curl are
\begin{equation}
	\begin{aligned}
		E_{ij}^{\mathrm{f}}(\xieta)&=\mu_0(\zeta)^2 E_{ij}^\square\Big(\vec{\mu}_{01}(\sxi),\vec{\mu}_{01}(\sxii)\Big)\,,\\
		\nabla\times E_{ij}^{\mathrm{f}}(\xieta)&=\mu_0(\zeta)^2\nabla\times
			E_{ij}^\square\Big(\vec{\mu}_{01}(\sxi),\vec{\mu}_{01}(\sxii)\Big)\\
				&\qquad\quad+2\mu_0(\zeta)\nabla\mu_0(\zeta)\times E_{ij}^\square\Big(\vec{\mu}_{01}(\sxi),\vec{\mu}_{01}(\sxii)\Big)\,,
	\end{aligned}
	\label{eq:PyrHcurlQuadFaceI}
\end{equation}
for $i=0,\ldots,p-1$ and $j=2,\ldots,p$. 
There are $p(p-1)$ shape functions in this family.

\subparagraph{Family II:}
The shape functions and their curl are
\begin{equation}
	\begin{aligned}
		E_{ij}^{\mathrm{f}}(\xieta)&=\mu_0(\zeta)^2 E_{ij}^\square\Big(\vec{\mu}_{01}(\sxii),\vec{\mu}_{01}(\sxi)\Big)\,,\\
		\nabla\times E_{ij}^{\mathrm{f}}(\xieta)&=\mu_0(\zeta)^2\nabla\times
			E_{ij}^\square\Big(\vec{\mu}_{01}(\sxii),\vec{\mu}_{01}(\sxi)\Big)\\
				&\qquad\quad+2\mu_0(\zeta)\nabla\mu_0(\zeta)\times E_{ij}^\square\Big(\vec{\mu}_{01}(\sxii),\vec{\mu}_{01}(\sxi)\Big)\,,
	\end{aligned}
	\label{eq:PyrHcurlQuadFaceII}
\end{equation}
for $i=0,\ldots,p-1$ and $j=2,\ldots,p$. 
Note the fact that the entries are permuted with respect to the first family.
There are $p(p-1)$ shape functions in this family.

\paragraph{\texorpdfstring{$H(\mathrm{curl})$}{Hcurl} Triangle Faces.} 
As with $H^1$ there will be two valid alternatives.
Take for example face 125.
The first alternative is to use the pyramid affine coordinates directly on $E_{ij}^\Tri$ \textit{and} multiply by $\mu_0(\sxii)^{-1}$.
This approach is discarded in favor of separating the effects of $\mu_0(\sxii)$ and $\vec{\nu}_{01}(\xi_1,\zeta)$ directly from $\vec{\lambda}_{12}(\xieta)$.
The resulting functions satisfy the necessary vanishing conditions using similar arguments to those used for mixed and triangle edges.
Also, the projection implied is the horizontal triangle face projection shown in Figure \ref{fig:PyramidProjectionHorizontalTriangle}.
There are $p(p-1)$ functions per triangle face, for a grand total of $4p(p-1)$ triangle face functions.

\subparagraph{Family I:} 
The shape functions and their curl are
\begin{equation}
	\begin{aligned}
		E_{ij}^{\mathrm{f}}(\xieta)&=\mu_c(\sxib)E_{ij}^\Tri(\vec{\nu}_{012}(\xi_a,\zeta))\,,\\
		\nabla\times E_{ij}^{\mathrm{f}}(\xieta)&=\mu_c(\sxib)\nabla\times E_{ij}^\Tri(\vec{\nu}_{012}(\xi_a,\zeta))
			+\nabla\mu_c(\sxib)\times E_{ij}^\Tri(\vec{\nu}_{012}(\xi_a,\zeta))\,,
	\end{aligned}
	\label{eq:PyrHcurlTriaFaceI}
\end{equation}
for $i\geq0$, $j\geq1$, $n=i+j=1,\ldots,p-1$, $(a,b)=(1,2),(2,1)$, and $c=0,1$. 
Every face has $\frac{1}{2}p(p-1)$ functions in this family.

\subparagraph{Family II:} 
The shape functions and their curl are
\begin{equation}
	\begin{aligned}
		E_{ij}^{\mathrm{f}}(\xieta)&=\mu_c(\sxib)E_{ij}^\Tri(\vec{\nu}_{120}(\xi_a,\zeta))\,,\\
		\nabla\times E_{ij}^{\mathrm{f}}(\xieta)&=\mu_c(\sxib)\nabla\times E_{ij}^\Tri(\vec{\nu}_{120}(\xi_a,\zeta))
			+\nabla\mu_c(\sxib)\times E_{ij}^\Tri(\vec{\nu}_{120}(\xi_a,\zeta))\,,
	\end{aligned}
	\label{eq:PyrHcurlTriaFaceII}
\end{equation}
for $i\geq0$, $j\geq1$, $n=i+j=1,\ldots,p-1$, $(a,b)=(1,2),(2,1)$, and $c=0,1$. 
Note the fact that the entries are $\vec{\nu}_{120}(\xi_a,\zeta)$ as opposed to $\vec{\nu}_{012}(\xi_a,\zeta)$.
Every face has $\frac{1}{2}p(p-1)$ functions in this family.

\subsubsection{\texorpdfstring{$H(\mathrm{curl})$}{Hcurl} Interior Bubbles}

These will be separated according to a Helmholtz decomposition.
Indeed, there are four families of interior bubbles, and the first family corresponds precisely to the gradients of $H^1$ interior bubble functions, so they have zero curl.
The other three families will essentially be generated by quadrilateral face functions in $H(\mathrm{curl})$ and $H^1$.
In all cases, the trace properties follow easily.
There is a grand total of $3p(p-1)^2$ interior bubble functions.

\subparagraph{Family I:}
These are the gradients of $H^1$ interior functions.
The shape functions and their curl are
\begin{equation}
	\begin{aligned}
		E_{ijk}^\mathrm{b}(\xieta)&=\phi_k^\E(\vec{\mu}_{01}(\zeta))
    			\nabla\phi_{ij}^\square\Big(\vec{\mu}_{01}(\sxi),\vec{\mu}_{01}(\sxii)\Big)\\
        		&\quad\qquad+\phi_{ij}^\square\Big(\vec{\mu}_{01}(\sxi),\vec{\mu}_{01}(\sxii)\Big)\nabla\phi_k^\E(\vec{\mu}_{01}(\zeta))\,,\\
    \nabla\times E_{ijk}^\mathrm{b}(\xieta)&=0\,,
	\end{aligned}
	\label{eq:PyrHcurlInteriorI}
\end{equation}
where $i=2,\ldots,p$, $j=2,\ldots,p$ and $k=2,\ldots,p$.
There are $(p-1)^3$ interior bubble functions in this family.

\subparagraph{Family II:}
The shape functions and their curl are
\begin{equation}
	\begin{aligned}
		E_{ijk}^{\mathrm{b}}(\xieta)&=\mu_0(\zeta)\phi_k^\E(\vec{\mu}_{01}(\zeta))
			E_{ij}^\square\Big(\vec{\mu}_{01}(\sxi),\vec{\mu}_{01}(\sxii)\Big)\,,\\
		\nabla\times E_{ijk}^{\mathrm{b}}(\xieta)&=\mu_0(\zeta)\phi_k^\E(\vec{\mu}_{01}(\zeta))\nabla\times
			E_{ij}^\square\Big(\vec{\mu}_{01}(\sxi),\vec{\mu}_{01}(\sxii)\Big)\\
				&\quad+\Big(\mu_0(\zeta)\nabla\phi_k^\E(\vec{\mu}_{01}(\zeta))+\phi_k^\E(\vec{\mu}_{01}(\zeta))\nabla\mu_0(\zeta)\Big)
					\times E_{ij}^\square\Big(\vec{\mu}_{01}(\sxi),\vec{\mu}_{01}(\sxii)\Big)\,,
	\end{aligned}
	\label{eq:PyrHcurlInteriorII}
\end{equation}
for $i=0,\ldots,p-1$, $j=2,\ldots,p$, and $k=2,\ldots,p$.
There are $p(p-1)^2$ shape functions in this family.

\subparagraph{Family III:}
The shape functions and their curl are
\begin{equation}
	\begin{aligned}
		E_{ijk}^{\mathrm{b}}(\xieta)&=\mu_0(\zeta)\phi_k^\E(\vec{\mu}_{01}(\zeta))
			E_{ij}^\square\Big(\vec{\mu}_{01}(\sxii),\vec{\mu}_{01}(\sxi)\Big)\,,\\
		\nabla\times E_{ijk}^{\mathrm{b}}(\xieta)&=\mu_0(\zeta)\phi_k^\E(\vec{\mu}_{01}(\zeta))\nabla\times
			E_{ij}^\square\Big(\vec{\mu}_{01}(\sxii),\vec{\mu}_{01}(\sxi)\Big)\\
				&\quad+\Big(\mu_0(\zeta)\nabla\phi_k^\E(\vec{\mu}_{01}(\zeta))+\phi_k^\E(\vec{\mu}_{01}(\zeta))\nabla\mu_0(\zeta)\Big)
					\times E_{ij}^\square\Big(\vec{\mu}_{01}(\sxii),\vec{\mu}_{01}(\sxi)\Big)\,,
	\end{aligned}
	\label{eq:PyrHcurlInteriorIII}
\end{equation}
for $i=0,\ldots,p-1$, $j=2,\ldots,p$, and $k=2,\ldots,p$.
Note the entries are permuted with respect to the second family of interior bubbles.
There are $p(p-1)^2$ shape functions in this family.

\subparagraph{Family IV:}
The shape functions for the final family of interior bubbles and their curl are
\begin{equation}
	\begin{aligned}
		E_{ij}^{\mathrm{b}}(\xieta)&=\phi_{ij}^\square\Big(\vec{\mu}_{01}(\sxii),\vec{\mu}_{01}(\sxi)\Big)
			n\mu_0(\zeta)^{n-1}\nabla\mu_0(\zeta)\,,\\
		\nabla\times E_{ij}^{\mathrm{b}}(\xieta)&=n\mu_0(\zeta)^{n-1}
			\nabla\phi_{ij}^\square\Big(\vec{\mu}_{01}(\sxii),\vec{\mu}_{01}(\sxi)\Big)\times\nabla\mu_0(\zeta)\,,
	\end{aligned}
	\label{eq:PyrHcurlInteriorIV}
\end{equation}
for $i=2,\ldots,p$, $j=2,\ldots,p$, and $n=\max\{i,j\}$.
There are $(p-1)^2$ shape functions in this family.

\subsection{\texorpdfstring{$H(\mathrm{div})$}{Hdiv} Shape Functions}

The dimension of the space $\mathcal{U}^{(2),p}$ is $3p^3+2p$.
The number of linearly independent shape functions coincides with that dimension.

\subsubsection{\texorpdfstring{$H(\mathrm{div})$}{Hdiv} Faces}

\paragraph{\texorpdfstring{$H(\mathrm{div})$}{Hdiv} Quadrilateral Face.} 
These are constructed exactly the same way as the $H^1$ and $H(\mathrm{curl})$ counterparts, but with the higher order blending function $\mu_0(\zeta)^3$ so that the resulting functions lie in the space.
In view of \eqref{eq:Vijsimplified}, it is clear that the function will only have a tangential component along the triangle faces, so that the normal trace vanishes at these faces, as required.
Moreover, again through \eqref{eq:Vijsimplified}, it is evident that the nonzero trace on the face is a quadrilateral $L^2$ face function.
The projection is once again depicted in Figure \ref{fig:PyramidProjectionQuad}.

In view of \eqref{eq:Vijsimplified}, the shape functions and their divergence are
\begin{equation}
	\begin{aligned}
		V_{ij}^\mathrm{f}(\xieta)&=\mu_0(\zeta)^3V_{ij}^\square\Big(\vec{\mu}_{01}(\sxi),\vec{\mu}_{01}(\sxii)\Big)\,,\\
    	\nabla\cdot V_{ij}^\mathrm{f}(\xieta)&=3\mu_0(\zeta)^2\nabla\mu_0(\zeta)\cdot
    		V_{ij}^\square\Big(\vec{\mu}_{01}(\sxi),\vec{\mu}_{01}(\sxii)\Big)\,,	
	\end{aligned}
	\label{eq:PyrHdivQuadFace}
\end{equation}
where $i=0,\ldots,p-1$ and $j=0,\ldots,p-1$.
Clearly, there are $p^2$ quadrilateral face functions.

\paragraph{\texorpdfstring{$H(\mathrm{div})$}{Hdiv} Triangle Faces.} 
The construction of the $H(\mathrm{div})$ triangle face functions is highly nontrivial.
Indeed, an analogous construction to the case of $H^1$ or $H(\mathrm{curl})$ does not work here.
This radicates in the definition of the space itself.
The issue is even present for the lowest order space, and in fact it is by looking at this space in detail that the problem is solved.

To summarize the construction, take for example face 125.
The detailed calculations are in Appendix \ref{app:pyrappendix}.
Proceeding as in the $H^1$ or $H(\mathrm{curl})$ case, one obtains the following two disheartening facts for the lowest order candidate functions,
\begin{equation*}
	\mu_0\big(\sxii\big)V_{00}^\Tri(\vec{\nu}_{012}(\xi_1,\zeta))\notin\mathcal{U}^{(2),1}\,,\qquad\quad
	\mu_0\big(\sxii\big)^{-1}V_{00}^\Tri(\vec{\lambda}_{125}(\xieta))\notin\mathcal{U}^{(2),1}\,.
\end{equation*}
However, they both satisfy the desired trace properties.
One could alternatively attempt a more direct construction, but issues arise constantly, either because the functions are not in the space, or because there are ``illegal'' derivatives of Legendre polynomials $P_i$, which in theory should not exist as they are elements of $L^2$ and are intended to approximate elements of that space (for example, a discontinuous function).
These issues do not arise when using $V_{ij}^\Tri$ in view of Lemma \ref{lem:divformula}, and this is one of the reasons why it is so convenient to use it.
Fortunately, there is a way to make this happen.
The key is to look at the unique lowest order function for a given face. 
The explicit formulas are given in \citet{Hiptmair99} and \citet{Nigam_Phillips_11}.
After scrupulous observation, one obtains
\begin{equation*}
	V_{00}^\mathrm{f}(\xieta)=\frac{1}{2}\Big(\mu_0\big(\sxii\big)V_{00}^\Tri(\vec{\nu}_{012}(\xi_1,\zeta))+
		\mu_0\big(\sxii\big)^{-1}V_{00}^\Tri(\vec{\lambda}_{125}(\xieta))\Big)\in\mathcal{U}^{(2),1}\,.
\end{equation*}
Hence, this suggests the following general formula for the face 125 shape functions,
\begin{equation*}
	V_{ij}^\mathrm{f}(\xieta)=\frac{1}{2}\Big(\mu_0\big(\sxii\big)V_{ij}^\Tri(\vec{\nu}_{012}(\xi_1,\zeta))+
		\mu_0\big(\sxii\big)^{-1}V_{ij}^\Tri(\vec{\lambda}_{125}(\xieta))\Big)\,,
\end{equation*}
for $i\geq0$, $j\geq0$, and $n=i+j=0,\ldots,p-1$. 
In Appendix \ref{app:pyrappendix} it is shown that these high order functions are in the correct space and that they sastisfy the trace properties.
Some could worry when seeing the factor $\mu_0\big(\sxii\big)^{-1}$, but this is in fact not a real singularity.
Indeed it is shown in Appendix \ref{app:pyrappendix} how to avoid it explicitly, along with an alternative formula convenient for computations.
Lastly, note the inherent projection is not unique, and in fact is a combination of horizontal and oblique triangle face projections.

Finally, in view of \eqref{eq:HdivtriangleRemark} the general shape functions are 
\begin{equation}
	\begin{aligned}
		V_{ij}^\mathrm{f}(\xieta)&=\frac{1}{2}\Big(\mu_c\big(\sxib\big)V_{ij}^\Tri(\vec{\nu}_{012}(\xi_a,\zeta))
			+\mu_c\big(\sxib\big)^{-1}V_{ij}^\Tri(\vec{\lambda}_{de5}(\xieta))\Big)\,,\\
    \nabla\cdot V_{ij}^\mathrm{f}(\xieta)&=\frac{1}{2}\Big(\nabla\mu_c\big(\sxib\big)\cdot V_{ij}^\Tri(\vec{\nu}_{012}(\xi_a,\zeta))
    	+\mu_c\big(\sxib\big)^{-1}\nabla\cdot V_{ij}^\Tri(\vec{\lambda}_{de5}(\xieta))\\
    		&\qquad\qquad\qquad\qquad\qquad\qquad\quad
    			-\mu_c\big(\sxib\big)^{-2}\nabla\mu_c(\sxib)\cdot V_{ij}^\Tri(\vec{\lambda}_{de5}(\xieta))\Big)\,,	
	\end{aligned}
	\label{eq:PyrHdivTriaFace}
\end{equation}
where $i\geq0$, $j\geq0$, $n=i+j=0,\ldots,p-1$, $(a,b)=(1,2),(2,1)$, $c=0,1$ and where $(d,e)$ depends on $(a,b,c)$ (in fact $\vec{\lambda}_{de5}(\xieta)=\mu_c\big(\sxib\big)\vec{\nu}_{012}(\xi_a,\zeta)$).
There are $\frac{1}{2}p(p+1)$ functions for each face, leading to a total of $2p(p-1)$ triangle face functions.

\subsubsection{\texorpdfstring{$H(\mathrm{div})$}{Hdiv} Interior Bubbles}

Like the $H(\mathrm{curl})$ interior bubbles, these will be separated according to a Helmholtz decomposition.
Indeed, the first three out of seven families of interior bubbles are precisely the curl of $H(\mathrm{curl})$ interior bubble functions, so they have zero divergence.
In all cases, the trace properties follow easily.
There is a grand total of $3p^2(p-1)$ interior bubble functions.

\subparagraph{Family I:}
The shape functions and their divergence are
\begin{equation}
	\begin{aligned}
		V_{ijk}^{\mathrm{b}}(\xieta)&=\mu_0(\zeta)\phi_k^\E(\vec{\mu}_{01}(\zeta))\nabla\times
			E_{ij}^\square\Big(\vec{\mu}_{01}(\sxi),\vec{\mu}_{01}(\sxii)\Big)\\
				&\quad+\Big(\mu_0(\zeta)\nabla\phi_k^\E(\vec{\mu}_{01}(\zeta))+\phi_k^\E(\vec{\mu}_{01}(\zeta))\nabla\mu_0(\zeta)\Big)
					\times E_{ij}^\square\Big(\vec{\mu}_{01}(\sxi),\vec{\mu}_{01}(\sxii)\Big)\,,\\
		\nabla\cdot V_{ijk}^{\mathrm{b}}(\xieta)&=0\,,
	\end{aligned}
	\label{eq:PyrHdivInteriorI}
\end{equation}
for $i=0,\ldots,p-1$, $j=2,\ldots,p$, and $k=2,\ldots,p$.
There are $p(p-1)^2$ shape functions in this family.

\subparagraph{Family II:}
The shape functions and their divergence are
\begin{equation}
	\begin{aligned}
		V_{ijk}^{\mathrm{b}}(\xieta)&=\mu_0(\zeta)\phi_k^\E(\vec{\mu}_{01}(\zeta))\nabla\times
			E_{ij}^\square\Big(\vec{\mu}_{01}(\sxii),\vec{\mu}_{01}(\sxi)\Big)\\
				&\quad+\Big(\mu_0(\zeta)\nabla\phi_k^\E(\vec{\mu}_{01}(\zeta))+\phi_k^\E(\vec{\mu}_{01}(\zeta))\nabla\mu_0(\zeta)\Big)
					\times E_{ij}^\square\Big(\vec{\mu}_{01}(\sxii),\vec{\mu}_{01}(\sxi)\Big)\,,\\
		\nabla\cdot V_{ijk}^{\mathrm{b}}(\xieta)&=0\,,
	\end{aligned}
	\label{eq:PyrHdivInteriorII}
\end{equation}
for $i=0,\ldots,p-1$, $j=2,\ldots,p$, and $k=2,\ldots,p$.
Note the entries are permuted with respect to the first family of interior bubbles.
There are $p(p-1)^2$ shape functions in this family.

\subparagraph{Family III:}
The shape functions and their divergence are
\begin{equation}
	\begin{aligned}
		V_{ij}^{\mathrm{b}}(\xieta)&=n\mu_0(\zeta)^{n-1}
			\nabla\phi_{ij}^\square\Big(\vec{\mu}_{01}(\sxii),\vec{\mu}_{01}(\sxi)\Big)\times\nabla\mu_0(\zeta)\,,\\
		\nabla\cdot V_{ij}^{\mathrm{b}}(\xieta)&=0\,,
	\end{aligned}
	\label{eq:PyrHdivInteriorIII}
\end{equation}
for $i=2,\ldots,p$, $j=2,\ldots,p$, and $n=\max\{i,j\}$.
There are $(p-1)^2$ shape functions in this family.

\subparagraph{Family IV:}
These have nonzero divergence and are generated by the quadrilateral face functions, but using $\mu_0(\zeta)^2\phi_k^\E(\vec{\mu}_{01}(\zeta))$ as a factor instead of $\mu_0(\zeta)^3$. 
In view of \eqref{eq:Vijsimplified}, the shape functions and their divergence are
\begin{equation}
	\begin{aligned}
		V_{ijk}^{\mathrm{b}}(\xieta)&\!=\!\mu_0(\zeta)^2\phi_k^\E(\vec{\mu}_{01}(\zeta))
			V_{ij}^\square\Big(\vec{\mu}_{01}(\sxi),\vec{\mu}_{01}(\sxii)\Big)\,,\\
		\nabla\!\cdot\! V_{ijk}^{\mathrm{b}}(\xieta)&\!=\!
				\Big(\mu_0(\zeta)^2\nabla\phi_k^\E(\vec{\mu}_{01}(\zeta))
					\!+\!2\mu_0(\zeta)\phi_k^\E(\vec{\mu}_{01}(\zeta))\nabla\mu_0(\zeta)\Big)
						\!\cdot\! V_{ij}^\square\Big(\vec{\mu}_{01}(\sxi),\vec{\mu}_{01}(\sxii)\Big)\,,
	\end{aligned}
	\label{eq:PyrHdivInteriorIV}
\end{equation}
for $i=0,\ldots,p-1$, $j=0,\ldots,p-1$, and $k=2,\ldots,p$.
There are $p^2(p-1)$ shape functions in this family.

\subparagraph{Family V:}
These have nonzero divergence and are expressed as the product of a power of $\mu_1(\zeta)$ with a curl. 
As a first step, define
\begin{equation}
\begin{aligned}
		V_{ij}^\trianglelefteq(\vec{s}_{01}^{\,x},\vec{s}_{01}^{\,y},t_0)&\!=\!\nabla\!\times\!\Big(\textstyle{\frac{1}{2}}t_0^2
			\Big(\phi_i^\E(\vec{s}_{01}^{\,x})\nabla\phi_j^\E(\vec{s}_{01}^{\,y})
					\!-\!\phi_j^\E(\vec{s}_{01}^{\,y})\nabla\phi_i^\E(\vec{s}_{01}^{\,x})\Big)\Big)\,,\\
			&\!=\!t_0^2\Big(\!\nabla\phi_i^\E(\vec{s}_{01}^{\,x})\!\times\!\nabla\phi_j^\E(\vec{s}_{01}^{\,y})\!\Big)
				\!+\!t_0\nabla t_0\!\times\!\!\Big(\!\phi_i^\E(\vec{s}_{01}^{\,x})\nabla\phi_j^\E(\vec{s}_{01}^{\,y})
					\!-\!\phi_j^\E(\vec{s}_{01}^{\,y})\nabla\phi_i^\E(\vec{s}_{01}^{\,x})\!\Big)\,,
	\end{aligned}
	\label{eq:PyrHdivBubblesAuxI}
\end{equation}
for $i=2,\ldots,p$, and $j=2,\ldots,p$.
Clearly, $\nabla\cdot V_{ij}^\trianglelefteq(\vec{s}_{01}^{\,x},\vec{s}_{01}^{\,y},t_0)=0$.
The shape functions and their divergence are
\begin{equation}
	\begin{aligned}
		V_{ij}^{\mathrm{b}}(\xieta)&=\mu_1(\zeta)^{n-1}V_{ij}^\trianglelefteq(\vec{\mu}_{01}(\sxi),\vec{\mu}_{01}(\sxii),\mu_0(\zeta))\,,\\
		\nabla\cdot	V_{ij}^{\mathrm{b}}(\xieta)&=(n-1)\mu_1(\zeta)^{n-2}\nabla\mu_1(\zeta)
			\cdot V_{ij}^\trianglelefteq(\vec{\mu}_{01}(\sxi),\vec{\mu}_{01}(\sxii),\mu_0(\zeta))\,,
	\end{aligned}
	\label{eq:PyrHdivInteriorV}
\end{equation}
for $i=2,\ldots,p$, $j=2,\ldots,p$, and with $n=\max\{i,j\}$. 
There are $(p-1)^2$ shape functions in this family.

\subparagraph{Family VI:}
These have nonzero divergence and are expressed as the product of a power of $\mu_1(\zeta)$ with a curl.
First define
\begin{equation}
		V_i^\trianglerighteq(\vec{s}_{01},\mu_1,t_0)=
			\nabla\Big(t_0^2\phi_i^\E(\vec{s}_{01})\Big)\times\nabla\mu_1
				=\Big(t_0^2\nabla\phi_i^\E(\vec{s}_{01})+2t_0\phi_i^\E(\vec{s}_{01})\nabla t_0\Big)\times\nabla\mu_1\,,
	\label{eq:PyrHdivBubblesAuxII}
\end{equation}
for $i=2,\ldots,p$.
Obviously, $\nabla\cdot V_i^\trianglerighteq(\vec{s}_{01},\mu_1,t_0)=0$.
The shape functions and their divergence are
\begin{equation}
	\begin{aligned}
		V_{i}^{\mathrm{b}}(\xieta)&=\mu_1(\zeta)^{i-1}V_i^\trianglerighteq(\vec{\mu}_{01}(\sxi),\mu_1(\sxii),\mu_0(\zeta))\,,\\
		\nabla\cdot	V_{i}^{\mathrm{b}}(\xieta)&=(i-1)\mu_1(\zeta)^{i-2}\nabla\mu_1(\zeta)
			\cdot V_i^\trianglerighteq(\vec{\mu}_{01}(\sxi),\mu_1(\sxii),\mu_0(\zeta))\,,
	\end{aligned}
	\label{eq:PyrHdivInteriorVI}
\end{equation}
for $i=2,\ldots,p$. 
There is a total of $(p-1)$ shape functions in this family.

\subparagraph{Family VII:}
Using \eqref{eq:PyrHdivBubblesAuxII}, the shape functions and their divergence are
\begin{equation}
	\begin{aligned}
		V_{j}^{\mathrm{b}}(\xieta)&=\mu_1(\zeta)^{j-1}V_j^\trianglerighteq(\vec{\mu}_{01}(\sxii),\mu_1(\sxi),\mu_0(\zeta))\,,\\
		\nabla\cdot	V_{j}^{\mathrm{b}}(\xieta)&=(j-1)\mu_1(\zeta)^{j-2}\nabla\mu_1(\zeta)
			\cdot V_j^\trianglerighteq(\vec{\mu}_{01}(\sxii),\mu_1(\sxi),\mu_0(\zeta))\,,
	\end{aligned}
	\label{eq:PyrHdivInteriorVII}
\end{equation}
for $j=2,\ldots,p$.
Note the entries are permuted with respect to the sixth family of interior bubbles.
There is a total of $(p-1)$ shape functions in this family.

\subsection{\texorpdfstring{$L^2$}{L2} Shape Functions}

The dimension of the space $\mathcal{U}^{(3),p}$ is $p^3$.
The same number of shape functions will span the space.

\subsubsection{\texorpdfstring{$L^2$}{L2} Interior}
Again, these are reminiscent of the shape functions for the hexahedron.
They are,
\begin{equation}
    \psi_{ijk}^\mathrm{b}(\xieta)=[P_i](\vec{\mu}_{01}(\sxi))[P_j](\vec{\mu}_{01}(\sxii))[P_k](\vec{\mu}_{01}(\zeta))
    	(\nabla\nu_1(\xi_1,\zeta)\!\!\times\!\!\nabla\nu_1(\xi_2,\zeta))\!\cdot\!\nabla\mu_1(\zeta)\,,
    \label{eq:PyrL2Interior}
\end{equation}
for $i=0,\ldots,p-1$, $j=0,\ldots,p-1$ and $k=0,\ldots,p-1$. 
There are $p^3$ interior functions.

\subsection{Orientations}

\begin{figure}[!ht]
\begin{center}
\includegraphics[scale=0.5]{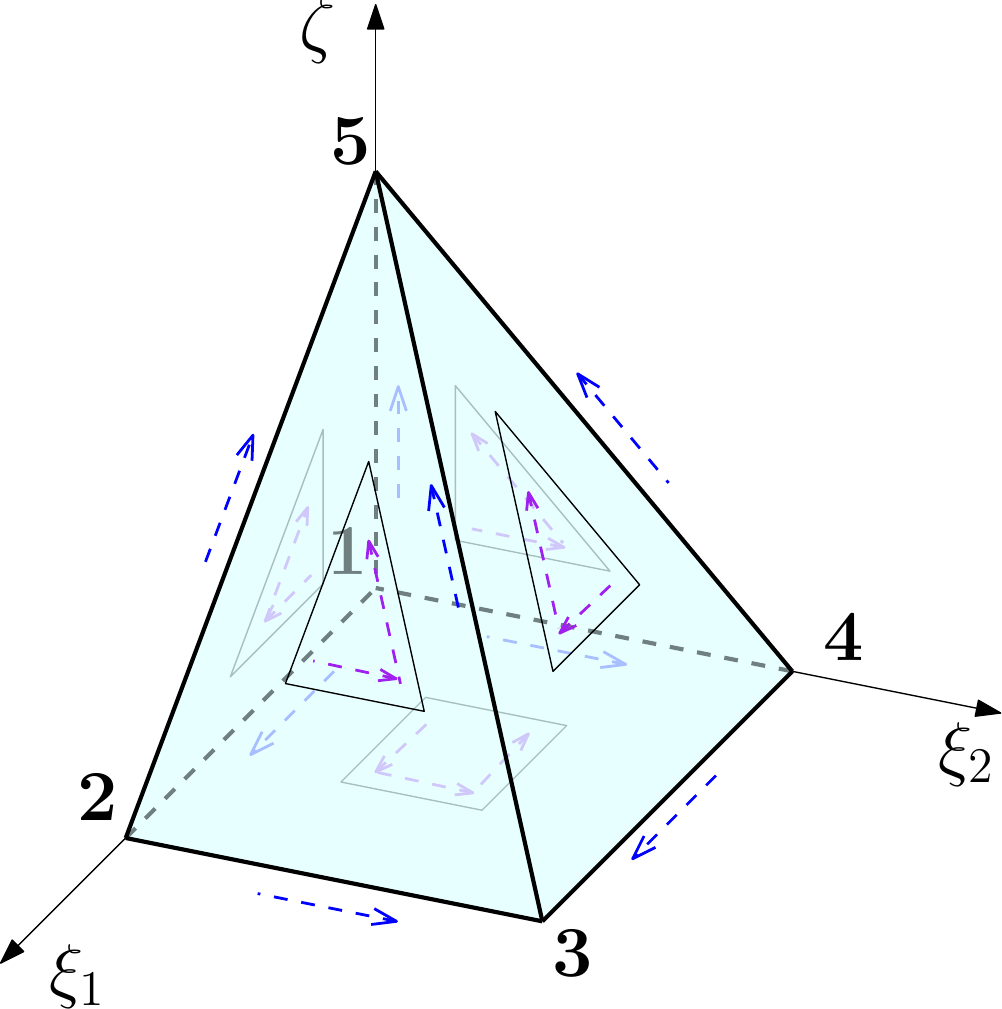}
\caption{Master pyramid with numbered vertices \textit{and} local edge and face orientations.}
\label{fig:MasterPyramidOrientations}
\end{center}
\end{figure}

The predefined \textit{local} edge and face orientations for the pyramid are illustrated in Figure \ref{fig:MasterPyramidOrientations}.
They represent the $\oo=0$ case.
The task at hand is to find the associated \textit{locally ordered} tuples of affine coordinates representing those local orientations.
The key is being aware of the relationships between the vertices and the affine coordinates, which have been explained in detail at the beginning of this section. 
Hence, once \S\ref{sec:HexaOrientations}, \S\ref{sec:TetOrientations} and \S\ref{sec:PrismOrientations} are consulted, it should be clear how to construct the orientation embedded shape functions. 
Consult Appendix \ref{app:pyrappendix} if the reader wants to avoid computational instabilities for the $H(\mathrm{div})$ triangle face functions.

%% file: conclusion.tex
\section{Conclusions}
\label{sec:Conclusions}

We have presented here a full systematic construction of hierarchical higher order shape functions for elements of ``all shapes'' (see Figure \ref{fig:elementsallshapes}) using the exact sequence logic.
Compatibility of the shape functions at the interelement boundaries is based on the idea of having a known fixed trace at the boundary which is extended (or lifted) to the rest of the element.
Hence, the shape functions can be used in hybrid meshes containing elements of all shapes.
Furthermore, due to the properties of the discrete spaces, interpolation estimates are ensured in any hybrid master element mesh and for all energy spaces.

The unified construction is based at its core in considering tensor products of polynomials.
This has positive implications from the computational point of view, and could result in the successful implementation of fast integration techniques as described in Appendix \ref{app:Integration}.

Also, the shape functions allow the polynomial order $p$ to vary accross a given mesh. 
For example, the quadrilateral, hexahedron and prism shape functions are naturally anisotropic and can have different orders in each direction.
More so, each edge and face in the mesh can have their own order $p$, independent of the order of the neighboring edges, faces and interiors in the mesh.
Hence, techniques that exploit the use of local $p$ adaptivity can be implemented using these shape functions.

As polynomial building blocks, Legendre and Jacobi polynomials are used in this work.
Our results show that the recursive formulas presented in this document to implement those polynomials seem to be the most accurate in comparison to other recursive formulas. 
The choices of Jacobi and Legendre polynomials are known to have extremely good sparsity and conditioning properties for typical projection problems \citep{Beuchler_Pillwein_Schoeberl_Zaglmayr_12}.
However, there is flexibility in this choice of polynomials, and it might be worth investigating if there are applications where different choices provide useful advantages (see Appendix \ref{app:GeneratingFamilies}).

All constructions are written in terms of affine coordinates and their gradient.
Indeed, the shape functions are valid for any (typical) master element geometry, provided the affine coordinates are computed.
We hope this has convinced the reader that the direct use of affine coordinates for \textit{all} elements (not just simplices) is the ideal approach, and that it motivates other researchers to also communicate in those general terms.
The polynomials we chose are \textit{shifted} to have the domain $(0,1)$ instead of the typical $(-1,1)$.
We claim this is the natural choice for construction of shape functions, since all affine cordinates have range $[0,1]$ and affine coordinates are \textit{the} natural inputs for the polynomials.
Hence, we encourage the implementation of the polynomials in the shifted domain.
The concept of polynomial homogenization was introduced and heavily used.
It is a particular form of scaling which is closely related to affine coordinates, and is a tool that provides natural extensions.
In fact, homogenization provides some level of geometrical intuition, since the projected affine coordinates arise naturally through this process.
Theoretically it is also a convenient tool since it results in homogeneous polynomials, which have many desirable properties.

Moreover, only \textit{eight} ancillary operators effectively generate all shape functions.
These ancillary operators are coordinate free, in the sense that the form of the operators is invariant with respect to any transformation. 
This is important, because it allows to transform nonlinearly to other geometries.
Hence, it suffices to compute the affine coordinates and their gradient in that deformed space.
Then, the shape functions resulting from the substitution of the deformed affine coordinates will precisely be the well defined pullback of the original shape functions.
This has both theoretical and practical implications.
For example, curved physical elements are deformations of the master element domains. 
This has the potential to result in more efficient computations when integrating (see Appendix \ref{app:Integration} for the basics of integration).

For the face and edge shape functions, the logic of projecting, evaluating and blending is used consistently for all elements and all energy spaces.
This provides a firm geometrical intuition of the expressions and formulas for the shape functions, which we hope the reader will appreciate.

Additionally, the shape functions can be converted to orientation embedded shape functions via only \textit{three} local-to-global permutation functions (one for edges, one for triangle faces, and one for quadrilateral faces).
These orientation embedded shape functions are extremely practical in many applications, especially in the implementation of constrained nodes in $hp$ methods.

All the characteristics above prove to be vital in the implementation of a code.
Indeed, the number of important routines which are called repeatedly is very small, and this minimizes the sources of errors, while allowing a very focused optimization of the implementation.
A complete Fortran 90 code supplements this document.\footnote{See the ESEAS library available at https://github.com/libESEAS/ESEAS.}
It provides an excellent guidance if the reader is ever interested in implementing this construction.  
The code has been tested thoroughly by numerically checking polynomial reproducibility and exact sequence properties.
This is described in Appendix \ref{app:verification}.
The shape functions for all elements and all spaces are conveniently summarized in Appendix \ref{app:ShapeFunctionTable}.


Lastly, special attention is given to the successful construction of the pyramid shape functions, which is rare in the literature.
A thorough geometric intuition for the pyramid was described, and the 3D pyramid affine-related coordinates were defined and analyzed.
The set of exact sequence spaces were taken from \citet{Nigam_Phillips_11}, and they are consistent with the fundamental first order elements described by \citet{Hiptmair99}.
We believe it to be the first time that an arbitrary high order construction of $H(\mathrm{div})$ shape functions has been implemented whilst respecting those lower order spaces (there have been others where either the lower order spaces have been larger, or simply different).
This may prove to be very valuable to other researchers even if only for comparison purposes.
For the pyramid, it might be possible to investigate better choices of polynomials for the bubbles, as this may provide better conditioning and sparsity properties.

To finalize, we hope this construction has been useful in its methoodology and that it motivates further research in this very rich area.

\paragraph{Acknowledgements.}
The work of Fuentes, Keith, Demkowicz and Nagaraj was supported with grants by AFOSR (FA9550-12-1-0484), NSF (DMS-1418822) and Sandia National Laboratories (1536119).

%% file: genfamilies.tex
\section{Polynomial Families}
\label{app:GeneratingFamilies}

Let $\mathcal{P}^i(x)=\mathrm{span}\{x^j:j=0,\ldots,i\}$ be the polynomials of order $i$, and $\mathcal{F}_p=\{P_i:i=0,\ldots,p-1\}$ and $\mathcal{F}_p^\alpha=\{P_j^\alpha:j=0,\ldots,p-1\}$ be sets of polynomials with domain $x\in[0,1]$. 
The families $\mathcal{F}_p$ and $\mathcal{F}_p^\alpha$ should satisfy that $\mathrm{span}(\mathcal{F}_p)=\mathrm{span}(\mathcal{F}_p^\alpha)=\mathcal{P}^{p-1}(x)$, $P_i\in\mathcal{P}^i(x)$ and $P_j^\alpha\in\mathcal{P}^j(x)$.
Moreover, the family $\mathcal{F}_p$ should satisfy the zero average property, $\int_0^1 P_i(x)\,\mathrm{d}x=0$, for all $i\geq1$.

These sets of conditions are very easy to satisfy. 
For example consider any sequence of polynomials of increasing order, say $1,x,x^2,x^3,\ldots,x^{p-1}$.
Then it is only a matter of adding a suitable constant to all polynomials of order $i\geq1$ such that their integral is $0$.
For example, $1,x,x^2,x^3,\ldots,x^p$ becomes $1,x-\frac{1}{2},x^2-\frac{1}{3},\ldots,x^{p-1}-\frac{1}{p}$, and this latter one is already a suitable family $\mathcal{F}_p$.

The elements of $\mathcal{F}_p$ and $\mathcal{F}_p^\alpha$ are thought of being elements of $L^2$, even though they are infinitely differentiable as polynomials.
Indeed, it is often desirable that the elements of $\mathcal{F}_p$ and $\mathcal{F}_p^\alpha$ satisfy certain properties in the $L^2$ (or weighted $L^2$) inner product, since they can result in considerably sparser finite element matrices.
In fact, orthogonality of the $P_i$, is generally seeked.
If there is no weight, orthogonality of the $P_i$ is attained uniquely (up to a constant) by the Legendre polynomials, and this is why they are the typical choice.
Meanwhile, the family $\mathcal{F}_p^\alpha$ can also be chosen wisely by taking into account a weighted $L^2$ space relevant to the triangle element.
If the $P_i$ are Legendre polynomials, the natural choice for the $P_j^\alpha$ is to be Jacobi polynomials with certain weights.

Now, define the polynomials $L_{i}\in\mathcal{P}^{i}(x)$ and $L_{j}^\alpha\in\mathcal{P}^{j}(x)$ from the $P_{i-1}\in\mathcal{F}_p$ and $P_{j-1}^\alpha\in\mathcal{F}_p^\alpha$ as
\begin{equation}
    L_{i}(x)=\int_0^x P_{i-1}(\tilde{x})\,\mathrm{d}\tilde{x}\,,\qquad\qquad 
    	L_{j}^\alpha(x)=\int_0^x P_{j-1}^\alpha(\tilde{x})\,\mathrm{d}\tilde{x}\,,
\end{equation}
for $i=2,\ldots,p$ and $j=1,\ldots,p$.
Clearly, it follows that $\mathcal{P}^{p}(x)=\mathcal{P}^1(x)\cup\{L_i:i=2,\ldots,p\}$ and $\mathcal{P}^{p}(x)=\mathcal{P}^0(x)\cup\{L_j^\alpha:j=1,\ldots,p\}$.
By construction, the $L_i$ and $L_j^\alpha$ are elements of $H^1$ and as a result their pointwise evaluation exists. 
In fact, due to the zero average property, 
\begin{equation}
	\begin{alignedat}{3}
		L_i(0)&=L_i(1)=0\,,\qquad\qquad&&\text{for }i\geq2\,,\\
		L_j^\alpha(0)&=0\,,\qquad\qquad&&\text{for }i\geq1\,.
	\end{alignedat}
	\label{eq:AppLiLjVanishingProperties}
\end{equation}



Now, apply the definition of scaling given in \eqref{eq:scaledpolyomials} to the polynomials, yielding $P_i(x;t)$, $P_j^\alpha(x;t)$, $L_i(x;t)$ and $L_j^\alpha(x;t)$, where $x\in[0,t]$.
In particular, note that
\begin{equation}
	L_{i+1}(x;t)=t^{i+1}L_{i+1}\Big(\frac{x}{t}\Big)=t^{i+1}\int_0^{\frac{x}{t}} P_{i}(\tilde{x})\,\mathrm{d}\tilde{x}
		=t^{i}\int_0^{x} P_{i}\Big(\frac{\tilde{y}}{t}\Big)\,\mathrm{d}\tilde{y}\,.
\end{equation}
Since the scaled polynomials are now bivariate polynomials in $(x,t)$, it is useful to find the derivatives in both of these variables.
Using the latter equality above,
\begin{equation}
	\frac{\partial}{\partial x}L_{i+1}(x;t)=t^iP_i\Big(\frac{x}{t}\Big)=P_i(x;t)\,.
\end{equation}
Meanwhile, using the other equality, it follows
\begin{equation}
	\frac{\partial}{\partial t}L_{i+1}(x;t)
		=(i+1)t^iL_{i+1}\Big(\frac{x}{t}\Big)+t^{i+1}P_{i}\Big(\frac{x}{t}\Big)\Big(\frac{-x}{t^2}\Big)
			=t^i\Big((i+1)L_{i+1}\Big(\frac{x}{t}\Big)-\Big(\frac{x}{t}\Big)P_{i}\Big(\frac{x}{t}\Big)\Big)\,.
\end{equation}
This suggests the definition
\begin{equation}
	R_i(x)=(i+1)L_{i+1}(x) - xP_i(x)\,,
	\label{eq:R_i_eq1}
\end{equation}
for $i\geq1$.
Using the leading order term of $P_i$, the reader may observe that $R_i$ is an order $i$ polynomial (\textit{not} order $i+1$), so the indexing with $i$ makes sense.
More importantly,
\begin{equation}
	\frac{\partial}{\partial t}L_{i+1}(x;t)
			=t^i\Big((i+1)L_{i+1}\Big(\frac{x}{t}\Big)-\Big(\frac{x}{t}\Big)P_{i}\Big(\frac{x}{t}\Big)\Big)
				=R_{i}(x;t)\,.
\end{equation}
Exactly the same analysis applies to the $L_{j+1}^\alpha$, meaning that
\begin{equation}
	\frac{\partial}{\partial x}L_{j+1}\alpha(x;t)=P_j^\alpha(x;t)\,,\qquad\qquad
		\frac{\partial}{\partial x}L_{j+1}^\alpha(x;t)=R_j^\alpha(x;t)\,,
\end{equation}
for $j\geq0$, and where
\begin{equation}
	R_j^\alpha(x)=(j+1)L_{j+1}^\alpha(x) - xP_j^\alpha(x)\,.
\end{equation}

As observed, all these properties are deduced at a very general level, and apply to any families $\mathcal{F}_p$ and $\mathcal{F}_p^\alpha$ satisfying the simple set of properties mentioned initially.
To conclude, it follows all the results deduced throughout the document still hold for these more general polynomial families, including the vanishing properties, the ancillary function properties, etc.
Hence, the reader may decide to change these families if desired.
This could potentially provide better sparsity properties depending on the problem.
Nevertheless, be aware that for the classical Laplace problem (and many related sets of problems), the ideal polynomial families are the Legendre and Jacobi polynomials, which are used throughout this document and are conveniently defined through recursive formulas.


%% file: pyrappendix.tex
\section{Pyramid Supplement}
\label{app:pyrappendix}

\begin{figure}[!ht]
\begin{center}
\includegraphics[scale=0.5]{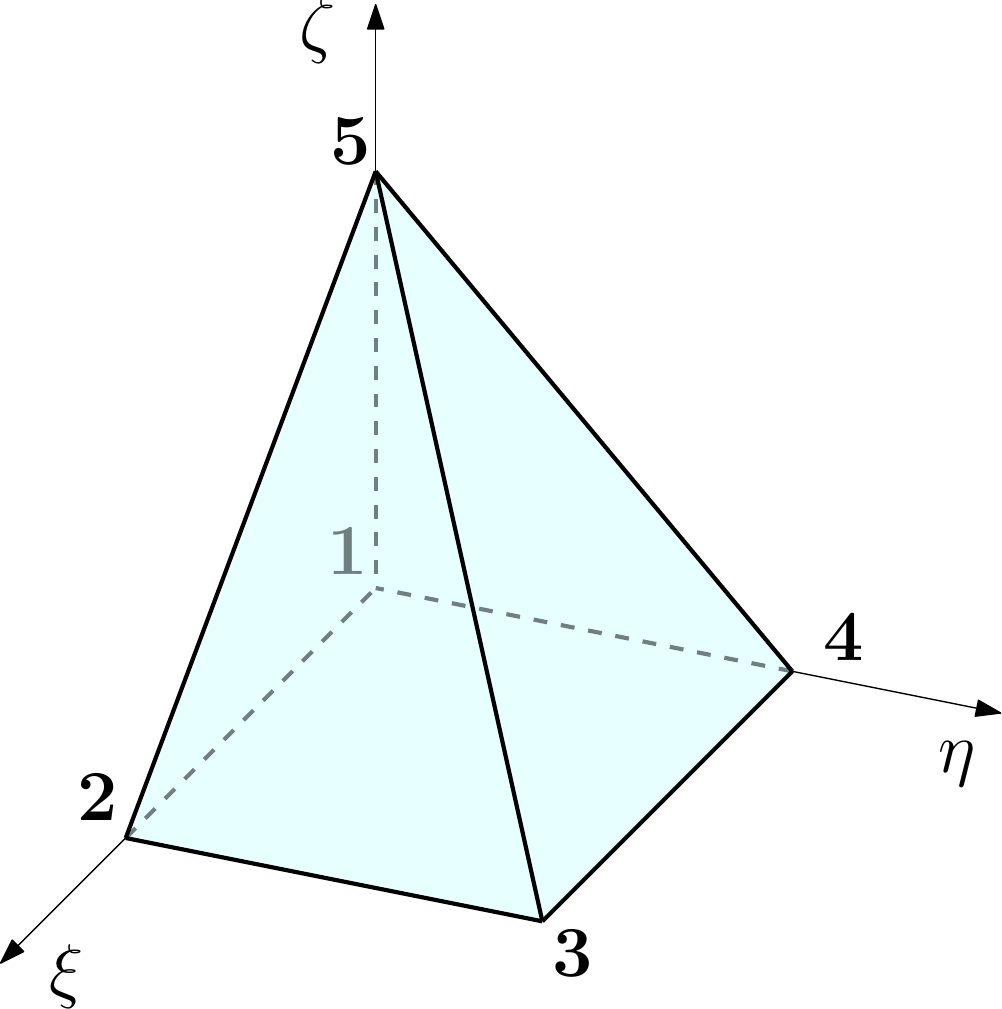}
\caption{Master pyramid with numbered vertices.}
\label{fig:MasterPyramidAppendix}
\end{center}
\end{figure}

The master pyramid is shown in Figure \ref{fig:MasterPyramidAppendix} in the $(\xi,\eta,\zeta)$ space (\textit{not} the $(\xi_1,\xi_2,\zeta)$ space).
More specifically, the definition is $\Omega=\{(\xi,\eta,\zeta)\in\R^3:\xi>0,\eta>0,\zeta>0,\xi+\zeta<1,\eta+\zeta<1\}$.

This supplement provides the proofs that the pyramid shape functions proposed in \S\ref{sec:Pyramid} are in the finite element spaces  defined in the fundamental work by \citet{Nigam_Phillips_11}.
Recall from \eqref{eq:pyramidsequence} that the discrete spaces forming the exact sequence are $\mathcal{U}^{(m),p}$ for $m=0,1,2,3$, where $m$ stands for the differential $m$-forms lying in each space.
The spaces $\mathcal{U}^{(m),p}$ are defined as
\begin{equation}
	\mathcal{U}_\infty^{(m),p}=\mathcal{V}^{(m),p}\cap\Gamma^{(m),p}\,,
\end{equation}
where the spaces $\mathcal{V}^{(m),p}$ are called the \textit{underlying spaces}, and the spaces $\Gamma^{(m),p}$ are called the \textit{compatibility spaces}.
These two families of spaces will be defined next.

First, note that the compatibility spaces ensure the elements in $\mathcal{U}^{(m),p}$ are compatible at the level of spaces with the other elements.
Indeed, they consist of those functions having their face traces lying on the appropriate 2D quadrilateral and triangle spaces in \eqref{eq:QuadES} and \eqref{eq:EStriangle} respectively.
More specifically, let $\mathrm{tr}_\Tri^{(m)}$ and $\mathrm{tr}_\square^{(m)}$ be the trace of the differential $m$-forms over the four pyramid triangle faces and the pyramid quadrilateral face respectively.
Recall that for $m=0$ the trace is the value of the function itself, for $m=1$ the trace is the 2D tangential component, while for $m=2$ the trace is the normal component.
With this in mind, the compatibility spaces are,\footnote{Note that in \eqref{eq:Pyramidcompatilibityspaces}, the very special property that $\mathcal{P}^p$, $\mathcal{N}^p$ and $\mathcal{RT}^p$ are affine invariant is used. If the spaces were different, one would need to take the 2D pullback of each triangle trace to the master triangle. The same holds for the quadrilateral trace, where in this case it is exploited that the quadrilateral face is the master quadrilateral and no transformation is needed.}
\begin{equation}
\begin{aligned}
	\Gamma^{(0),p}&=\{\phi\in H^1:\mathrm{tr}_\Tri^{(0)}(\phi)\in\mathrm{tr}_\Tri^{(0)}(\mathcal{P}^p(\xi,\eta,\zeta)),\,\,
		\mathrm{tr}_\square^{(0)}(\phi)\in\mathcal{Q}^{p,p}(\xi,\eta)\}\,,\\
	\Gamma^{(1),p}&=\{E\in H(\mathrm{curl}):\mathrm{tr}_\Tri^{(1)}(E)\in\mathrm{tr}_\Tri^{(1)}(\mathcal{N}^p(\xi,\eta,\zeta)),\,\,
		\mathrm{tr}_\square^{(1)}(E)\in\mathcal{Q}^{p-1,p}(\xi,\eta)\!\times\!\mathcal{Q}^{p,p-1}(\xi,\eta)\}\,,\\
	\Gamma^{(2),p}&=\{V\in H(\mathrm{div}):\mathrm{tr}_\Tri^{(2)}(V)\in\mathrm{tr}_\Tri^{(2)}(\mathcal{RT}^p(\xi,\eta,\zeta)),\,\,
		\mathrm{tr}_\square^{(2)}(V)\in\mathcal{Q}^{p-1,p-1}(\xi,\eta)\}\,,\\
	\Gamma^{(3),p}&=\{\psi\in L^2\}\,.
\end{aligned}
\label{eq:Pyramidcompatilibityspaces}
\end{equation}

Fortunately, at the level of shape functions, dealing with the compatibility spaces is not as intimidating as it might look.
All that is required is that the shape functions satisfy the dimensional hierarchy, so that their nonzero face traces correspond to the lower dimensional shape functions, which are known to lie in the correct space.
Therefore, if a shape function satisfies the required trace properties, it \textit{automatically} belongs to the appropriate compatibility space.

The main difficulty lies in showing that the shape functions belong to the underlying spaces $\mathcal{V}^{(m),p}$.
In fact, these spaces are not nicely defined directly on the master pyramid.
For this reason, it is more convenient to define them on a deformed space where the symmetries are more evident, and then use the inverse pullback to the master pyramid.
This process is explained in what follows.


\begin{figure}[!ht]
\begin{center}
\includegraphics[scale=0.5]{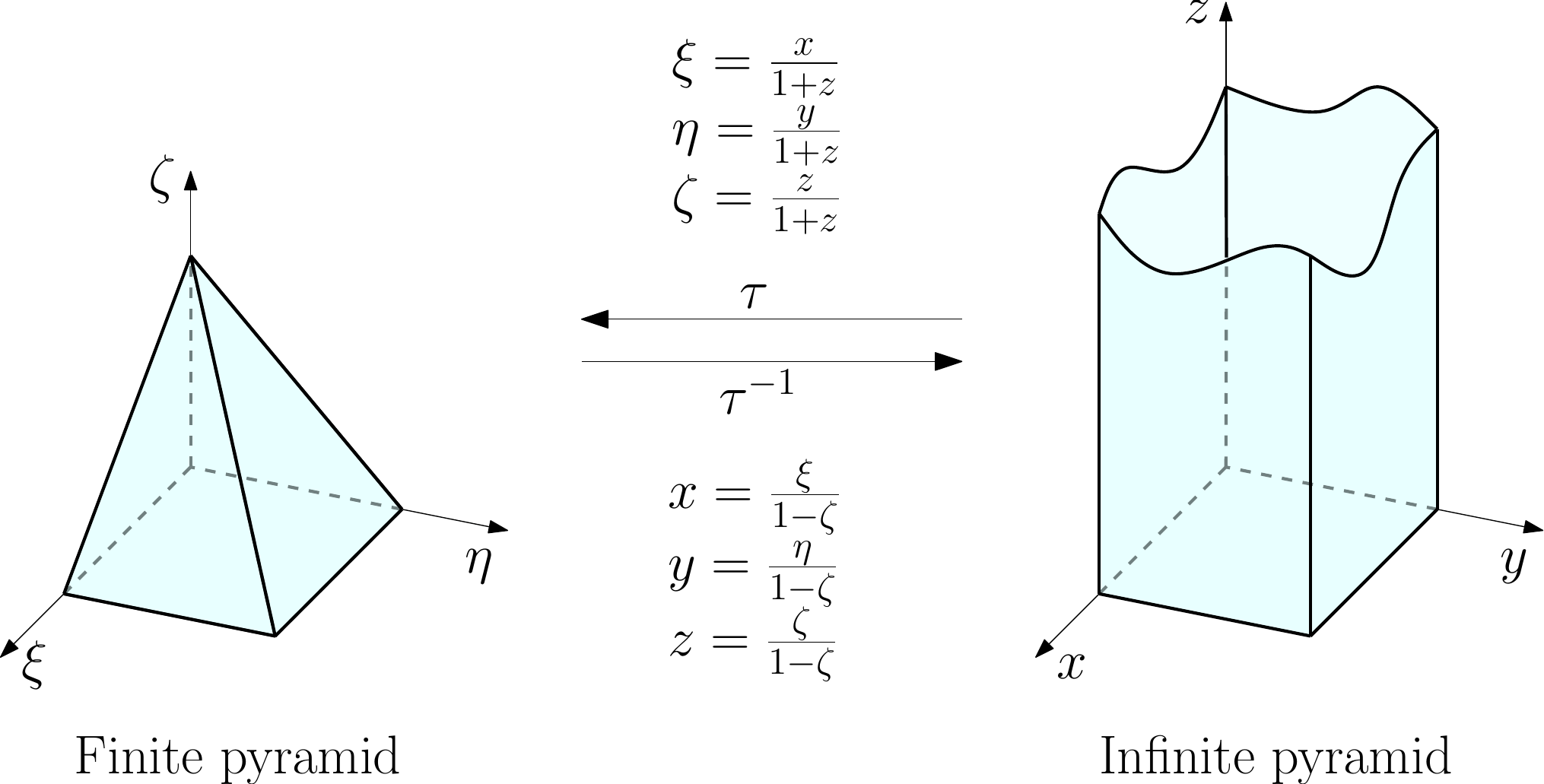}
\caption{Transformation from the infinite pyramid to the master pyramid.}
\label{fig:PyramidSingleTransform}
\end{center}
\end{figure}

The deformed space is usually chosen as a cube, but \citet{Nigam_Phillips_11} chose it to be an \textit{infinite pyramid}, which is defined as $\Omega_\infty=\{(x,y,z)\in\R^3:0<x<1,0<y<1,z>0\}$.
The first step is to consider the transformation $\tau:\Omega_\infty\rightarrow\Omega$ from the intinite pyramid $\Omega_\infty$ to the pyramid $\Omega$, which is given by the component equations,
\begin{equation}
	\xi=\frac{x}{1+z},\quad\qquad\eta=\frac{y}{1+z},\quad\qquad\zeta=\frac{z}{1+z}\,.
\end{equation}
Clearly, $\tau$ is a diffeomorphism between these open sets (\textit{not} when including the boundary).
The inverse $\tau^{-1}:\Omega\rightarrow\Omega_\infty$ is given by the component equations,
\begin{equation}
	x=\frac{\xi}{1-\zeta},\quad\qquad y=\frac{\eta}{1-\zeta},\quad\qquad z=\frac{\zeta}{1-\zeta}\,.
\end{equation}
The transformation is depicted in Figure \ref{fig:PyramidSingleTransform}.
Take note of the following useful expressions resulting from this transformation,
\begin{equation*}
	1-x=\frac{1-\xi-\zeta}{1-\zeta},\quad\quad 1-y=\frac{1-\eta-\zeta}{1-\zeta},\quad\quad
	1+z=\frac{1}{1-\zeta},\quad\quad \frac{1}{1+z}=1-\zeta\,.
\end{equation*}

Next consider the following isomorphic mappings,
\begin{equation}
\begin{aligned}
	\tau^{-*,(m)}:&\mathcal{V}_\infty^{(m),p}\longrightarrow\mathcal{V}^{(m),p}=\tau^{-*,(m)}(\mathcal{V}_\infty^{(m),p})\,,\\
		\tau^{*,(m)}:&\mathcal{V}^{(m),p}\longrightarrow\mathcal{V}_\infty^{(m),p}\,,
\end{aligned}
\end{equation}
where $\tau^{-*,(m)}$ and $\tau^{*,(m)}$ are the inverse pullback and pullback mappings induced by $\tau$.
Since the spaces $\mathcal{V}^{(m),p}$ and $\mathcal{V}_\infty^{(m),p}$ are isomorphic, it is mathematically irrelevant which of the two spaces is actually defined, since the other space can be determined through the corresponfing pullback mapping.
However, sometimes there are practical reasons to explicitly define one set of spaces over the other.
Indeed, it is more opportune to define the deformed underlying spaces $\mathcal{V}_\infty^{(m),p}$, which are\footnote{Note there is misprint in \citet{Nigam_Phillips_11} in (3.8) when presenting the equivalent characterization of $\mathcal{V}_\infty^{(1),p}$. The definition here corrects that, and is consistent with the calculations in \S3.3 of \citet{Nigam_Phillips_11}.}
\begin{equation}
\begin{aligned}
	\mathcal{V}_\infty^{(0),p}&=\{\phi\!\in\!\mathcal{Q}_p^{p,p,p}:
		\nabla\phi\!\in\!\mathcal{Q}_p^{p-1,p,p-1}\!\times\!\mathcal{Q}_p^{p,p-1,p-1}\!\times\!\mathcal{Q}_{p+1}^{p,p,p-1}\}\,,\\
	\mathcal{V}_\infty^{(1),p}&=\{E\!\in\!
		\mathcal{Q}_{p+1}^{p-1,p,p}\!\times\!\mathcal{Q}_{p+1}^{p,p-1,p}\!\times\!\mathcal{Q}_{p+1}^{p,p,p-1}:
			\nabla\!\times\! E\!\in\!
				\mathcal{Q}_{p+2}^{p,p-1,p-1}\!\times\!\mathcal{Q}_{p+2}^{p-1,p,p-1}\!\times\!\mathcal{Q}_{p+1}^{p-1,p-1,p-1}\}\,,\\
	\mathcal{V}_\infty^{(2),p}&=\{V\!\in\!
		\mathcal{Q}_{p+2}^{p,p-1,p-1}\!\times\!\mathcal{Q}_{p+2}^{p-1,p,p-1}\!\times\!\mathcal{Q}_{p+2}^{p-1,p-1,p}:
			\nabla\cdot V\!\in\!\mathcal{Q}_{p+3}^{p-1,p-1,p-1}\}\,,\\
	\mathcal{V}_\infty^{(3),p}&=\{\psi\!\in\!\mathcal{Q}_{p+3}^{p-1,p-1,p-1}\}\,,
\end{aligned}
\end{equation}
where $\mathcal{Q}_{n}^{p,q,r}=\mathcal{Q}_{n}^{p,q,r}(x,y,z)$ are the $n$-weighted tensor polynomial spaces. 
They are defined as
\begin{equation}
	\mathcal{Q}_{n}^{p,q,r}(x,y,z)
		=\Big\{\frac{\psi(x,y,z)}{(1+z)^n}:\psi(x,y,z)\in\mathcal{Q}^{p,q,r}(x,y,z)\Big\}\,.
\end{equation}
Note the useful inclusion $\mathcal{Q}_{n}^{p,q,r}\subseteq\mathcal{Q}_{n+1}^{p,q,r+1}$ which holds for these rational polynomial spaces.

Lastly, note that the pullbacks $\tau^{*,(m)}$ take different forms depending on $m$.
Indeed, if $J_\tau$ is the Jacobian of the transformation $\tau$, the pullbacks are,
\begin{equation}
\begin{alignedat}{3}
	\tau^{*,(0)}\phi&=\phi\circ\tau\,,\quad&&\phi\in H^1\,,\\
	\tau^{*,(1)}E&=J_\tau^{\T}(E\circ\tau)\,,\quad && E\in H(\mathrm{curl})\,,\\
	\tau^{*,(2)}V&=\det(J_\tau)J_\tau^{-1}(V\circ\tau)\,,\quad && V\in H(\mathrm{div})\,,\\
	\tau^{*,(3)}\psi&=\det(J_\tau)(\psi\circ\tau)\,,\quad && \psi\in L^2\,.
\end{alignedat}
\label{eq:pullbacksgeneral}
\end{equation}
The same relations hold for the inverse pullbacks $\tau^{-*,(m)}$, but replacing $\tau$ by $\tau^{-1}$, $J_\tau$ by $J_{\tau^{-1}}$, and the domains by their isometrically isomorphic counterparts,\footnote{The spaces are iso\textit{metrically} isomorphic if the appropriate weights are added to the definition of the norm.} $H_\infty^1=\tau^{*,(0)}(H^1)$, $H(\mathrm{curl})_\infty=\tau^{*,(1)}(H(\mathrm{curl}))$, $H(\mathrm{div})_\infty=\tau^{*,(2)}(H(\mathrm{div}))$, and $L_\infty^2=\tau^{*,(3)}(L^2)$.
Even though they will be unnecessary, in the interest of completeness these pullback mappings are written explicitly below,
\begin{equation}
\begin{alignedat}{3}
	\tau^{*,(0)}\phi&=\phi\circ\tau\,,\quad
		&&\phi\in H^1\,,\\
	\tau^{*,(1)}E&=\frac{1}{(1+z)^2}\begin{pmatrix}1+z&0&0\\0&1+z&0\\-x&-y&1\end{pmatrix}(E\circ\tau)\,,\quad
		&& E\in H(\mathrm{curl})\,,\\
	\tau^{*,(2)}V&=\frac{1}{(1+z)^3}\begin{pmatrix}1&0&x\\0&1&y\\0&0&1+z\end{pmatrix}(V\circ\tau)\,,\quad
		&& V\in H(\mathrm{div})\,,\\
	\tau^{*,(3)}\psi&=\frac{1}{(1+z)^4}(\psi\circ\tau)\,,\quad
		&&\psi\in L^2\,.
\end{alignedat}
\label{eq:Pyramidpullbacks}
\end{equation}
The inverse pullbacks explicitly are,
\begin{equation}
\begin{alignedat}{3}
	\tau^{-*,(0)}\phi&=\phi\circ\tau^{-1}\,,\quad 
		&&\phi\in H_\infty^1\,,\\
	\tau^{-*,(1)}E&=\frac{1}{(1-\zeta)^2}\begin{pmatrix}1-\zeta&0&0\\0&1-\zeta&0\\\xi&\eta&1\end{pmatrix}(E\circ\tau^{-1})\,,\quad 
		&& E\in H(\mathrm{curl})_\infty\,,\\
	\tau^{-*,(2)}V&=\frac{1}{(1-\zeta)^3}\begin{pmatrix}1&0&-\xi\\0&1&-\eta\\0&0&1-\zeta\end{pmatrix}(V\circ\tau^{-1})\,,\quad 
		&& V\in H(\mathrm{div})_\infty\,,\\
	\tau^{-*,(3)}\psi&=\frac{1}{(1-\zeta)^4}(\psi\circ\tau^{-1})\,,\quad 
		&&\psi\in L_\infty^2\,.
\end{alignedat}
\label{eq:Pyramidinversepullbacks}
\end{equation}

To prove the shape functions defined in \S\ref{sec:Pyramid} are in the underlying spaces, in general one would pull them back (using \eqref{eq:Pyramidpullbacks}) and check if they belong to the spaces $\mathcal{V}_\infty^{(m),p}$.
However, due to the \textit{coordinate free} definitions of the ancillary functions and the shape functions in general, it is not necessary to find the pullback explicitly.
Instead, one simply finds the trivial pullback of all the sets of affine coordinates and their gradients in the $(x,y,z)$ coordinates.
Then, the only task is to evaluate all the shape functions with these affine coordinates and the result will be the desired pullback of the original shape function.
This is why the mappings \eqref{eq:Pyramidpullbacks} and \eqref{eq:Pyramidinversepullbacks} become redundant for our shape functions.

In view of these comments, it is useful to have all the affine coordinates and their gradients in the $(x,y,z)$ space.
The triangle affine coordinates (see \eqref{eq:PyramidTriCoord}) are
\begin{equation}
	\begin{gathered}
		\nu_0(x,z)=\sxxz\,,\qquad\nu_1(x,z)=\sxz\,,\qquad\nu_2(z)=\szzz\,,\\
		\nu_0(y,z)=\syyz\,,\qquad\nu_1(y,z)=\syz\,,\qquad\nu_2(z)=\szzz\,.
	\end{gathered}
\end{equation}
Their gradient (see \eqref{eq:PyramidTriCoordGrad}) is
\begin{equation}
	\begin{gathered}
		\nabla\nu_0(x,z)=\textstyle{\frac{1}{(1+z)^2}}\bigg(\begin{smallmatrix}-(1+z)\\[2pt]0\\[2pt]-(1-x)\end{smallmatrix}\bigg)\,,\qquad
			\nabla\nu_1(x,z)=\textstyle{\frac{1}{(1+z)^2}}\bigg(\begin{smallmatrix}(1+z)\\[2pt]0\\[2pt]-x\end{smallmatrix}\bigg)\,,\qquad
				\nabla\nu_2(z)=\textstyle{\frac{1}{(1+z)^2}}\bigg(\begin{smallmatrix}0\\[2pt]0\\[2pt]1\end{smallmatrix}\bigg)\,,\\
		\nabla\nu_0(y,z)=\textstyle{\frac{1}{(1+z)^2}}\bigg(\begin{smallmatrix}0\\[2pt]-(1+z)\\[2pt]-(1-y)\end{smallmatrix}\bigg)\,,\qquad
			\nabla\nu_1(y,z)=\textstyle{\frac{1}{(1+z)^2}}\bigg(\begin{smallmatrix}0\\[2pt](1+z)\\[2pt]-y\end{smallmatrix}\bigg)\,,\qquad
				\nabla\nu_2(z)=\textstyle{\frac{1}{(1+z)^2}}\bigg(\begin{smallmatrix}0\\[2pt]0\\[2pt]1\end{smallmatrix}\bigg)\,.
	\end{gathered}
\end{equation}
The sets of quadrilateral scaled 1D affine coordinates (see \eqref{eq:PyramidQuadCoord}) are
\begin{equation}
	\begin{gathered}
		\mu_0(x)=1-x\,,\quad\qquad\mu_1(x)=x\,,\\
		\mu_0(y)=1-y\,,\quad\qquad\mu_1(y)=y\,.
	\end{gathered}
\end{equation}
Their gradient (see \eqref{eq:PyramidQuadCoordGrad}) is
\begin{equation}
	\begin{gathered}
		\nabla\mu_0(x)=\bigg(\begin{smallmatrix}-1\\[2pt]0\\[2pt]0\end{smallmatrix}\bigg)\,,\quad\qquad
    	\nabla\mu_1(x)=\bigg(\begin{smallmatrix}1\\[2pt]0\\[2pt]0\end{smallmatrix}\bigg)\,,\\
		\nabla\mu_0(y)=\bigg(\begin{smallmatrix}0\\[2pt]-1\\[2pt]0\end{smallmatrix}\bigg)\,,\quad\qquad
    	\nabla\mu_1(y)=\bigg(\begin{smallmatrix}0\\[2pt]1\\[2pt]0\end{smallmatrix}\bigg)\,.
	\end{gathered}
\end{equation}
The vertical 1D affine coordinates (see \eqref{eq:PyramidZCoord}) are
\begin{equation}
		\mu_0(\szzz)=1-\szzz=\szz\,,\quad\qquad\mu_1(\szzz)=\szzz\,.
\end{equation}
Their gradient (see \eqref{eq:PyramidZCoordGrad}) is
\begin{equation}
		\nabla\mu_0(\szzz)=\textstyle{\frac{1}{(1+z)^2}}\bigg(\begin{smallmatrix}0\\[2pt]0\\[2pt]-1\end{smallmatrix}\bigg)\,,\quad\qquad
			\nabla\mu_1(\szzz)=\textstyle{\frac{1}{(1+z)^2}}\bigg(\begin{smallmatrix}0\\[2pt]0\\[2pt]1\end{smallmatrix}\bigg)\,.		
\end{equation}
Finally, the pyramid 3D affine coordinates (see \eqref{eq:PyramidAffineCoord}) are
\begin{equation}
	\begin{gathered}
    \lambda_1(x,y,z)=\textstyle{\frac{(1-x)(1-y)}{1+z}}\,,\qquad\qquad
    \lambda_2(x,y,z)=\textstyle{\frac{x(1-y)}{1+z}}\,,\qquad\qquad
    \lambda_3(x,y,z)=\textstyle{\frac{xy}{1+z}}\,,\\
    \lambda_4(x,y,z)=\textstyle{\frac{(1-x)y}{1+z}}\,,\qquad\qquad
    \lambda_5(x,y,z)=\textstyle{\frac{z}{1+z}}\,.
	\end{gathered}
\end{equation}
Their gradient (see \eqref{eq:PyramidAffineCoordGrad}) is
\begin{equation}
	\begin{gathered}
    \nabla\lambda_1(x,y,z)=\begin{pmatrix}\frac{-(1-y)}{1+z}\\\frac{-(1-x)}{1+z}\\
        \frac{-(1-x)(1-y)}{(1+z)^2}\end{pmatrix}\,,\;
    \nabla\lambda_2(x,y,z)=\begin{pmatrix}\frac{(1-y)}{1+z}\\\frac{-x}{1+z}\\
        \frac{-x(1-y)}{(1+z)^2}\end{pmatrix}\,,\;
    \nabla\lambda_3(x,y,z)=\begin{pmatrix}\frac{y}{1+z}\\\frac{x}{1+z}\\
        \frac{-xy}{(1+z)^2}\end{pmatrix}\,,\\
    \nabla\lambda_4(x,y,z)=\begin{pmatrix}\frac{-y}{1+z}\\\frac{(1-x)}{1+z}\\
        \frac{-(1-x)y}{(1+z)^2}\end{pmatrix}\,,\;
    \nabla\lambda_5(z)=\begin{pmatrix}0\\0\\\frac{1}{(1+z)^2}\end{pmatrix}\,.
	\end{gathered}
\end{equation}

Before beginning with the proofs, take note of the following useful results.
\begin{lemma}
\label{lem:LegendreDecomp}
Let $L_i$ and $L_j^\alpha$ be the integrated Legendre and Jacobi polynomials of order $i\geq2$ and $j\geq1$.
Then there exist homogeneous polynomials $[\psi_{i-2}]\in\tilde{\mathcal{P}}^{i-2}(s_0,s_1)$ and $[\chi_{j-1}]\in\tilde{\mathcal{P}}^{j-1}(t_0,t_1)$ such that\footnote{This result is not limited to Legendre and Jacobi polynomials and, as reflected in the proof, applies as well to the general polynomial families presented in Appendix \ref{app:GeneratingFamilies} (see \eqref{eq:AppLiLjVanishingProperties}).}
\begin{equation}
	\begin{aligned}
		{}[L_i](s_0,s_1)&=s_0s_1[\psi_{i-2}](s_0,s_1)\,,\\
		[L_j^\alpha](t_0,t_1)&=t_1[\chi_{j-1}](t_0,t_1)\,.\\
	\end{aligned}	
\end{equation}
\end{lemma} 
\begin{proof}
The integrated Legendre polynomials $L_i(s_1)$ for $i\geq2$ vanish at $s_1=0$ and $s_1=1$ (see \eqref{eq:Lvanishatendpoints}).
Hence, they must take the form
\begin{equation*}
	L_i(s_1)=s_1(1-s_1)\psi_{i-2}(s_1)\,,
\end{equation*}
for some polynomial $\psi_{i-2}\in\mathcal{P}^{i-2}(s_1)$.
After homogenization one obtains the desired result,
\begin{equation*}
\begin{aligned}
	{}[L_i](s_0,s_1)&=L_i\Big(\frac{s_1}{s_0+s_1}\Big)(s_0+s_1)^{i}
		=\frac{s_1}{s_0+s_1}\Big(1-\frac{s_1}{s_0+s_1}\Big)\psi_{i-2}\Big(\frac{s_1}{s_0+s_1}\Big)(s_0+s_1)^i\\
	&=\frac{s_0s_1}{(s_0+s_1)^2}(s_0+s_1)^2\psi_{i-2}\Big(\frac{s_1}{s_0+s_1}\Big)(s_0+s_1)^{i-2}
			=s_0s_1[\psi_{i-2}](s_0,s_1)\,,
\end{aligned}
\end{equation*}
where $[\psi_{i-2}]\in\tilde{\mathcal{P}}^{i-2}(s_0,s_1)$ is the homogenization of $\psi_{i-2}$.
The result involving $L_j^\alpha(t_1)$ follows using exactly the same reasoning and that it vanishes at $t_1=0$ (see \eqref{eq:Lalphavanishzero}). 
\end{proof}

\begin{remark}
Let $L_i$ and $L_j^\alpha$ be the integrated Legendre and Jacobi polynomials of order $i\geq2$ and $j\geq1$. 
Then there exists a homogeneous polynomial $[\psi_{i-2},\chi_{j-1}]\in\tilde{\mathcal{P}}^{i+j-3}(s_0,s_1,s_2)$ such that
\begin{equation}
	[L_i,L_j^\alpha](s_0,s_1,s_2)=[L_i](s_0,s_1)[L_j^\alpha](s_0+s_1,s_2)=s_0s_1s_2[\psi_{i-2},\chi_{j-1}](s_0,s_1,s_2)\,.
	\label{eq:LiLjDecomp}
\end{equation}
\end{remark} 

\begin{remark}
Let $L_k$ be the integrated Legendre polynomials of order $k\geq2$. Then, 
\begin{equation}
		[L_k](\vec{\mu}_{01}(\szzz))\!=\![L_k](\szz,\szzz)\!=\!\frac{z}{(1+z)^2}[\psi_{k-2}](\szz,\szzz)
			\!=\!\frac{z}{(1+z)^k}[\psi_{k-2}](1,z)\!\in\!\mathcal{Q}_k^{0,0,k-1}\,,
	\label{eq:PyrphikZSpace}
\end{equation}
where $[\psi_{k-2}]\in\tilde{\mathcal{P}}^{k-2}(s_0,s_1)$ is the homogeneous polynomial in Lemma \ref{lem:LegendreDecomp}.
\end{remark}

\subsection{\texorpdfstring{$H^1$}{H1} Shape Functions}

For ease of reference, note the deformed underlying space for $m=0$ is
\begin{equation*}
	\mathcal{V}_\infty^{(0),p}=\{\phi\!\in\!\mathcal{Q}_p^{p,p,p}:
		\nabla\phi\!\in\!\mathcal{Q}_p^{p-1,p,p-1}\!\times\!\mathcal{Q}_p^{p,p-1,p-1}\!\times\!\mathcal{Q}_{p+1}^{p,p,p-1}\}\,.
\end{equation*}

\subsubsection {\texorpdfstring{$H^1$}{H1} Vertices}

The shape functions for the vertices are given by \eqref{eq:PyrH1Vertex}, and they are precisely the 3D pyramid affine coordinates.
Take for example any quadrilateral vertex, say $v_1$, and the top vertex $v_5$, which merits its own attention.
Their corresponding shape functions already satisfy the trace properties as discussed when they were defined, so they lie in the lowest order compatibility space $\Gamma^{(0),1}$.
Also, the vertex function for $v_1$ satisfies
\begin{equation}
\begin{aligned}
	\phi^\mathrm{v}(x,y,z)&=\lambda_1(x,y,z)=\textstyle{\frac{(1-x)(1-y)}{1+z}}\in\mathcal{Q}_1^{1,1,1}\,,\\
		\nabla\phi^\mathrm{v}(x,y,z)&=\nabla\lambda_1(x,y,z)
			=\Bigg(\begin{smallmatrix}\frac{-(1-y)}{1+z}\\\frac{-(1-x)}{1+z}\\\frac{-(1-x)(1-y)}{(1+z)^2}\end{smallmatrix}\Bigg)
				\in\Bigg(\begin{smallmatrix}\mathcal{Q}_1^{0,1,0}\\[2pt]\mathcal{Q}_1^{1,0,0}\\[2pt]
					\mathcal{Q}_2^{1,1,0}\end{smallmatrix}\Bigg)\,.
	\label{eq:PyrQuadVertexProof}
\end{aligned}
\end{equation}
The same holds analogously for the other three quadrilateral vertices.
Similarly, the top vertex function satisfies
\begin{equation}
\begin{aligned}
	\phi^\mathrm{v}(x,y,z)&=\lambda_5(x,y,z)=\textstyle{\frac{z}{1+z}}\in\mathcal{Q}_1^{1,1,1}\,,\\
		\nabla\phi^\mathrm{v}(x,y,z)&=\nabla\lambda_5(x,y,z)
			=\Bigg(\begin{smallmatrix}0\\[2pt]0\\[2pt]\frac{1}{(1+z)^2}\end{smallmatrix}\Bigg)
				\in\Bigg(\begin{smallmatrix}0\\[2pt]0\\[2pt]
					\mathcal{Q}_2^{0,0,0}\end{smallmatrix}\Bigg)
						\subseteq\Bigg(\begin{smallmatrix}\mathcal{Q}_1^{0,1,0}\\[2pt]\mathcal{Q}_1^{1,0,0}\\[2pt]
							\mathcal{Q}_2^{1,1,0}\end{smallmatrix}\Bigg)\,.
	\label{eq:PyrTopVertexProof}
\end{aligned}
\end{equation}
Therefore, all deformed vertex functions lie in the lowest order underlying space $\mathcal{V}_\infty^{(0),1}$, so that the original shape functions are in the lowest order space $\mathcal{U}^{(0),1}$.

\subsubsection {\texorpdfstring{$H^1$}{H1} Edges}

\paragraph{\texorpdfstring{$H^1$}{H1} Mixed Edges.}
The shape functions for mixed edges are presented in \eqref{eq:PyrH1MixedEdge}.
For every mixed edge they are labeled as $\phi_i^\mathrm{e}$, where $i=2,\ldots,p$.
As discussed when they were defined, they satisfy the desired trace properties, so they lie in the compatibility space $\Gamma^{(0),i}$.
To see they also lie in the underlying space take as an example mixed edge 12, and note that
\begin{equation}
\begin{aligned}
	\phi_i^\mathrm{e}(x,y,z)&=\mu_0(y)\phi_i^\E(\vec{\nu}_{01}(x,z))
		\!=\!\mu_0(y)[L_i](\textstyle{\frac{1-x}{1+z}},\textstyle{\frac{x}{1+z}})
			\!=\!\textstyle{\frac{1}{(1+z)^i}}\underbrace{\mu_0(y)[L_i](1-x,x)}_{A(x,y)\in\mathcal{Q}^{i,i}(x,y)}
				\!\in\!\mathcal{Q}_i^{i,i,i}\,,\\
	\nabla\phi_i^\mathrm{e}(x,y,z)&=
		\textstyle{\frac{1}{(1+z)^i}}\underbrace{\nabla A(x,y)}_{\in\mathcal{Q}^{i-1,i,0}\!\times\!\mathcal{Q}^{i,i-1,0}\!\times\!\{0\}}
			+A(x,y)\underbrace{\nabla\textstyle{\frac{1}{(1+z)^i}}}_{\in\{0\}\!\times\!\{0\}\!\times\!\mathcal{Q}_{i+1}^{0,0,0}}
				\!\in\!\Bigg(\begin{smallmatrix}\mathcal{Q}_i^{i-1,i,i-1}\\[2pt]\mathcal{Q}_i^{i,i-1,i-1}\\[2pt]	
					\mathcal{Q}_{i+1}^{i,i,i-1}\end{smallmatrix}\Bigg)\,.
\end{aligned}
\end{equation}
Here it was used that $[L_i]$ is a homogenized polynomial, so in particular \eqref{eq:ScalingProperty} holds.
Therefore, the deformed shape functions are in $\mathcal{V}_\infty^{(0),i}$, so that the shape functions lie in $\mathcal{U}^{(0),i}$.
Naturally, the same calculations hold for all other mixed edges.

\paragraph{\texorpdfstring{$H^1$}{H1} Triangle Edges.}
The shape functions for triangle edges are defined in \eqref{eq:PyrH1TriaEdge}, and are labeled as $\phi_i^\mathrm{e}$, where $i=2,\ldots,p$.
They satisfy the desired trace properties, so they lie in the compatibility space $\Gamma^{(0),i}$.
As an example take triangle edge 15.
Then it follows
\begin{equation}
\begin{aligned}
	\phi_i^\mathrm{e}(x,y,z)&=\phi_i^\E(\vec{\lambda}_{15}(x,y,z))
		\!=\![L_i](\textstyle{\frac{(1-x)(1-y)}{1+z}},\textstyle{\frac{z}{1+z}})
			\!=\!\textstyle{\frac{1}{(1+z)^i}}
				\underbrace{[L_i]((1-x)(1-y),z)}_{A(x,y,z)\in\mathcal{Q}^{i-1,i-1,i-1}}\!\in\!\mathcal{Q}_i^{i,i,i}\,,\\
	\nabla\phi_i^\mathrm{e}(x,y,z)&=\textstyle{\frac{1}{(1+z)^i}}\!\!\!\!\!\!
		\underbrace{\nabla A(x,y,z)}_{\in\mathcal{Q}^{i-2,i-1,i-1}\!\times\!\mathcal{Q}^{i-1,i-2,i-1}
			\!\times\!\mathcal{Q}^{i-1,i-1,i-2}}\!\!\!\!\!\!
				+A(x,y,z)\underbrace{\nabla\textstyle{\frac{1}{(1+z)^i}}}_{\in\{0\}\!\times\!\{0\}\!\times\!\mathcal{Q}_{i+1}^{0,0,0}}
			\!\in\!\Bigg(\begin{smallmatrix}\mathcal{Q}_i^{i-1,i,i-1}\\[2pt]\mathcal{Q}_i^{i,i-1,i-1}\\[2pt]
				\mathcal{Q}_{i+1}^{i,i,i-1}\end{smallmatrix}\Bigg)\,,
\end{aligned}
\end{equation}
where $[L_i]((1-x)(1-y),z)=A(x,y,z)\in\mathcal{Q}^{i-1,i-1,i-1}$ due to Lemma \ref{lem:LegendreDecomp}, and where it is used that $\mathcal{Q}_i^{i-1,i-1,i-2}\subseteq\mathcal{Q}_{i+1}^{i-1,i-1,i-1}$.
Hence, the shape functions are in the correct underlying space and they belong to $\mathcal{U}^{(0),i}$.
Analogous calculations hold for the other triangle edges.

\subsubsection {\texorpdfstring{$H^1$}{H1} Faces}

\paragraph{\texorpdfstring{$H^1$}{H1} Quadrilateral Face.}
The quadrilateral face functions are defined in \eqref{eq:PyrH1QuadFace} and identified as $\phi_{ij}^\mathrm{f}$, for $i=2,\ldots,p$ and $j=2,\ldots,p$.
Again, the functions are known to satisfy the desired trace properties, so they lie in the compatibility space $\Gamma^{(0),n}$, where $n=\max\{i,j\}$.
To see they are in the underlying space, note
\begin{equation}
\begin{aligned}
	\phi_{ij}^\mathrm{f}(x,y,z)&=\mu_0(\szzz)\phi_{ij}^\square(\vec{\mu}_{01}(x),\vec{\mu}_{01}(y))
		\!=\!\textstyle{\frac{1}{1+z}}
			\underbrace{\phi_{ij}^\square(\vec{\mu}_{01}(x),\vec{\mu}_{01}(y))}_{A(x,y)\in\mathcal{Q}^{i,j}(x,y)}
				\!\in\!\mathcal{Q}_1^{i,j,0}\subseteq\mathcal{Q}_n^{n,n,n}\,,\\
	\nabla\phi_{ij}^\mathrm{f}(x,y,z)&=\textstyle{\frac{1}{1+z}}\!\!\!\!\!\!
		\underbrace{\nabla A(x,y)}_{\in\mathcal{Q}^{i-1,j,0}\!\times\!\mathcal{Q}^{i,j-1,0}\!\times\!\{0\}}\!\!\!\!\!\!
				+A(x,y)\underbrace{\nabla\textstyle{\frac{1}{1+z}}}_{\in\{0\}\!\times\!\{0\}\!\times\!\mathcal{Q}_{2}^{0,0,0}}
			\!\in\!\Bigg(\begin{smallmatrix}\mathcal{Q}_n^{n-1,n,n-1}\\[2pt]\mathcal{Q}_n^{n,n-1,n-1}\\[2pt]
				\mathcal{Q}_{n+1}^{n,n,n-1}\end{smallmatrix}\Bigg)\,,
\end{aligned}
\end{equation}
where the inclusions of the type $\mathcal{Q}_1^{i,j,0}\subseteq\mathcal{Q}_n^{i,j,n-1}\subseteq\mathcal{Q}_n^{n,n,n}$ are used repeatedly.
Note that the factor $\mu_0(\szzz)$ was \textit{required} in order for the function to be in the correct space (see the first two components of the gradient).
It follows the shape functions belong to $\mathcal{U}^{(0),n}$, where $n=\max\{i,j\}$.

\paragraph{\texorpdfstring{$H^1$}{H1} Triangle Faces.} 
The triangle face functions are defined in \eqref{eq:PyrH1TriaFace} and are labeled as $\phi_{ij}^\mathrm{f}$, for $i\geq2$, $j\geq1$, and $n=i+j=3,\ldots,p$.
They satisfy the trace properties, so they lie in $\Gamma^{(0),n}$.
Take for instance face 125, and observe the functions also lie in the corresponding underlying space, because
\begin{equation}
\begin{aligned}
	\phi_{ij}^\mathrm{f}(x,y,z)&=\mu_0(y)\phi_{ij}^\Tri(\vec{\nu}_{012}(x,z))
			\!=\!\textstyle{\frac{1}{(1+z)^n}}
				\underbrace{\mu_0(y)[L_i,L_j^\alpha](1-x,x,z)}_{A(x,y,z)\in\mathcal{Q}^{n-1,n-1,n-2}}\!\in\!\mathcal{Q}_n^{n,n,n}\,,\\
	\nabla\phi_{ij}^\mathrm{f}(x,y,z)&=\textstyle{\frac{1}{(1+z)^n}}\!\!\!\!\!\!\!\!\!\!\!\!\!\!
		\underbrace{\nabla A(x,y,z)}_{\in\mathcal{Q}^{n-2,n-1,n-2}\!\times\!\mathcal{Q}^{n-1,n-2,n-2}
			\!\times\!\mathcal{Q}^{n-1,n-1,n-3}}\!\!\!\!\!\!\!\!\!\!\!\!\!\!
				+A(x,y,z)\underbrace{\nabla\textstyle{\frac{1}{(1+z)^n}}}_{\in\{0\}\!\times\!\{0\}\!\times\!\mathcal{Q}_{n+1}^{0,0,0}}
			\!\!\!\in\!\Bigg(\begin{smallmatrix}\mathcal{Q}_n^{n-1,n,n-1}\\[2pt]\mathcal{Q}_n^{n,n-1,n-1}\\[2pt]
				\mathcal{Q}_{n+1}^{n,n,n-1}\end{smallmatrix}\Bigg)\,.
\end{aligned}
\end{equation}
Here, $[L_i,L_j^\alpha](1-x,x,z)=A(x,y,z)\in\mathcal{Q}^{n-1,n-1,n-2}$ by \eqref{eq:LiLjDecomp}.
Hence, $\phi_{ij}^\mathrm{f}$ is in the correct space $\mathcal{U}^{(0),n}$, where $n=i+j$.
The same follows for the other triangle faces.

\subsubsection{\texorpdfstring{$H^1$}{H1} Interior Bubbles}
The interior bubbles are defined in \eqref{eq:PyrH1Interior}, and identified as $\phi_{ijk}^\mathrm{b}$, where $i=2,\ldots,p$, $j=2,\ldots,p$ and $k=2,\ldots,p$.
They satisfy the vanishing properties along the whole boundary so they are trivially in $\Gamma^{(0),n}$, for $n=\max\{i,j,k\}$.
They also belong to the underlying spaces $\mathcal{V}_\infty^{(0),n}$, since
\begin{equation}
\begin{aligned}
	\phi_{ij}^\mathrm{f}(x,y,z)&=
		\underbrace{\phi_k^\E(\vec{\mu}_{01}(\szzz))}_{B(z)\in\mathcal{Q}_k^{0,0,k-1}}
			\underbrace{\phi_{ij}^\square(\vec{\mu}_{01}(x),\vec{\mu}_{01}(y))}_{A(x,y)\in\mathcal{Q}^{i,j}(x,y)}
				\!\in\!\mathcal{Q}_k^{i,j,k-1}\subseteq\mathcal{Q}_n^{n,n,n}\,,\\
	\nabla\phi_{ij}^\mathrm{f}(x,y,z)&=B(z)\!\!\!\!\!\!
		\underbrace{\nabla A(x,y)}_{\in\mathcal{Q}^{i-1,j,0}\!\times\!\mathcal{Q}^{i,j-1,0}\!\times\!\{0\}}\!\!\!\!\!\!
				+A(x,y)\underbrace{\nabla B(z)}_{\in\{0\}\!\times\!\{0\}\!\times\!\mathcal{Q}_{k+1}^{0,0,k-1}}
			\!\in\!\Bigg(\begin{smallmatrix}\mathcal{Q}_n^{n-1,n,n-1}\\[2pt]\mathcal{Q}_n^{n,n-1,n-1}\\[2pt]
				\mathcal{Q}_{n+1}^{n,n,n-1}\end{smallmatrix}\Bigg)\,.
	\label{eq:PyrH1InteriorProof}
\end{aligned}
\end{equation}
Here, \eqref{eq:PyrphikZSpace} was explicitly used.
It follows all the interior bubbles lie in $\mathcal{U}^{(0),n}$, where $n=\max\{i,j,k\}$.

\subsection{\texorpdfstring{$H(\mathrm{curl})$}{Hcurl} Shape Functions}

Recall the deformed underlying space for $m=1$ is
\begin{equation*}
		\mathcal{V}_\infty^{(1),p}=\{E\!\in\!
		\mathcal{Q}_{p+1}^{p-1,p,p}\!\times\!\mathcal{Q}_{p+1}^{p,p-1,p}\!\times\!\mathcal{Q}_{p+1}^{p,p,p-1}:
			\nabla\!\times\! E\!\in\!
				\mathcal{Q}_{p+2}^{p,p-1,p-1}\!\times\!\mathcal{Q}_{p+2}^{p-1,p,p-1}\!\times\!\mathcal{Q}_{p+1}^{p-1,p-1,p-1}\}\,.
\end{equation*}

\subsubsection{\texorpdfstring{$H(\mathrm{curl})$}{Hcurl} Edges}

\paragraph{\texorpdfstring{$H(\mathrm{curl})$}{Hcurl} Mixed Edges.}
The mixed edge shape functions are given in \eqref{eq:PyrHcurlMixedEdge}, and are labeled as $E_i^\mathrm{e}$ for $i=0,\ldots,p-1$.
As discussed before, the trace properties are satisfied, meaning that they belong to $\Gamma^{(1),i+1}$.
To see they are in the underlying space, consider edge 12 and note that
\begin{equation}
\begin{aligned}
	E_i^\mathrm{e}(x,y,z)&\!=\!\mu_0(y)E_i^\E(\textstyle{\frac{1-x}{1+z}},\textstyle{\frac{x}{1+z}})
		\!=\!\textstyle{\frac{1}{(1+z)^{i+2}}}\mu_0(y)E_i^\E(\overbrace{1-x}^{\tilde{s}_0(x)},\overbrace{x}^{\tilde{s}_1(x)})\\
	&\!=\!\textstyle{\frac{1}{(1+z)^{i+2}}}\underbrace{\mu_0(y)[P_i](\tilde{s}_0(x),\tilde{s}_1(x))
			(\overbrace{\tilde{s}_0(x)\nabla\tilde{s}_1(x)\!-\!\tilde{s}_1(x)\nabla\tilde{s}_0(x)}^{
				\in\mathcal{Q}^{0,0,0}\!\times\!\mathcal{Q}^{0,0,0}\!\times\!\{0\}})}_{
					A(x,y)\in\mathcal{Q}^{i,1,0}\!\times\!\mathcal{Q}^{1,i,0}\!\times\!\{0\}}
						\!\in\!\!\Bigg(\begin{smallmatrix}\mathcal{Q}_{i+2}^{i,i+1,i+1}\\[2pt]\mathcal{Q}_{i+2}^{i+1,i,i+1}\\[2pt]	
							\mathcal{Q}_{i+2}^{i+1,i+1,i}\end{smallmatrix}\Bigg)\,,\\
	\nabla\!\!\times\!\! E_i^\mathrm{e}(x,y,z)&\!=\!
		\textstyle{\frac{1}{(1+z)^{i+2}}}\underbrace{\nabla\!\times\! A(x,y)}_{
			\in\{0\}\!\times\!\{0\}\!\times\!\mathcal{Q}^{i,i,0}}
			+\underbrace{\nabla\textstyle{\frac{1}{(1+z)^{i+2}}}\!\times\! A(x,y)}_{
				\in\mathcal{Q}_{i+3}^{1,i,0}\!\times\!\mathcal{Q}_{i+3}^{i,1,0}\!\times\!\{0\}}
					\!\in\!\!\Bigg(\begin{smallmatrix}\mathcal{Q}_{i+3}^{i+1,i,i}\\[2pt]\mathcal{Q}_{i+3}^{i,i+1,i}\\[2pt]	
						\mathcal{Q}_{i+2}^{i,i,i}\end{smallmatrix}\Bigg)\,.
\end{aligned}
\end{equation}
Hence, the pullback of the shape functions lies in $\mathcal{V}_\infty^{(1),i+1}$, so the original shape functions lie in $\mathcal{U}^{(1),i+1}$ as desired.
The same calculations hold for the other mixed edges.

\paragraph{\texorpdfstring{$H(\mathrm{curl})$}{Hcurl} Triangle Edges.}
The triangle edge functions are defined in \eqref{eq:PyrHcurlTriaEdge}, and are identified as $E_i^\mathrm{e}$, where $i=0,\ldots,p-1$.
Again, the trace properties are satisfied, so that they are inside $\Gamma^{(1),i+1}$.
Take for instance edge 15.
Making use of \eqref{eq:curlsEiE}, \eqref{eq:PyrQuadVertexProof} and \eqref{eq:PyrTopVertexProof}, notice that
\begin{equation}
\begin{aligned}
	E_i^\mathrm{e}(x,y,z)&\!=\!E_i^\E(\lambda_1(x,y,z),\lambda_5(z))
			\!=\!\textstyle{\frac{1}{(1+z)^{i+2}}}E_i^\E(\overbrace{(1-x)(1-y)}^{\tilde{s}_0(x,y)},\overbrace{z}^{\tilde{s}_1(z)})\\
	&\!=\!\textstyle{\frac{1}{(1+z)^{i+2}}}[P_i](\tilde{s}_0(x,y),\tilde{s}_1(z))
			(\overbrace{\tilde{s}_0(x,y)\nabla\tilde{s}_1(z)\!-\!\tilde{s}_1(z)\nabla\tilde{s}_0(x,y)}^{
				\in\mathcal{Q}^{0,1,1}\!\times\!\mathcal{Q}^{1,0,1}\!\times\!\mathcal{Q}^{1,1,0}})
						\!\in\!\!\Bigg(\begin{smallmatrix}\mathcal{Q}_{i+2}^{i,i+1,i+1}\\[2pt]\mathcal{Q}_{i+2}^{i+1,i,i+1}\\[2pt]	
							\mathcal{Q}_{i+2}^{i+1,i+1,i}\end{smallmatrix}\Bigg)\,,\\
	\nabla\!\!\times\!\! E_i^\mathrm{e}(x,y,z)&\!=\!
		\textstyle{\frac{(i+2)}{(1+z)^{i}}}
			[P_i](\tilde{s}_0(x,y),\tilde{s}_1(z))\underbrace{\overbrace{\nabla\lambda_1(x,y,z)}^{
				\in\mathcal{Q}_1^{0,1,0}\!\times\!\mathcal{Q}_1^{1,0,0}\!\times\!\mathcal{Q}_2^{1,1,0}}
					\!\times\!\overbrace{\nabla\lambda_5(z)}^{
					\in\{0\}\!\times\!\{0\}\!\times\!\mathcal{Q}_2^{0,0,0}}}_{
						\in\mathcal{Q}_3^{1,0,0}\!\times\!\mathcal{Q}_3^{0,1,0}\!\times\!\{0\}}
							\!\in\!\!\Bigg(\begin{smallmatrix}\mathcal{Q}_{i+3}^{i+1,i,i}\\[2pt]\mathcal{Q}_{i+3}^{i,i+1,i}\\[2pt]	
								\mathcal{Q}_{i+2}^{i,i,i}\end{smallmatrix}\Bigg)\,.
\end{aligned}
\end{equation}
Therefore, the shape functions belong to $\mathcal{V}_\infty^{(1),i+1}$, and as a result lie in $\mathcal{U}^{(1),i+1}$.
The same reasoning is attached to the other triangle edges.

\subsubsection{\texorpdfstring{$H(\mathrm{curl})$}{Hcurl} Faces}

\paragraph{\texorpdfstring{$H(\mathrm{curl})$}{Hcurl} Quadrilateral Face.}
The two closely related families of quadrilateral face functions are presented in \eqref{eq:PyrHcurlQuadFaceI} and \eqref{eq:PyrHcurlQuadFaceII}. 
They are labeled as $E_{ij}^\mathrm{f}$, for $i=0,\ldots,p-1$ and $j=2,\ldots,p$.
As usual, the functions are shown to satisfy the desired trace properties, so they lie in the compatibility space $\Gamma^{(1),n}$, where $n=\max\{i+1,j\}$.
They lie in the underlying space because
\begin{equation}
\begin{aligned}
	E_{ij}^\mathrm{f}(x,y,z)&\!=\!\textstyle{\frac{1}{(1+z)^2}}
			\underbrace{E_{ij}^\square(\vec{\mu}_{01}(x),\vec{\mu}_{01}(y))}_{A(x,y)
				\in\mathcal{Q}^{i,j,0}\!\times\!\mathcal{Q}^{j,i,0}\!\times\!\{0\}}
					\!\in\!\Bigg(\begin{smallmatrix}\mathcal{Q}_2^{n-1,n,0}\\[2pt]\mathcal{Q}_2^{n,n-1,0}\\[2pt]0\end{smallmatrix}\Bigg)
						\subseteq\Bigg(\begin{smallmatrix}\mathcal{Q}_{n+1}^{n-1,n,n}\\[2pt]\mathcal{Q}_{n+1}^{n,n-1,n}\\[2pt]
							\mathcal{Q}_{n+1}^{n,n,n-1}\end{smallmatrix}\Bigg)\,,\\
	\nabla\!\times\! E_{ij}^\mathrm{f}(x,y,z)&\!=\!\textstyle{\frac{1}{(1+z)^2}}\!\!\!\!\!\!
		\underbrace{\nabla\times A(x,y)}_{\in\{0\}\!\times\!\{0\}\!\times\!\mathcal{Q}^{n-1,n-1,0}}\!\!\!\!\!\!
				+\underbrace{\overbrace{\nabla\textstyle{\frac{1}{(1+z)^2}}}^{\in\{0\}\!\times\!\{0\}\!\times\!\mathcal{Q}_{3}^{0,0,0}}
					\!\!\!\!\times\,\, A(x,y)}_{\in\mathcal{Q}_{3}^{j,i,0}\!\times\!\mathcal{Q}_{3}^{i,j,0}\!\times\!\{0\}}
						\!\in\!\Bigg(\begin{smallmatrix}\mathcal{Q}_{n+2}^{n,n-1,n-1}\\[2pt]\mathcal{Q}_{n+2}^{n-1,n,n-1}\\[2pt]
							\mathcal{Q}_{n+1}^{n-1,n-1,n-1}\end{smallmatrix}\Bigg)\,,
	\label{eq:PyrHcurlQuadFaceProof}
\end{aligned}
\end{equation}
where the inclusions of the type $\mathcal{Q}_2^{i,j,0}\subseteq\mathcal{Q}_{n+1}^{i,j,n-1}\subseteq\mathcal{Q}_{n+1}^{n-1,n,n}$ are used repeatedly.
Notice that the factor $\mu_0(\szzz)^2$ was \textit{required} in order for the function to be in the correct space (see the last component of the curl).
It follows the shape functions belong to $\mathcal{U}^{(1),n}$, where $n=\max\{i+1,j\}$.
The calculations are invariant to permutations, so the result holds for both families.

\paragraph{\texorpdfstring{$H(\mathrm{curl})$}{Hcurl} Triangle Faces.}
The two closely related families of triangle face functions are defined in \eqref{eq:PyrHcurlTriaFaceI} and \eqref{eq:PyrHcurlTriaFaceII}. 
They are identified as $E_{ij}^\mathrm{f}$, for $i\geq0$, $j\geq1$, and $n=i+j=1,\ldots,p-1$.
They satisfy the trace properties, so they lie in $\Gamma^{(1),n}$.
Take for instance face 125.
Using Lemma \ref{lem:LegendreDecomp} it follows that
\begin{equation}
\begin{aligned}
	E_{ij}^\mathrm{f}(x,y,z)&=\mu_0(y)E_{ij}^\Tri(\vec{\nu}_{012}(x,z))
		=\mu_0(y)[P_i,L_{j}^\alpha](\vec{\nu}_{012}(x,z))E_{0}^\E(\vec{\nu}_{01}(x,z))\\
			&=\underbrace{\overbrace{[P_i,\chi_{j-1}](\vec{\nu}_{012}(x,z))}^{\in\mathcal{Q}_{n-1}^{n-1,n-1,n-1}}
				\overbrace{\mu_0(y)\nu_2(x,z)E_0^\E(\vec{\nu}_{01}(x,z))}^{
					\in\mathcal{Q}_3^{1,1,1}\!\times\!\mathcal{Q}_3^{1,1,1}\!\times\!\mathcal{Q}_3^{2,2,0}}}_{
						\in\mathcal{Q}_{n+2}^{n,n,n}\!\times\!\mathcal{Q}_{n+2}^{n,n,n}\!\times\!\mathcal{Q}_{n+2}^{n+1,n+1,n-1}}
							\in\Bigg(\begin{smallmatrix}\mathcal{Q}_{n+2}^{n,n+1,n+1}\\[2pt]\mathcal{Q}_{n+2}^{n+1,n,n+1}\\[2pt]
								\mathcal{Q}_{n+2}^{n+1,n+1,n}\end{smallmatrix}\Bigg)\,,\\
	\nabla \!\times\! E_{ij}^\mathrm{f}(x,y,z)
		&\in\nabla\!\times\!\Bigg(\begin{smallmatrix}\mathcal{Q}_{n+2}^{n,n,n}\\[2pt]\mathcal{Q}_{n+2}^{n,n,n}\\[2pt]
			\mathcal{Q}_{n+2}^{n+1,n+1,n-1}\end{smallmatrix}\Bigg)
				\subseteq\Bigg(\begin{smallmatrix}\mathcal{Q}_{n+3}^{n+1,n,n}\\[2pt]\mathcal{Q}_{n+3}^{n,n+1,n}\\[2pt]
					\mathcal{Q}_{n+2}^{n,n,n}\end{smallmatrix}\Bigg)\,.
\end{aligned}
\end{equation}
Note the calculations hold regardless of permutations in $\vec{\nu}_{012}(x,z)$, meaning that both families lie in the underlying space $\mathcal{V}_\infty^{(1),n+1}$.
Hence, the shape functions $E_{ij}^\mathrm{f}$ are in the correct space $\mathcal{U}^{(1),n+1}$, where $n=i+j$.
The same follows for the other triangle faces.

\subsubsection{\texorpdfstring{$H(\mathrm{curl})$}{Hcurl} Interior Bubbles}

\subparagraph{Family I:}
The first family is defined by \eqref{eq:PyrHcurlInteriorI} and it is the gradient of the $H^1$ bubbles.
They are labeled as $E_{ijk}^\mathrm{b}$, where $i=2,\ldots,p$, $j=2,\ldots,p$, and $k=2,\ldots,p$. 
They satisfy the trace properties and lie in $\Gamma^{(1),n}$, for $n=\max\{i,j,k\}$.
As gradients of $\mathcal{U}^{(0),n}$ they lie in the underlying space. 
Indeed, from \eqref{eq:PyrH1InteriorProof} it follows
\begin{equation}
\begin{aligned}
	E_{ijk}^\mathrm{b}(x,y,z)&=\nabla\Big(\phi_k^\E(\vec{\mu}_{01}(\szzz))\phi_{ij}^\square(\vec{\mu}_{01}(x),\vec{\mu}_{01}(y))\Big)
		\!\in\!\Bigg(\begin{smallmatrix}\mathcal{Q}_n^{n-1,n,n-1}\\[2pt]\mathcal{Q}_n^{n,n-1,n-1}\\[2pt]
			\mathcal{Q}_{n+1}^{n,n,n-1}\end{smallmatrix}\Bigg)
				\subseteq\Bigg(\begin{smallmatrix}\mathcal{Q}_{n+1}^{n-1,n,n}\\[2pt]\mathcal{Q}_{n+1}^{n,n-1,n}\\[2pt]
					\mathcal{Q}_{n+1}^{n,n,n-1}\end{smallmatrix}\Bigg)\,,\\
	\nabla\!\times\! E_{ijk}^\mathrm{b}(x,y,z)&=0
		\!\in\!\Bigg(\begin{smallmatrix}\mathcal{Q}_{n+2}^{n,n-1,n-1}\\[2pt]\mathcal{Q}_{n+2}^{n-1,n,n-1}\\[2pt]
			\mathcal{Q}_{n+1}^{n-1,n-1,n-1}\end{smallmatrix}\Bigg)\,.
\end{aligned}
\end{equation}
Hence, the interior bubbles lie in $\mathcal{U}^{(1),n}$, for $n=\max\{i,j,k\}$.

\subparagraph{Families II and III:}
These families are presented in \eqref{eq:PyrHcurlInteriorII} and \eqref{eq:PyrHcurlInteriorIII}, and are identified as $E_{ijk}^\mathrm{b}$, with $i=0,\ldots,p-1$, $j=2,\ldots,p$, and $k=2,\ldots,p$.
They satisfy the vanishing trace properties and belong to $\Gamma^{(1),n}$, for $n=\max\{i+1,j,k\}$.
Using \eqref{eq:PyrHcurlQuadFaceProof} and \eqref{eq:PyrphikZSpace} it follows
\begin{equation}
\begin{aligned}
	E_{ijk}^\mathrm{b}(x,y,z)&\!=\!
		\underbrace{\textstyle{\frac{1}{1+z}}\phi_k^\E(\vec{\mu}_{01}(\szzz))}_{B(z)\in\mathcal{Q}_{k+1}^{0,0,k-1}} 
		\underbrace{E_{ij}^\square(\vec{\mu}_{01}(x),\vec{\mu}_{01}(y))}_{A(x,y)
			\in\mathcal{Q}^{i,j,0}\!\times\!\mathcal{Q}^{j,i,0}\!\times\!\{0\}}
				\!\in\!\Bigg(\begin{smallmatrix}\mathcal{Q}_{n+1}^{n-1,n,n-1}\\[2pt]\mathcal{Q}_{n+1}^{n,n-1,n-1}\\[2pt]0\end{smallmatrix}\Bigg)
					\subseteq\Bigg(\begin{smallmatrix}\mathcal{Q}_{n+1}^{n-1,n,n}\\[2pt]\mathcal{Q}_{n+1}^{n,n-1,n}\\[2pt]
						\mathcal{Q}_{n+1}^{n,n,n-1}\end{smallmatrix}\Bigg)\,,\\
	\nabla\!\times\! E_{ijk}^\mathrm{b}(x,y,z)&\!=\!B(z)\!\!\!\!\!\!
		\underbrace{\nabla\times A(x,y)}_{\in\{0\}\!\times\!\{0\}\!\times\!\mathcal{Q}^{n-1,n-1,0}}\!\!\!\!\!\!
				+\underbrace{\overbrace{\nabla B(z)}^{\in\{0\}\!\times\!\{0\}\!\times\!\mathcal{Q}_{k+2}^{0,0,k-1}}
					\!\!\!\!\!\!\times\,\,\, A(x,y)}_{\in\mathcal{Q}_{k+2}^{j,i,k-1}\!\times\!\mathcal{Q}_{k+2}^{i,j,k-1}\!\times\!\{0\}}
						\in\Bigg(\begin{smallmatrix}\mathcal{Q}_{n+2}^{n,n-1,n-1}\\[2pt]\mathcal{Q}_{n+2}^{n-1,n,n-1}\\[2pt]
							\mathcal{Q}_{n+1}^{n-1,n-1,n-1}\end{smallmatrix}\Bigg)\,.
\end{aligned}
\label{eq:PyrHcurlInteriorIIProof}
\end{equation}
The analysis is independent of permutations of the entries, so it holds for both families.
Therefore, the functions of both families are elements of $\mathcal{U}^{(1),n}$, for $n=\max\{i+1,j,k\}$.

\subparagraph{Family IV:}
The family is shown in \eqref{eq:PyrHcurlInteriorIV}, with the functions being identified as $E_{ij}^{\mathrm{b}}$, for $i=2,\ldots,p$, and $j=2,\ldots,p$.
They have vanishing trace, so they lie in $\Gamma^{(1),n}$ for $n=\max\{i,j\}$.
Moreover,
\begin{equation}
\begin{aligned}
	E_{ij}^\mathrm{b}(x,y,z)&=
		\underbrace{\phi_{ij}^\square(\vec{\mu}_{01}(x),\vec{\mu}_{01}(y))}_{A(x,y)\in\mathcal{Q}^{i,j}(x,y)}
			\underbrace{\nabla\textstyle{\frac{1}{(1+z)^n}}}_{B(z)\in\{0\}\!\times\!\{0\}\!\times\!\mathcal{Q}_{n+1}^{0,0,0}}
				\!\in\!\Bigg(\begin{smallmatrix}0\\[2pt]0\\[2pt]\mathcal{Q}_{n+1}^{i,j,0}\end{smallmatrix}\Bigg)
					\subseteq\Bigg(\begin{smallmatrix}\mathcal{Q}_{n+1}^{n-1,n,n}\\[2pt]\mathcal{Q}_{n+1}^{n,n-1,n}\\[2pt]
						\mathcal{Q}_{n+1}^{n,n,n-1}\end{smallmatrix}\Bigg)\,,\\
	\nabla\!\times\! E_{ij}^\mathrm{b}(x,y,z)&=
		\underbrace{\nabla A(x,y)}_{\in\mathcal{Q}^{i-1,j,0}\!\times\!\mathcal{Q}^{i,j-1,0}\!\times\!\{0\}}\!\!\!\!\!\!\times\,\,B(z)
			\!\in\!\Bigg(\begin{smallmatrix}\mathcal{Q}_{n+1}^{i,j-1,0}\\[2pt]\mathcal{Q}_{n+1}^{i-1,j,0}\\[2pt]0\end{smallmatrix}\Bigg)
				\subseteq\Bigg(\begin{smallmatrix}\mathcal{Q}_{n+2}^{n,n-1,n-1}\\[2pt]\mathcal{Q}_{n+2}^{n-1,n,n-1}\\[2pt]
					\mathcal{Q}_{n+1}^{n-1,n-1,n-1}\end{smallmatrix}\Bigg)\,.
\end{aligned}
\label{eq:PyrHcurlInteriorIVProof}
\end{equation}
Hence, the functions belong to $\mathcal{U}^{(1),n}$, for $n=\max\{i,j\}$.

\subsection{\texorpdfstring{$H(\mathrm{div})$}{Hdiv} Shape Functions}

Recall the deformed underlying space for $m=2$ is
\begin{equation*}
	\mathcal{V}_\infty^{(2),p}=\{V\!\in\!
		\mathcal{Q}_{p+2}^{p,p-1,p-1}\!\times\!\mathcal{Q}_{p+2}^{p-1,p,p-1}\!\times\!\mathcal{Q}_{p+2}^{p-1,p-1,p}:
			\nabla\cdot V\!\in\!\mathcal{Q}_{p+3}^{p-1,p-1,p-1}\}\,.
\end{equation*}

\subsubsection{\texorpdfstring{$H(\mathrm{div})$}{Hdiv} Faces}

\paragraph{\texorpdfstring{$H(\mathrm{div})$}{Hdiv} Quadrilateral Face.}
The quadrilateral face functions are defined in \eqref{eq:PyrHdivQuadFace}, and they are labeled as $V_{ij}^\mathrm{f}$, for $i=0,\ldots,p-1$ and $j=0,\ldots,p-1$.
The functions satisfy the desired trace properties, so they lie in the compatibility space $\Gamma^{(2),n}$, where $n=\max\{i+1,j+1\}$.
They lie in the underlying space because
\begin{equation}
\begin{aligned}
	V_{ij}^\mathrm{f}(x,y,z)&\!=\!\textstyle{\frac{1}{(1+z)^3}}
		\underbrace{V_{ij}^\square(\vec{\mu}_{01}(x),\vec{\mu}_{01}(y))}_{A(x,y)\in\{0\}\!\times\!\{0\}\!\times\!\mathcal{Q}^{i,j,0}}
			\!\in\!\Bigg(\begin{smallmatrix}0\\[2pt]0\\[2pt]\mathcal{Q}_3^{n-1,n-1,0}\end{smallmatrix}\Bigg)
				\subseteq\Bigg(\begin{smallmatrix}\mathcal{Q}_{n+2}^{n,n-1,n-1}\\[2pt]\mathcal{Q}_{n+2}^{n-1,n,n-1}\\[2pt]
					\mathcal{Q}_{n+2}^{n-1,n-1,n}\end{smallmatrix}\Bigg)\,,\\
	\nabla\!\cdot\! V_{ij}^\mathrm{f}(x,y,z)&\!=
		\underbrace{\nabla\textstyle{\frac{1}{(1+z)^3}}}_{\in\{0\}\!\times\!\{0\}\!\times\!\mathcal{Q}_{4}^{0,0,0}}
			\!\!\!\!\cdot\,\,\, A(x,y)\in\mathcal{Q}_{4}^{i,j,0}
				\subseteq\mathcal{Q}_{n+3}^{n-1,n-1,n-1}\,.
\end{aligned}
\label{eq:PyrHdivQuadFaceProof}
\end{equation}
Notice that the factor $\mu_0(\szzz)^3$ was \textit{required} in order for the function to be in the correct space (see the divergence).
It follows the shape functions belong to $\mathcal{U}^{(2),n}$, where $n=\max\{i+1,j+1\}$.

\paragraph{\texorpdfstring{$H(\mathrm{div})$}{Hdiv} Triangle Faces.}
This is by far the most difficult construction (and proof).
Since only a sketch was shown before, more details will be given here in relation to those functions.
As an example consider face 125.

The typical approach is to use the ancillary functions with appropriately chosen affine coordinates as their input, and possibly some blending factor to ensure the trace properties are satisfied.
Indeed, this directly yields all the shape functions up to this point.
However, the two most sensible options for the shape functions fail to be even in the lowest order space,
\begin{equation}
\begin{aligned}
	\mu_0(y)V_{00}^\Tri(\vec{\nu}_{012}(x,z))
		&=\textstyle{\frac{1}{(1+z)^3}}\bigg(\begin{smallmatrix}0\\[2pt]-(1-y)\\[2pt]0\end{smallmatrix}\bigg)
			\in\Bigg(\begin{smallmatrix}\mathcal{Q}_{3}^{1,0,0}\\[2pt]\mathcal{Q}_{3}^{0,1,0}\\[2pt]
				\mathcal{Q}_{3}^{0,0,1}\end{smallmatrix}\Bigg)\,,\\
	\nabla\!\cdot\!\Big(\mu_0(y)V_{00}^\Tri(\vec{\nu}_{012}(x,z))\Big)
		&=\textstyle{\frac{1}{(1+z)^3}}=\textstyle{\frac{1+z}{(1+z)^4}}\notin\mathcal{Q}_{4}^{0,0,0}\,,\\
	\textstyle{\frac{1}{\mu_0(y)}}V_{00}^\Tri(\vec{\lambda}_{125}(x,y,z))
		&=\textstyle{\frac{1}{(1+z)^3}}\bigg(\begin{smallmatrix}0\\[2pt]-(1-y)\\[2pt]z\end{smallmatrix}\bigg)
			\in\Bigg(\begin{smallmatrix}\mathcal{Q}_{3}^{1,0,0}\\[2pt]\mathcal{Q}_{3}^{0,1,0}\\[2pt]
				\mathcal{Q}_{3}^{0,0,1}\end{smallmatrix}\Bigg)\,,\\
	\nabla\!\cdot\!\Big(\textstyle{\frac{1}{\mu_0(y)}}V_{00}^\Tri(\vec{\lambda}_{125}(x,y,z))\Big)
		&=\textstyle{\frac{2-z}{(1+z)^4}}\notin\mathcal{Q}_{4}^{0,0,0}\,.		
\end{aligned}
\label{eq:PyrFailedAttempts}
\end{equation}

One option is to depart from the spaces defined by \citet{Nigam_Phillips_11}, by adding elements to $\mathcal{U}^{(2),p}$.
However, this has the negative effect that the discrete $L^2$ space, $\mathcal{U}^{(3),p}$, also has to be modified to satisfy the exact sequence property.
Thus, one could have scenarios where the lowest order discrete space for $L^2$ is the constants plus other functions, as opposed to simply the constants.
Although valid, this deviates from the other elements, where the lowest order discrete $L^2$ space is always the constants.
Moreover, this option adds degrees of freedom to the construction.
Hence, modifying the spaces $\mathcal{U}^{(2),p}$ and $\mathcal{U}^{(3),p}$ is a last resort, and should be avoided when possible.

A second option, is not to use the ancillary functions.
However, using the ancillary functions directly is not a capricious decision.
The main reason is that they automatically guarantee that no ``illegal'' derivatives are present in the divergence of the shape functions.
Indeed, the traces should be of the form of 2D $L^2$ functions for a given face, so intuitively, elements of $L^2$ should be involved in the expressions for the shape functions.
However, the divergence of the shape functions cannot involve derivatives of those traces, since they are elements of $L^2$, which in general do not have derivatives.
Indeed, the traces for the triangle faces in our shape functions involve combinations of Legendre and Jacobi polynomials, $P_i$ and $P_j^\alpha$, which are the representatives of $L^2$. 
As expected, up to now the divergence of $H(\mathrm{div})$ face functions has not involved any derivatives of $P_i$ and $P_j^\alpha$.
This is not a triviality, since in general one would expect the derivatives to be there, but due to Lemma \ref{lem:divformula}, the derivatives disappear if the ancillary function $V_{ij}^\Tri$ is utilized.
Therefore, amongst other reasons, it is highly desirable to use the ancillary functions.

With these facts in mind, the message is then to at least try to persist in the use of the ancillary functions without modifying the space.
The key to making this possible relies in observing closely the lowest order space.
Here, the first order shape functions are known explicitly for each face and have been deduced and presented by \citet{Hiptmair99} and \citet{Nigam_Phillips_11} amongst others.
For face 125, the pullback of that shape function is
\begin{equation}
	V_{00}^\mathrm{f}(x,y,z)=\textstyle{\frac{1}{(1+z)^3}}\bigg(\begin{smallmatrix}0\\[2pt]-(1-y)\\[2pt]\frac{1}{2}z\end{smallmatrix}\bigg)
			\in\Bigg(\begin{smallmatrix}\mathcal{Q}_{3}^{1,0,0}\\[2pt]\mathcal{Q}_{3}^{0,1,0}\\[2pt]
				\mathcal{Q}_{3}^{0,0,1}\end{smallmatrix}\Bigg)\,,\quad\qquad
	\nabla\!\cdot\!V_{00}^\mathrm{f}(x,y,z)=\textstyle{\frac{3}{2(1+z)^4}}\in\mathcal{Q}_{4}^{0,0,0}\,.
\end{equation}
As a result, any proposed shape functions should match this expression in the lowest order case.
Even though this was not the case for the previous attempts shown in \eqref{eq:PyrFailedAttempts}, a close observation reveals that a linear combination of those attempts \textit{does} lead to the desired lowest order shape function. 
That is,\footnote{In fact, the case $\frac{1}{2}(V_{00}^\Tri(\nu_0(x,z),\lambda_2(x,y,z),\lambda_5(z))+V_{00}^\Tri(\lambda_1(x,y,z),\nu_1(x,z),\lambda_5(z)))$ also gives the desired lowest order shape function. However, the higher order version of this expression unfortunately is not in the space $\mathcal{U}^{(2),p}$.}
\begin{equation}
\begin{aligned}
	\textstyle{\frac{1}{2}}\Big(\mu_0(y)V_{00}^\Tri(\vec{\nu}_{012}(x,z))
		+\textstyle{\frac{1}{\mu_0(y)}}V_{00}^\Tri(\vec{\lambda}_{125}(x,y,z))\Big)
			&=\textstyle{\frac{1}{2(1+z)^3}}\bigg(\begin{smallmatrix}0\\[2pt]-(1-y)\\[2pt]0\end{smallmatrix}\bigg)
				+\textstyle{\frac{1}{2(1+z)^3}}\bigg(\begin{smallmatrix}0\\[2pt]-(1-y)\\[2pt]z\end{smallmatrix}\bigg)\\
	&=\textstyle{\frac{1}{(1+z)^3}}\bigg(\begin{smallmatrix}0\\[2pt]-(1-y)\\[2pt]\frac{1}{2}z\end{smallmatrix}\bigg)
		=V_{00}^\mathrm{f}(x,y,z)\,.
\end{aligned}
\end{equation}
Immediately, this suggests the higher order expression for the shape functions,
\begin{equation}
\begin{aligned}
	V_{ij}^\mathrm{f}(x,y,z)&\!=\!\textstyle{\frac{1}{2}}\Big(\mu_0(y)V_{ij}^\Tri(\vec{\nu}_{012}(x,z))
		+\textstyle{\frac{1}{\mu_0(y)}}V_{ij}^\Tri(\vec{\lambda}_{125}(x,y,z))\Big)\,,\\
	\nabla\!\cdot\! V_{ij}^\mathrm{f}(x,y,z)&\!=\!\textstyle{\frac{1}{2}}\Big(\nabla\mu_0(y)\cdot V_{ij}^\Tri(\vec{\nu}_{012}(x,z))
    	+\textstyle{\frac{1}{\mu_0(y)}}\nabla\!\cdot\!V_{ij}^\Tri(\vec{\lambda}_{125}(x,y,z))\\
    		&\qquad\qquad\qquad\qquad\qquad\qquad\quad
    			-\textstyle{\frac{1}{\mu_0(y)^2}}\nabla\mu_0(y)\cdot V_{ij}^\Tri(\vec{\lambda}_{125}(x,y,z))\Big)\,,
\end{aligned}
\label{eq:PyrHdivTriaFaceAppen}
\end{equation}
where $i\geq0$, $j\geq0$ and $n=i+j=0,\ldots,p-1$.
Here, it was used that $\nabla\cdot V_{ij}^\Tri(\vec{\nu}_{012}(x,z))=0$ by \eqref{eq:HdivtriangleRemark}.
Clearly, no derivatives of the $L^2$ representatives are present in the expression for the divergence, since they are not present neither in the terms involving $V_{ij}^\Tri$ directly, nor in the terms involving $\nabla\cdot V_{ij}^\Tri$ in view of Lemma \ref{lem:divformula}.
It remains to show the trace properties, and that the $V_{ij}^\mathrm{f}$ are in the underlying space $\mathcal{V}_\infty^{(2),n+1}$, where $n=i+j$.

In fact, each of the two components of the shape function satisfies the trace properties.
To see this, note that $V_{ij}^\Tri(\vec{\nu}_{012}(x,z))$ has entries independent of $y$, meaning that only the second component is nonzero.
This component is tangent to faces 235, 145 and 1234, so that the normal component vanishes on those three faces.
Meanwhile in the opposite face 435, $\mu_0(y)$ vanishes, so the shape function also vanishes, while at the face itself $\mu_0(y)=1$, and the shape function (in the pyramid $(\xi,\eta,\zeta)$ coordinates) takes the form of a 2D triangle $L^2$ face function, as desired.
Therefore, the component $\mu_0(y)V_{ij}^\Tri(\vec{\nu}_{012}(x,z))$ satisfies the trace properties.
For the other component, notice that
\begin{equation*}
	\textstyle{\frac{1}{\mu_0(y)}}V_{ij}^\Tri(\vec{\lambda}_{125}(x,y,z))
		\!=\![P_i,P_j^\alpha](\vec{\lambda}_{125}(x,y,z))\Big(\mu_0(y)V_{00}^\Tri(\vec{\nu}_{012}(x,z))
			+\lambda_5(z)\nabla\mu_0(y)\times E_{0}^\E(\vec{\nu}_{01}(x,z))\Big)\,.
\end{equation*}
Hence, the second component is further decoupled into two terms, with the first one already satisfying the desired properties by the previous analysis. 
The second term vanishes at all faces as required, because $\nabla\mu_0(y)$ is normal to faces 125 and 435, while $E_{0}^\E(\vec{\nu}_{01}(x,z))$ is normal to the faces 235 and 145, meaning that $\nabla\mu_0(y)\times E_{0}^\E(\vec{\nu}_{01}(x,z))$ is tangent to all the triangle faces, so its normal trace vanishes. 
Finally at the quadrilateral face it vanishes because $\lambda_5(z)=0$ there.
It follows $\mu_0(y)^{-1}V_{ij}^\Tri(\vec{\lambda}_{125}(x,y,z))$ satisfies the trace properties as well, and so $V_{ij}^\mathrm{f}\in\Gamma^{(2),n+1}$ for $n=i+j$.

To prove the shape functions lie in the underlying space, simply recall that $[P_i,P_j^\alpha]\in\tilde{\mathcal{P}}^{i+j}(s_0,s_1,s_2)$ is a homogeneous polynomial, so it suffices to assume it is a monomial $[P_i,P_j^\alpha](s_0,s_1,s_2)=s_0^as_1^bs_2^c$ of order $n=i+j=a+b+c$.
Calculating explicitly for face 125 gives,
\begin{equation}
	\begin{aligned}
		V_{ij}^\mathrm{f}(x,y,z)&=\textstyle{\frac{1}{2(1+z)^{n+3}}}
			\bigg(\begin{smallmatrix}0\\[2pt]-(1-y)\cdot(1-x)^ax^b(1+(1-y)^{a+b})z^c\\[2pt]
				z\cdot(1-x)^ax^b(1-y)^{a+b}z^c\end{smallmatrix}\bigg)
					\in\Bigg(\begin{smallmatrix}\mathcal{Q}_{n+3}^{n+1,n,n}\\[2pt]\mathcal{Q}_{n+3}^{n,n+1,n}\\[2pt]
						\mathcal{Q}_{n+3}^{n,n,n+1}\end{smallmatrix}\Bigg)\,,\\
		\nabla\!\cdot\!V_{ij}^\mathrm{f}(x,y,z)&=\displaystyle{\frac{z^c(1-x)^ax^b((n+2-z)(1-y)^{a+b}+(1+z))}{2(1+z)^{n+4}}}
			\in\mathcal{Q}_{n+4}^{n,n,?}\,.
	\end{aligned}
\end{equation}
In the expression for the divergence, the power of $z$ is either $c$ or $c+1$.
If $a+b=0$, then $c=n$ and $\nabla\!\cdot\!V_{ij}^\mathrm{f}(x,y,z)=\frac{(n+3)z^n}{2(1+z)^{n+4}}\in\mathcal{Q}_{n+4}^{n,n,n}$.
If $a+b\geq1$, then $c+1\leq n$ and again $\nabla\!\cdot\!V_{ij}^\mathrm{f}(x,y,z)\in\mathcal{Q}_{n+4}^{n,n,n}$.
Therefore in all cases, $V_{ij}^\mathrm{f}(x,y,z)\in\mathcal{V}_\infty^{(2),n+1}$, where $n=i+j$, and as a result $V_{ij}^\mathrm{f}\in\mathcal{U}^{(2),n+1}$.
Naturally, an analogous result holds for all other triangle faces.

This part concludes with the observation that the expressions in \eqref{eq:PyrHdivTriaFaceAppen} have factors $\frac{1}{\mu_0(y)}$ and $\frac{1}{\mu_0(y)^2}$, which appear to be singularities (on face 435).
However, they are not real singularities, since
\begin{equation}
\begin{aligned}
	\textstyle{\frac{1}{\mu}}V_{ij}^\Tri(\mu s_0,\mu s_1, s_2)
		&\!=\![P_i,P_j^\alpha](\mu s_0,\mu s_1, s_2)(\mu V_{00}^\Tri(s_0,s_1,s_2)\!+\!s_2\nabla\mu\!\times\! E_0^\E(s_0,s_1))\,,\\
	\nabla\!\cdot\!\Big(\!\textstyle{\frac{1}{\mu}}V_{ij}^\Tri(\mu s_0,\mu s_1, s_2)\!\Big)\!
		&\!=\![P_i,P_j^\alpha](\mu s_0,\mu s_1, s_2)\nabla\mu\!\cdot\!\!
			\Big(\!(i\!+\!j\!+\!3)E_0^\E(s_0,s_1)\!\!\times\!\!\nabla s_2\!-\!V_{00}^\Tri(s_0,s_1,s_2)\!\Big)\,\!,\!\!\!\!
\end{aligned}
\label{eq:SingularityAvoidNoOri}
\end{equation}
for $i\geq0$, $j\geq0$ and $n=i+j=0,\ldots,p-1$.
Here, in the expression for the divergence the term $\mu(\nabla s_0\times\nabla s_1)\cdot\nabla s_2$ is not shown, because it is assumed that $s_0+s_1+s_2=1$, since this is the property $\vec{\nu}_{012}$ satisfies.
This expression can be very useful from a computational standpoint, because the benign singularity can be a problem when computing the shape function at interior points very close to face 125.
Indeed, if orientations are not being taken into account, the expression above is sufficient to avoid those problematic terms.
Nevertheless if orientations are considered, the solution is more technical.
We present an approach that is convenient from a computational point of view.
First define,
\begin{equation}
	\kappa(\oo)=\begin{cases}0\quad&\text{if }\oo=0,1,2\,,\\1\quad&\text{if }\oo=3,4,5\,.\end{cases}
\end{equation}
Then define the following function,
\begin{equation}
\begin{aligned}
	\tilde{V}_{ij}^\Tri(s_0,s_1,s_2,\mu,\oo_0,\oo)&\!=\![P_i,P_j^\alpha](\sigma_\oo^\Tri(\sigma_{\oo_0}^\Tri(\mu s_0,\mu s_1,s_2)))
		\Big(\mu V_{00}^\Tri(\sigma_\oo^\Tri(\sigma_{\oo_0}^\Tri(s_0,s_1,s_2)))\\
			&\qquad\qquad\qquad\qquad\qquad\qquad
				+s_2\nabla\mu\!\times\! E_0^\E(\sigma_{\kappa(\oo)}^\E(\sigma_{\kappa(\oo_0)}^\E(s_0,s_1)))\Big)\,,\\
	\nabla\!\cdot\!\tilde{V}_{ij}^\Tri(s_0,s_1,s_2,\mu,\oo_0,\oo)
		&\!=\![P_i,P_j^\alpha](\sigma_\oo^\Tri(\sigma_{\oo_0}^\Tri(\mu s_0,\mu s_1,s_2)))\nabla\mu\\
			&\qquad\qquad\qquad\qquad
				\cdot\Big(\!(i\!+\!j\!+\!3)E_0^\E(\sigma_{\kappa(\oo)}^\E(\sigma_{\kappa(\oo_0)}^\E(s_0,s_1)))\!\times\!\nabla s_2\\
					&\qquad\qquad\qquad\qquad\qquad\qquad\qquad\qquad
						-V_{00}^\Tri(\sigma_\oo^\Tri(\sigma_{\oo_0}^\Tri(s_0,s_1,s_2)))\!\Big)\,,
\end{aligned}
\end{equation}
for $i\geq0$, $j\geq0$ and $n=i+j=0,\ldots,p-1$.
Then, the \textit{orientation embedded} shape functions for each face (see \eqref{eq:PyrHdivTriaFace}) are
\begin{equation}
	\begin{aligned}
		V_{ij}^\mathrm{f}(\xieta)
			&\!=\!\textstyle{\frac{1}{2}}\Big(\mu_c(\sxib)V_{ij}^\Tri(\sigma_\oo^\Tri(\sigma_{\oo_0}^\Tri(\vec{\nu}_{012}(\xi_a,\zeta))))
				+\tilde{V}_{ij}^\Tri(\vec{\nu}_{012}(\xi_a,\zeta),\mu_c(\sxib),\oo_0,\oo)\Big)\,,\\
    \nabla\!\cdot\! V_{ij}^\mathrm{f}(\xieta)
    	&\!=\!\textstyle{\frac{1}{2}}\Big(\nabla\mu_c(\sxib)
    		\!\cdot\! V_{ij}^\Tri(\sigma_\oo^\Tri(\sigma_{\oo_0}^\Tri(\vec{\nu}_{012}(\xi_a,\zeta))))
    	\!+\!\nabla\!\cdot\!\tilde{V}_{ij}^\Tri(\vec{\nu}_{012}(\xi_a,\zeta),\mu_c(\sxib),\oo_0,\oo)\Big)\,,	
	\end{aligned}
	\label{eq:SingularityAvoidYesOri}
\end{equation}
where $i\geq0$, $j\geq0$, $n=i+j=0,\ldots,p-1$, $(a,b)=(1,2),(2,1)$, $c=0,1$, and where $\oo_0$ is chosen such that $\sigma_{\oo_0}^\Tri(\vec{\nu}_{012}(\xi_a,\zeta))$ is the locally oriented triplet representing that face.
This is necessary because $s_2$ is treated differently than $(s_0,s_1)$ in the definition of $\tilde{V}_{ij}^\Tri$.

\subsubsection{\texorpdfstring{$H(\mathrm{div})$}{Hdiv} Interior Bubbles}

\subparagraph{Families I and II:}
These families are presented in \eqref{eq:PyrHdivInteriorI} and \eqref{eq:PyrHdivInteriorII} and are the curl of $H(\mathrm{curl})$ interior bubbles.
They are identified as $V_{ijk}^\mathrm{b}$, with $i=0,\ldots,p-1$, $j=2,\ldots,p$, and $k=2,\ldots,p$.
They satisfy the vanishing trace properties and belong to $\Gamma^{(2),n}$, for $n=\max\{i+1,j,k\}$.
Using \eqref{eq:PyrHcurlInteriorIIProof} it follows
\begin{equation}
\begin{aligned}
	V_{ijk}^\mathrm{b}(x,y,z)&=
		\underbrace{\nabla\!\times\!\Big(\textstyle{\frac{1}{1+z}}\phi_k^\E(\vec{\mu}_{01}(\szzz))
			E_{ij}^\square(\vec{\mu}_{01}(x),\vec{\mu}_{01}(y))\Big)}_{
				\in\mathcal{Q}_{n+2}^{n,n-1,n-1}\!\times\!\mathcal{Q}_{n+2}^{n-1,n,n-1}\!\times\!\mathcal{Q}_{n+1}^{n-1,n-1,n-1}}
						\in\Bigg(\begin{smallmatrix}\mathcal{Q}_{n+2}^{n,n-1,n-1}\\[2pt]\mathcal{Q}_{n+2}^{n-1,n,n-1}\\[2pt]
							\mathcal{Q}_{n+2}^{n-1,n-1,n}\end{smallmatrix}\Bigg)\,,\\
	\nabla\!\cdot\! V_{ij}^\mathrm{b}(x,y,z)&\!=\!0\in\mathcal{Q}_{n+3}^{n-1,n-1,n-1}\,.
\end{aligned}
\end{equation}
The calculations are invariant to permutations of the entries, so it holds for both families.
Therefore, the functions of both families are elements of $\mathcal{U}^{(2),n}$, where $n=\max\{i+1,j,k\}$.

\subparagraph{Family III:}
This family is defined in \eqref{eq:PyrHdivInteriorIII} and it is the curl of $H(\mathrm{curl})$ interior bubbles.
The functions are labeled as $V_{ij}^{\mathrm{b}}$, for $i=2,\ldots,p$, and $j=2,\ldots,p$.
They have vanishing trace, so they lie in $\Gamma^{(2),n}$ for $n=\max\{i,j\}$.
Moreover, using \eqref{eq:PyrHcurlInteriorIVProof} gives
\begin{equation}
\begin{aligned}
	V_{ij}^\mathrm{b}(x,y,z)&=
		\nabla\!\times\!\Big(\phi_{ij}^\square(\vec{\mu}_{01}(x),\vec{\mu}_{01}(y))\nabla\textstyle{\frac{1}{(1+z)^n}}\Big)
			\in\Bigg(\begin{smallmatrix}\mathcal{Q}_{n+1}^{i,j-1,0}\\[2pt]\mathcal{Q}_{n+1}^{i-1,j,0}\\[2pt]0\end{smallmatrix}\Bigg)
				\subseteq\Bigg(\begin{smallmatrix}\mathcal{Q}_{n+2}^{n,n-1,n-1}\\[2pt]\mathcal{Q}_{n+2}^{n-1,n,n-1}\\[2pt]
					\mathcal{Q}_{n+2}^{n-1,n-1,n}\end{smallmatrix}\Bigg)\,,\\
	\nabla\!\cdot\! V_{ijk}^\mathrm{b}(x,y,z)&=0\in\mathcal{Q}_{n+3}^{n-1,n-1,n-1}\,.
\end{aligned}
\end{equation}
Hence, the functions are in $\mathcal{U}^{(2),n}$, for $n=\max\{i,j\}$.

\subparagraph{Family IV:}
This family is shown in \eqref{eq:PyrHdivInteriorIV} and the functions are identified as $V_{ijk}^\mathrm{b}$, for $i=0,\ldots,p-1$, $j=0,\ldots,p-1$, and $k=2,\ldots,p$.
They satisfy the vanishing trace properties and lie in the compatibility space $\Gamma^{(2),n}$, where $n=\max\{i+1,j+1,k\}$.
Using \eqref{eq:PyrHdivQuadFaceProof} and \eqref{eq:PyrphikZSpace} it follows
\begin{equation}
\begin{aligned}
	V_{ijk}^\mathrm{b}(x,y,z)&\!=\!
		\underbrace{\textstyle{\frac{1}{(1+z)^2}}\phi_k^\E(\vec{\mu}_{01}(\szzz))}_{B(z)\in\mathcal{Q}_{k+2}^{0,0,k-1}}
			\underbrace{V_{ij}^\square(\vec{\mu}_{01}(x),\vec{\mu}_{01}(y))}_{A(x,y)\in\{0\}\!\times\!\{0\}\!\times\!\mathcal{Q}^{i,j,0}}
			\!\in\!\!\Bigg(\!\begin{smallmatrix}0\\[2pt]0\\[2pt]\mathcal{Q}_{n+2}^{n-1,n-1,n-1}\end{smallmatrix}\!\Bigg)
				\!\!\subseteq\!\!\Bigg(\begin{smallmatrix}\mathcal{Q}_{n+2}^{n,n-1,n-1}\\[2pt]\mathcal{Q}_{n+2}^{n-1,n,n-1}\\[2pt]
					\mathcal{Q}_{n+2}^{n-1,n-1,n}\end{smallmatrix}\Bigg)\,,\\
	\nabla\!\cdot\! V_{ijk}^\mathrm{b}(x,y,z)&\!=\!
		\underbrace{\nabla B(z)}_{
			\in\{0\}\!\times\!\{0\}\!\times\!\mathcal{Q}_{k+3}^{0,0,k-1}}
				\!\!\!\!\cdot\,\,\, A(x,y)\in\mathcal{Q}_{k+3}^{i,j,k-1}
					\subseteq\mathcal{Q}_{n+3}^{n-1,n-1,n-1}\,.
\end{aligned}
\end{equation}
Therefore, the functions belong to $\mathcal{U}^{(2),n}$, for $n=\max\{i+1,j+1,k\}$.

\subparagraph{Family V:}
This family is defined in \eqref{eq:PyrHdivInteriorV}, with the functions being labeled as $V_{ij}^\mathrm{b}$, for $i=2,\ldots,p$, and $j=2,\ldots,p$.
They vanish at the pyramid boundary and belong to the compatibility space $\Gamma^{(2),n}$, where $n=\max\{i,j\}$.
Using the expression in \eqref{eq:PyrHdivBubblesAuxI}, it follows
\begin{equation}
\begin{aligned}
	V_{ij}^\mathrm{b}(x,y,z)&\!=\!
		\underbrace{(\textstyle{\frac{z}{1+z}})^{n-1}}_{\in\mathcal{Q}_{n-1}^{0,0,n-1}}
			\underbrace{V_{ij}^\trianglelefteq(\vec{\mu}_{01}(x),\vec{\mu}_{01}(y),\mu_0(\szzz))}_{A(x,y)
				\in\mathcal{Q}_3^{i,j-1,0}\!\times\!\mathcal{Q}_3^{i-1,j,0}\!\times\!\mathcal{Q}_2^{i-1,j-1,0}}
						\!\in\!\Bigg(\begin{smallmatrix}\mathcal{Q}_{n+2}^{n,n-1,n-1}\\[2pt]\mathcal{Q}_{n+2}^{n-1,n,n-1}\\[2pt]
							\mathcal{Q}_{n+2}^{n-1,n-1,n}\end{smallmatrix}\Bigg)\,,\\
	\nabla\!\cdot\! V_{ij}^\mathrm{b}(x,y,z)&\!=\!
		\underbrace{\nabla (\textstyle{\frac{z}{1+z}})^{n-1}}_{
			\in\{0\}\!\times\!\{0\}\!\times\!\mathcal{Q}_{n}^{0,0,n-2}}
				\!\!\!\!\cdot\,\,\, A(x,y)\in\mathcal{Q}_{n+2}^{i-1,j-1,n-2}
					\subseteq\mathcal{Q}_{n+3}^{n-1,n-1,n-1}\,.
\end{aligned}
\end{equation}
As a result, the functions are in $\mathcal{U}^{(2),n}$, for $n=\max\{i,j\}$.

\subparagraph{Families VI and VII:}
These families are presented in \eqref{eq:PyrHdivInteriorVI} and \eqref{eq:PyrHdivInteriorVII}, with the functions being labeled as $V_{i}^\mathrm{b}$ and $V_{j}^\mathrm{b}$ respectively, where $i=2,\ldots,p$, and $j=2,\ldots,p$.
They satisfy the trace properties and therefore are in the space $\Gamma^{(2),i}$ and $\Gamma^{(2),j}$ respectively.
Using the expression in \eqref{eq:PyrHdivBubblesAuxII}, it follows
\begin{equation}
\begin{aligned}
	V_{i}^\mathrm{b}(x,y,z)&\!=\!
		\underbrace{(\textstyle{\frac{z}{1+z}})^{i-1}}_{\in\mathcal{Q}_{i-1}^{0,0,i-1}}
			\underbrace{V_i^\trianglerighteq(\vec{\mu}_{01}(x),\mu_1(y),\mu_0(\szzz))}_{A(x,y)
				\in\mathcal{Q}_3^{i,0,0}\!\times\!\{0\}\!\times\!\mathcal{Q}_2^{i-1,0,0}}
						\!\in\!\Bigg(\begin{smallmatrix}\mathcal{Q}_{i+2}^{i,i-1,i-1}\\[2pt]\mathcal{Q}_{i+2}^{i-1,i,i-1}\\[2pt]
							\mathcal{Q}_{i+2}^{i-1,i-1,i}\end{smallmatrix}\Bigg)\,,\\
	\nabla\!\cdot\! V_{i}^\mathrm{b}(x,y,z)&\!=\!
		\underbrace{\nabla (\textstyle{\frac{z}{1+z}})^{i-1}}_{
			\in\{0\}\!\times\!\{0\}\!\times\!\mathcal{Q}_{i}^{0,0,i-2}}
				\!\!\!\!\cdot\,\,\, A(x,y)\in\mathcal{Q}_{i+2}^{i-1,0,i-2}
					\subseteq\mathcal{Q}_{i+3}^{i-1,i-1,i-1}\,.
\end{aligned}
\end{equation}
Symmetric arguments apply to the seventh family $V_j^\mathrm{b}(x,y,z)$, meaning that the functions are in $\mathcal{U}^{(2),i}$ and $\mathcal{U}^{(2),j}$ respectively.

\subsection{\texorpdfstring{$L^2$}{L2} Shape Functions}

Recall the deformed underlying space for $m=3$ is
\begin{equation*}
	\mathcal{V}_\infty^{(3),p}=\mathcal{Q}_{p+3}^{p-1,p-1,p-1}\,.
\end{equation*}

\subsubsection{\texorpdfstring{$L^2$}{L2} Interior}
The interior shape functions are defined in \eqref{eq:PyrL2Interior} and identified as $\psi_{ijk}^\mathrm{b}$ for $i=0,\ldots,p-1$, $j=0,\ldots,p-1$ and $k=0,\ldots,p-1$. 
They trivially satisfy the compatibility properties, since there are none to satisfy, meaning they lie in $\Gamma^{(3),n}$, for $n=\max\{i+1,j+1,k+1\}$.
In this case the factor making the expression coordinate free is $(\nabla\nu_1(x,z)\!\!\times\!\!\nabla\nu_1(y,z))\!\cdot\!\nabla\mu_1(\szzz)=\frac{1}{(1+z)^4}$.
It follows,
\begin{equation}
	\psi_{ijk}^\mathrm{b}(x,y,z)
		=\textstyle{\frac{1}{(1+z)^4}}\underbrace{[P_i](\vec{\mu}_{01}(x))[P_j](\vec{\mu}_{01}(y))[P_k](\vec{\mu}_{01}(\szzz))}_{
			\in\mathcal{Q}_{k}^{i,j,k}}\in\mathcal{Q}_{k+4}^{i,j,k}\subseteq\mathcal{Q}_{n+3}^{n-1,n-1,n-1}\,.
\end{equation}
Hence, the interior functions lie in $\mathcal{U}^{(3),n}$ for $n=\max\{i+1,j+1,k+1\}$.

%% file: integration.tex
\section{Integration}
\label{app:Integration}

\subsection{Coordinate Changes}

\begin{figure}[!ht]
\begin{center}
\includegraphics[scale=0.45]{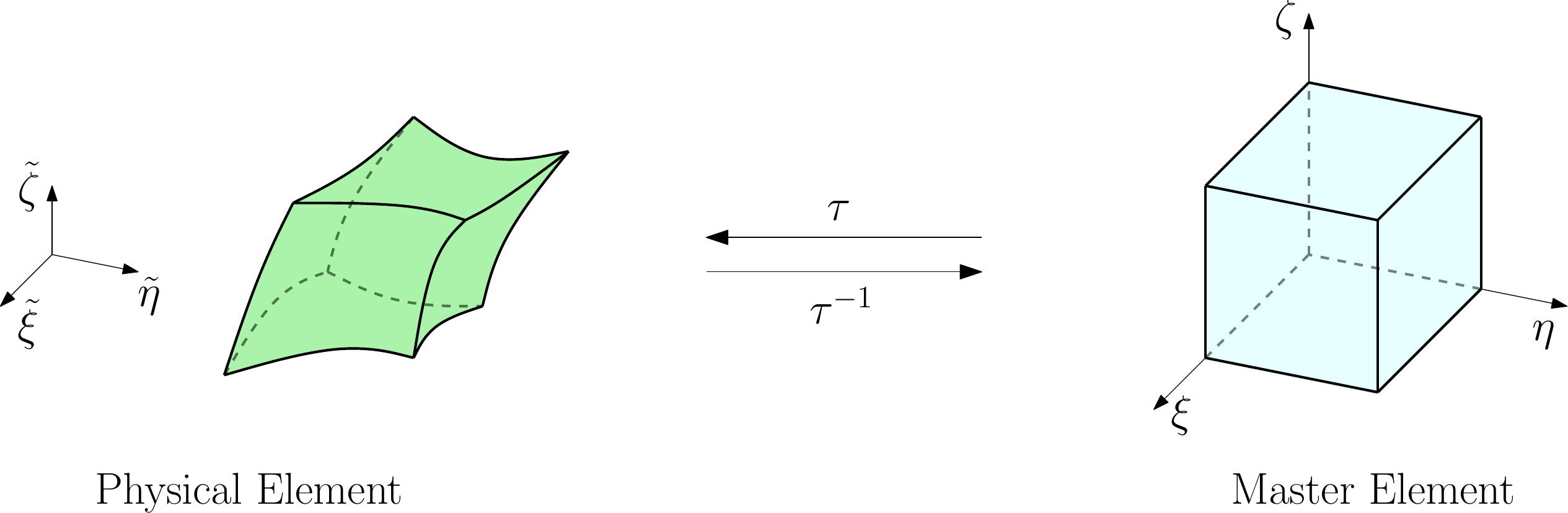}
\caption{Transformation from the master element to the physical element.}\label{fig:curvilineartransform}
\end{center}
\end{figure}

At its core, the finite element method advocates carrying out integration over a master element domain instead of the original physical element.
It makes the method very feasible from a computational standpoint.
This involves a change of variables $\tau:\Omega\rightarrow\tilde{\Omega}$, from the master element domain $\Omega$ to the physical domain $\tilde{\Omega}$, which is assumed to be known.
This is illustrated in Figure \ref{fig:curvilineartransform}.
Note the change of variables $\tau$ is in general a nonlinear mapping.

Indeed, consider a ``physical'' integrand $f$ which is a function of variables in the different energy spaces and their differential form.
These variables are in the physical system of coordinates.
For instance, take $\phi_{\tilde{\Omega}}$, $E_{\tilde{\Omega}}$, $V_{\tilde{\Omega}}$ and $\psi_{\tilde{\Omega}}$ to represent variables in $H^1$, $H(\mathrm{curl})$, $H(\mathrm{div})$ and $L^2$ respectively.
Their corresponding differentials are $\nabla_{\tilde{\Omega}}\phi_{\tilde{\Omega}}$, $\nabla_{\tilde{\Omega}}\!\times\!E_{\tilde{\Omega}}$ and $\nabla_{\tilde{\Omega}}\!\cdot\!V_{\tilde{\Omega}}$.
However, it is their pullbacks to the master element domain, denoted with the subscript $\Omega$, which are known, since the shape functions are defined in the master element domain.\footnote{Unless the affine coordinates and their gradient are written in the physical system of coordinates, in which case one can simply substitute them in the expressions for the shape functions. This is due to the coordinate free nature of the shape functions.}
Making use of the appropriate pullback mapping for each of the variables as written in \eqref{eq:pullbacksgeneral}, this yields\footnote{In 2D, $\nabla_{\Omega}\!\times\!E_{\Omega}$ is in $L^2$, so the correct expression in the last line would be $\frac{1}{\det(J_\tau)}\nabla_{\Omega}\!\times\!E_{\Omega}$ instead of the 3D expression $\frac{1}{\det(J_\tau)}J_\tau\nabla_{\Omega}\!\times\!E_{\Omega}$. In 1D, $E_{\Omega}$ and $V_{\Omega}$ do not even exist, so they would be ignored throughout.}
\begin{equation}
	\begin{aligned}
		\mathcal{I}_{\tilde{\Omega}}&\!=\!\int_{\tilde{\Omega}}\!
			f\Big(\phi_{\tilde{\Omega}},\nabla_{\tilde{\Omega}}\phi_{\tilde{\Omega}},
				E_{\tilde{\Omega}},\nabla_{\tilde{\Omega}}\!\times\!E_{\tilde{\Omega}}, 
					V_{\tilde{\Omega}},\nabla_{\tilde{\Omega}}\!\cdot\!V_{\tilde{\Omega}}, \psi_{\tilde{\Omega}}\Big)\mathrm{d}\tilde{\Omega}\\
			&\!=\!\int_{\Omega}\!f\Big(\phi_{\tilde{\Omega}}\circ\tau, (\nabla_{\tilde{\Omega}}\phi_{\tilde{\Omega}})\circ\tau,
				E_{\tilde{\Omega}}\circ\tau,(\nabla_{\tilde{\Omega}}\!\times\!E_{\tilde{\Omega}})\circ\tau, 
					V_{\tilde{\Omega}}\circ\tau,(\nabla_{\tilde{\Omega}}\!\cdot\!V_{\tilde{\Omega}})\circ\tau,
						\psi_{\tilde{\Omega}}\circ\tau\Big)\!\det(J_\tau)\mathrm{d}\Omega\\
			&\!=\!\int_{\Omega}\!f\Big(\phi_{\Omega},J_\tau^{-\T}\nabla_{\Omega}\phi_{\Omega},
				J_\tau^{-\T}E_{\Omega},\textstyle{\frac{1}{\det(J_\tau)}}J_\tau\nabla_{\Omega}\!\times\!E_{\Omega},\\
					&\qquad\qquad\qquad\qquad\qquad\qquad\qquad\qquad\quad
						\textstyle{\frac{1}{\det(J_\tau)}}J_\tau V_{\Omega},\textstyle{\frac{1}{\det(J_\tau)}}\nabla_{\Omega}\!\cdot\!V_{\Omega},
							\textstyle{\frac{1}{\det(J_\tau)}}\psi_{\Omega}\Big)\!\det(J_\tau)\mathrm{d}\Omega\,,
	\end{aligned}
	\label{eq:integraluptomaster}
\end{equation}
where $J_\tau$ is the Jacobian matrix of the transformation $\tau:\Omega\rightarrow\tilde{\Omega}$.

Now, the integration is at least in a well known master element domain.
However, this is still a ``difficult'' domain over which to integrate (with the exception of the 1D segment, the 2D quadrilateral and the 3D hexahedron).
Hence, it is desirable to make one further change of coordinates to a ``nicer'' integration domain.
This is denoted by $\tau_Q:Q\rightarrow\Omega$, where $Q=(0,1)$ in 1D, $Q=(0,1)^2$ in 2D and $Q=(0,1)^3$ in 3D.
The transformations for each element are nicely depicted in Figure \ref{fig:integrationtransforms}. 
Some readers may have a strong preference for the integration domains $\tilde{Q}=(-1,1)$ in 1D, $\tilde{Q}=(-1,1)^2$ in 2D and $\tilde{Q}=(-1,1)^3$ in 3D.
If that is the case, then simply make the substitutions $x=\frac{\tilde{x}+1}{2}$, $y=\frac{\tilde{y}+1}{2}$ and $z=\frac{\tilde{z}+1}{2}$ in those expressions shown in Figure \ref{fig:integrationtransforms}.


\begin{figure}[!ht]
\begin{center}
\includegraphics[scale=0.45]{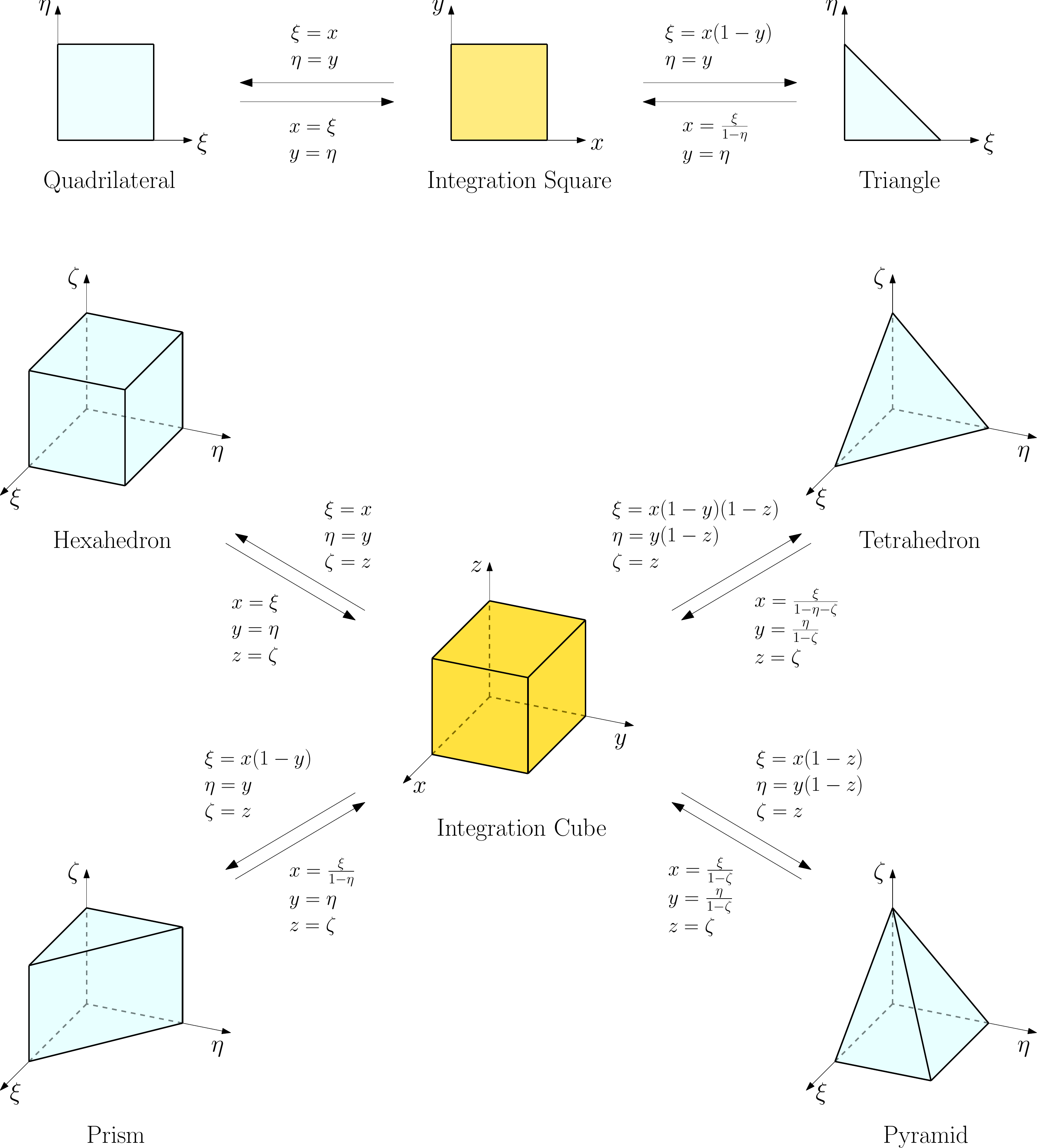}
\caption{Transformations from each element to a nice integration domain.}\label{fig:integrationtransforms}
\end{center}
\end{figure}


The original integral in \eqref{eq:integraluptomaster} finally becomes
\begin{equation}
	\begin{aligned}
		\mathcal{I}_{\tilde{\Omega}}
			&\!=\!\int_{Q}\!f\Big(\phi_{\Omega}\circ\tau_Q,(J_\tau^{-\T}\nabla_{\Omega}\phi_{\Omega})\circ\tau_Q,
				(J_\tau^{-\T}E_{\Omega})\circ\tau_Q,(\textstyle{\frac{1}{\det(J_\tau)}}J_\tau\nabla_{\Omega}\!\times\!E_{\Omega})\circ\tau_Q,\\
					&\qquad\qquad
						(\textstyle{\frac{1}{\det(J_\tau)}}J_\tau V_{\Omega})\circ\tau_Q,
							(\textstyle{\frac{1}{\det(J_\tau)}}\nabla_{\Omega}\!\cdot\!V_{\Omega})\circ\tau_Q,
								(\textstyle{\frac{1}{\det(J_\tau)}}\psi_{\Omega})\circ\tau_Q\Big)\!\det(J_{\tau_Q})\det(J_\tau)\mathrm{d}Q\,,
	\end{aligned}
	\label{eq:integraluptoQ}
\end{equation}
where $J_{\tau_Q}$ is the Jacobian matrix of the transformation $\tau_Q:Q\rightarrow\Omega$.

\subsection{Fast integration}

To actually calculate the integral, the typical approach is to use Gaussian quadrature.
In $Q$ (or $\tilde{Q}$) the quadrature points and weights are well known and taken from the literature, and this is part of the reason why integration over the physical domain was reduced to integration over $Q$ in \eqref{eq:integraluptoQ}.

However, as the number of spatial dimensions $N$ increases from 1D to 3D, the cost grows quickly with $p$.
Indeed, to construct a typical finite element stiffness matrix, integrals usually reduce to the form
\begin{equation*}
	\mathcal{I}_{\tilde{\Omega}}=\int_{Q}u_Iv_J\mathrm{d}Q\,,
\end{equation*}
where $I=1,\ldots,\mathcal{O}(p^N)$ and $J=1,\ldots,\mathcal{O}(p^N)$.
The cost to integrate each term is $\mathcal{O}(p^N)$ as well, because there are $p+1$ quadrature points in each spatial dimension.
Hence, with a straightforward implementation, the cost to integrate all terms is $\mathcal{O}(p^{3N})$, so it is $\mathcal{O}(p^3)$ in 1D, $\mathcal{O}(p^6)$ in 2D and $\mathcal{O}(p^9)$ in 3D.
This constitutes a problem for $N\geq2$ and high $p$, so it is highly desirable to improve the integration cost.

Fortunately, the integration cost can be reduced if there exists a decoupling of either $u_I$ or $v_J$ in $x$, $y$ and $z$ (the variables after transforming to $Q$, not necessarily the variables of the physical or master element domain). 
Assume the decoupling is in $v_J$, where it takes the form $v_J(x,y)=v_{j_x}^x(x,y)v_{j_y}^y(y)$ in 2D and $v_J(x,y,z)=v_{j_x}^x(x,y,z)v_{j_y}^y(y,z)v_{j_z}^z(z)$ in 3D, with $j_x,j_y,j_z=1,\ldots,\mathcal{O}(p)$. 
Then, by reorganizing the operations and storing some coefficients, the cost is reduced to $\mathcal{O}(p^5)$ in 2D, and to $\mathcal{O}(p^7)$ in 3D.
Some of the details are in \citet{hpbook2}.
With the shape functions presented in this text, regardless of the element shape and the the associated topological entity, such a decoupling is to be expected, so this acceleration to $\mathcal{O}(p^{2N+1})$ is possible.
This technique based on a tensor product decoupling is typically called \textit{fast quadrature}.

Naturally, there are other fast integration techniques different from the fast quadrature described above which might also be applicable, but further research is required.

%% file: verification.tex
\section{Verification}
\label{app:verification}


One of the most important tests is to numerically confirm the polynomial approximability properties of the spaces spanned by the shape functions.
Coupled with the exact sequence property of the discrete spaces, this ensures all well known interpolation inequalities.
More specifically, for an affinely transformed master element mesh, let $W^p$ be the span of the $H^1$ basis functions of order $p$ (being composed piecewise by shape functions), and similarly with $Q^p$, $V^p$ and $Y^p$ for the spaces $H(\mathrm{curl})$, $H(\mathrm{div})$ and $L^2$ respectively.
Then, one has to check that $\mathcal{P}^p\subseteq W^p$, $(\mathcal{P}^{p-1})^3\subseteq Q^p$, $(\mathcal{P}^{p-1})^3\subseteq V^p$ and $\mathcal{P}^{p-1}\subseteq Y^p$.

To do this, first consider an arbitrary $u$ in a given energy space $U$ approximated by a discrete space $U_h$.
Clearly,
\begin{equation}
	u\in U_h \Leftrightarrow \mathrm{dist}(u,U_h)=\min_{u_h\in U_h}||u-u_h||_U=0\,.
\end{equation}
Hence, given $u\in U$ the task is to compute $\mathrm{dist}(u,U_h)=||u-u_h||_U$, with $u_h\in U_h$ being the element where the minimum is attained.
Fortunately $u_h$ is computed from a variational problem equivalent to the projection (distance) problem.
It is,
\begin{equation}
	||u-u_h||_U=\mathrm{dist}(u,U_h)\Leftrightarrow
		\begin{cases}
			u_h\in U_h\,,&{}\\
			b(u_h,v_h)=\langle u_h,v_h \rangle_U=\langle u,v_h \rangle_U=\ell_u(v_h)& \forall v_h\in U_h.
		\end{cases}
	\label{eq:variationalprojection}
\end{equation}
Naturally, the inner product is different depending on the energy space $U$.
They are,
\begin{equation}
  \begin{aligned}
	  \langle \phi_1,\phi_2\rangle_{H^1}&=\int_\Omega (\phi_1\phi_2+\nabla\phi_1\cdot\nabla\phi_2)\mathrm{d}\Omega\,,\\
	  \langle E_1,E_2\rangle_{H(\mathrm{curl})}&=\int_\Omega (E_1\cdot E_2+(\nabla\times E_1)\cdot(\nabla\times E_2))\mathrm{d}\Omega\,,\\
		\langle V_1,V_2\rangle_{H(\mathrm{div})}&=\int_\Omega (V_1\cdot V_2+(\nabla\cdot V_1)(\nabla\cdot V_2))\mathrm{d}\Omega\,,\\
		\langle \psi_1,\psi_2\rangle_{L^2}&=\int_\Omega \psi_1\psi_2\mathrm{d}\Omega\,.
	\end{aligned}
\end{equation}

Then, the task is to determine whether each element $u$ of a monomial basis for the polynomial spaces in question lies in $U_h$.
This is achieved by solving the variational problem in \eqref{eq:variationalprojection} and checking that the relative error $\frac{||u-u_h||_U}{||u||_U}=\frac{\mathrm{dist}(u,U_h)}{||u||_U}$ is in the range of machine zero.
Thus, for example to ensure that $\mathcal{P}^p\subseteq W^p$ one must check that all monomials of the form $x^iy^jz^k$ for $i+j+k\leq p$ lie in $W^p$, the span of the $H^1$ basis functions.
Similar procedures hold for $H(\mathrm{curl})$, $H(\mathrm{div})$ and $L^2$.
These tests are called \textit{polynomial reproducibility} tests.

The polynomial reproducibility tests are successful when using the code associated with this work.
The tests are done on a series of meshes, including a four element hybrid mesh with one element of each type, as depicted in Figure \ref{fig:VerificationMesh}.
By doing this on a mesh, there is the additional value of implicitly verifying compatibility of the shape functions across the boundaries of the elements.
Indeed, it should be checked that the polynomial reproducibililty tests pass under all possible orientations of each face and edge in the mesh, which is the case for our code.

\begin{figure}[!ht]
\begin{center}
\includegraphics[scale=0.5]{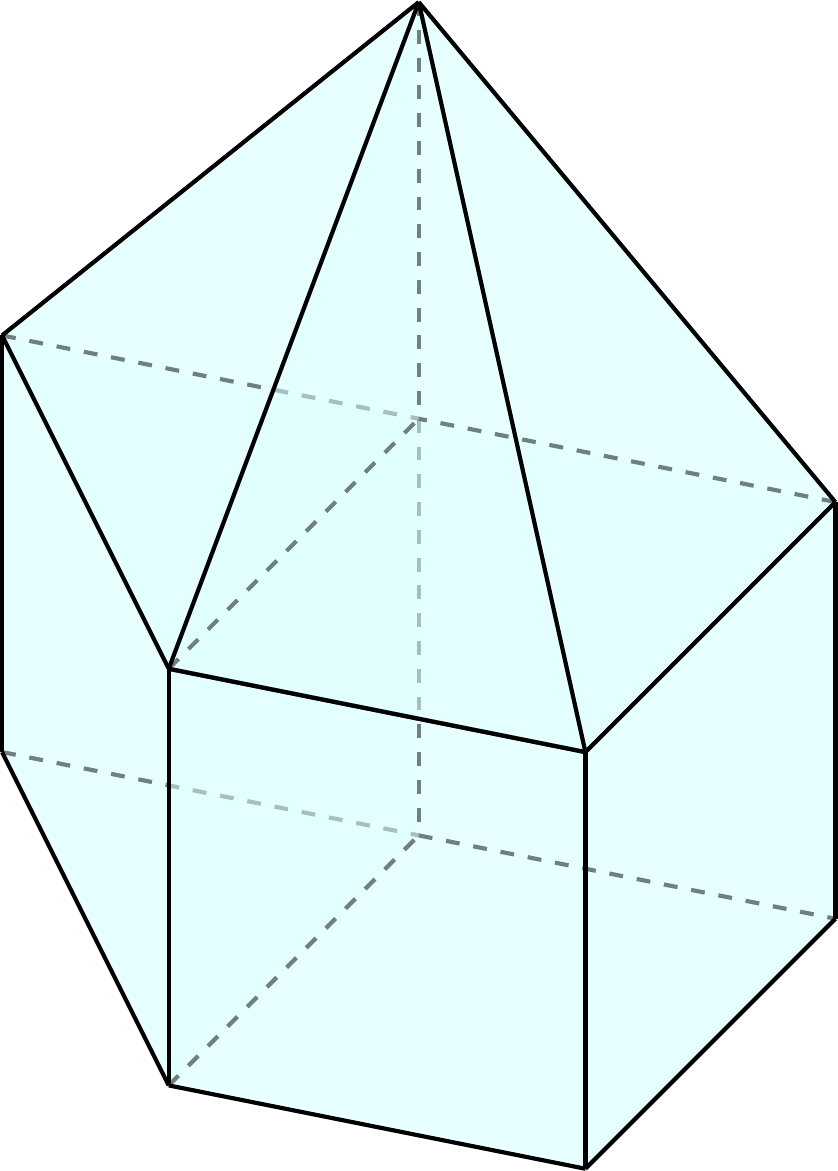}
\caption{A hybrid mesh used to verify polynomial reproducibility.}
\label{fig:VerificationMesh}
\end{center}
\end{figure}


Another convenient test is to verify some aspects of the exact sequence property of the discrete spaces.
For this, consider a fixed element and the discrete spaces $W^p$, $Q^p$, $V^p$ and $Y^p$ conforming to $H^1$, $H(\mathrm{curl})$, $H(\mathrm{div})$ and $L^2$ respectively.
The discrete spaces are precisely the span of the corresponding shape functions.

Hence, for example consider an $H^1$ shape function $\phi\in W^p$.
Then the idea is to confirm numerically that $\nabla\phi\in Q^p$.
This is done as described in the polynomial reproducibility tests, where the computed projection of $\nabla\phi$ to $Q^p$ is $E_h$, which is given as a linear combination of the $H(\mathrm{curl})$ shape functions $E\in Q^p$.
Therefore, one should obtain that that $\frac{||\nabla\phi-E_h||_{H(\mathrm{curl})}}{||\nabla\phi||_{H(\mathrm{curl})}}$ is in the range of machine zero.
Moreover, one can additionally check that the coefficients of the linear combination for $E_h$ make sense.
For instance, if $\phi\in W^p$ is originally an $H^1$ interior bubble, then $\nabla\phi$ is also an $H(\mathrm{curl})$ interior bubble and as a result is in the span of the $H(\mathrm{curl})$ interior shape functions (meaning the coefficients associated to $H(\mathrm{curl})$ edge and face shape functions are zero).
Similarly, if $\phi\in W^p$ is an $H^1$ face shape function, then $\nabla\phi$ is in the span of the the $H(\mathrm{curl})$ face functions associated to the same face \textit{and} the $H(\mathrm{curl})$ interior bubbles.
Naturally this applies to other topological entities and to the different energy spaces.

These verifications are successful when using the code that supplements this text.

%% file: tables.tex
\section{Tables}
\label{app:ShapeFunctionTable}

\subsection{Polynomials}

\renewcommand{\arraystretch}{1.35}
\begin{center}
\begin{tabular}
{|C{0.35cm} L{11.9cm} C{2.9cm}|N}
\hline
\multicolumn{3}{|c|}{\large \textbf{Polynomials}} &\\
\hline
\multicolumn{3}{|c|}{Legendre} &\\
\hline
{}& {} &{} &\\[-23pt]
\multicolumn{3}{|l|}{Shifted and scaled Legendre polynomials for $t\geq0$ and $x\in[0,t]$:} &\\
\multicolumn{3}{|l|}{
$\begin{aligned}
	P_0(x;t)&=1\\
	P_1(x;t)&=2x-t\\
	iP_i(x;t)&=(2i-1)(2x-t)P_{i-1}(x;t)-(i-1)t^2P_{i-2}(x;t) \quad \text{for }\, i\geq2
\end{aligned}$} &\\
\multicolumn{3}{|l|}{Integrated Legendre polynomials:} &\\
\multicolumn{3}{|l|}{
$	\begin{aligned}
		2(2i-1)L_i(x; t)&=P_i(x;t)-t^2P_{i-2}(x;t) \quad\text{for }\, i\geq2
	\end{aligned}$} &\\
\multicolumn{3}{|l|}{Scaling differential of integrated Legendre polynomials:} &\\
\multicolumn{3}{|l|}{
$	\begin{aligned}
		\textstyle{\frac{\partial}{\partial t}}L_{i+1}(x;t)
			&=R_i(x;t)=-\textstyle{\frac{1}{2}}(P_i(x;t)+tP_{i-1}(x;t))\quad\text{for }\,i\geq1
	\end{aligned}$} &\\
\multicolumn{3}{|l|}{Homogenized polynomials:} &\\
\multicolumn{3}{|l|}{
$	\begin{aligned}{}
		[P_i](s_0,s_1)&=P_i(s_1;s_0+s_1)\quad\text{for }\, i\geq0\\
		[R_i](s_0,s_1)&=R_i(s_1;s_0+s_1)\quad\text{for }\, i\geq1\\
		[L_i](s_0,s_1)&=L_i(s_1;s_0+s_1)\quad\text{for }\, i\geq2\\
		\nabla[L_i](s_0,s_1)&=[P_{i-1}](s_0,s_1)\nabla s_1+[R_{i-1}](s_0,s_1)\nabla(s_0+s_1)
			\quad\text{for }\, i\geq2
	\end{aligned}$} &\\	
\multicolumn{3}{|l|}{$\,\,$ where $s_0$ and $s_1$ are functions of some spatial variable in $\R^N$, so $\nabla s_0,\nabla s_1\in\R^N$, for $N=1,2,3$.} &\\[-17pt]
{}& {} &{} &\\[-5pt]
\multicolumn{3}{r}{\footnotesize\textit{continued on next page}} &\\[-18.5pt]
\hline
\end{tabular}
\end{center}
\renewcommand{\arraystretch}{1}

\newpage

\renewcommand{\arraystretch}{1.35}
\begin{center}
\begin{tabular}
{|C{0.35cm} L{11.9cm} C{2.9cm}|N}
\multicolumn{3}{l}{\footnotesize\textit{continued from previous page}} &\\[-2pt]
\hline
\multicolumn{3}{|c|}{\large \textbf{Polynomials}} &\\
\hline
\multicolumn{3}{|c|}{Jacobi} &\\
\hline
{}& {} &{} &\\[-23pt]
\multicolumn{3}{|l|}{Shifted and scaled Jacobi polynomials for $\alpha>-1$, $t\geq0$ and $x\in[0,t]$:} &\\
\multicolumn{3}{|l|}{
$	\begin{aligned}
		P_0^\alpha(x;t) &= 1 \\
		P_1^\alpha(x;t) &= 2x-t+\alpha x\\
		a_j{P}_j^\alpha(x;t)&=b_j(c_j(2x-t)+\alpha^2 t)P_{j-1}^\alpha(x;t)-d_jt^2 P_{j-2}^\alpha(x;t)
		\quad \text{for }\, j\geq2
	\end{aligned}$} &\\
\multicolumn{3}{|l|}{$\,\,$ where} &\\
\multicolumn{3}{|l|}{
$	\begin{aligned}{}
		\qquad a_j &=  2j(j+\alpha)\, (2j + \alpha - 2)\\
		\qquad b_j &=  2j + \alpha - 1\\
		\qquad c_j &=  (2j+\alpha)\, (2j + \alpha - 2)\\
		\qquad d_j &=  2(j+\alpha-1)\, (j-1)\, (2j+\alpha)
	\end{aligned}$} &\\
\hline
\end{tabular}
\end{center}
\renewcommand{\arraystretch}{1}

\newpage
\subsection{Ancillary Operators}

\renewcommand{\arraystretch}{1.35}
\begin{center}
\begin{tabular}
{|C{0.35cm} L{11.9cm}|C{2.9cm}|N}
\hline
\multicolumn{3}{|c|}{\large \textbf{Ancillary Operators}} &\\
\hline
\multicolumn{3}{|c|}{\small In all cases, $s_0$, $s_1$, $s_2$, $t_0$ and $t_1$ are functions of some spatial variable in $\R^N$.} &\\
\hline
\multicolumn{3}{|c|}{$H^1$} &\\
\hline
{}& {} &{} &\\[-23pt]
\multicolumn{2}{|c|}{Operator}& Indices &\\ \hline
\multicolumn{2}{|l|}{Edge} &
\multirow{2}{*}[-12pt]{\small$\begin{gathered}
N\!=\!1,2,3\\
i\!=\!2,\ldots,p_s
\end{gathered}$} &\\
\multicolumn{2}{|l|}{
$	\begin{aligned}
		\phi^\E_i(s_0,s_1)&=[L_i](s_0,s_1)\\
		\nabla\phi^\E_i(s_0,s_1)&=\nabla[L_i](s_0,s_1)
	\end{aligned}$} &{}&\\
\multicolumn{2}{|l|}{Quadrilateral Face} &
\multirow{2}{*}[-5pt]{\small$\begin{gathered}
N\!=\!2,3\\
i\!=\!2,\ldots,p_s\\
j\!=\!2,\ldots,p_t
\end{gathered}$} &\\
\hline
\end{tabular}
\end{center}
\renewcommand{\arraystretch}{1}

\newpage

\renewcommand{\arraystretch}{1.35}
\begin{center}
\begin{tabular}
{|C{0.35cm} L{11.9cm}|C{2.9cm}|N}
\multicolumn{3}{l}{\footnotesize\textit{continued from previous page}} &\\[-2pt]
\hline
\multicolumn{3}{|c|}{\large \textbf{Ancillary Operators}} &\\
\hline
\multicolumn{3}{|c|}{$H(\mathrm{div})$} &\\
\hline
{}& {} &{} &\\[-23pt]
\multicolumn{2}{|c|}{Operator}& Indices &\\ \hline
\multicolumn{2}{|l|}{Quadrilateral Face$^\dagger$} &
\multirow{2}{*}[-12pt]{\small$\begin{gathered}
N\!=\!3\\
i\!=\!0,\ldots,p_s\!-\!1\\
j\!=\!0,\ldots,p_t\!-\!1
\end{gathered}$} &\\
\multicolumn{2}{|l|}{
$	\begin{aligned}
		V_{ij}^{\square}(s_0,s_1,t_0,t_1)&=E_i^\E(s_0,s_1)\!\times\! E_j^\E(t_0,t_1)\\
		\nabla\!\cdot\! V_{ij}^{\square}(s_0,s_1,t_0,t_1)&=E_j^\E(t_0,t_1)\cdot(\nabla\!\times\! E_i^\E(s_0,s_1))\\
			&\qquad\qquad\qquad\qquad-E_i^\E(s_0,s_1)\cdot(\nabla\!\times\! E_j^\E(t_0,t_1))
	\end{aligned}$} &{}&\\
\multicolumn{2}{|l|}{Triangle Face$^\ddagger$} &
\multirow{2}{*}[-12pt]{\small$\begin{gathered}
N\!=\!3\\
i\geq0,\,j\geq0\\
i\!+\!j\!=\!0,\ldots,p_s\!-\!1
\end{gathered}$} &\\
\hline
\end{tabular}
\end{center}
\renewcommand{\arraystretch}{1}

\newpage
\subsection{Segment}

\renewcommand{\arraystretch}{1.35}
\begin{center}
\begin{tabular}
{|C{0.35cm} L{11.9cm}|C{2.9cm}|N}
\hline
\multicolumn{3}{|c|}{\large \textbf{Segment}} &\\
\hline
\multicolumn{3}{|c|}{Geometry} &\\
\hline
\multicolumn{3}{|c|}{} &\\[-17pt]
\multicolumn{3}{|c|}{\begin{minipage}{0.9\textwidth}
	\begin{center}
		\includegraphics[scale=0.5]{./figures/MasterSegment.pdf}
	\end{center}
\end{minipage}} &\\
\multicolumn{3}{|c|}{Affine Coordinates} &\\
\multicolumn{3}{|c|}{
$	\begin{alignedat}{4}
		\mu_0&=1-\xi\,\qquad\qquad &&\phantom{\nabla}\mu_1&&=\xi\\
		\nabla\mu_0&=-1\,\qquad\qquad &&\nabla\mu_1&&=1
	\end{alignedat}$} &\\
\multicolumn{3}{|c|}{\small\fbox{The order in the pair $\vec{\mu}_{01}=(\mu_0,\mu_1)$ is $p$.}} &\\
\multicolumn{3}{|c|}{} &\\[-17pt]
\hline
\multicolumn{3}{|c|}{$H^1$} &\\
\hline
{}& {} &{} &\\[-23pt]
\multicolumn{2}{|c|}{Shape Functions}& Indices &\\ 
\hline
\multicolumn{2}{|l|}{Vertices} &
\multirow{2}{*}[-12pt]{\small$\begin{gathered}
	a\!=\!0,1
\end{gathered}$} &\\
\multicolumn{2}{|l|}{
$	\begin{aligned}
		\phi^\mathrm{v}&=\mu_a\\
		\nabla\phi^\mathrm{v}&=\nabla\mu_a
	\end{aligned}$} &{}&\\
\multicolumn{2}{|l|}{Edges} &
\multirow{2}{*}[-12pt]{\small$\begin{gathered}
	i\!=\!2,\ldots,p
\end{gathered}$} &\\
\multicolumn{2}{|l|}{
$	\begin{aligned}
		\phi_i^\mathrm{e}&=\phi_i^\E(\vec{\mu}_{01})\,\\
		\nabla\phi_i^\mathrm{e}&=\nabla\phi_i^\E(\vec{\mu}_{01})
	\end{aligned}$} &{}&\\[-17pt]
{}& {} &{} &\\
\hline
\multicolumn{3}{|c|}{$L^2$} &\\
\hline
{}& {} &{} &\\[-23pt]
\multicolumn{2}{|c|}{Shape Functions}& Indices &\\ 
\hline
\multicolumn{2}{|l|}{Edges} &
\multirow{2}{*}[-12pt]{\small$\begin{gathered}
	i\!=\!0,\ldots,p\!-\!1
\end{gathered}$} &\\
\multicolumn{2}{|l|}{
$	\begin{aligned}
		\psi^\mathrm{e}_i=[P_i](\vec{\mu}_{01})\nabla\mu_1
	\end{aligned}$} &{}&\\[-17pt]
{}& {} &{} &\\
\hline
\end{tabular}
\end{center}
\renewcommand{\arraystretch}{1}

\newpage
\subsection{Quadrilateral}

\renewcommand{\arraystretch}{1.35}
\begin{center}
\begin{tabular}
{|C{0.35cm} L{11.9cm}|C{2.9cm}|N}
\hline
\multicolumn{3}{|c|}{\large \textbf{Quadrilateral}} &\\
\hline
\multicolumn{3}{|c|}{Geometry} &\\
\hline
\multicolumn{3}{|c|}{} &\\[-17pt]
\multicolumn{3}{|c|}{\begin{minipage}{0.9\textwidth}
	\begin{center}
		\includegraphics[scale=0.5]{./figures/MasterQuad.pdf}
	\end{center}
\end{minipage}} &\\
\multicolumn{3}{|c|}{Affine Coordinates} &\\
\multicolumn{3}{|c|}{
$	\begin{alignedat}{4}
		\mu_0^{\,\xi_1}&=1-\xi_1\,\qquad\qquad &&\phantom{\nabla}\mu_1^{\,\xi_1}&&=\xi_1\\
		\nabla\mu_0^{\,\xi_1}&=\Big(\begin{smallmatrix}-1\\[2pt]0\end{smallmatrix}\Big)\qquad\qquad 
			&&\nabla\mu_1^{\,\xi_1}&&=\Big(\begin{smallmatrix}1\\[2pt]0\end{smallmatrix}\Big)\\
		\mu_0^{\,\xi_2}&=1-\xi_2\,\qquad\qquad &&\phantom{\nabla}\mu_1^{\,\xi_2}&&=\xi_2\\
		\nabla\mu_0^{\,\xi_2}&=\Big(\begin{smallmatrix}0\\[2pt]-1\end{smallmatrix}\Big)\qquad\qquad 
			&&\nabla\mu_1^{\,\xi_2}&&=\Big(\begin{smallmatrix}0\\[2pt]1\end{smallmatrix}\Big)
	\end{alignedat}$} &\\
\multicolumn{3}{|c|}{} &\\[-17pt]
\multicolumn{3}{|c|}{\small\fbox{\parbox[c][25pt][c]{0.4\textwidth}{The order in the pair $\vec{\mu}_{01}^{\,\xi_1}=(\mu_0^{\,\xi_1},\mu_1^{\,\xi_1})$ is $p_1$.\\The order in the pair $\vec{\mu}_{01}^{\,\xi_2}=(\mu_0^{\,\xi_2},\mu_1^{\,\xi_2})$ is $p_2$.}}} &\\
\hline
\end{tabular}
\end{center}
\renewcommand{\arraystretch}{1}

\newpage

\renewcommand{\arraystretch}{1.35}
\begin{center}
\begin{tabular}
{|C{0.35cm} L{11.9cm}|C{2.9cm}|N}
\multicolumn{3}{l}{\footnotesize\textit{continued from previous page}} &\\[-2pt]
\hline
\multicolumn{3}{|c|}{\large \textbf{Quadrilateral}} &\\
\hline
\multicolumn{3}{|c|}{$H(\mathrm{curl})^\star$} &\\
\hline
{}& {} &{} &\\[-23pt]
\multicolumn{2}{|c|}{Shape Functions}& Indices &\\ 
\hline
\multicolumn{2}{|l|}{Edges} &
\multirow{2}{*}[-12pt]{\small$\begin{gathered}
	i\!=\!0,\ldots,p_a\!-\!1\\
	(a,b)\!=\!(1,2),(2,1)\\
	c\!=\!0,1
\end{gathered}$} &\\
\multicolumn{2}{|l|}{
$	\begin{aligned}
		E_i^\mathrm{e}&=\mu_c^{\,\xi_b}E_i^\E(\vec{\mu}_{01}^{\,\xi_a})\\
		\nabla\!\times\! E_i^\mathrm{e}&=\nabla\mu_c^{\,\xi_b}\!\times\! E_i^\E(\vec{\mu}_{01}^{\,\xi_a})
	\end{aligned}$} &{}&\\
\multicolumn{2}{|l|}{Face} & {}&\\
\hline
\end{tabular}
\end{center}
\renewcommand{\arraystretch}{1}

\newpage
\subsection{Triangle}

\renewcommand{\arraystretch}{1.35}
\begin{center}
\begin{tabular}
{|C{0.35cm} L{11.9cm}|C{2.9cm}|N}
\hline
\multicolumn{3}{|c|}{\large \textbf{Triangle}} &\\
\hline
\multicolumn{3}{|c|}{Geometry} &\\
\hline
\multicolumn{3}{|c|}{} &\\[-17pt]
\multicolumn{3}{|c|}{\begin{minipage}{0.9\textwidth}
	\begin{center}
		\includegraphics[scale=0.5]{./figures/MasterTri.pdf}
	\end{center}
\end{minipage}} &\\
\multicolumn{3}{|c|}{Affine Coordinates} &\\
\multicolumn{3}{|c|}{
$	\begin{alignedat}{6}
		\nu_0&=1-x_1-x_2\,\qquad\qquad &&\phantom{\nabla}\nu_1&&=x_1\qquad\qquad &&\phantom{\nabla}\nu_2&&=x_2\\
		\nabla\nu_0&=\Big(\begin{smallmatrix}-1\\[2pt]-1\end{smallmatrix}\Big)\qquad\qquad 
			&&\nabla\nu_1&&=\Big(\begin{smallmatrix}1\\[2pt]0\end{smallmatrix}\Big)\qquad\qquad 
				&&\nabla\nu_2&&=\Big(\begin{smallmatrix}0\\[2pt]1\end{smallmatrix}\Big)
	\end{alignedat}$} &\\
\multicolumn{3}{|c|}{} &\\[-17pt]
\multicolumn{3}{|c|}{\small\fbox{The order in the triplet $\vec{\nu}_{012}=(\nu_0,\nu_1,\nu_2)$ and all its subpairs is $p$.}} &\\
\hline
\end{tabular}
\end{center}
\renewcommand{\arraystretch}{1}

\newpage

\renewcommand{\arraystretch}{1.35}
\begin{center}
\begin{tabular}
{|C{0.35cm} L{11.9cm}|C{2.9cm}|N}
\multicolumn{3}{l}{\footnotesize\textit{continued from previous page}} &\\[-2pt]
\hline
\multicolumn{3}{|c|}{\large \textbf{Triangle}} &\\
\hline
\multicolumn{3}{|c|}{$H(\mathrm{curl})^\star$} &\\
\hline
{}& {} &{} &\\[-23pt]
\multicolumn{2}{|c|}{Shape Functions}& Indices &\\ 
\hline
\multicolumn{2}{|l|}{Edges} &
\multirow{2}{*}[-12pt]{\small$\begin{gathered}
	i\!=\!0,\ldots,p\!-\!1\\
	0\leq a<b\leq2
\end{gathered}$} &\\
\multicolumn{2}{|l|}{
$	\begin{aligned}
		E_i^\mathrm{e}&=E_i^\E(\vec{\nu}_{ab})\\
		\nabla\!\times\! E_i^\mathrm{e}&=\nabla\!\times\! E_i^\E(\vec{\nu}_{ab})
	\end{aligned}$} &{}&\\
\multicolumn{2}{|l|}{Face} & {}&\\
\hline
\end{tabular}
\end{center}
\renewcommand{\arraystretch}{1}

\newpage
\subsection{Hexahedron}

\renewcommand{\arraystretch}{1.35}
\begin{center}
\begin{tabular}
{|C{0.35cm} L{11.9cm} C{2.9cm}|N}
\hline
\multicolumn{3}{|c|}{\large \textbf{Hexahedron}} &\\
\hline
\multicolumn{3}{|c|}{Geometry} &\\
\hline
\multicolumn{3}{|c|}{} &\\[-10pt]
\multicolumn{3}{|c|}{\begin{minipage}{0.9\textwidth}
	\begin{center}
		\includegraphics[scale=0.5]{./figures/MasterHexa.pdf}
	\end{center}
\end{minipage}} &\\
\multicolumn{3}{|c|}{Affine Coordinates} &\\
\multicolumn{3}{|c|}{
$	\begin{alignedat}{4}
		\mu_0^{\,\xi_1}&=1-\xi_1\,\qquad\qquad &&\phantom{\nabla}\mu_1^{\,\xi_1}&&=\xi_1\\
		\nabla\mu_0^{\,\xi_1}&=\bigg(\begin{smallmatrix}-1\\[2pt]0\\[2pt]0\end{smallmatrix}\bigg)\qquad\qquad 
			&&\nabla\mu_1^{\,\xi_1}&&=\bigg(\begin{smallmatrix}1\\[2pt]0\\[2pt]0\end{smallmatrix}\bigg)\\
		\mu_0^{\,\xi_2}&=1-\xi_2\,\qquad\qquad &&\phantom{\nabla}\mu_1^{\,\xi_2}&&=\xi_2\\
		\nabla\mu_0^{\,\xi_2}&=\bigg(\begin{smallmatrix}0\\[2pt]-1\\[2pt]0\end{smallmatrix}\bigg)\qquad\qquad 
			&&\nabla\mu_1^{\,\xi_2}&&=\bigg(\begin{smallmatrix}0\\[2pt]1\\[2pt]0\end{smallmatrix}\bigg)\\
		\mu_0^{\,\xi_3}&=1-\xi_3\,\qquad\qquad &&\phantom{\nabla}\mu_1^{\,\xi_3}&&=\xi_3\\
		\nabla\mu_0^{\,\xi_3}&=\bigg(\begin{smallmatrix}0\\[2pt]0\\[2pt]-1\end{smallmatrix}\bigg)\qquad\qquad 
			&&\nabla\mu_1^{\,\xi_3}&&=\bigg(\begin{smallmatrix}0\\[2pt]0\\[2pt]1\end{smallmatrix}\bigg)
	\end{alignedat}$} &\\
\multicolumn{3}{|c|}{} &\\[-7pt]
\multicolumn{3}{|c|}{\small\fbox{\parbox[c][40pt][c]{0.4\textwidth}{The order in the pair $\vec{\mu}_{01}^{\,\xi_1}=(\mu_0^{\,\xi_1},\mu_1^{\,\xi_1})$ is $p_1$.\\The order in the pair $\vec{\mu}_{01}^{\,\xi_2}=(\mu_0^{\,\xi_2},\mu_1^{\,\xi_2})$ is $p_2$.\\The order in the pair $\vec{\mu}_{01}^{\,\xi_3}=(\mu_0^{\,\xi_3},\mu_1^{\,\xi_3})$ is $p_3$.}}} &\\
\multicolumn{3}{|c|}{} &\\[-17pt]
{}& {} &{} &\\[-5pt]
\multicolumn{3}{r}{\footnotesize\textit{continued on next page}} &\\[-18.5pt]
\hline
\end{tabular}
\end{center}
\renewcommand{\arraystretch}{1}

\newpage

\renewcommand{\arraystretch}{1.35}
\begin{center}
\begin{tabular}
{|C{0.35cm} L{11.9cm}|C{2.9cm}|N}
\multicolumn{3}{l}{\footnotesize\textit{continued from previous page}} &\\[-2pt]
\hline
\multicolumn{3}{|c|}{\large \textbf{Hexahedron}} &\\
\hline
\multicolumn{3}{|c|}{$H^1$} &\\
\hline
{}& {} &{} &\\[-23pt]
\multicolumn{2}{|c|}{Shape Functions}& Indices &\\ 
\hline
\multicolumn{2}{|l|}{Vertices} &
\multirow{2}{*}[-12pt]{\small$\begin{gathered}
	a\!=\!0,1,\,b\!=\!0,1\\
	c\!=\!0,1
\end{gathered}$} &\\
\multicolumn{2}{|l|}{
$	\begin{aligned}
		\phi^\mathrm{v}&=\mu_a^{\,\xi_1}\mu_b^{\,\xi_2}\mu_c^{\,\xi_3}\\
		\nabla\phi^\mathrm{v}&=\mu_a^{\,\xi_1}\mu_b^{\,\xi_2}\nabla\mu_c^{\,\xi_3}+\mu_c^{\,\xi_3}\mu_a^{\,\xi_1}\nabla\mu_b^{\,\xi_2}
			+\mu_b^{\,\xi_2}\mu_c^{\,\xi_3}\nabla\mu_a^{\,\xi_1}
	\end{aligned}$} &{}&\\
\multicolumn{2}{|l|}{Edges} &
\multirow{2}{*}[-12pt]{\small$\begin{gathered}
	i\!=\!2,\ldots,p_a\\[-2pt]
	(a,b,c)\!=\!(1,2,3),\\[-2pt]
	(2,3,1),(3,1,2)\\[-2pt]
	d\!=\!0,1,\,e\!=\!0,1,
\end{gathered}$} &\\
\hline
\end{tabular}
\end{center}
\renewcommand{\arraystretch}{1}

\newpage

\renewcommand{\arraystretch}{1.35}
\begin{center}
\begin{tabular}
{|C{0.35cm} L{11.9cm}|C{2.9cm}|N}
\multicolumn{3}{l}{\footnotesize\textit{continued from previous page}} &\\[-2pt]
\hline
\multicolumn{3}{|c|}{\large \textbf{Hexahedron}} &\\
\hline
\multicolumn{3}{|c|}{$H(\mathrm{curl})$} &\\
\hline
{}& {} &{} &\\[-23pt]
\multicolumn{2}{|c|}{Shape Functions}& Indices &\\ 
\hline
\multicolumn{2}{|l|}{Edges} &
\multirow{2}{*}[-12pt]{\small$\begin{gathered}
	i\!=\!0,\ldots,p_a\!-\!1\\[-2pt]
	(a,b,c)\!=\!(1,2,3),\\[-2pt]
	(2,3,1),(3,1,2)\\[-2pt]
	d\!=\!0,1,\,e\!=\!0,1,
\end{gathered}$} &\\
\multicolumn{2}{|l|}{
$	\begin{aligned}
		E_i^\mathrm{e}&=\mu_e^{\,\xi_c}\mu_d^{\,\xi_b}E_i^\E(\vec{\mu}_{01}^{\,\xi_a})\\
		\nabla\!\times\! E_i^\mathrm{e}&=\Big(\mu_e^{\,\xi_c}\nabla\mu_d^{\,\xi_b}+\mu_d^{\,\xi_b}\nabla\mu_e^{\,\xi_c}\Big)
			\!\times\! E_i^\E(\vec{\mu}_{01}^{\,\xi_a})
	\end{aligned}$} &{}&\\[25pt]
\multicolumn{2}{|l|}{Faces} & {}&\\
\hline
\end{tabular}
\end{center}
\renewcommand{\arraystretch}{1}

\newpage

\renewcommand{\arraystretch}{1.35}
\begin{center}
\begin{tabular}
{|C{0.35cm} L{11.9cm}|C{2.9cm}|N}
\multicolumn{3}{l}{\footnotesize\textit{continued from previous page}} &\\[-2pt]
\hline
\multicolumn{3}{|c|}{\large \textbf{Hexahedron}} &\\
\hline
\multicolumn{3}{|c|}{$H(\mathrm{div})$} &\\
\hline
{}& {} &{} &\\[-23pt]
\multicolumn{2}{|c|}{Shape Functions}& Indices &\\ 
\hline
\multicolumn{2}{|l|}{Faces} & 
\multirow{2}{*}[-12pt]{\small$\begin{gathered}
	i\!=\!0,\ldots,p_a\!-\!1\\[-2pt]
	j\!=\!0,\ldots,p_b\!-\!1\\[-2pt]
	(a,b,c)\!=\!(1,2,3),\\[-2pt]
	(2,3,1),(3,1,2)\\[-2pt]
	d\!=\!0,1
\end{gathered}$} &\\
\multicolumn{2}{|l|}{
$	\begin{aligned}
		V_{ij}^\mathrm{f}&=\mu_d^{\,\xi_c}V_{ij}^{\square}(\vec{\mu}_{01}^{\,\xi_a},\vec{\mu}_{01}^{\,\xi_b})\\
		\nabla\!\cdot\! V_{ij}^\mathrm{f}&=
			\nabla\mu_d^{\,\xi_c}\!\cdot\! V_{ij}^{\square}(\vec{\mu}_{01}^{\,\xi_a},\vec{\mu}_{01}^{\,\xi_b})
	\end{aligned}$} &{}&\\[35pt]
\multicolumn{2}{|l|}{Interior} & {}&\\
\hline
\end{tabular}
\end{center}
\renewcommand{\arraystretch}{1}

\newpage
\subsection{Tetrahedron}

\renewcommand{\arraystretch}{1.35}
\begin{center}
\begin{tabular}
{|C{0.35cm} L{11.9cm} C{2.9cm}|N}
\hline
\multicolumn{3}{|c|}{\large \textbf{Tetrahedron}} &\\
\hline
\multicolumn{3}{|c|}{Geometry} &\\
\hline
\multicolumn{3}{|c|}{} &\\[-10pt]
\multicolumn{3}{|c|}{\begin{minipage}{0.9\textwidth}
	\begin{center}
		\includegraphics[scale=0.5]{./figures/MasterTet.pdf}
	\end{center}
\end{minipage}} &\\
\multicolumn{3}{|c|}{Affine Coordinates} &\\
\multicolumn{3}{|c|}{
$	\begin{alignedat}{8}
		\lambda_0&=1-x_1-x_2-x_3\qquad\quad&&\phantom{\nabla}\lambda_1&&=x_1\qquad\quad
			&&\phantom{\nabla}\lambda_2&&=x_2\qquad\quad&&\phantom{\nabla}\lambda_3&&=x_3\\
		\nabla\lambda_0&=\bigg(\begin{smallmatrix}-1\\[2pt]-1\\[2pt]-1\end{smallmatrix}\bigg)\qquad\quad 
			&&\nabla\lambda_1&&=\bigg(\begin{smallmatrix}1\\[2pt]0\\[2pt]0\end{smallmatrix}\bigg)\qquad\quad 
				&&\nabla\lambda_2&&=\bigg(\begin{smallmatrix}0\\[2pt]1\\[2pt]0\end{smallmatrix}\bigg)\qquad\quad
					&&\nabla\lambda_3&&=\bigg(\begin{smallmatrix}0\\[2pt]0\\[2pt]1\end{smallmatrix}\bigg)
	\end{alignedat}$} &\\
\multicolumn{3}{|c|}{} &\\[-8pt]
\multicolumn{3}{|c|}{\small\fbox{The order in the quadruple $\vec{\lambda}_{0123}=(\lambda_0,\lambda_1,\lambda_2,\lambda_3)$ and all its subtuples is $p$.}} &\\[-5pt]
{}& {} &{} &\\[-5pt]
\multicolumn{3}{r}{\footnotesize\textit{continued on next page}} &\\[-18.5pt]
\hline
\end{tabular}
\end{center}
\renewcommand{\arraystretch}{1}

\newpage

\renewcommand{\arraystretch}{1.35}
\begin{center}
\begin{tabular}
{|C{0.35cm} L{11.9cm}|C{2.9cm}|N}
\multicolumn{3}{l}{\footnotesize\textit{continued from previous page}} &\\[-2pt]
\hline
\multicolumn{3}{|c|}{\large \textbf{Tetrahedron}} &\\
\hline
\multicolumn{3}{|c|}{$H^1$} &\\
\hline
{}& {} &{} &\\[-23pt]
\multicolumn{2}{|c|}{Shape Functions}& Indices &\\ 
\hline
\multicolumn{2}{|l|}{Vertices} &
\multirow{2}{*}[-12pt]{\small$\begin{gathered}
	a\!=\!0,1,2,3
\end{gathered}$} &\\
\multicolumn{2}{|l|}{
$	\begin{aligned}
		\phi^\mathrm{v}&=\lambda_a\\
		\nabla\phi^\mathrm{v}&=\nabla\lambda_a
	\end{aligned}$} &{}&\\
\multicolumn{2}{|l|}{Edges} &
\multirow{2}{*}[-12pt]{\small$\begin{gathered}
	i\!=\!2,\ldots,p\\
	0\leq a<b\leq3
\end{gathered}$} &\\
\hline
\end{tabular}
\end{center}
\renewcommand{\arraystretch}{1}

\newpage

\renewcommand{\arraystretch}{1.35}
\begin{center}
\begin{tabular}
{|C{0.35cm} L{11.9cm}|C{2.9cm}|N}
\multicolumn{3}{l}{\footnotesize\textit{continued from previous page}} &\\[-2pt]
\hline
\multicolumn{3}{|c|}{\large \textbf{Tetrahedron}} &\\
\hline
\multicolumn{3}{|c|}{$H(\mathrm{curl})$} &\\
\hline
{}& {} &{} &\\[-23pt]
\multicolumn{2}{|c|}{Shape Functions}& Indices &\\ 
\hline
\multicolumn{2}{|l|}{Edges} &
\multirow{2}{*}[-12pt]{\small$\begin{gathered}
	i\!=\!0,\ldots,p\!-\!1\\
	0\leq a<b\leq3
\end{gathered}$} &\\
\multicolumn{2}{|l|}{
$	\begin{aligned}
		E_i^\mathrm{e}&=E_i^\E(\vec{\lambda}_{ab})\\
		\nabla\!\times\! E_i^\mathrm{e}&=\nabla\times E_i^\E(\vec{\lambda}_{ab})
	\end{aligned}$} &{}&\\
\multicolumn{2}{|l|}{Faces} & {}&\\
\hline
\end{tabular}
\end{center}
\renewcommand{\arraystretch}{1}

\newpage

\renewcommand{\arraystretch}{1.35}
\begin{center}
\begin{tabular}
{|C{0.35cm} L{11.9cm}|C{2.9cm}|N}
\multicolumn{3}{l}{\footnotesize\textit{continued from previous page}} &\\[-2pt]
\hline
\multicolumn{3}{|c|}{\large \textbf{Tetrahedron}} &\\
\hline
\multicolumn{3}{|c|}{$H(\mathrm{div})$} &\\
\hline
{}& {} &{} &\\[-23pt]
\multicolumn{2}{|c|}{Shape Functions}& Indices &\\ 
\hline
\multicolumn{2}{|l|}{Faces} & 
\multirow{2}{*}[-12pt]{\small$\begin{gathered}
	i\geq0,\,j\geq0\\[-2pt]
	i\!+\!j\!=\!0,\ldots,p\!-\!1\\[-2pt]
	0\leq a<b<c\leq3
\end{gathered}$} &\\
\multicolumn{2}{|l|}{
$	\begin{aligned}
		V_{ij}^\mathrm{f}&=V_{ij}^\Tri(\vec{\lambda}_{abc})\\
		\nabla\!\cdot\! V_{ij}^\mathrm{f}&=\nabla\!\cdot\! V_{ij}^\Tri(\vec{\lambda}_{abc}(x))
	\end{aligned}$} &{}&\\
\multicolumn{2}{|l|}{Interior} & {}&\\
\hline
\end{tabular}
\end{center}
\renewcommand{\arraystretch}{1}

\newpage
\subsection{Prism}

\renewcommand{\arraystretch}{1.35}
\begin{center}
\begin{tabular}
{|C{0.35cm} L{11.9cm} C{2.9cm}|N}
\hline
\multicolumn{3}{|c|}{\large \textbf{Prism}} &\\
\hline
\multicolumn{3}{|c|}{Geometry} &\\
\hline
\multicolumn{3}{|c|}{} &\\[-10pt]
\multicolumn{3}{|c|}{\begin{minipage}{0.9\textwidth}
	\begin{center}
		\includegraphics[scale=0.5]{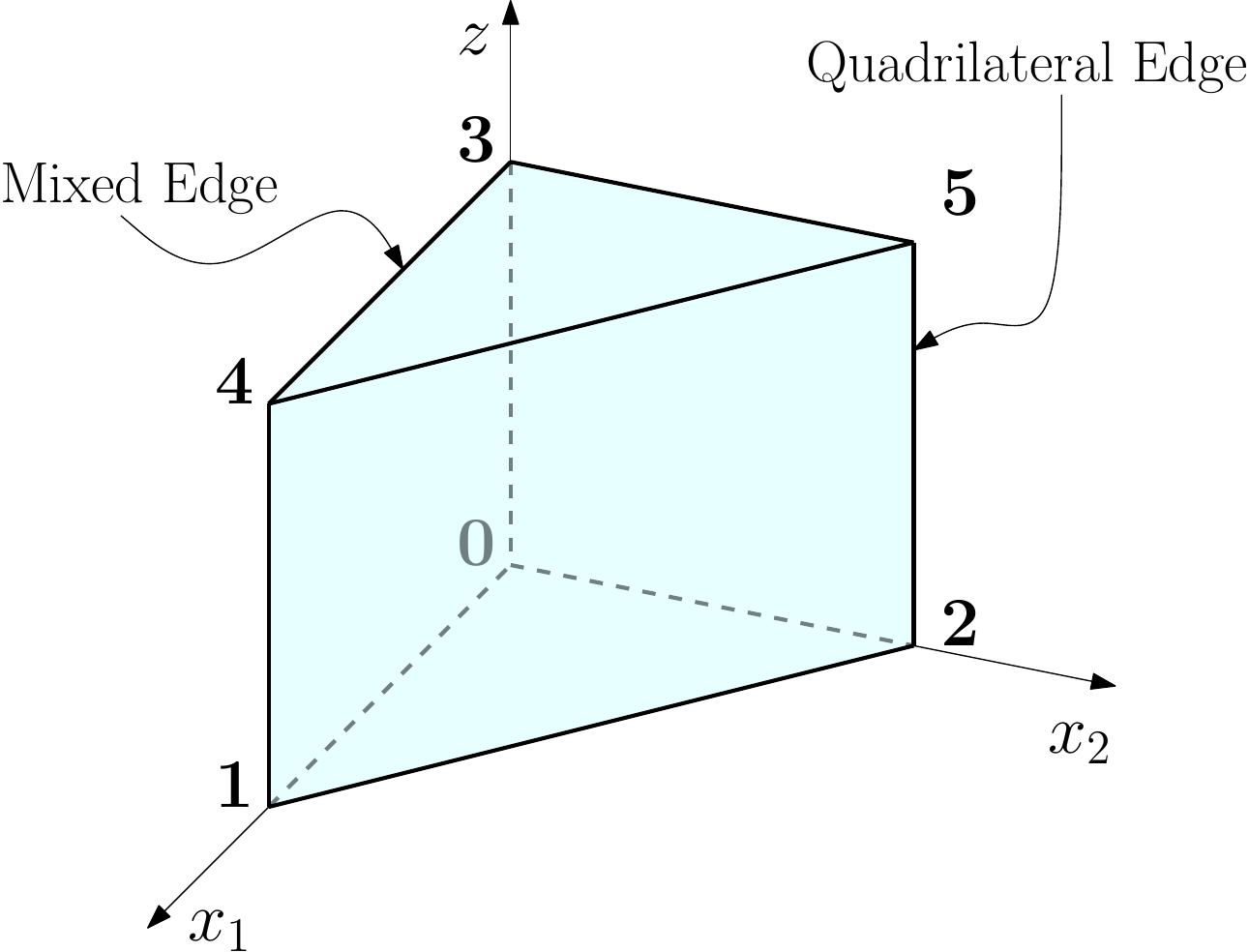}
	\end{center}
\end{minipage}} &\\
\multicolumn{3}{|c|}{Affine Coordinates} &\\
\multicolumn{3}{|c|}{
$	\begin{alignedat}{6}
		\nu_0&=1-x_1-x_2\,\qquad\qquad &&\phantom{\nabla}\nu_1&&=x_1\qquad\qquad &&\phantom{\nabla}\nu_2&&=x_2\\
		\nabla\nu_0&=\bigg(\begin{smallmatrix}-1\\[2pt]-1\\[2pt]0\end{smallmatrix}\bigg)\qquad\qquad 
			&&\nabla\nu_1&&=\bigg(\begin{smallmatrix}1\\[2pt]0\\[2pt]0\end{smallmatrix}\bigg)\qquad\qquad 
				&&\nabla\nu_2&&=\bigg(\begin{smallmatrix}0\\[2pt]1\\[2pt]0\end{smallmatrix}\bigg)
	\end{alignedat}$} &\\[30pt]
\multicolumn{3}{|c|}{
$	\begin{alignedat}{4}
		\mu_0&=1-z\,\qquad\qquad &&\phantom{\nabla}\mu_1&&=z\\
		\nabla\mu_0&=\bigg(\begin{smallmatrix}0\\[2pt]0\\[2pt]-1\end{smallmatrix}\bigg)\qquad\qquad 
			&&\nabla\mu_1&&=\bigg(\begin{smallmatrix}0\\[2pt]0\\[2pt]1\end{smallmatrix}\bigg)
	\end{alignedat}$} &\\
\multicolumn{3}{|c|}{} &\\[-8pt]
\multicolumn{3}{|c|}{\small\fbox{\parbox[c][25pt][c]{0.6\textwidth}{The order in the triplet $\vec{\nu}_{012}=(\nu_0,\nu_1,\nu_2)$ and all its subpairs is $p$.\\The order in the pair $\vec{\mu}_{01}=(\mu_0,\mu_1)$ is $q$.}}} &\\
{}& {} &{} &\\[-5pt]
\multicolumn{3}{r}{\footnotesize\textit{continued on next page}} &\\[-18.5pt]
\hline
\end{tabular}
\end{center}
\renewcommand{\arraystretch}{1}

\newpage

\renewcommand{\arraystretch}{1.35}
\begin{center}
\begin{tabular}
{|C{0.35cm} L{11.9cm}|C{2.9cm}|N}
\multicolumn{3}{l}{\footnotesize\textit{continued from previous page}} &\\[-2pt]
\hline
\multicolumn{3}{|c|}{\large \textbf{Prism}} &\\
\hline
\multicolumn{3}{|c|}{$H^1$} &\\
\hline
{}& {} &{} &\\[-23pt]
\multicolumn{2}{|c|}{Shape Functions}& Indices &\\ 
\hline
\multicolumn{2}{|l|}{Vertices} &
\multirow{2}{*}[-12pt]{\small$\begin{gathered}
	a\!=\!0,1,2\\
	c\!=\!0,1
\end{gathered}$} &\\
\multicolumn{2}{|l|}{
$	\begin{aligned}
		\phi^\mathrm{v}&=\nu_a\mu_c\\
		\nabla\phi^\mathrm{v}&=\mu_c\nabla\nu_a+\nu_a\nabla\mu_c
	\end{aligned}$} &{}&\\
\multicolumn{2}{|l|}{Mixed Edges} &
\multirow{2}{*}[-10pt]{\small$\begin{gathered}
	i\!=\!2,\ldots,p\\[-2pt]
	0\leq a<b\leq2\\[-2pt]
	c\!=\!0,1
\end{gathered}$} &\\
\hline
\end{tabular}
\end{center}
\renewcommand{\arraystretch}{1}

\newpage

\renewcommand{\arraystretch}{1.35}
\begin{center}
\begin{tabular}
{|C{0.35cm} L{11.9cm}|C{2.9cm}|N}
\multicolumn{3}{l}{\footnotesize\textit{continued from previous page}} &\\[-2pt]
\hline
\multicolumn{3}{|c|}{\large \textbf{Prism}} &\\
\hline
\multicolumn{3}{|c|}{$H(\mathrm{curl})$} &\\
\hline
{}& {} &{} &\\[-23pt]
\multicolumn{2}{|c|}{Shape Functions}& Indices &\\ 
\hline
\multicolumn{2}{|l|}{Mixed Edges} &
\multirow{2}{*}[-10pt]{\small$\begin{gathered}
	i\!=\!0,\ldots,p\!-\!1\\[-2pt]
	0\leq a<b\leq2\\[-2pt]
	c\!=\!0,1
\end{gathered}$} &\\
\multicolumn{2}{|l|}{
$	\begin{aligned}
		E_i^\mathrm{e}&=\mu_cE_i^\E(\vec{\nu}_{ab})\\
		\nabla\!\times\! E_i^\mathrm{e}&=\mu_c\nabla\!\times\! E_i^\E(\vec{\nu}_{ab})+\nabla\mu_c\!\times\! E_i^\E(\vec{\nu}_{ab})
	\end{aligned}$} &{}&\\
\multicolumn{2}{|l|}{Quadrilateral Edges} &
\multirow{2}{*}[-10pt]{\small$\begin{gathered}
	i\!=\!0,\ldots,q\!-\!1\\
	a=0,1,2
\end{gathered}$} &\\
\hline
\end{tabular}
\end{center}
\renewcommand{\arraystretch}{1}

\newpage

\renewcommand{\arraystretch}{1.35}
\begin{center}
\begin{tabular}
{|C{0.35cm} L{11.9cm}|C{2.9cm}|N}
\multicolumn{3}{l}{\footnotesize\textit{continued from previous page}} &\\[-2pt]
\hline
\multicolumn{3}{|c|}{\large \textbf{Prism}} &\\
\hline
\multicolumn{3}{|c|}{$H(\mathrm{curl})$} &\\
\hline
{}& {} &{} &\\[-23pt]
\multicolumn{2}{|c|}{Shape Functions}& Indices &\\ 
\hline
\multicolumn{2}{|l|}{Interior} & {}&\\[-5pt]
\multicolumn{2}{|l|}{\small $\,$Family I} &
\multirow{2}{*}[-12pt]{\small$\begin{gathered}
	i\geq0,\,j\geq1\\[-2pt]
	i\!+\!j\!=\!1,\ldots,p\!-\!1\\[-2pt]
	k\!=\!2,\ldots,q
\end{gathered}$} &\\
\multicolumn{2}{|l|}{
$	\begin{aligned}
		E_{ijk}^\mathrm{b}&=\phi_k^\E(\vec{\mu}_{01})E_{ij}^\Tri(\vec{\nu}_{012})\\
		\nabla\!\times\! E_{ijk}^\mathrm{b}&=\phi_k^\E(\vec{\mu}_{01})\nabla\!\times\! E_{ij}^\Tri(\vec{\nu}_{012})
			+\nabla\phi_k^\E(\vec{\mu}_{01})\!\times\! E_{ij}^\Tri(\vec{\nu}_{012})
	\end{aligned}$} &{}&\\
\hline
\end{tabular}
\end{center}
\renewcommand{\arraystretch}{1}

\newpage

\renewcommand{\arraystretch}{1.35}
\begin{center}
\begin{tabular}
{|C{0.35cm} L{11.9cm}|C{2.9cm}|N}
\multicolumn{3}{l}{\footnotesize\textit{continued from previous page}} &\\[-2pt]
\hline
\multicolumn{3}{|c|}{\large \textbf{Prism}} &\\
\hline
\multicolumn{3}{|c|}{$H(\mathrm{div})$} &\\
\hline
{}& {} &{} &\\[-23pt]
\multicolumn{2}{|c|}{Shape Functions}& Indices &\\ 
\hline
\multicolumn{2}{|l|}{Triangle Faces} & 
\multirow{2}{*}[-12pt]{\small$\begin{gathered}
	i\geq0,\,j\geq0\\[-2pt]
	i\!+\!j\!=\!0,\ldots,p\!-\!1\\[-2pt]
	c\!=\!0,1
\end{gathered}$} &\\
\multicolumn{2}{|l|}{
$	\begin{aligned}
		V_{ij}^\mathrm{f}&=\mu_cV_{ij}^\Tri(\vec{\nu}_{012})\\
		\nabla\!\cdot\! V_{ij}^\mathrm{f}&=\nabla\mu_c\!\cdot\! V_{ij}^\Tri(\vec{\nu}_{012})
	\end{aligned}$} &{}&\\
\multicolumn{2}{|l|}{Quadrilateral Faces} & 
\multirow{2}{*}[-12pt]{\small$\begin{gathered}
	i\!=\!0,\ldots,p\!-\!1\\[-2pt]
	j\!=\!0,\ldots,q\!-\!1\\[-2pt]
	0\leq a<b\leq2
\end{gathered}$} &\\
\hline
\end{tabular}
\end{center}
\renewcommand{\arraystretch}{1}

\newpage
\subsection{Pyramid}

\renewcommand{\arraystretch}{1.35}
\begin{center}
\begin{tabular}
{|C{0.35cm} L{11.9cm} C{2.9cm}|N}
\hline
\multicolumn{3}{|c|}{\large \textbf{Pyramid}} &\\
\hline
\multicolumn{3}{|c|}{Geometry} &\\
\hline
\multicolumn{3}{|c|}{} &\\[-10pt]
\multicolumn{3}{|c|}{\begin{minipage}{0.4\textwidth}
	\begin{center}
		\includegraphics[scale=0.5]{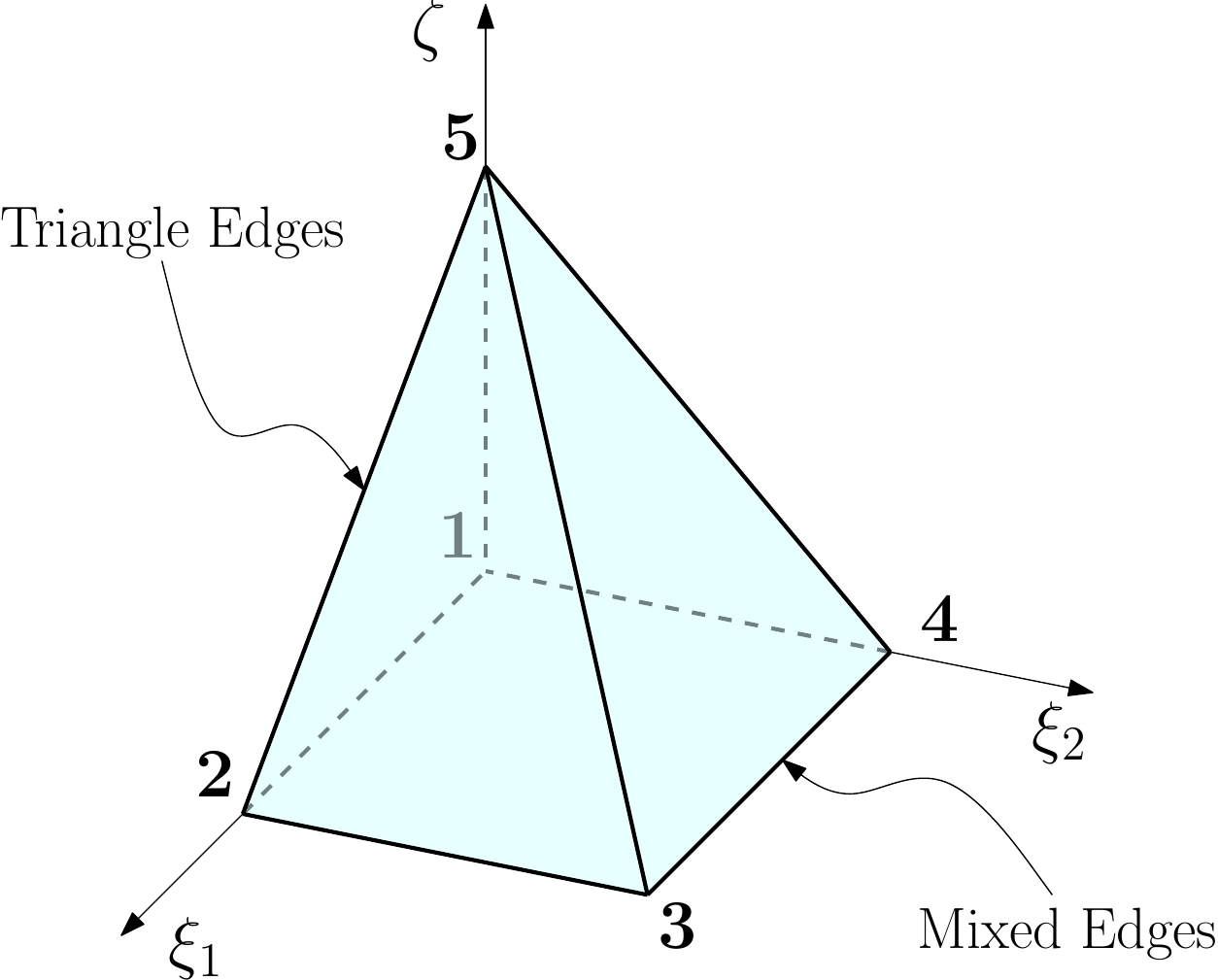}
	\end{center}
\end{minipage}
\begin{minipage}{0.55\textwidth}
	\begin{center}
		Affine Coordinates
		$	\begin{alignedat}{4}
		{}&{}&&{}&&{}\\[-7pt]
		\mu_0^{\,\zeta,\xi_1}&=1-\sxi\,\quad &&\phantom{\nabla}\mu_1^{\,\zeta,\xi_1}&&=\sxi\\
		\nabla\mu_0^{\,\zeta,\xi_1}&=
			\textstyle{\frac{1}{(1-\zeta)^2}}\bigg(\begin{smallmatrix}-(1-\zeta)\\[2pt]0\\[2pt]-\xi_1\end{smallmatrix}\bigg)\quad 
				&&\nabla\mu_1^{\,\zeta,\xi_1}&&=
					\textstyle{\frac{1}{(1-\zeta)^2}}\bigg(\begin{smallmatrix}(1-\zeta)\\[2pt]0\\[2pt]\xi_1\end{smallmatrix}\bigg)\\
		\mu_0^{\,\zeta,\xi_2}&=1-\sxii\,\quad &&\phantom{\nabla}\mu_1^{\,\zeta,\xi_2}&&=\sxii\\
		\nabla\mu_0^{\,\zeta,\xi_2}&=
			\textstyle{\frac{1}{(1-\zeta)^2}}\bigg(\begin{smallmatrix}0\\[2pt]-(1-\zeta)\\[2pt]-\xi_2\end{smallmatrix}\bigg)\quad 
				&&\nabla\mu_1^{\,\zeta,\xi_2}&&=
					\textstyle{\frac{1}{(1-\zeta)^2}}\bigg(\begin{smallmatrix}0\\[2pt](1-\zeta)\\[2pt]\xi_2\end{smallmatrix}\bigg)\\
		\mu_0^{\,\zeta}&=1-\zeta\,\quad &&\quad\phantom{\nabla}\mu_1^{\,\zeta}&&=\zeta\\
		\nabla\mu_0^{\,\zeta}&=\bigg(\begin{smallmatrix}0\\[2pt]0\\[2pt]-1\end{smallmatrix}\bigg)\quad 
			&&\quad\nabla\mu_1^{\,\zeta}&&=\bigg(\begin{smallmatrix}0\\[2pt]0\\[2pt]1\end{smallmatrix}\bigg)
	\end{alignedat}$
	\end{center}
\end{minipage}} &\\
\multicolumn{3}{|c|}{} &\\
\hline
\end{tabular}
\end{center}
\renewcommand{\arraystretch}{1}

\newpage

\renewcommand{\arraystretch}{1.35}
\begin{center}
\begin{tabular}
{|C{0.35cm} L{11.9cm}|C{2.9cm}|N}
\multicolumn{3}{l}{\footnotesize\textit{continued from previous page}} &\\[-2pt]
\hline
\multicolumn{3}{|c|}{\large \textbf{Pyramid}} &\\
\hline
\multicolumn{3}{|c|}{$H^1$} &\\
\hline
{}& {} &{} &\\[-23pt]
\multicolumn{2}{|c|}{Shape Functions}& Indices &\\ 
\hline
\multicolumn{2}{|l|}{Vertices} &
\multirow{2}{*}[-12pt]{\small$\begin{gathered}
	a\!=\!1,2,3,4,5
\end{gathered}$} &\\
\multicolumn{2}{|l|}{
$	\begin{aligned}
		\phi^\mathrm{v}&=\lambda_a\\
		\nabla\phi^\mathrm{v}&=\nabla\lambda_a
	\end{aligned}$} &{}&\\
\multicolumn{2}{|l|}{Mixed Edges} &
\multirow{2}{*}[-10pt]{\small$\begin{gathered}
	i\!=\!2,\ldots,p\\[-2pt]
	(a,b)\!=\!(1,2),(2,1)\\[-2pt]
	c\!=\!0,1
\end{gathered}$} &\\
\hline
\end{tabular}
\end{center}
\renewcommand{\arraystretch}{1}

\newpage

\renewcommand{\arraystretch}{1.35}
\begin{center}
\begin{tabular}
{|C{0.35cm} L{11.9cm}|C{2.9cm}|N}
\multicolumn{3}{l}{\footnotesize\textit{continued from previous page}} &\\[-2pt]
\hline
\multicolumn{3}{|c|}{\large \textbf{Pyramid}} &\\
\hline
\multicolumn{3}{|c|}{$H(\mathrm{curl})$} &\\
\hline
{}& {} &{} &\\[-23pt]
\multicolumn{2}{|c|}{Shape Functions}& Indices &\\ 
\hline
\multicolumn{2}{|l|}{Mixed Edges} &
\multirow{2}{*}[-10pt]{\small$\begin{gathered}
	i\!=\!0,\ldots,p\!-\!1\\[-2pt]
	(a,b)\!=\!(1,2),(2,1)\\[-2pt]
	c\!=\!0,1
\end{gathered}$} &\\
\multicolumn{2}{|l|}{
$	\begin{aligned}
		E_i^\mathrm{e}&=\mu_c^{\,\zeta,\xi_b}E_i^\E(\vec{\nu}_{01}^{\,\zeta,\xi_a})\\
		\nabla\!\times\! E_i^\mathrm{e}&=\mu_c^{\,\zeta,\xi_b}\nabla\!\times\! E_i^\E(\vec{\nu}_{01}^{\,\zeta,\xi_a})
        +\nabla\mu_c^{\,\zeta,\xi_b}\!\times\! E_i^\E(\vec{\nu}_{01}^{\,\zeta,\xi_a})
	\end{aligned}$} &{}&\\
\multicolumn{2}{|l|}{Triangle Edges} &
\multirow{2}{*}[-10pt]{\small$\begin{gathered}
	i\!=\!0,\ldots,p\!-\!1\\
	a=1,2,3,4
\end{gathered}$} &\\
\hline
\end{tabular}
\end{center}
\renewcommand{\arraystretch}{1}

\newpage

\renewcommand{\arraystretch}{1.35}
\begin{center}
\begin{tabular}
{|C{0.35cm} L{11.9cm}|C{2.9cm}|N}
\multicolumn{3}{l}{\footnotesize\textit{continued from previous page}} &\\[-2pt]
\hline
\multicolumn{3}{|c|}{\large \textbf{Pyramid}} &\\
\hline
\multicolumn{3}{|c|}{$H(\mathrm{curl})$} &\\
\hline
{}& {} &{} &\\[-23pt]
\multicolumn{2}{|c|}{Shape Functions}& Indices &\\ 
\hline
\multicolumn{2}{|l|}{Interior} & {}&\\[-5pt]
\multicolumn{2}{|l|}{\small $\,$Family I} &
\multirow{2}{*}[-12pt]{\small$\begin{gathered}
	i\!=\!2,\ldots,p\\[-2pt]
	j\!=\!2,\ldots,p\\[-2pt]
	k\!=\!2,\ldots,p
\end{gathered}$} &\\
\multicolumn{2}{|l|}{
$	\begin{aligned}
		E_{ijk}^\mathrm{b}&=\phi_k^\E(\vec{\mu}_{01}^{\,\zeta})
			\nabla\phi_{ij}^\square(\vec{\mu}_{01}^{\,\zeta,\xi_1},\vec{\mu}_{01}^{\,\zeta,\xi_2})
        		+\phi_{ij}^\square(\vec{\mu}_{01}^{\,\zeta,\xi_1},\vec{\mu}_{01}^{\,\zeta,\xi_2})\nabla\phi_k^\E(\vec{\mu}_{01}^{\,\zeta})\\
		\nabla\!\times\! E_{ijk}^\mathrm{b}&=0
	\end{aligned}$} &{}&\\
\hline
\end{tabular}
\end{center}
\renewcommand{\arraystretch}{1}

\newpage

\renewcommand{\arraystretch}{1.35}
\begin{center}
\begin{tabular}
{|C{0.35cm} L{11.9cm}|C{2.9cm}|N}
\multicolumn{3}{l}{\footnotesize\textit{continued from previous page}} &\\[-2pt]
\hline
\multicolumn{3}{|c|}{\large \textbf{Pyramid}} &\\
\hline
\multicolumn{3}{|c|}{$H(\mathrm{div})$} &\\
\hline
{}& {} &{} &\\[-23pt]
\multicolumn{2}{|c|}{Shape Functions}& Indices &\\ 
\hline
\multicolumn{2}{|l|}{Quadrilateral Face} & 
\multirow{2}{*}[-12pt]{\small$\begin{gathered}
	i\!=\!0,\ldots,p\!-\!1\\
	j\!=\!0,\ldots,q\!-\!1
\end{gathered}$} &\\
\multicolumn{2}{|l|}{
$	\begin{aligned}
		V_{ij}^\mathrm{f}&=(\mu_0^{\,\zeta})^3V_{ij}^\square(\vec{\mu}_{01}^{\,\zeta,\xi_1},\vec{\mu}_{01}^{\,\zeta,\xi_2})\\
		\nabla\!\cdot\! V_{ij}^\mathrm{f}&=3(\mu_0^{\,\zeta})^2\nabla\mu_0^{\,\zeta}\!\cdot\!
    		V_{ij}^\square(\vec{\mu}_{01}^{\,\zeta,\xi_1},\vec{\mu}_{01}^{\,\zeta,\xi_2})
	\end{aligned}$} &{}&\\
\multicolumn{2}{|l|}{Triangle Faces$^\star$} & 
\multirow{2}{*}[-12pt]{\small$\begin{gathered}
	i\geq0,\,j\geq0\\[-2pt]
	i\!+\!j\!=\!0,\ldots,p\!-\!1\\[-2pt]
	(a,b)\!=\!(1,2),(2,1)\\[-2pt]
	c\!=\!0,1
\end{gathered}$} &\\
\hline\end{tabular}
\end{center}
\renewcommand{\arraystretch}{1}

\newpage

\renewcommand{\arraystretch}{1.35}
\begin{center}
\begin{tabular}
{|C{0.35cm} L{11.9cm}|C{2.9cm}|N}
\multicolumn{3}{l}{\footnotesize\textit{continued from previous page}} &\\[-2pt]
\hline
\multicolumn{3}{|c|}{\large \textbf{Pyramid}} &\\
\hline
\multicolumn{3}{|c|}{$H(\mathrm{div})$} &\\
\hline
{}& {} &{} &\\[-23pt]
\multicolumn{2}{|c|}{Shape Functions}& Indices &\\ 
\hline
\multicolumn{2}{|l|}{Interior} & {}&\\[-5pt]
\multicolumn{2}{|l|}{\small $\,$Family IV} &
\multirow{2}{*}[-12pt]{\small$\begin{gathered}
	i\!=\!0,\ldots,p\!-\!1\\[-2pt]
	j\!=\!0,\ldots,p\!-\!1\\[-2pt]
	k\!=\!2,\ldots,p
\end{gathered}$} &\\
\multicolumn{2}{|l|}{
$	\begin{aligned}
		V_{ijk}^\mathrm{b}&=(\mu_0^{\,\zeta})^2\phi_k^\E(\vec{\mu}_{01}^{\,\zeta})
			V_{ij}^\square(\vec{\mu}_{01}^{\,\zeta,\xi_1},\vec{\mu}_{01}^{\,\zeta,\xi_2})\\
		\nabla\!\cdot\! V_{ijk}^\mathrm{b}&=\Big((\mu_0^{\,\zeta})^2\nabla\phi_k^\E(\vec{\mu}_{01}^{\,\zeta})
					+2\mu_0^{\,\zeta}\phi_k^\E(\vec{\mu}_{01}^{\,\zeta})\nabla\mu_0^{\,\zeta}\Big)
						\!\cdot\! V_{ij}^\square(\vec{\mu}_{01}^{\,\zeta,\xi_1},\vec{\mu}_{01}^{\,\zeta,\xi_2})
	\end{aligned}$} &{}&\\
\hline
\end{tabular}
\end{center}
\renewcommand{\arraystretch}{1}

\newpage

\subsection{Orientations}

\renewcommand{\arraystretch}{1.35}
\begin{center}
\begin{tabular}
{|C{0.35cm} L{11.9cm} C{2.9cm}|N}
\hline
\multicolumn{3}{|c|}{\large \textbf{Orientations}} &\\
\hline
\multicolumn{3}{|c|}{Local Orientations} &\\
\hline
{}& {} &{} &\\[-15pt]
\multicolumn{3}{|c|}{} &\\[-5pt]
\multicolumn{3}{|r|}{\begin{minipage}{0.9\textwidth}
	\begin{center}
		\includegraphics[scale=0.5]{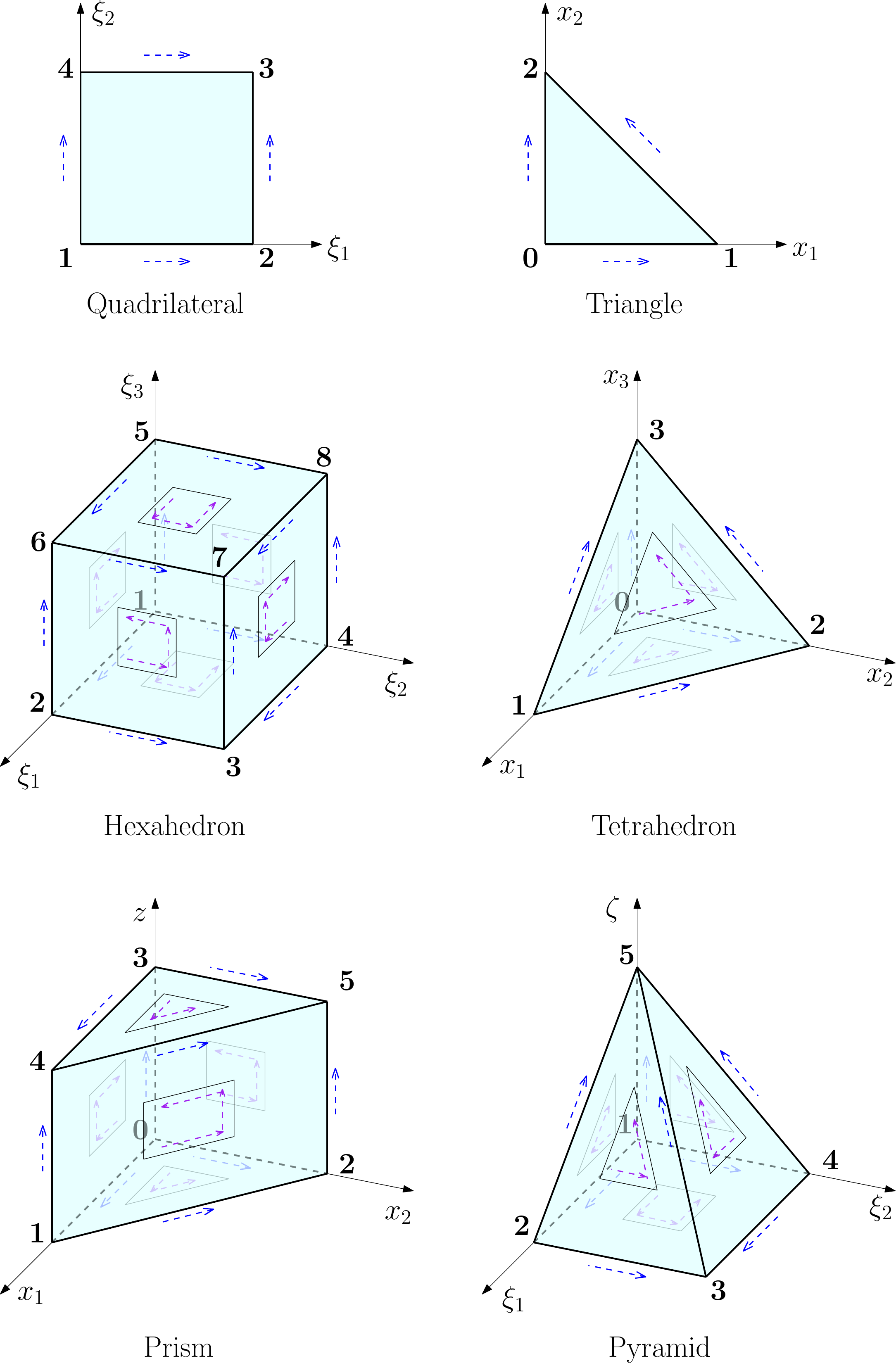}
	\end{center}
\end{minipage}} &\\[-10pt]
\multicolumn{3}{|c|}{} &\\[-17pt]
{}& {} &{} &\\[-5pt]
\multicolumn{3}{r}{\footnotesize\textit{continued on next page}} &\\[-18.5pt]
\hline
\end{tabular}
\end{center}
\renewcommand{\arraystretch}{1}

\newpage

\renewcommand{\arraystretch}{1.35}
\begin{center}
\begin{tabular}
{|C{0.35cm} L{11.9cm} C{2.9cm}|N}
\multicolumn{3}{l}{\footnotesize\textit{continued from previous page}} &\\[-2pt]
\hline
\multicolumn{3}{|c|}{\large \textbf{Orientations}} &\\
\hline
\multicolumn{3}{|c|}{Local-to-Global Permutation Functions} &\\
\hline
{}& {} &{} &\\[-15pt]
\multicolumn{3}{|l|}{Edge} &\\
\multicolumn{3}{|l|}{
$	\begin{gathered}
		\quad\quad\quad\!\!\!\sigma_\oo^\E(s_0,s_1)=\begin{cases}\sigma_0^\E(s_0,s_1)=(s_0,s_1)&\quad\text{if  }\,\oo=0\\
			\sigma_1^\E(s_0,s_1)=(s_1,s_0)&\quad\text{if  }\,\oo=1\end{cases}
	\end{gathered}$} &\\
\multicolumn{3}{|l|}{Quadrilateral Face} &\\
\multicolumn{3}{|l|}{
$	\begin{gathered}
			\sigma_\oo^\square(s_0,s_1,t_0,t_1)=\begin{cases}
		\sigma_0^\square(s_0,s_1,t_0,t_1)=(s_0,s_1,t_0,t_1)&\quad\text{if  }\,\oo=0\\
		\sigma_1^\square(s_0,s_1,t_0,t_1)=(t_0,t_1,s_1,s_0)&\quad\text{if  }\,\oo=1\\
		\sigma_2^\square(s_0,s_1,t_0,t_1)=(s_1,s_0,t_1,t_0)&\quad\text{if  }\,\oo=2\\
		\sigma_3^\square(s_0,s_1,t_0,t_1)=(t_1,t_0,s_0,s_1)&\quad\text{if  }\,\oo=3\\
		\sigma_4^\square(s_0,s_1,t_0,t_1)=(t_0,t_1,s_0,s_1)&\quad\text{if  }\,\oo=4\\
		\sigma_5^\square(s_0,s_1,t_0,t_1)=(s_1,s_0,t_0,t_1)&\quad\text{if  }\,\oo=5\\
		\sigma_6^\square(s_0,s_1,t_0,t_1)=(t_1,t_0,s_1,s_0)&\quad\text{if  }\,\oo=6\\
		\sigma_7^\square(s_0,s_1,t_0,t_1)=(s_0,s_1,t_1,t_0)&\quad\text{if  }\,\oo=7\end{cases}
	\end{gathered}$} &\\
\hline
\end{tabular}
\end{center}
\renewcommand{\arraystretch}{1}